\documentclass[10pt]{amsart}
\usepackage{pgf,tikz}
\usepackage{braids}
\usetikzlibrary{decorations.pathreplacing}
\usepackage{caption}
\usepackage{amssymb}
\usepackage{stmaryrd}
\usepackage{mathrsfs}
\usepackage[all]{xy}

\makeatletter
  \def\@wrindex#1{%
    \protected@write\@indexfile{}%
      {\string\indexentry{#1}{ \S\thesubsection (p.\thepage)}}
    \endgroup
  \@esphack}

\makeatother

\makeindex

\newcommand{\AAA}{\varpi}
\newcommand{\BB}{\rho}
\newcommand{\DR}{\mathrm{DR}}
\newcommand{\B}{\mathrm{B}}
\newcommand{\KZ}{\mathrm{KZ}}
\newcommand{\topo}{\mathrm{top}}
\newcommand{\fil}{\mathrm{fil}}
\newcommand{\gr}{\mathrm{gr}}

\usepackage{rotating}
\usepackage{amscd}
\usepackage{epsfig}

\author{Benjamin Enriquez}
\author{Hidekazu Furusho}

\address{Institut de Recherche Math\'{e}matique Avanc\'{e}e, UMR 7501, 
Universit\'{e} de Strasbourg et CNRS, 7
rue Ren\'{e} Descartes, 67000 Strasbourg, France}
\email{enriquez@math.unistra.fr}

\address{Graduate School of Mathematics, Nagoya University, 
Furo-cho, Chikusa-ku, Nagoya, 464-8602, Japan}
\email{furusho@math.nagoya-u.ac.jp}

\date{March 7, 2021}
\newtheorem{thm}{Theorem}[section]
\newtheorem{lem}[thm]{Lemma}
\newtheorem{lemdef}[thm]{Lemma-Definition}
\newtheorem{cor}[thm]{Corollary}
\newtheorem{prop}[thm]{Proposition}

{\theoremstyle{definition} \newtheorem{rem}[thm]{Remark}}
{\theoremstyle{definition} \newtheorem{defn}[thm]{Definition}}

{\theoremstyle{remark} }

\numberwithin{equation}{subsection}
\numberwithin{figure}{section}

\begin{document}

\baselineskip 16pt 

\title[The Betti side of the double shuffle theory. I. The harmonic coproducts]{The Betti side of the 
double shuffle theory. \\  I. The harmonic coproducts}

\begin{abstract}
This paper is the first in a series which aims at: (a) giving a proof that the
associator relations between multizeta values imply the double shuffle and regularization
(DSR) ones, alternative to that of the second-named author's 2010 paper; 
(b) enhancing Racinet's construction of 
a torsor structure over the $\mathbb Q$-scheme of DSR relations to an explicit bitorsor structure.

In this paper, we revisit Racinet's original DSR formalism, whose main character is an
algebra coproduct, called the harmonic coproduct, and we introduce a variant which is a
module coproduct; we explain the `de Rham' nature of this formalism and construct a `Betti'
counterpart of it; we show how both formalisms can be interpreted in terms of geometry,
following the ideas of Deligne and Terasoma's unfinished 2005 preprint; 
we use Bar-Natan's interpretation of associators as functors from
the category of parenthesized braids to that of chord diagrams to show that any associator
relates the Betti and de Rham geometric objects, both in the `algebraic' and in the `module'
setups; we derive that any associator relates the Betti and de Rham algebra coproducts, as
well as their module counterparts. These results will be used in the next parts of the series. 
\end{abstract}
\bibliographystyle{amsalpha+}
\maketitle

{\footnotesize \tableofcontents}

\section*{Introduction}

\subsection{The context}

\subsubsection{The bitorsor of mixed Tate motives over $\mathbb Z$}

\cite{DG} contains the construction of a $\mathbb Q$-linear neutral Tannakian category $\mathcal T:=
\mathrm{MT}(\mathbb Z)_{\mathbb Q}$ of mixed Tate motives over $\mathbb Z$, equipped with `Betti' and `de Rham' fiber 
functors $\omega_\B,\omega_\DR:\mathcal T\to\mathrm{Vec}_{\mathbb Q}$ to the category of finite dimensional 
$\mathbb Q$-vector spaces. It gives rise to a $\mathbb Q$-scheme ${\mathsf{Isom}}^{\otimes}_{\mathcal T}
(\omega_\B,\omega_\DR) : \mathbf k\mapsto{\mathrm{Isom}}^{\otimes}_{\mathcal T}(\omega_\B\otimes\mathbf k,
\omega_\DR\otimes\mathbf k)$, and to `motivic Galois' $\mathbb Q$-group schemes ${\mathsf{Aut}}^{\otimes}_{\mathcal T}
(\omega_?):\mathbf k\mapsto{\mathrm{Aut}}^{\otimes}_{\mathcal T}(\omega_X\otimes\mathbf k)$, for 
$?\in\{\B,\DR\}$ (denoted $G^?$ in \cite{DG}), $\mathbf k$ running over $\mathbb Q$-algebras. Then 
${\mathsf{Isom}}^{\otimes}_{\mathcal T}(\omega_\B,\omega_\DR)(\mathbf k)$ is equipped with commuting free and transitive 
actions of ${\mathsf{Aut}}^{\otimes}_{\mathcal T}(\omega_\DR)(\mathbf k)$ from the left, and of 
${\mathsf{Aut}}^{\otimes}_{\mathcal T}(\omega_\B)(\mathbf k)$ from the right.

The triple formed by this set and these two groups is a {\it bitorsor,}
meaning that the set is equipped with commuting left and right, free and transitive actions of
the groups, and the assignment of the triple to $\mathbf k$ is then a {\it bitorsor $\mathbb Q$-scheme.}

The category $\mathcal T$ is closely connected with the family of real 
numbers known as multiple zeta values (MZVs), which are given by 
\begin{equation}\label{def:MZV:3jan2020}
\forall s\geq 1,k_1\geq 1,\ldots,k_{s-1}\geq1,k_s\geq 2, \quad \zeta(k_1,\ldots,k_s)=\sum_{0<n_1<\cdots<n_s}
{1\over n_1^{k_1}\cdots n_s^{k_s}};
\index{zeta(k_1,\ldots,k_s)@$\zeta(k_1,\ldots,k_s)$} 
\end{equation}
it follows from their integral expressions that these numbers, as well as\footnote{We set $\mathrm{i}:=\sqrt{-1}\index{i@$\mathrm{i}$}$.}  
$2\pi\mathrm{i}$, are periods of mixed Tate motives. 

\subsubsection{The bitorsor of associators}

Combining \cite{Dr} and \cite{LM}, one can exhibit a family of algebraic relations satisfied by the MZVs and $2\pi\mathrm{i}$, the 
`associator relations'. The corresponding $\mathbb Q$-scheme is the scheme\footnote{An affine $\mathbb Q$-scheme (resp., 
$\mathbb Q$-group scheme) $\mathsf X$ can be identified with a representable functor from the category of $\mathbb Q$-algebras 
to that of sets (resp. groups), denoted $\mathbf k\mapsto \mathsf X(\mathbf k)$, for $\mathbf k$ running over 
$\mathbb Q$-algebras.
} of associators 
$\mathsf M$; 
it is equipped with commuting free and transitive actions of affine $\mathbb Q$-group schemes 
$\mathsf{GT}$ and $\mathsf{GRT}$ (the Grothendieck-Teichm\"uller group 
scheme and its graded version) from the left and from the right, so that $\mathsf M$ is a bitorsor $\mathbb Q$-scheme. 

In \cite{An}, it is explained that there is a bitorsor $\mathbb Q$-scheme morphism 
${\mathsf{Isom}}_{\mathcal T}^{\otimes}(\omega_\B,\omega_\DR)\to \mathsf M$, i.e., a morphism
between these $\mathbb Q$-schemes, together with compatible $\mathbb Q$-group scheme morphisms 
${\mathsf{Aut}}_{\mathcal T}^{\otimes}(\omega_\B)\to\mathsf{GT}$ and  
${\mathsf{Aut}}_{\mathcal T}^{\otimes}(\omega_\DR)\to \mathsf{GRT}$. 

\subsubsection{The double shuffle torsor}

In \cite{IKZ,Rac} (see also \cite{G}), another family of algebraic relations satisfied by MZVs and $2\pi\mathrm{i}$ was exhibited, the 
`double shuffle and regularization' (DSR) relations. The corresponding $\mathbb Q$-scheme was constructed in \cite{Rac}: it is a 
functor\footnote{$\mathsf{DMR}$ stands for the French `double m\'elange et r\'egularisation'.} 
$\mathsf{DMR}^{\DR,\B}:=(\mathbf k\mapsto \sqcup_{\mu\in\mathbf k^\times}\{\mu\}\times 
\mathsf{DMR}_\mu(\mathbf k))$, which is equipped with a free and transitive left action of a $\mathbb Q$-group scheme
$\mathsf{DMR}^\DR:=(\mathbf k\mapsto \mathbf k^\times\ltimes(\mathsf{DMR}_0(\mathbf k),\circledast))$, therefore 
with a torsor $\mathbb Q$-scheme structure. An interpretation of $\mathsf{DMR}^\DR$ in terms of a stabilizer group was
obtained in \cite{EFu}. 

\subsubsection{Relations between the associator and double shuffle schemes}

The unfinished preprint \cite{DT}, based on the ideas of Deligne's letter to Racinet (April 21, 2001), contains an attempt 
at constructing a $\mathbb Q$-scheme inclusion $\mathsf M\hookrightarrow\mathsf{DMR}^{\DR,\B}$, using constructions involving 
the categories of perverse sheaves over the moduli space $\mathfrak M_{0,4}$ and $\mathfrak M_{0,5}$. 

In \cite{F}, a $\mathbb Q$-scheme inclusion $\mathsf M\hookrightarrow\mathsf{DMR}^{\DR,\B}$ was constructed; a compatible 
$\mathbb Q$-group scheme inclusion $\mathsf{GRT}^{\mathrm{op}}\hookrightarrow \mathsf{DMR}^\DR$ can 
be constructed, so that 
this $\mathbb Q$-scheme inclusion can be upgraded to a $\mathbb Q$-torsor inclusion. This work is based on the construction, 
based on bar-complex techniques and the combinatorics of multiple polylogarithms, of explicit linear forms on the algebra 
$U(\mathfrak p_5)$, in which the associator relations take place, and on study of the interaction of these linear forms
with the associator relations. 

\subsection{The motivation and objectives of the series of papers} 

The main motivation of this series of papers are (1) to gain a better understanding of the inclusion 
$\mathsf M\subset\mathsf{DMR}^{\DR,\B}$, in particular whether it is an equality (2) making explicit the bitorsor structure of 
$\mathsf{DMR}^{\DR,\B}$. In order to reach (1), it seems useful to obtain a conceptual interpretation of this inclusion. 

The objectives of the series of papers are therefore: (a) to extend the interpretation of the group $\mathsf{DMR}^\DR$ in 
terms of a stabilizer (\cite{EFu}) to the torsor $\mathsf{DMR}^{\DR,\B}$, for which the ideas of \cite{DT} will prove to be useful; 
(b) to give a proof of 
the inclusion $\mathsf M\hookrightarrow\mathsf{DMR}^{\DR,\B}$ based on this interpretation and on the interpretation of 
$\mathsf M$ obtained in \cite{BN}; (c) to enhance the torsor structure of $\mathsf{DMR}^{\DR,\B}$ into a bitorsor structure.

The material of the series is distributed as follows: the present part I is preparational; part II (\cite{EF2}) will be devoted to objectives (a) and (b), and part III (\cite{EF3}) to objective (c). 

\subsection{The contents of the present paper (part I of the series)}

After treating categorical and topological preliminaries (Part 1), we revisit the DSR formalism and construct its Betti counterpart 
(Part 2); interpret both formalisms in terms of braids and moduli spaces (Part 3); and derive the compatibility of their main 
constituents with associators (Part 4). The corresponding compatibility statements are Theorem 10.9 (compatibility of harmonic 
algebra coproducts with associators) and Theorem 11.13 (compatibility of harmonic module coproducts with associators). These 
statements will play a key role in the next papers of the series and should be viewed as the main results of the present paper. 
The necessary material for understanding their formulation is to be found in: \S\S\ref{sect:assoc:20032018} and 
\ref{subsect:Gamma:fun:20032018} (associators and their $\Gamma$-functions), \S\S\ref{sect:1:1:1:crm} and 
\ref{sect:completions:DR} (definitions of the algebra $\hat{\mathcal W}^{\DR}$ and of the module 
$\hat{\mathcal M}^{\DR}$), \S\S\ref{sect:tcDaD} and \ref{sect:completions:DR} (definitions of the de Rham harmonic algebra 
coproduct $\hat\Delta^{\mathcal W,\DR}$ and of its module version $\hat\Delta^{\mathcal M,\DR}$), \S\S\ref{sect:2:1:28oct} 
and \ref{sect:compl:28oct} (definitions of the algebra $\hat{\mathcal W}^{\B}$ and of the module $\hat{\mathcal M}^{\B}$), 
\S\S\ref{sect:2:3:28oct} and \ref{sect:compl:28oct} (definitions of the Betti harmonic coproduct $\hat\Delta^{\mathcal W,\B}$ 
and of its module version $\hat\Delta^{\mathcal M,\B}$), and \S\ref{sect:3:3:28oct} (definitions of the algebra isomorphism 
$\mathrm{comp}^{\mathcal W,(1)}_{(\mu,\Phi)}$ and of the module isomorphism 
$\mathrm{comp}^{\mathcal M,(10)}_{(\mu,\Phi)}$).

\subsubsection{Conventions}\label{conventions:0703}
 
Throughout the paper, $\mathbf k$\index{k@$\mathbf k$} is a commutative and associative $\mathbb Q$-algebra.
By a $\mathbf k$-algebra we mean a $\mathbf k$-module, equipped with a $\mathbf k$-bilinear associative product; 
a  $\mathbf k$-algebra is called unital if it contains a (necessarily unique) unity element.  
 We denote by $A^\times$ the group of invertible elements of an associative unital $\mathbf k$-algebra $A$. For 
$u\in A^\times$, we denote by $\mathrm{Ad}(u)$ the inner automorphism of $A$ induced by $u$, so 
$\mathrm{Ad}(u)(a)=uau^{-1}$ for $a\in A$. If $f:A\to B$ is a morphism of $\mathbf k$-algebras, if $M,N$ are respectively an 
$A$-module and a $B$-module, then a $\mathbf k$-module morphism $g:M\to N$ is said to be {\it compatible with $f$} 
iff $g(am)=f(a)g(m)$ for any $a\in A$, $m\in M$. 

\subsubsection{Acknowledgements}

The collaboration of both authors has been supported by grants JSPS KAKENHI JP15KK0159 and JP18H01110. 

\part{Categorical preliminaries}\label{part:categorical}

In this section, we introduce various symmetric monoidal categories of $\mathbf k$-modules and functors relating them. 
This material will be used through the rest of the paper. 

\begin{defn}\label{def:cat}
(1) $\mathbf k\text{-mod}_{\gr}$\index{k-mod_gr@$\mathbf k\text{-mod}_{\gr}$} is the category where objects are $\mathbb Z$-graded $\mathbf k$-modules and where morphisms are $\mathbf k$-module morphisms of degree $0$; 

(2) $\mathbf k\text{-mod}_{\fil}$\index{k-mod_fil@$\mathbf k\text{-mod}_{\fil}$} is the category where objects are pairs $(M,i\mapsto F^iM)$, where $M$ is a $\mathbf k$-module
and $i\mapsto F^iM$ is a map $\mathbb Z\to\{\mathbf k$-submodules of M$\}$, which is decreasing (i.e. $F^iM\supset F^{i+1}M$ for 
$i\in\mathbb Z$), and where the set of morphisms from $(M,i\mapsto F^iM)$ to $(N,i\mapsto F^iN)$ is the set of $\mathbf k$-module morphisms 
$f:M\to N$, which are compatible with the filtrations on both sides. We denote by $F^if:F^iM\to F^iN$ the induced $\mathbf k$-module map
corresponding to $i\in\mathbb Z$; 

(3) $\mathbf k\text{-mod}_{\gr,+}$\index{k-mod_gr+@$\mathbf k\text{-mod}_{\gr,+}$} is the full subcategory of $\mathbf k\text{-mod}_{\gr}$ of bounded below $\mathbb Z$-graded $\mathbf k$-modules, i.e. of modules $M=\oplus_{i\in\mathbb Z}M_i$,  such that there exists $i(M)\in\mathbb Z$, such that $M_i=0$ for $i<i(M)$; 

(4) $\mathbf k\text{-mod}_{\fil,+}$\index{k-mod_fil+@$\mathbf k\text{-mod}_{\fil,+}$} 
is the full subcategory of $\mathbf k\text{-mod}_{\fil}$ of pairs $(M,i\mapsto F^iM)$, such that 
there exists $i(M)\in\mathbb Z$, such that $F^{i(M)}M=M$;

(5) $\mathbf k\text{-mod}_{\topo}$\index{k-mod_top@$\mathbf k\text{-mod}_{\topo}$} is the full subcategory of $\mathbf k\text{-mod}_{\fil,+}$ of pairs $(M,i\mapsto F^iM)$, such that 
the map $M\to \lim_{\leftarrow i}M/F^iM$ is a $\mathbf k$-module isomorphism, i.e. $M$ is complete 
for the topology defined by $i\mapsto F^iM$. 
\end{defn}

\subsubsection*{Functors}\label{sect:functors}

There are tautological functors to each category from each of its full subcategories. We will denote some of these functors
as follows $\mathrm{taut}_{\topo}^{\fil,+}:\mathbf k\text{-mod}_{\topo}\to\mathbf k\text{-mod}_{\fil,+}$\index{tautfiltop@$\mathrm{taut}_{\topo}^{\fil,+}$} and 
$\mathrm{taut}_{\gr,+}^{\fil,+}:\mathbf k\text{-mod}_{\gr,+}\to\mathbf k\text{-mod}_{\fil,+}$\index{tautgrfil@$\mathrm{taut}_{\gr,+}^{\fil,+}$}. 

There is also an `associated graded' functor $\gr:\mathbf k\text{-mod}_{\fil}\to\mathbf k\text{-mod}_{\gr}$\index{gr@$\gr$}, which takes 
$(M,i\mapsto F^iM)$ to $\mathrm{gr}(M):=\oplus_{i\in\mathbb Z}\mathrm{gr}_i(M)$, where $\mathrm{gr}_i(M):=F^iM/F^{i+1}M$ for 
$i\in\mathbb Z$. It induces a functor $\mathrm{gr}:\mathbf k\text{-mod}_{\fil,+}\to\mathbf k\text{-mod}_{\gr,+}$. 

For each $\alpha\in\mathbb Z$, there are functors $F^\alpha:\mathbf k\text{-mod}_{\fil}\to\mathbf k\text{-mod}_{\fil,+}$ and 
$(-)_{\geq\alpha}:\mathbf k\text{-mod}_{\gr}\to\mathbf k\text{-mod}_{\gr,+}$, defined by $F^\alpha(M,i\mapsto F^iM):=(F^\alpha M,i\mapsto F^iM\cap F^\alpha M)$, and $M_{\geq\alpha}:=\oplus_{i\geq\alpha}M_i$ for $M=\oplus_{i\in\mathbb Z}M_i$. 

Finally, there is a `completion' functor $(-)^\wedge:\mathbf k\text{-mod}_{\fil,+}\to\mathbf k\text{-mod}_{\topo}$, given by 
$(M,i\mapsto F^iM)\mapsto(\underset{j}{\varprojlim}
M/F^jM,i\mapsto\underset{j}{\varprojlim}F^iM/F^jM)$; it gives rise to a functor 
$(-)^\wedge\circ\mathrm{taut}_{\gr,+}^{\fil,+}:\mathbf k\text{-mod}_{\gr,+}\to\mathbf k\text{-mod}_{\topo}$. 

A diagram of all these functors and categories is as follows
$$
\xymatrix{
 \mathbf k\text{-mod}_{\gr}\ar^{(-)_{\geq0}}[r]& \mathbf k\text{-mod}_{\gr,+}\ar@<2pt>^{\mathrm{taut}_{\gr,+}^{\fil,+}}[d]& \\ 
\mathbf k\text{-mod}_{\fil} \ar_{F^0}[r]\ar^{\mathrm{\gr}}[u]&\mathbf k\text{-mod}_{\fil,+} \ar@<2pt>^{\mathrm{\gr}}[u] 
\ar@<2pt>^{(-)^\wedge}[r]& \mathbf k\text{-mod}_{\topo}\ar@<2pt>^{\mathrm{taut}_{\topo}^{\fil,+}}[l]}
$$

\begin{lem}\label{lemma:injectivity:completion} 
1) The functor $(-)^\wedge\circ\mathrm{taut}_{\gr,+}^{\fil,+}(f) : \mathbf k\text{-mod}_{\gr,+}\to\mathbf k\text{-mod}_{\topo}$ is exact. 

2) If $f$ is a morphism in $\mathbf k\text{-mod}_{\fil,+}$ such that the morphism ${\gr}(f)$ in $\mathbf k\text{-mod}_{\gr,+}$ is 
injective, then the morphism $\hat f$ in $\mathbf k\text{-mod}_{\topo}$ is injective. 
\end{lem}

\proof 1) Let $0\to A\to B\to C\to 0$ be an exact sequence in $\mathbf k\text{-mod}_{\gr,+}$. Then for each $k\in\mathbb Z$, 
its degree $k$ component $0\to A_k\to B_k\to C_k\to 0$  is exact. The image of $0\to A\to B\to C\to 0$ by 
$(-)^\wedge\circ\mathrm{taut}_{\gr,+}^{\fil,+}$ is the sequence $0\to (\oplus_{k<0}A_k)\oplus(\prod_{k\geq0}A_k)\to
(\oplus_{k<0}B_k)\oplus(\prod_{k\geq0}B_k)\to(\oplus_{k<0}C_k)\oplus (\prod_{k\geq0}C_k)\to 0$, which is then exact. 

2) Let $f:M\to N$ be a morphism in $\mathbf k\text{-mod}_{\fil,+}$ with ${\gr}(f)$ injective. Let $i_0:=\mathrm{min}(i(M),i(N))$. 
For any $k\geq i_0$, the map $f_{\leq k}:M/F^kM\to N/F^kN$ induced by $f$ is compatible with the filtrations on both sides induced by 
the images of $F^iM$ and $F^iN$, $i\leq k$. The associated graded map is $\oplus_{i=i_0}^{k-1}\mathrm{gr}_i(f)$, which is injective. 
Therefore $f_{\leq k}$ is injective. So  $\hat f$ is injective as well.  \hfill\qed\medskip 

\subsubsection*{Symmetric monoidal structures}

Each of the categories from Definition \ref{def:cat}, (1)-(4), is equipped with a symmetric monoidal structure, induced at the level of objects 
by the tensor product of $\mathbf k$-modules. The category $\mathbf k\text{-mod}_{\topo}$ is also equipped with a symmetric monoidal structure
arising from the completed tensor product $(M,i\mapsto F^M)\hat\otimes(N,i\mapsto F^N):=(\underset{j}{\varprojlim}(M\otimes N)/F^j(M\otimes N),
i\mapsto\underset{j}{\varprojlim}F^i(M\otimes N)/F^j(M\otimes N))$. For $X,Y$ objects in $\mathbf k\text{-mod}_{\topo}$, one has therefore 
$X\hat\otimes Y=(\mathrm{taut}_{\topo}^{\fil,+}(X)\otimes\mathrm{taut}_{\topo}^{\fil,+}(Y))^\wedge$. The functors in the previous paragraphs are then all monoidal. 

\part{The de Rham and Betti frameworks of double shuffle theory}\label{part:2:2jan2020}

In this part, we recall from \cite{Rac} the DSR formalism and complement it by the introduction of a harmonic module coproduct 
$\Delta^{\mathcal M,\DR}$ (\S\ref{sect:1}). We construct a variant of this formalism, where its basic ingredient, the free algebra 
in two generators, is replaced by the group algebra of the free group in the same number of generators (\S\ref{sect:2:28oct}). 
We then construct families of isomorphisms relating the various ingredients of the formalism of \S\ref{sect:1} with their 
counterparts of \S\ref{sect:2:28oct} (\S\ref{sect:fci:24dec2019}). The formalism of \S\ref{sect:1} (resp. \S\ref{sect:2:28oct}) 
is called `de Rham' (resp. `Betti') for reasons explained in \S\ref{sect:3:4:28oct}.

\section{The de Rham framework of double shuffle theory}\label{sect:1}

We recall from \cite{Rac} the basic formalism of double shuffle theory. Its main ingredients are: $\mathbf k$-bialgebras 
$(\mathcal V^{\DR},\Delta^{\mathcal V,\DR})$ and $(\mathcal W^{\DR},\Delta^{\mathcal W,\DR})$, a $\mathbf k$-coalgebra 
$(\mathcal M^{\DR},\Delta^{\mathcal M,\DR})$, related by an algebra inclusion $\mathcal W^{\DR}\hookrightarrow\mathcal V^{\DR}$ 
and a $\mathcal V^{\DR}$-module structure on $\mathcal M^{\DR}$, inducing a free rank one $\mathcal W^{\DR}$-module structure 
on $\mathcal M^{\DR}$, compatible with the coproducts; completions of all these structures; and a ``$\Gamma$-function'' map 
$\mathcal G(\hat{\mathcal V}^\DR)\to\mathbf k[[t]]$, where the source is the set of group-like elements of the completion of 
$\hat{\mathcal V}^\DR$. 

We introduce the algebra and module structures in \S\ref{sect:1:1:1:crm}, the coproducts in \S\ref{sect:tcDaD}, the completions 
in \S\ref{sect:completions:DR}, the $\Gamma$-function map in \S\ref{sect:Gamma}. We explain the relation with the formalism 
of \cite{Rac} in \S\ref{sect:rel:Rac}. 

\subsection{The algebras $\mathcal V^{\DR}$, $\mathcal W^{\DR}$ and the module $\mathcal M^{\DR}$}
\label{sect:1:1:1:crm}

Let $\mathfrak f_2$ be the free $\mathbb Z$-graded $\mathbf k$-Lie algebra over generators $e_0$ and $e_1$\index{e_0, e_1@ $e_0$, $e_1$} of degree 1 and let 
$$
\mathcal V^{\DR}:=U\mathfrak f_2\index{V^DR@$\mathcal V^{\DR}$}
$$
be its enveloping algebra; it is a $\mathbb Z$-graded algebra whose components of degree $<0$ vanish, therefore an algebra in 
$\mathbf k\text{-mod}_{\mathrm{gr},+}$. 

Then
$$
\mathcal W^{\DR}:=\mathbf k\oplus (U\mathfrak f_2)e_1\subset U\mathfrak f_2=\mathcal V^{\DR} \index{W^DR@$\mathcal W^{\DR}$}
$$
is a $\mathbb Z$-graded subalgebra of $\mathcal V^{\DR}$, hence a subalgebra of it in $\mathbf k\text{-mod}_{\gr,+}$. 

The quotient $\mathbf k$-module 
$$
\mathcal M^\DR:=\mathcal V^{\DR}/\mathcal V^{\DR}e_0 
\index{M^DR@$\mathcal M^{\DR}$}
$$
is an object in $\mathbf k\text{-mod}_{\gr,+}$; it is a left module over $\mathcal V^\DR$, and by restriction also 
over $\mathcal W^\DR$. Denote by $1_\DR$\index{1DR@$1_\DR$} the class 
of $1\in\mathcal V^{\DR}$ in $\mathcal M^\DR$. The map
\begin{equation}\label{def:pi:DR}
(-)\cdot 1_\DR 
: \mathcal V^\DR\to\mathcal M^\DR,\quad a\mapsto a\cdot 1_\DR
\index{-1_DR@$(-)\cdot 1_\DR$}
\end{equation} 
is a surjective morphism in  $\mathbf k\text{-mod}_{\gr,+}$, and its kernel is equal to $(U\mathfrak f_2)e_0$. 
It follows from the direct sum decomposition $U\mathfrak f_2=\mathbf k\oplus (U\mathfrak f_2)e_0\oplus (U\mathfrak f_2)e_1$ that  
$\mathcal M^\DR$ is free of rank one as a $\mathcal W^\DR$-module, generated by $1_\DR$. We denote by 
\begin{equation}\label{iso:W:M:DR}
\mathcal W^\DR\to\mathcal M^\DR,\quad a\mapsto a\cdot 1_\DR
\end{equation}
the corresponding isomorphism of $\mathcal W^\DR$-modules. 

\subsection{The coproducts $\Delta^{\mathcal V,\DR},\Delta^{\mathcal W,\DR}$ and $\Delta^{\mathcal M,\DR}$}
\label{sect:tcDaD}

We denote by $\Delta^{\mathcal V,\DR}:\mathcal V^{\DR}\to(\mathcal V^{\DR})^{\otimes2}$\index{DeltaVDR@$\Delta^{\mathcal V,\DR}$} the enveloping algebra 
coproduct; it is determined by the conditions that $e_0,e_1$ are primitive. 

One shows that $\mathcal W^{\DR}$ is freely generated by the elements
$y_n:=-e_0^{n-1}e_1$\index{y_n@$y_n$}, where $n\geq 1$. One denotes by 
$$
\Delta^{\mathcal W,\DR}:\mathcal W^{\DR}\to(\mathcal W^{\DR})^{\otimes2} 
\index{DeltaWDR@$\Delta^{\mathcal W,\DR}$}
$$
the algebra morphism determined by 
\begin{equation}\label{image:yn}
\forall n\geq 1,\quad \Delta^{\mathcal W,\DR}(y_n)=y_n\otimes1+1\otimes y_n+\sum_{k=1}^{n-1}y_k\otimes y_{n-k};  
\end{equation}
it equips $\mathcal W^{\DR}$ with a Hopf algebra structure. 

We denote by 
$$
\Delta^{\mathcal M,\DR}:\mathcal M^{\DR}\to(\mathcal M^{\DR})^{\otimes2} 
\index{DeltaMDR@$\Delta^{\mathcal M,\DR}$}
$$
the $\mathbf k$-module morphism determined by the conditions that $\Delta^{\mathcal M,\DR}(1_\DR)=(1_\DR)^{\otimes2}$, and that it 
is compatible with the module structures on both sides and with the algebra morphism $\Delta^{\mathcal W,\DR}$.  Then 
$(\mathcal M^{\DR},\Delta^{\mathcal M,\DR})$ is a cocommutative and coassociative coalgebra. 

\begin{rem}
As a Hopf algebra,  $(\mathcal W^\DR,\Delta^{\mathcal W,\DR})$ is isomorphic to the universal enveloping algebra $U(\mathfrak g^\DR)$ of the 
free Lie algebra $\mathfrak g^\DR$ with generators $(u_k)_{k\geq1}$, the isomorphism being given by $\mathfrak g^\DR[[t]]\ni
\sum_{k\geq1}t^ku_k\mapsto\mathrm{log}(1+\sum_{k\geq1}t^ky_k)\in\mathcal W^\DR[[t]]$. 
\end{rem}

\subsection{Completions}\label{sect:completions:DR}

The pairs $({\mathcal V}^\DR,\Delta^{\mathcal V,\DR})$, $({\mathcal W}^\DR,\Delta^{\mathcal W,\DR})$ 
(resp. $({\mathcal M}^\DR,\Delta^{\mathcal M,\DR})$) are Hopf algebras (resp. is a coalgebra) in $\mathbf k\text{-mod}_{\gr,+}$
(see Definition \ref{def:cat}). Their images by the functor $(-)^\wedge\circ\mathrm{taut}_{\gr,+}^{\fil,+}:\mathbf k\text{-mod}_{\gr,+}
\to\mathbf k\text{-mod}_{\topo}$ (see \S\ref{sect:functors}) are topological Hopf algebras 
$(\hat{\mathcal V}^\DR\index{V^DR^@$\hat{\mathcal V}^\DR$},\hat\Delta^{\mathcal V,\DR}\index{DeltaVDR^@$\hat\Delta^{\mathcal V,\DR}$})$ and 
$(\hat{\mathcal W}^\DR\index{W^DR^@$\hat{\mathcal W}^\DR$},\hat\Delta^{\mathcal W,\DR}\index{DeltaWDR^@$\hat\Delta^{\mathcal W,\DR}$})$, and a topological coalgebra 
$(\hat{\mathcal M}^\DR\index{M^DR^@$\hat{\mathcal M}^\DR$},\hat\Delta^{\mathcal M,\DR}\index{DeltaMDR^@$\hat\Delta^{\mathcal M,\DR}$})$; the topological $\mathbf k$-module $\hat{\mathcal M}^\DR$ is then a free
rank 1 module over the topological $\mathbf k$-algebra $\hat{\mathcal W}^\DR$. 

\subsection{$\Gamma$-functions}\label{sect:Gamma}

Define a map 
$$
\hat{\mathcal V}^\DR
\to \mathbf k[[t]]^\times, \quad g\mapsto \Gamma_g, 
$$
where 
\begin{equation}\label{def:Gamma:fun:of:NC:series}
\Gamma_g(t):=\mathrm{exp}\Big(\sum_{n\geq 1}{(-1)^{n+1}\over n}(g|e_0^{n-1}e_1)t^n\Big)\in\mathbf k[[t]]^\times 
\index{Gamma@$\Gamma_g$}
\end{equation}
and where $g\mapsto ((g|w))_w$
is the map $\hat{\mathcal V}^\DR\to\mathbf k^{\{\text{words in }e_0,e_1\}}$ such that the identity $g=\sum_{w}(g|w)w$ holds. 

\subsection{Relation with double shuffle theory}\label{sect:rel:Rac}

In order to explain how the above material is used in the theory of double shuffle relations between the MZVs given by 
\eqref{def:MZV:3jan2020}, we introduce the notation $\hat{\mathcal V}_{\mathbb C}^\DR$,  
$\hat{\mathcal M}_{\mathbb C}^\DR$ (resp. ${\mathcal V}_{\mathbb Q}^\DR$) for the spaces $\hat{\mathcal V}^\DR$,  
$\hat{\mathcal M}^\DR$ (resp. ${\mathcal V}^\DR$) when $\mathbf k=\mathbb C$ (resp. $\mathbf k=\mathbb Q$). 

The series $\varphi_{\mathrm{KZ}}\in\mathbb C\langle\langle A,B\rangle\rangle^\times$
from \S2 in \cite{Dr} can be viewed as an element of $(\hat{\mathcal V}_{\mathbb C}^\DR)^\times$ 
via the isomorphism given by $A\mapsto e_0$, $B\mapsto e_1$. It is a generating 
series for the MZVs, namely 
$$
\varphi_\KZ=1+\sum_{m,k_1,\ldots,k_{m}>0,k_1>1}
(-1)^m\zeta(k_1,\ldots,k_m)w(k_1,\ldots,k_m),
\index{phiKZ@$\varphi_\KZ$}
$$
where $w(k_1,\ldots,k_m)$ is an explicit element in $e_0^{k_1-1}e_1\cdots e_0^{k_m-1}e_1+e_1\mathcal V^\DR_{\mathbb Q}+\mathcal V^\DR_{\mathbb Q} e_0$
(see \cite{LM} and \cite{Fu0}, Proposition 3.2.3). 

The double shuffle relations between MZVs can be formulated as follows (see \cite{Rac}): 
\begin{itemize}
\item[(1)] $\varphi_\KZ$ is group-like for the coproduct $\hat\Delta^{\mathcal V,\DR}$; 

\item[(2)] $\big(\Gamma_{\varphi_\KZ}^{-1}(-e_1)\cdot\varphi_\KZ\big)\cdot 1_\DR\in\hat{\mathcal M}^\DR_{\mathbb C}$
is group-like for the coproduct $\hat\Delta^{\mathcal M,\DR}$.  
\end{itemize}
The correspondence between the present formalism and that of \cite{Rac} is as follows. 
\begin{center}
\begin{tabular}{|l|l|l|l|l|l|l|l|l|}
\hline
 this paper & $e_0,e_1$ & $\mathcal V^\DR$, $\hat{\mathcal V}^\DR$ & $\Delta^{\mathcal V,\DR}$, $\hat\Delta^{\mathcal V,\DR}$ & $\mathcal W^\DR$, $\hat{\mathcal W}^\DR$ & $\Delta^{\mathcal W,\DR}$, $\hat\Delta^{\mathcal W,\DR}$ \\ \hline
\cite{Rac} & $x_0,-x_1$ & $\mathbf k\langle x_0,x_1\rangle$, &  $\Delta$, $\hat\Delta$&$\mathbf k\oplus\mathbf k\langle x_0,x_1\rangle x_1$,  & $\Delta_\star$, $\hat\Delta_\star$ \\ 
&  & $\mathbf k\langle\langle x_0,x_1\rangle\rangle$  &  &degree completion  & \\ \hline
\end{tabular}
\end{center}
\begin{center}
\begin{tabular}{|l|l|l|l|l|l|l|l|l|}
\hline
   $\mathcal M^\DR$, $\hat{\mathcal M}^\DR$ &
   $\Delta^{\mathcal M,\DR}$, $\hat\Delta^{\mathcal M,\DR}$ &
    $(-)\cdot 1_\DR:\mathcal V^\DR\to\mathcal M^\DR$, $(-)\cdot 1_\DR:\hat{\mathcal V}^\DR\to\hat{\mathcal M}^\DR$  
    \\ \hline
 $\mathbf k\langle x_0,x_1\rangle /\mathbf k\langle x_0,x_1\rangle x_0$,  &
  $\Delta_\star$, $\hat\Delta_\star$ &
    $\pi_Y :\mathbf k\langle x_0,x_1\rangle\to\mathbf k\langle x_0,x_1\rangle /\mathbf k\langle x_0,x_1\rangle x_0$, 
\\ 
degree completion  & & degree completion  \\ \hline
\end{tabular}
\end{center}

\section{The Betti framework of double shuffle theory}\label{sect:2:28oct}

We construct a Betti version of the double shuffle formalism from \S\ref{sect:1}. The analogues of $\mathcal V^{\DR}$, $\mathcal W^{\DR}$ and 
$\mathcal M^{\DR}$ are a pair of filtered $\mathbf k$-algebras $\mathcal V^{\B}$, $\mathcal W^{\B}$, and a filtered $\mathbf k$-module 
$\mathcal M^{\B}$ (\S\ref{sect:2:1:28oct}). An algebra presentation of $\mathcal W^{\B}$ is established in \S\ref{sect:pres:W:l:B}, and used in 
\S\ref{sect:2:3:28oct} to define a Hopf algebra coproduct $\Delta^{\mathcal W,\B}$ on $\mathcal W^{\B}$ as well as a coassociative 
coproduct $\Delta^{\mathcal M,\B}$ on $\mathcal M^\B$; a Hopf algebra coproduct $\Delta^{\mathcal V,\B}$ on $\mathcal V^\B$ is also 
defined. In \S\ref{sect:filt:28oct}, we show the compatibility of the coproducts with the filtrations, and we relate 
the resulting associated graded objects with the material of \S\ref{sect:1}. We define the corresponding topological 
Hopf algebras and coassociative coalgebra in \S\ref{sect:compl:28oct}. 

\subsection{The algebras $\mathcal V^{\B}$, $\mathcal W^{\B}$ and the module $\mathcal M^{\B}$}\label{sect:2:1:28oct}

Let $F_2$\index{F_2@$F_2$} be the free group with generators $X_0$, $X_1$\index{X_0, X_1@$X_0$, $X_1$}. Let 
$$
\mathcal V^\B:=\mathbf kF_2
\index{V^B@$\mathcal V^\B$}
$$
be its group algebra. It is equipped with a decreasing algebra filtration $\mathcal V^\B=F^0\mathcal V^\B\supset F^1\mathcal V^\B\supset\cdots$, where for $i\geq 0$, $F^i\mathcal V^\B$\index{F^iVB@$F^i\mathcal V^\B$} is the $i$-th power of the augmentation ideal $\mathcal V^\B_+$ of $\mathcal V^\B=\mathbf kF_2$.
Then 
$$
\mathcal W^\B:=\mathbf k\oplus\mathcal V^\B(X_1-1)\subset\mathcal V^\B
\index{W^B@$\mathcal W^\B$}
$$
is a subalgebra of $\mathcal V^\B$; it is equipped with the induced filtration given by 
$$
\forall i\geq0,\quad F^i\mathcal W^\B:=F^i\mathcal V^\B\cap\mathcal W^\B.
\index{F^iWB@$F^i\mathcal W^\B$}
$$ 
Then $\mathcal V^\B$ is an algebra in $\mathbf k\text{-mod}_{\fil,+}$ and $\mathcal W^\B$ is a subalgebra. 

The quotient $\mathbf k$-module 
$$
\mathcal M^\B:=\mathcal V^\B/\mathcal V^\B(X_0-1)
\index{M^B@$\mathcal M^\B$}
$$
is a left module over $\mathcal V^\B$, and by restriction over $\mathcal W^\B$. Let $1_\B$\index{1B@$1_\B$} be the 
class of $1\in\mathcal V^\B$ in $\mathcal M^\B$. The map  
$$
(-)\cdot 1_\B 
:\mathcal V^\B\to\mathcal M^\B, \quad a\mapsto a\cdot 1_\B
\index{-1_B@$(-)\cdot 1_\B$}
$$
is surjective, with kernel $(\mathbf kF_2)(X_0-1)$. The $\mathbf k$-module $\mathcal M^\B$ is equipped with a decreasing filtration 
given by $F^i\mathcal M^\B:=(F^i\mathcal V^\B)\cdot 1_\B$\index{F^iMB@$F^i\mathcal M^\B$}  for $i\geq 0$, compatible with the filtrations of $\mathcal V^\B$ and $\mathcal W^\B$. 
Therefore $\mathcal M^\B$ is a $\mathcal V^\B$-module and a $\mathcal W^\B$-module in $\mathbf k\text{-mod}_{\fil,+}$. 

\begin{lem} [Proposition 6.2.6  in \cite{Wei}] \label{lem:decomp:ZZG}
Let $G$ be the free group over generators $g_1,\dots,g_n$, let $\mathbb ZG$ be its group 
algebra with coefficients in $\mathbb Z$ and let $(\mathbb ZG)_0$ be the augmentation ideal of
this algebra. Then $(\mathbb ZG)_0$ is freely generated, as a left $\mathbb ZG$-module, by the family $(g_1-1,\ldots,g_n-1)$. 
\end{lem}
The proof in \cite{Wei} relies on the following statement (not explicitly stated there)
$$
\forall N\geq 1,\quad \#\{g\in G\mid g\neq1,\ell(g)<N\}=\#\{(g,i)\in G\times[\![1,n]\!]\mid\ell(g)<N,\ell(gg_i)<N\}
$$
which may be proved by counting arguments. Here is an alternative proof of this lemma. 

\medskip 
{\it Alternative proof of Lemma \ref{lem:decomp:ZZG}.} There is a unique morphism of left $\mathbb ZG$-modules
$$
a:(\mathbb ZG)^{\oplus n}\to(\mathbb ZG)_0, \quad (f_1,\ldots,f_n)\mapsto \sum_{i=1}^n f_i(g_i-1). 
$$
It follows from \cite{Fox}, equation (2.3) (see also the argument in Proposition 6.2.6 from \cite{Wei}) that the map $a$ is surjective. 

The injectivity of $a$ may then be established in two ways. 

(First proof of injectivity of $a$.) In \cite{Fox}, endomorphisms ${\partial\over\partial g_i}$ ($i\in[\![1,n]\!]$) of the 
$\mathbb Z$-module $\mathbb ZG$ are constructed. These endomorphisms 
are characterized by the following properties: ${\partial\over\partial g_i}(1)=0$,  ${\partial\over\partial g_i}(g_j)=\delta_{ij}$
($j\in[\![1,n]\!]$), and ${\partial\over\partial g_i}(uv)=u{\partial\over\partial g_i}(v)+{\partial\over\partial g_i}(u)\epsilon(v)$ for $u,v$
in $\mathbb ZG$, where $\epsilon:\mathbb ZG\to\mathbb Z$ is the counit map (augmentation). These properties 
imply the identities ${\partial\over\partial g_i}(f(g_j-1))=\delta_{ij}f$ for $f\in\mathbb ZG$, $i\in[\![1,n]\!]$. As a consequence, 
the $\mathbb Z$-module map 
$$
b:(\mathbb ZG)_0\to(\mathbb ZG)^{\oplus n}, \quad f\mapsto ({\partial\over\partial g_1}f,\ldots,
{\partial\over\partial g_n}f)
$$
is such that $b\circ a=\mathrm{id}$. It follows that $a$ is injective and therefore an isomorphism of $\mathbb ZG$-modules.  

(Second proof of injectivity of $a$.) Let $\mathbb Z\langle\langle x_1,\ldots,x_n\rangle\rangle$ be the ring of formal series in 
noncommutative variables $x_1,\ldots,x_n$, with coefficients in $\mathbb Z$. The group morphism 
$G\to\mathbb Z\langle\langle x_1,\ldots,x_n\rangle\rangle^\times$ given by $g_i\mapsto 1+x_i$  for $i\in[\![1,n]\!]$
gives rise to an algebra morphism $r:\mathbb ZG\to\mathbb Z\langle\langle x_1,\ldots,x_n\rangle\rangle$,  
which is injective (see \cite{Bbk}, Ch.\ 2, Sec.\ 5, no.\ 3, Thm.\ 1). 
The map $r\circ a:\mathbb ZG^{\oplus n}\to \mathbb Z\langle\langle x_1,\ldots,x_n\rangle\rangle$ is  
given by $(f_1,\ldots,f_n)\mapsto \sum_{i=1}^n r(f_i)x_i$ and is therefore equal to the composed map 
$\mathbb ZG^{\oplus n}\stackrel{r^{\oplus n}}{\to} \mathbb Z\langle\langle x_1,\ldots,x_n\rangle\rangle^{\oplus n}\stackrel{\varphi}{\to}
\mathbb Z\langle\langle x_1,\ldots,x_n\rangle\rangle$, where $\varphi$ is given by $(r_1,\ldots,r_n)\mapsto\sum_{i=1}^n r_ix_i$. 
This map is injective, as is $r^{\oplus n}$, which implies that $r\circ a$ is injective, showing the injectivity of $a$. 
\hfill\qed\medskip 

Lemma \ref{lem:decomp:ZZG} implies the direct sum decomposition 
\begin{equation}\label{decomp:kF2}
\mathcal V^\B=\mathbf k 1\oplus\mathcal V^\B(X_1-1)\oplus\mathcal V^\B(X_0-1)
=\mathcal W^\B \oplus \mathcal V^\B(X_0-1),  
\end{equation}
which implies that $\mathcal M^\B$ is free of rank one as a $\mathcal W^\B$-module, generated by the class 
$1_\B$ of $1\in \mathbf kF_2$. We denote by 
\begin{equation}\label{iso:W:M:B}
\mathcal W^\B\to\mathcal M^\B,\quad a\mapsto a\cdot 1_\B
\end{equation}
the corresponding isomorphism of $\mathcal W^\B$-modules. 

\subsection{Presentation of the algebra $\mathcal W^\B$}\label{sect:pres:W:l:B}

Set for $k\in\mathbb Z$, 
$$
\xi_k^+:=X_0^k(X_1-1), \quad \mathrm{and} \quad \xi_k^-:=X_0^k(X_1^{-1}-1)\index{xi @ $\xi_k^+$, $\xi_k^-$}.
$$ These are elements of $\mathcal W^{\B}\subset \mathbf k F_2$. 

\begin{lem}\label{lem:Wpresentation1}
The algebra $\mathcal W^{\B}$ is the algebra  whose
generators are $(\xi_k^+)_{k\in\mathbb Z}$, $(\xi_k^-)_{k\in\mathbb Z}$, and relations are 
\begin{equation}\label{relations1}
\forall k\in\mathbb Z, \quad\xi_k^+\xi_0^-=\xi_k^-\xi_0^+=-\xi_k^+-\xi_k^-. 
\end{equation}
\end{lem}

\proof 
Let ${\mathcal A}$ be the associative non-commutative algebra generated by 
$(\xi_k^+)_{k\in\mathbb Z}$, $(\xi_k^-)_{k\in\mathbb Z}$
divided by the two-sided ideal generated by the relation given in \eqref{relations1}.

One checks that there is an algebra morphism
$$
f:{\mathcal A}\to\mathbf k\oplus\mathbf k F_2(X_1-1)=\mathcal W^\B,
$$ 
uniquely determined by 
$$\xi_k^+\mapsto X_0^k(X_1-1) \text{ and }\xi_k^-\mapsto X_0^k(X_1^{-1}-1).$$ 
We will prove that $f$ is an algebra isomorphism.

Let us first prove that $f$ is surjective. For this, it will suffice to prove that for any word $w$ in $X_0^{\pm1},X_1^{\pm1}$, there exists a 
polynomial without constant term $P_w$ over a set of noncommutative variables $\{(\tilde\xi_k^+)_{k\in\mathbb Z},\tilde\xi_0^-\}$ 
indexed by $\mathbb Z\sqcup\{0\}$, such that 
$$
w\cdot (X_1-1)=P_w((\xi_k^+)_{k\in\mathbb Z},\xi_0^-). 
$$
We argue by induction on the length $|w|$ of $w$. If $w=1$, then $P_w=\tilde\xi_0^+$. Assume the statement for $|w|<n$ 
and let $w$ be a word of length $n$. Then $w=gw'$, where $g\in\{X_0^{\pm1},X_1^{\pm1}\}$ and $w'$ has length $n-1$. 
If $g=X_1^{\pm1}$, then 
$$
w\cdot(X_1-1)=X_1^{\pm1}w'(X_1-1)=P_w((\xi_k^+)_{k\in\mathbb Z},\xi_0^-),
$$ 
where $P_w
:=(1+\tilde\xi_0^\pm)P_{w'}
$. 
If $g=X_0^{\pm1}$, then 
$$
w\cdot(X_1-1)=X_0^{\pm1}w'(X_1-1)=P_w((\xi_k^+)_{k\in\mathbb Z},\xi_0^-),
$$
where we set 
$$
P_w
:=\sum_{k\in\mathbb Z}\tilde\xi_{k\pm1}^+(P_{w'})_k^+
-\tilde\xi_{\pm1}^+(1+\tilde\xi_0^-)(P_{w'})_0^-
, 
$$
in which the expressions $(P_{w'})_k^+$ and $(P_{w'})_0^-$ have the following meaning: for $P$ a polynomial 
in the noncommutative variables $(\tilde\xi_k^+)_{k\in\mathbb Z},\tilde\xi_0^-$ without constant term, 
$(P)_k^+$ and $(P)_0^-$ are the polynomials in the same variables such that
$$
P=\sum_{k\in\mathbb Z}\tilde\xi_k^+(P)_k^++\tilde\xi_0^-(P)_0^-. 
$$ 
This proves that $f:{\mathcal A}\to\mathbf k\oplus\mathbf k F_2(X_1-1)$ is surjective. 

To prove that $f$ is injective, we will: 
\begin{itemize}
\item[a)] equip the vector space ${\mathcal A}[t^{\pm1}]$ with an algebra structure $*$, such that ${\mathcal A}\to
{\mathcal A}[t^{\pm1}]$, $a\mapsto a\otimes1$ is an algebra morphism; 
\item[b)] construct an algebra isomorphism $F:{\mathcal A}[t^{\pm1}]\to\mathbf k F_2$; 
\item[c)] check that $F$ extends $f$, in the sense that the following diagram commutes
\begin{equation}\label{diagram:f:F}
\xymatrix{
{\mathcal A} \ar^{\!\!\!\!\!\!\!\!\!\!\!\!\!\!\!\!\!\!\!\!\!\!\!\!\!f}[r] \ar@{^(->}_{-\otimes1}[d] &
\mathbf k\oplus\mathbf k F_2(X_1-1)
\ar@{^(->}[d]\\
{\mathcal A}[t^{\pm1}]\ar_F[r] & \mathbf k F_2}
\end{equation}
\end{itemize}
 The above diagram implies that $f$ is injective, which finally implies that it is an isomorphism. 

\underline{(a) Construction of an algebra structure on ${\mathcal A}[t^{\pm1}]$.} 
Let $\mathbf k\langle(\tilde\xi_k^{\pm})_{k\in\mathbb Z}\rangle_+$ 
be the ideal of $\mathbf k\langle(\tilde\xi_k^{\pm})_{k\in\mathbb Z}\rangle$ generated by the generators 
$(\tilde\xi_k^{\pm})_{k\in\mathbb Z}$. We have a direct sum decomposition 
$\mathbf k\langle(\tilde\xi_k^{\pm})_{k\in\mathbb Z}\rangle=\mathbf k\oplus\mathbf k\langle(\tilde\xi_k^{\pm})_{k\in\mathbb Z}\rangle_+$. 
Left multiplication induces a linear isomorphism 
$$
(\oplus_{k\in\mathbb Z}(\mathbf k\tilde\xi_k^+\oplus\mathbf k\tilde\xi_k^-))\otimes\mathbf k\langle(\tilde\xi_k^{\pm})_{k\in\mathbb Z}\rangle\stackrel{\sim}{\to}\mathbf k\langle(\tilde\xi_k^{\pm})_{k\in\mathbb Z}\rangle_+,  
$$
therefore there is a linear automorphism $T$ of $\mathbf k\langle(\tilde\xi_k^{\pm})_{k\in\mathbb Z}\rangle_+$, 
uniquely determined by 
\begin{equation}\label{rel:def:T}
\forall k\in\mathbb Z, \forall a\in\mathbf k\langle(\tilde\xi_k^{\pm})_{k\in\mathbb Z}\rangle, \quad 
T(\tilde\xi_k^\pm a)=\tilde\xi_{k+1}^\pm a.
\end{equation}
Let $I$ be the two-sided ideal of $\mathbf k\langle(\tilde\xi_k^{\pm})_{k\in\mathbb Z}\rangle$ generated by (\ref{relations1}). 
 Then $I\subset\mathbf k\langle(\tilde\xi_k^{\pm})_{k\in\mathbb Z}\rangle_+$. Set $\mathcal A^+:=
\mathbf k\langle(\tilde\xi_k^{\pm})_{k\in\mathbb Z}\rangle_+/I$, then 
$$
{\mathcal A}=\mathbf k\oplus {\mathcal A}^+. 
$$
Let $V\subset\mathbf k\langle(\tilde\xi_k^{\pm})_{k\in\mathbb Z}\rangle_+$ be the linear span of all 
$\tilde\xi_k^+\tilde\xi_0^-+\tilde\xi_k^++\tilde\xi_k^-$ 
and $\tilde\xi_k^-\tilde\xi_0^++\tilde\xi_k^++\tilde\xi_k^-$, for $k\in\mathbb Z$. Then 
$$
I=V\cdot \mathbf k\langle(\tilde\xi_k^{\pm})_{k\in\mathbb Z}\rangle+
\mathbf k\langle(\tilde\xi_k^{\pm})_{k\in\mathbb Z}\rangle_+\cdot V\cdot \mathbf k\langle(\tilde\xi_k^{\pm})_{k\in\mathbb Z}\rangle. 
$$
It follows from the definition of $T$ that $T^{\pm1}$ maps the subspace $\mathbf k\langle(\tilde\xi_k^{\pm})_{k\in\mathbb Z}\rangle_+\cdot V
\cdot \mathbf k\langle(\tilde\xi_k^{\pm})_{k\in\mathbb Z}\rangle$ to itself. Moreover, for any 
$a\in\mathbf k\langle(\tilde\xi_k^{\pm})_{k\in\mathbb Z}\rangle$, one has 
$$
T^{\pm1}((\tilde\xi_k^+\tilde\xi_0^-+\tilde\xi_k^++\tilde\xi_k^-)a)=(\tilde\xi_{k\pm1}^+\tilde\xi_0^-+\tilde\xi_{k\pm1}^++\tilde\xi_{k\pm1}^-)a,\quad
T^{\pm1}((\tilde\xi_k^-\tilde\xi_0^++\tilde\xi_k^++\tilde\xi_k^-)a)=(\tilde\xi_{k\pm1}^-\tilde\xi_0^++\tilde\xi_{k\pm1}^++\tilde\xi_{k\pm1}^-)a, 
$$
which implies that $T^{\pm1}$ maps $V\cdot\mathbf k\langle(\tilde\xi_k^{\pm})_{k\in\mathbb Z}\rangle$ to itself. All this implies that 
$T^{\pm1}(I)\subset I$, so $T(I)=I$. It follows that $T$ induces a linear automorphism of ${\mathcal A}^+$. Relation 
(\ref{rel:def:T}) then implies: 
\begin{equation}\label{special:property:T}
\text{for }P,P'\in {\mathcal A}^+\text{ and }a\in\mathbb Z,\text{ we have }T^a(PP')=T^a(P)P'. 
\end{equation}

We now define a bilinear map 
$$
*:{\mathcal A}[t^{\pm1}]^{\otimes2}\to {\mathcal A}[t^{\pm1}]
$$
as follows: 
\begin{itemize}
\item[1)] for $a,b\in\mathbb Z$, $(1\otimes t^a)*(1\otimes t^b)=1\otimes t^{a+b}$; 
\item[2)] for $a\in\mathbb Z$, $P\in {\mathcal A}^+$, 
$$
(1\otimes t^a)*(P\otimes t^b):=T^a(P)\otimes t^b, \quad
(P\otimes t^a)*(1\otimes t^b):=P\otimes t^{a+b}, 
$$
\item[3)] for $a,b\in\mathbb Z$, and $P,Q\in {\mathcal A}^+$, 
$$
(P\otimes t^a)*(Q\otimes t^b):=(P\cdot T^a(Q))\otimes t^b,  
$$
where $\cdot$ is the product in ${\mathcal A}^+$. 
\end{itemize}

It then follows from a case analysis, using (\ref{special:property:T}), that 
\begin{equation}
\text{the product }*\text{ defines an associative algebra structure on }{\mathcal A}[t^{\pm1}].  
\end{equation}

\underline{(b) Construction of an algebra isomorphism $({\mathcal A}[t^{\pm1}],*)\simeq\mathbf k F_2$.} The elements 
$1\otimes t$ and $1\otimes t^{-1}$ of $({\mathcal A}[t^{\pm1}],*)$ are mutually inverse; so are the elements 
$1+\xi_0^+$ and $1+\xi_0^-$ of ${\mathcal A}$, and therefore the elements $(1+\xi_0^+)\otimes1$ and $(1+\xi_0^-)\otimes1$ of 
${\mathcal A}[t^{\pm1}]$. All this implies that there is an algebra morphism 
$$
G:\mathbf k F_2\to({\mathcal A}[t^{\pm1}],*)
$$
uniquely determined by 
$$
X_0^{\pm1}\mapsto 1\otimes t^{\pm1},\quad X_1^{\pm1}\mapsto (1+\xi_0^\pm)\otimes1. 
$$
One checks that the algebra homomorphism $f: {\mathcal A}\to \mathbf k\oplus\mathbf k F_2(X_1-1)$ 
restricts to a morphism of algebras without unit 
$$
f^+:{\mathcal A}^+\to\mathbf k F_2(X_1-1), 
$$
and that the following diagram commutes
$$
\xymatrix{ {\mathcal A}^+\ar^{\!\!\!\!\!\!\!\!\!\!\!\!\!\!\!\!\!f^+}[r]\ar_T[d]&\mathbf k F_2(X_1-1)\ar^{X_0\cdot-}[d]
\\ {\mathcal A}^+\ar_{\!\!\!\!\!\!\!\!\!\!\!\!\!\!\!\!\!\!\!f^+}[r]& \mathbf k F_2(X_1-1) }
$$
Define now a linear map 
$$
F:{\mathcal A}[t^{\pm1}]\to\mathbf k F_2
$$
by 
$$
1\otimes t^a\mapsto X_0^a, \quad P\otimes t^b\mapsto f^+(P)X_0^b
$$
for $a,b\in\mathbb Z$, $P\in {\mathcal A}^+$. Then: 
\begin{itemize}
\item[1)] for $a,b\in\mathbb Z$, $F(1\otimes t^a)F(1\otimes t^b)=X_0^aX_0^b=X_0^{a+b}=F(1\otimes t^{a+b})=F((1\otimes t^a)*(1\otimes t^b))$; 
\item[2)] for $a\in\mathbb Z$, $P\in {\mathcal A}^+$, 
$$
F(1\otimes t^a)F(P\otimes1)=X_0^a f^+(P)= f^+(T^a(P))=F(T^a(P)\otimes1)=F((1\otimes t^a)*(P\otimes1)), 
$$
$$
F(P\otimes1)F(1\otimes t^a)=f^+(P)X_0^a=F(P\otimes t^a)=F((P\otimes1)*(1\otimes t^a));
$$
\item[3)]for $a,b\in\mathbb Z$, $P,Q\in {\mathcal A}^+$, 
\begin{align*}
& F(P\otimes t^a)F(Q\otimes t^b)=f^+(P)X_0^a f^+(Q)X_0^b= f^+(P) f^+(T^a(Q))X_0^b
\\ & = f^+(P T^a(Q))X_0^b=F(PT^a(Q)\otimes t^b)=F((P\otimes t^a)*(Q\otimes t^b)). 
\end{align*}
\end{itemize}
All this implies that {\it $F$ is an algebra morphism ${\mathcal A}[t^{\pm1}]\to\mathbf k F_2$.}
Let us now show that the algebra morphisms $F,G$ are mutually inverse. 

We have 
$$
F\circ G(X_0^{\pm1})=F(1\otimes t^{\pm1})=X_0^{\pm1}, 
$$
$$ 
F\circ G(X_1^{\pm1})=F((1+\xi_0^\pm)\otimes1)=1+F(\xi_0^\pm\otimes1)=1+f^+(\xi_0^\pm)=1+(X_1^{\pm1}-1)=X_1^{\pm1}, 
$$
so $F\circ G=\mathrm{id}_{\mathbf k F_2}$. 
 
Similarly, $G\circ F$ is an algebra endomorphism of $({\mathcal A}[t^{\pm1}],*)$. 

Given the structure of $*$, a collection of generators of this algebra is given by $\{g\otimes 1,1\otimes t^{\pm1}\}$, 
where $g$ runs over a collection of algebraic generators of ${\mathcal A}$, so such a collection of algebraic 
generators is $\{(\xi_k^\pm\otimes1)_{k\in\mathbb Z},1\otimes t^{\pm1}\}$. One has 
\begin{align*}
& G\circ F(\xi_k^\pm\otimes1)=G(f^+(\xi_k^\pm))=G(X_0^k(X_1^{\pm1}-1))=G(X_0)^{*k}*(G(X_1^{\pm1})-1)
\\ & =(1\otimes t^{\pm1})^{*k}*(\xi_0^\pm\otimes1) =\xi_k^\pm\otimes1, 
\end{align*}
so $G\circ F=\mathrm{id}_{{\mathcal A}[t^{\pm1}]}$. 

\underline{(c) Diagram involving $f$ and $F$.} The commutation of (\ref{diagram:f:F}) follows from the definition of $F$. 
\hfill\qed\medskip 

Lemma \ref{lem:Wpresentation1} leads to the following presentation:

\begin{prop}\label{new:prop:pres}
The associative $\mathbf k$-algebra $\mathcal W^{\B}$ is generated by the unity element $1$ and the elements
\begin{align}
& Y_n^+:=(X_0-1)^{n-1}X_0(1-X_1) \qquad\qquad (n>0), \\
& Y_n^-:=(X_0^{-1}-1)^{n-1}X_0^{-1}(1-X_1^{-1}) \qquad (n>0),  \notag \\
& X_1 \text{   and } X_1^{-1}, \notag
\index{Y_n @ $Y_n^+$, $ Y_n^-$}
\end{align}
with defining relations
\begin{equation}\label{equation for X1}
X_1X_1^{-1}=X_1^{-1}X_1=1.
\end{equation}
\end{prop}

\begin{proof}
By the presentations in the previous lemma, 
we have $\xi_k^-=-\xi_k^+(1+\xi_0^-)$ 
and $\xi_k^+=-\xi^-_k(1+\xi_0^+)$, 
from which we can deduce that 
$\mathcal A$ is generated by $\xi_k^+$ and $\xi_{-k}^-$ for $k\geqslant 0$
with defining equation \eqref{relations1} only for $k=0$.

It is immediate to see that $(Y_n^+)_{n>0}$ (resp. $(Y_n^-)_{n>0}$)
can be expressed by 
linear combinations of  $(\xi_{n}^+)_{n>0}$ (resp. $(\xi_{n}^-)_{n>0}$) and vice versa
and that the equation \eqref{relations1} for $k=0$ is equivalent to 
\eqref{equation for X1},
from which our claim follows.
\end{proof}

\subsection{The coproducts $\Delta^{\mathcal V,\B},\Delta^{\mathcal W,\B}$ and $\Delta^{\mathcal M,\B}$}\label{sect:2:3:28oct}

We denote by $\Delta^{\mathcal V,\B}:\mathcal V^\B\to(\mathcal V^\B)^{\otimes2}$
\index{DeltaVB@$\Delta^{\mathcal V,\B}$}
the group algebra coproduct; it is determined by the  
condition that the elements $X_0^{\pm1}$, $X_1^{\pm1}$ are group-like. 

\begin{prop}\label{lemma:tiauamsdt}
There is a unique algebra morphism 
$$
\Delta^{\mathcal W,\B}:\mathcal W^\B\to(\mathcal W^\B)^{\otimes2}
\index{DeltaWB@$\Delta^{\mathcal W,\B}$}
$$
such that 
\begin{equation}\label{formula:Delta:W:B}
\Delta^{\mathcal W,\B}(X_1^{\pm1})=X_1^{\pm1}\otimes X_1^{\pm1}, \quad
\Delta^{\mathcal W,\B}(Y_k^\pm)=Y_k^\pm\otimes1+1\otimes Y_k^\pm+\sum_{k',k''>0,k'+k"=k} 
Y^\pm_{k'}\otimes Y^\pm_{k''}\quad\text{for any }k\geq1.
\end{equation}
 It equips $\mathcal W^\B$ with a cocommutative Hopf algebra structure with antipode given by 
$X_1^{\pm1}\mapsto X_1^{\mp1}$, $Y_k^\pm\mapsto \sum_{a\geq1}(-1)^a\sum_{k_1+\cdots+k_a=k}
Y^\pm_{k_1}\cdots Y^\pm_{k_a}$.  
\end{prop}

\proof The relations between the generators (see Proposition \ref{new:prop:pres}) are obviously 
preserved, therefore $\Delta^{\mathcal W,\B}$ is well-defined. One directly checks the 
cocommutativity and coassociativity of $\Delta^{\mathcal W,\B}$.  One also checks that for $t$ a 
formal parameter, the series $1+\sum_{k\geq1}t^kY_k^{\pm}$ are group-like. It follows that 
$\Delta^{\mathcal W,\B}$ admits the announced antipode. \hfill\qed\medskip 

We denote by 
$$
\Delta^{\mathcal M,\B}:\mathcal M^{\B}\to(\mathcal M^{\B})^{\otimes2}
\index{DeltaMB@$\Delta^{\mathcal M,\B}$}
$$
the $\mathbf k$-module morphism determined by the conditions that $\Delta^{\mathcal M,\B}(1_\B)=(1_\B)^{\otimes2}$, and that it 
is compatible with the module structures on both sides and with the algebra morphism $\Delta^{\mathcal W,\B}$.  
Then $(\mathcal M^{\B},\Delta^{\mathcal M,\B})$ is a cocommutative and coassociative coalgebra. 

\begin{rem}
If $\mathfrak g$ is a Lie algebra and $\Gamma$ is a group acting on $\mathfrak g$ by Lie algebra 
automorphisms, then the semidirect product $U(\mathfrak g)\rtimes\Gamma$ is 
equipped with a cocommutative Hopf algebra structure, defined by the conditions that the elements 
of $\mathfrak g$ are primitive and the elements of $\Gamma$ are group-like. The Hopf algebra 
$(\mathcal W^\B,\Delta^{\mathcal W,\B})$ is isomorphic to $U(\mathfrak g^\B)\rtimes\mathbb Z$, 
with $\mathfrak g^\B$ the free Lie algebra over generators 
$(U_k^{\pm,(\ell)})_{k\geq 1,\ell\in\mathbb Z}$ and the action of $\mathbb Z$ on 
$\mathfrak g^\B$ being given by $1\cdot U_k^{\pm,(\ell)}=U_k^{\pm,(\ell+1)}$. The isomorphism 
is given by 
$\mathfrak g^\B[[t]]\ni\sum_{k\geq1}t^kU_k^{\pm,(\ell)}\mapsto X_1^\ell\cdot\mathrm{log}(1+\sum_{k\geq1}t^kY_k^\pm)\cdot X_1^{-\ell}\in\mathcal W^\B[[t]]$, $\mathbb Z\ni\pm 1\mapsto 
X_1^{\pm1}\in\mathcal W^\B$. 
\end{rem}

\begin{rem} The automorphism of $\mathcal V^\B$ induced by $X_i\mapsto X_i^{-1}$, $i=0,1$ restricts to a Hopf 
algebra automorphism $\theta$ of $(\mathcal W^\B,\Delta^{\mathcal W,\B})$ such that $X_1^\pm\mapsto X_1^\mp$ 
and $Y_k^\pm\mapsto Y_k^\mp$ for $k>0$. \end{rem}

\subsection{Filtrations and associated graded objects}\label{sect:filt:28oct}

\subsubsection{Associated graded object of $(\mathcal V^\B,\Delta^{\mathcal V,\B})$}\label{sect:241:28oct}

Recall the decreasing algebra filtration of the algebra $\mathcal V^\B$ (see \S\ref{sect:2:1:28oct}). 
The associated graded algebra $\mathrm{gr}(\mathcal V^\B)=\oplus_{n\geq0}\mathrm{gr}_n(\mathcal V^\B)
=\oplus_{n\geq0}F^n\mathcal V^\B/F^{n+1}\mathcal V^\B$ is generated in degree 1; according to \cite{Bbk}, chap. 2, \S5, no. 4, Theorem 2, 
there is an isomorphism of graded algebras
$$
\mathrm{gr}(\mathcal V^\B)\stackrel{\sim}{\to}\mathcal V^\DR, \quad 
\mathrm{gr}_1(\mathcal V^\B)\ni [X_i-1]\mapsto e_i\in (\text{degree 1 part of }\mathcal V^\DR)
\subset\mathcal V^\DR. 
$$
The algebra morphism $\Delta^{\mathcal V,\B}:\mathcal V^\B\to(\mathcal V^\B)^{\otimes2}$ is then compatible with the filtrations on both sides, 
and the associated graded algebra morphism coincides with $\Delta^{\mathcal V,\DR}:\mathcal V^\DR\to(\mathcal V^\DR)^{\otimes2}$. 

\subsubsection{Filtration on $(\mathcal W^\B,\Delta^{\mathcal W,\B})$ and associated graded object}

Recall the decreasing algebra filtration of the algebra $\mathcal W^\B$ (see \S\ref{sect:2:1:28oct}). 

\begin{lem}\label{lemma:20190926}
For $n\geq1$, one has 
$$
F^n\mathcal W^\B=F^{n-1}\mathcal V^\B\cdot(X_1-1)
$$
(equality of subspaces of $\mathcal V^\B$). 
\end{lem}

\proof Set $\underline e_i^\pm:=X_i^{\pm1}-1\in\mathcal V^\B$ for $i=0,1$. 
By (\ref{decomp:kF2}), one has 
\begin{equation}\label{decomp:VB+}
\mathcal V^\B_+=\mathcal V^\B\cdot\underline e_0^+\oplus\mathcal V^\B\cdot\underline e_1^+,
\end{equation} which implies
\begin{equation}\label{id:FnV}
\forall n\geq 0,\quad F^n\mathcal V^\B=(\mathcal V^\B_+)^n=\sum_{(a_1,\ldots,a_n)\in\{0,1\}^n}\mathcal V^\B\underline e_{a_1}^+\mathcal V^\B
\cdots\mathcal V^\B\underline e_{a_n}^+. 
\end{equation}
Since $F^n\mathcal V^\B$ is a right ideal in $\mathcal V^\B$, one has $F^n\mathcal V^\B=F^n\mathcal V^\B\cdot \mathcal V^\B$, which together with 
\eqref{id:FnV} implies 
\begin{equation}\label{id:FnV:bis}
\forall n\geq 0,\quad F^n\mathcal V^\B=\sum_{(a_1,\ldots,a_n)\in\{0,1\}^n}\mathcal V^\B\underline e_{a_1}^+\mathcal V^\B
\cdots\mathcal V^\B\underline e_{a_n}^+\mathcal V^\B. 
\end{equation}
Let $n\geq 1$. \eqref{id:FnV} implies 
$$
F^n\mathcal V^\B=\sum_{
\begin{matrix}
\scriptstyle{(a_1\ldots,a_{n-1})} \\  \ \ \scriptstyle{\in\{0,1\}^{n-1}}
\end{matrix}
}\mathcal V^\B\underline e_{a_1}^+\mathcal V^\B
\cdots\mathcal V^\B\underline e_{a_{n-1}}^+\mathcal V^\B\underline e_0^+
+
\sum_{
\begin{matrix}
\scriptstyle{(a_1\ldots,a_{n-1})} \\  \ \ \scriptstyle{\in\{0,1\}^{n-1}}
\end{matrix}
}
\mathcal V^\B\underline e_{a_1}^+\mathcal V^\B
\cdots\mathcal V^\B\underline e_{a_{n-1}}^+\mathcal V^\B\underline e_1^+, 
$$
which together with \eqref{decomp:VB+} and  \eqref{id:FnV:bis} for $n$ replaced by $n-1$ implies
\begin{equation}\label{useful:for:later:13jan2020}
F^n\mathcal V^\B=0\oplus F^{n-1}\mathcal V^\B\cdot \underline e_0^+\oplus F^{n-1}\mathcal V^\B\cdot \underline e_1^+\subset 
\mathbf k 1\oplus \mathcal V^\B\cdot \underline e_0^+\oplus \mathcal V^\B\cdot \underline e_1^+=\mathcal V^\B. 
\end{equation}
As $F^n\mathcal W^\B=F^n\mathcal V^\B\cap (\mathbf k 1\oplus 0\oplus \mathcal V^\B\cdot \underline e_1^+)$, one obtains 
$F^n\mathcal W^\B=F^{n-1}\mathcal V^\B\cdot \underline e_1^+$ 
\hfill\qed\medskip 

\begin{prop}\label{prop:computation:grVB:grWB}
The morphism of graded algebras $\mathrm{gr}(\mathcal W^\B)\to\mathrm{gr}(\mathcal V^\B)$ induced by the compatibility of the inclusion 
$\mathcal W^\B\subset\mathcal V^\B$ with the filtrations is injective. 
The image of $\mathrm{gr}(\mathcal W^\B)$ under the isomorphism $\mathrm{gr}(\mathcal V^\B)\simeq\mathcal V^\DR$ is equal to $\mathcal W^\DR\subset\mathcal V^\DR$.    
\end{prop}

\proof The injectivity of the map $\mathrm{gr}(\mathcal W^\B)\to\mathrm{gr}(\mathcal V^\B)$ follows from the fact that the filtration of 
$\mathcal W^\B$ is induced by  that of $\mathcal V^\B$. We have $\mathcal W^\B=\mathbf k\oplus F^1\mathcal W^\B$, which implies that 
\begin{equation}\label{id:decomp:grW}
\mathrm{gr}(\mathcal W^\B)=\mathbf k1\oplus\mathrm{gr}(F^1\mathcal W^\B)\subset\mathcal V^\DR, 
\end{equation}
where 
$$
\mathrm{gr}(F^1\mathcal W^\B)\subset\oplus_{n>0}\mathcal V^\DR_n. 
$$
The sequence of linear maps $\mathcal V^\B\stackrel{-\cdot(X_1-1)}{\longrightarrow}F^1\mathcal W^\B\subset\mathcal V^\B$ is compatible with the  
filtrations whose $n$th steps are given by $F^n\mathcal V^\B$, $F^{n+1}\mathcal W^\B$,  and $F^{n+1}\mathcal V^\B$. The 
associated graded sequence of maps $\mathcal V^\DR\simeq
\mathrm{gr}(\mathcal V^\B)\to\mathrm{gr}(F^1\mathcal W^\B)\to\mathrm{gr}(\mathcal V^\B)\simeq\mathcal V^\DR$ is given by 
$\mathcal V^\DR\to\mathcal V^\DR$, $x\mapsto xe_1$. It follows from Lemma \ref{lemma:20190926} that the first map 
$\mathrm{gr}(\mathcal V^\B)\to\mathrm{gr}(F^1\mathcal W^\B)$ is surjective. All this implies the identification 
$\mathrm{gr}(F^1\mathcal W^\B)\simeq\mathcal V^\DR e_1$, which together with \eqref{id:decomp:grW} implies the claim. 
\hfill\qed\medskip 

We now describe explicitly the filtration of $\mathcal W^\B$. Set $\underline e_i^\pm:=X_i^{\pm1}-1\in\mathcal V^\B$ for $i=0,1$,  
$$
\forall a_+,a_-\geq 0,\quad \underline y_{a_+,a_-}^\pm:=(\underline e_0^+)^{a_+}(\underline e_0^-)^{a_-}\underline e_1^\pm\in\mathcal W^\B, 
$$
and 
$$
\forall a\geq 1,\quad V_a:=\mathrm{Span}_{\mathbf k}( \underline y_{a_+,a_-}^\epsilon|a_++a_-\geq a-1, \quad \epsilon\in\{+,-\})\subset  \mathcal W^\B. 
$$
 
 \begin{lem}\label{lemma:2:9:20190926}
 1) For $n\geq1$, $F^n\mathcal W^\B$ is equal to the sum $\sum_{k\geq1}\sum_{a_1,\ldots,a_k|a_1+\cdots+a_k\geq n}V_{a_1}\cdots V_{a_k}$ of products (for the algebra structure of $\mathcal W^\B$) of subspaces of $\mathcal W^\B$. 
 
 2) $F^0\mathcal W^\B=\mathcal W^\B=\mathbf k1\oplus F^1\mathcal W^\B$. 
 \end{lem}

\proof One has $\mathcal V^\B=\mathbf k1\oplus F^1\mathcal V^\B$ and $\mathcal W^\B\supset\mathbf k1$, which implies 2). Let $n\geq1$. By Lemma \ref{lemma:20190926}, $F^n\mathcal W^\B=F^{n-1}\mathcal V^\B\cdot \underline e_1^+$. Since 
$X_1$ is invertible in $\mathcal V^\B$, one has $\mathcal V^\B\cdot X_1^{-1}=\mathcal V^\B$, which together with \eqref{id:FnV:bis} implies $F^{n-1}\mathcal V^\B\cdot X_1^{-1}=F^{n-1}\mathcal V^\B$. All this together with 
$X_1^{-1}\cdot\underline e_1^+=- \underline e_1^-$ implies 
\begin{equation}\label{id:FnW}
F^n\mathcal W^\B=F^{n-1}\mathcal V^\B\cdot \underline e_1^+ + F^{n-1}\mathcal V^\B\cdot \underline e_1^-.
\end{equation}
One has $\mathcal V^\B_+=\sum_{s\geq1}\sum_{((a_1,\epsilon_1),\ldots,(a_s,\epsilon_s))\in(\{0,1\}\times\{+,-\})^s}\mathbf k\cdot 
\underline e_{a_1}^{\epsilon_1}\cdots\underline e_{a_s}^{\epsilon_s}$, therefore 
$$
\forall n\geq 0,\quad F^n\mathcal V^\B=\sum_{s\geq n}\sum_{((a_1,\epsilon_1),\ldots,(a_s,\epsilon_s))\in(\{0,1\}\times\{+,-\})^s}\mathbf k\cdot 
\underline e_{a_1}^{\epsilon_1}\cdots\underline e_{a_s}^{\epsilon_s}, 
$$
which together with \eqref{id:FnW} implies 
$$
\forall n\geq 1,\quad F^n\mathcal W^\B=\sum_{s\geq n}\sum_{((a_1,\epsilon_1),\ldots,(a_{s-1},\epsilon_{s-1}))\in(\{0,1\}\times\{+,-\})^{s-1},
\epsilon_s\in\{+,-\}}\mathbf k\cdot 
\underline e_{a_1}^{\epsilon_1}\cdots\underline e_{a_{s-1}}^{\epsilon_{s-1}}\underline e_{1}^{\epsilon_{s}}, 
$$
The fact that $\underline e_0^+$ and $\underline e_0^-$ commute implies that the right-hand side of this equality can be 
identified with the sum of subspaces announced in 1). 
\hfill\qed\medskip 

Define linear maps 
$$
\underline\xi^\pm:\mathbf k[t,t^{-1}]\to\mathcal W^\B, \quad f\mapsto\underline\xi^\pm(f):=f(X_0)(X_1^{\pm1}-1).  
$$
\begin{lem}\label{lem:what:is:Va}
For $a\geq1$, one has 
$$
V_a=\underline\xi^+(((t-1)^{a-1}))+\underline\xi^-(((t-1)^{a-1}))\subset\mathcal W^\B, 
$$
where $((t-1)^{a-1}):=(t-1)^{a-1}\mathbf k[t,t^{-1}]\subset\mathbf k[t,t^{-1}]$ is the ideal of $\mathbf k[t,t^{-1}]$
generated by $(t-1)^{a-1}$. 
\end{lem}

\proof For $a',a''\geq0$, one has $\underline y_{a',a''}^\pm=\underline\xi^\pm((-t^{-1})^{a''}(t-1)^{a'+a''})$, therefore $V_a=\underline\xi^+(W_a)
+\underline\xi^-(W_a)$, where 
$$
W_a:=\mathrm{Span}\{t^{-a''}(t-1)^{a'+a''}|a',a''\geq0,a'+a''\geq a-1\}\subset\mathbf k[t,t^{-1}]. 
$$
One has 
\begin{align*}
& W_a=(t-1)^{a-1}\cdot\Big(\mathrm{Span}(1,\ldots,t^{1-a})+(t-1)\mathrm{Span}(1,\ldots,t^{-a})
+(t-1)^2\mathrm{Span}(1,\ldots,t^{-a-1})+\cdots\Big)
\\ & =(t-1)^{a-1}\cdot\Big(\mathrm{Span}(1,\ldots,t^{1-a})+\mathrm{Span}(t,t^{-a})
+\mathrm{Span}(t^2,\ldots,t^{-a-1})+\cdots\Big)=((t-1)^{a-1}). 
\end{align*}
\hfill\qed\medskip 

\begin{lem}\label{lemma:Delta:sharp:23012018} 
For any $k\in\mathbb Z$, one has 
\begin{equation}\label{cop:+}
\Delta^{\mathcal W,\B}(X_0^k(1-X_1))
=X_0^k(1-X_1)\otimes 1+1\otimes X_0^k(1-X_1)
+\sum_{i=1}^{k-1} X_0^{i}(1-X_1) \otimes X_0^{k-i}(1-X_1),  
\end{equation} 
\begin{equation}\label{cop:-}
\Delta^{\mathcal W,\B}(X_0^k(1-X_1^{-1})) 
=X_0^k(1-X_1^{-1})\otimes 1+1\otimes X_0^k(1-X_1^{-1})
-\sum_{i=0}^{k} X_0^{i}(1-X_1^{-1}) \otimes X_0^{k-i}(1-X_1^{-1}), 
\end{equation}
where for $A$ an abelian group, for $f:\mathbb Z\to A$ be a function and for $a,b\in\mathbb Z$, we set 
\begin{equation}\label{convention:sums}
\sum_{k=a}^b f(k):=\begin{cases} f(a)+\cdots+f(b) & \text{ if }b\geq a, \\ 0 & \text{ if }b=a-1, \\ 
-f(a-1)-\cdots-f(b+1)& \text{ if }b<a-1\end{cases}
\end{equation}
\end{lem}

\proof Let $t$ be a formal parameter. As we have seen, the series 
$s_\pm(t) :=1+\sum_{k\geq1}t^kY_k^\pm$ are group-like for $\Delta^{\mathcal W,\B}$. 
One computes these series as follows 
$$s_\pm(t)=1+tX_0^{\pm1}\{1-t(X_0^{\pm1}-1)\}^{-1}(1-X_1^{\pm1})=(1+t-tX_0^{\pm1})^{-1}
(1+t-tX_0^{\pm1}X_1^{\pm1})=\tilde s_\pm(u), $$
where $u :=t/(1+t)$ and $\tilde s_\pm(u ) :=(1-uX_0^{\pm1})^{-1}(1-uX_0^{\pm1}X_1^{\pm1})$. 
It follows that the series $\tilde s_\pm(u)$ are group-like. Together with their expansions 
$\tilde s_\pm(u)=1+\sum_{k\geq1}u^k X_0^{\pm k}(1-X_1^{\pm1})$, this implies \eqref{cop:+} for 
$k\geq1$ and \eqref{cop:-} for $k\leq-1$. Since $X_1^{\pm1}$ are group-like, \eqref{cop:+} for $k$
is equivalent to \eqref{cop:-} for $-k$, which implies \eqref{cop:+} for $k\leq-1$ and \eqref{cop:+} for 
$k\geq1$ . Finally, \eqref{cop:+} and \eqref{cop:-} for $k=0$ are direct consequences of the fact that 
$X_1^{\pm1}$ are group-like.   \hfill \qed\medskip 

\begin{lemdef}\label{tswdbtcotlmwti}
The assignments $f\mapsto {{t'f(t'')-t''f(t')}\over{t'-t''}}$, $f\mapsto {{t'f(t')-t''f(t'')}\over{t'-t''}}$  define linear maps $\mathbf k[t,t^{-1}]\to
\mathbf k[t',(t')^{-1},t'',(t'')^{-1}]$.  We denote by 
$$
\mathrm{Op}^\pm:\mathbf k[t,t^{-1}]\mapsto\mathbf k[t,t^{-1}]^{\otimes 2}
$$
the composition of this linear map with the isomorphism of their target algebra with $\mathbf k[t,t^{-1}]^{\otimes 2}$ 
given by $t'\mapsto t\otimes 1$, $t''\mapsto1\otimes t$. 
\end{lemdef}

\proof For $a\in\mathbb Z$, the image of $t^a$ by the first map is $-t't''(t^{\prime a-2}+\cdots
+t^{\prime\prime a-2})$ if $a\geq 1$, and $(t't'')^a(t^{\prime -a}+\cdots+t^{\prime\prime -a})$ if $a\leq 0$, which implies the result
in this case. The result in the second case follows from the identity $\forall f\in\mathbf k[t,t^{-1}]$, 
$\mathrm{Op}^+(f)+\mathrm{Op}^-(f)=f\otimes1+1\otimes f$.  \hfill\qed\medskip  

\begin{lem}\label{lemma:special:form:of:image:of:underline:X}
The map $\Delta^{\mathcal W,\B}:\mathcal W^{\B}\to(\mathcal W^{\B})^{\otimes2}$ is such that 
$$
\forall f\in\mathbf k[t,t^{-1}],\quad \underline\xi^\pm(f)\mapsto\underline\xi^\pm(f)\otimes1+1\otimes\underline{\xi}^\pm(f)
+(\underline{\xi}^\pm)^{\otimes2}(\mathrm{Op}^\pm(f)). 
$$
\end{lem}
\proof Let $k\in\mathbb Z$ and $f:=t^k$. Then  
\begin{align*}
&\underline{\xi}^+(f)=X_0^k(X_1-1)\mapsto \underline{\xi}^+(f)\otimes1+1\otimes\underline{\xi}^+(f)-\sum_{i=1}^{k-1}
X_0^{k-i}(X_1-1)\otimes X_0^i(X_1-1)
= \underline{\xi}^+(f)\otimes1
\\ & +1\otimes\underline{\xi}^+(f)+(\underline{{\xi}}^+)^{\otimes2}(-\sum_{i=1}^{k-1}t^{\prime k-i} t^{\prime\prime i})
=\underline{\xi}^+(f)\otimes1+1\otimes\underline{{\xi}}^+(f)+(\underline{{\xi}}^+)^{\otimes2}({{t'f(t'')-t''f(t')}\over{t'-t''}}),  
\end{align*}
which implies the result in the case of $\underline{\xi}^+$. The result for $\underline{\xi}^-$ follows from this, from 
$\underline{\xi}^-(f)=-\underline{\xi}^+(f)X_1^{-1}$, and from the relation between  $\mathrm{Op}^+$ and $\mathrm{Op}^-$
(see proof of Lemma \ref{tswdbtcotlmwti}). 
\hfill \qed\medskip 

\begin{lem}\label{lemma:image:of:ideal:t-1:to:the:n-1}
For $n\geq 2$, each of the maps $\mathrm{Op}^\pm$ defined in 
Lemma-Definition \ref{tswdbtcotlmwti}
takes $((t-1)^{n-1})$ (see Lemma \ref{lem:what:is:Va}) to $\sum_{a=1}^{n-1}((t-1)^{a-1})\otimes((t-1)^{n-a-1})$. 
\end{lem}

\proof Assume that $f\in((t-1)^{n-1})$, that is $f(t)=(t-1)^{n-1}a(t)$, with $a(t)\in\mathbf k[t^{\pm1}]$. Then 
\begin{align*}
& {{t'f(t'')-t''f(t')}\over{t'-t''}}={{t'(t''-1)^{n-1}a(t'')-t''(t'-1)^{n-1}a(t')}\over{t'-t''}}
\\ & =(t''-1)^{n-1}{t'{a(t'')}-t''{a(t')}\over t'-t''}+t''{a(t')}{(t''-1)^{n-1}-(t'-1)^{n-1}\over t'-t''}
\end{align*}
The second term of the last line belongs to the announced space $\sum_{a=1}^{n-1}((t-1)^{a-1})\otimes((t-1)^{n-a-1})$, 
while the first term belongs to $\mathbf k[t,t^{-1}]\otimes ((t-1)^{n-1})$, which is contained in this space. This proves the result in the case 
of $\mathrm{Op}^+$. This and the  relation between  $\mathrm{Op}^+$ and $\mathrm{Op}^-$ (see proof of Lemma \ref{tswdbtcotlmwti}) imply 
the result in the case of $\mathrm{Op}^-$. \hfill \qed\medskip 

\begin{prop}\label{lemma:compatibility:Delta:l/r:ideals}
The morphism $\Delta^{\mathcal W,\B}$ is compatible with the filtration of $\mathcal W^\B$, in other words, if 
$n\geq 0$, then $\Delta^{\mathcal W,\B}(F^n\mathcal W^\B)\subset\sum_{a,b\geq0|a+b=n}F^a\mathcal W^\B\otimes F^b\mathcal W^\B$. 
\end{prop}

\proof Lemmas \ref{lem:what:is:Va}, \ref{lemma:special:form:of:image:of:underline:X} and \ref{lemma:image:of:ideal:t-1:to:the:n-1} imply 
that $\forall a\geq1$, $\Delta^{\mathcal W,\B}(V_a)\subset V_a\otimes1+1\otimes V_a+\sum_{a',a''\geq1,a'+a''=a}V_{a'}\otimes V_{a''}$. 
The result then follows from Lemma \ref{lemma:2:9:20190926}. 
\hfill \qed\medskip 

\begin{prop}\label{lem:49:23:08:2017}
The associated graded morphism of $\Delta^{\mathcal W,\B}:\mathcal W^{\B}\to(\mathcal W^{\B})^{\otimes2}$ 
with respect to the filtrations of both sides identifies with $\Delta^{\mathcal W,\DR}:\mathcal W^{\DR}\to(\mathcal W^{\DR})^{\otimes2}$ 
under the isomorphisms $\mathrm{gr}(\mathcal W^{\B})\simeq\mathcal W^{\DR}$. 
\end{prop}

\proof This follows from the fact that the elements $y_n,n\geq1$ generate $\mathcal W^{\DR}$, from the second part of \eqref{formula:Delta:W:B} 
for $\pm=+$, and from the fact that for $n\geq1$, the class of $Y_n^+=X_0(X_0-1)^{n-1}(1-X_1)\in F^1\mathcal W^{\B}$ in 
$\mathrm{gr}_1(\mathcal W^{\B})$ is $e_0^{n-1}\cdot(-e_1)=y_n$.  \hfill\qed\medskip 

\subsubsection{Associated graded object of $(\mathcal M^\B,\Delta^{\mathcal M,\B})$}\label{sect:243:28oct}

Recall the decreasing filtration $\mathcal M^\B=F^0\mathcal M^\B\supset F^1\mathcal M^\B\supset\cdots$ on the $\mathbf k$-module 
$\mathcal M^\B$, where for $i\geq 0$, $F^i\mathcal M^\B=F^i\mathcal V^\B\cdot 1_\B$ (see \S\ref{sect:2:1:28oct}). 

\begin{lem}\label{lem:eq:two:filtrations:13jan2020}
For $i\geq0$, one has $F^i\mathcal M^\B=F^i\mathcal W^\B\cdot 1_\B$. 
\end{lem}

\proof For $i\geq1$, one has $F^i\mathcal V^\B\cdot 1_\B=\Big(F^{i-1}\mathcal V^\B(X_0-1)+F^{i-1}\mathcal V^\B(X_1-1)\Big)\cdot 1_\B
=F^{i-1}\mathcal V^\B(X_0-1)\cdot 1_\B+F^{i-1}\mathcal V^\B(X_1-1)\cdot 1_\B=F^{i-1}\mathcal V^\B(X_1-1)\cdot 1_\B=F^i\mathcal W^\B
\cdot 1_\B$, where the first (resp. last) equality follows from \eqref{useful:for:later:13jan2020} (resp. Lemma \ref{lemma:20190926}). 
\hfill\qed\medskip 

\begin{lem}\label{lem:sect:243:28oct}
$(\mathcal M^\B,\Delta^{\mathcal M,\B})$ is a coalgebra in $\mathbf k\text{-mod}_{\mathrm{fil,+}}$; its image under $\mathrm{gr}$
is the coalgebra $(\mathcal M^\DR,\Delta^{\mathcal M,\DR})$ in $\mathbf k\text{-mod}_{\mathrm{gr,+}}$. 
\end{lem}

\proof The first statement follows from Proposition \ref{lemma:compatibility:Delta:l/r:ideals} and from 
$\Delta^{\mathcal M,\B}(1_\B)=(1_\B)^{\otimes2}$. The morphism $(-)\cdot 1_\B:\mathcal V^\B\to\mathcal M^\B$ in 
$\mathbf k\text{-mod}_{\mathrm{fil,+}}$ induces a morphism $\mathcal V^\DR=\mathrm{gr}(\mathcal V^\B)
\stackrel{\mathrm{gr}((-)\cdot 1_\B)}{\to}\mathrm{gr}(\mathcal M^\B)$. If one equips $\mathcal V^\B$ with the shifted 
filtration $F^n(\mathcal V^\B[1]):=F^{n-1}(\mathcal V^\B)$, then $(-)\cdot(X_0-1):\mathcal V^\B[1]\to\mathcal V^\B$
is a morphism in $\mathbf k\text{-mod}_{\mathrm{fil,+}}$, such that the composed morphism 
$\mathcal V^\B[1]\stackrel{(-)\cdot(X_0-1)}{\to}\mathcal V^\B\stackrel{(-)\cdot 1_\B}{\to}\mathcal M^\B$ is zero. 
Applying $\mathrm{gr}$, one obtains that the composed morphism $\mathcal V^\DR[1]\stackrel{(-)\cdot e_0}{\to}
\mathcal V^\DR\to\mathrm{gr}(\mathcal M^\B)$ is zero, therefore a morphism $\mathcal M^\DR\to\mathrm{gr}(\mathcal M^\B)$.  

By Lemma \ref{lem:eq:two:filtrations:13jan2020}, the map $\mathcal W^\B\to\mathcal M^\B$, $x\mapsto x\cdot 1^\B$ is an 
isomorphism in $\mathbf k\text{-mod}_{\mathrm{fil,+}}$, therefore it gives rise to an isomorphism 
$\mathrm{gr}(\mathcal W^\B)\to\mathrm{gr}(\mathcal M^\B)$. The following diagram 
commutes
\begin{equation}\label{comm:diag:W:M:23dec2019}
\xymatrix{ \mathrm{gr}(\mathcal W^\B)\ar^{\mathrm{gr}((-)\cdot 1_\B)}_{\simeq}[rr]&& \mathrm{gr}(\mathcal M^\B)\\ \mathcal W^\DR\ar^\simeq_{(-)\cdot 1_\DR}[rr]\ar^\simeq[u]&& \mathcal M^\DR\ar[u]}
\end{equation}
which implies that $\mathcal M^\DR\to\mathrm{gr}(\mathcal M^\B)$ is an isomorphism. In the following diagram 
$$
\xymatrix{
\mathrm{gr}(\mathcal W^\B)^{\otimes2}\ar^{(\mathrm{gr}((-)\cdot 1_\B))^{\otimes2}}_{\simeq}[rrr]&&&\mathrm{gr}(\mathcal M^\B)^{\otimes2} \\ 
&\mathrm{gr}(\mathcal W^\B)\ar^{\mathrm{gr}((-)\cdot 1_\B)}_{\simeq}[r]\ar^{\mathrm{gr}(\Delta^{\mathcal W,\B})}[ul]&
\mathrm{gr}(\mathcal M^\B)\ar_{\mathrm{gr}(\Delta^{\mathcal M,\B})}[ur]& \\ 
&\mathcal W^\DR\ar_{(-)\cdot 1_\DR}^{\simeq}[r]\ar^{\Delta^{\mathcal W,\DR}}[dl]\ar^\simeq[u]&
\mathcal M^\DR\ar_{\Delta^{\mathcal M,\DR}}[dr]\ar_\simeq[u]& \\ 
(\mathcal W^\DR)^{\otimes2}\ar_{((-)\cdot 1_\DR)^{\otimes2}}^\simeq[rrr]\ar^\simeq[uuu]&&&(\mathcal M^\DR)^{\otimes2}\ar_\simeq[uuu]
}
$$
the inner square commutes by \eqref{comm:diag:W:M:23dec2019}, the external square commutes as it is the tensor square of 
\eqref{comm:diag:W:M:23dec2019}, the lower quadrangle commutes by the definition of $\Delta^{\mathcal M,\DR}$, the left quadrangle 
commutes by Proposition \ref{lem:49:23:08:2017}, the upper quadrangle commutes because $(-)\cdot 1_\B$ induces an isomorphism 
of the coalgebras $(\mathcal W^\B,\Delta^{\mathcal W,\B})$ and $(\mathcal M^\B,\Delta^{\mathcal M,\B})$ in 
$\mathbf k\text{-mod}_{\mathrm{fil,+}}$.  It follows that the right quadrangle commutes. 
\hfill\qed\medskip 

One checks that $\mathcal M^\B$ is a module in $\mathbf k\text{-mod}_{\mathrm{fil,+}}$ over $\mathcal V^\B$, therefore over $\mathcal W^\B$, 
and that the image of this module structure under $\mathrm{gr}$
is the module structure of $\mathcal M^\DR$ over $\mathcal V^\DR$ and $\mathcal W^\DR$ in 
$\mathbf k\text{-mod}_{\mathrm{gr,+}}$.

\subsection{Completions}\label{sect:compl:28oct}

Extending the above definitions of $F^k\mathcal X^\B$ by $F^k\mathcal X^\B:=\mathcal X^\B$ for $k<0$ and 
$\mathcal X\in\{\mathcal V,\mathcal W,\mathcal M\}$, the pairs $({\mathcal V}^\B,\Delta^{\mathcal V,\B})$, 
$({\mathcal W}^\B,\Delta^{\mathcal W,\B})$ (resp. $({\mathcal M}^\B,\Delta^{\mathcal M,\B})$) are Hopf algebras (resp. is a 
coassociative coalgebra) in $\mathbf k\text{-mod}_{\fil,+}$. Applying the functor $(-)^\wedge:\mathbf k\text{-mod}_{\fil,+}
\to\mathbf k\text{-mod}_{\topo}$ (see \S\ref{sect:functors}), one obtains topological Hopf algebras
$(\hat{\mathcal V}^\B\index{V^B^@$\hat{\mathcal V}^\B$},\hat\Delta^{\mathcal V,\B}\index{DeltaVB^@$\hat\Delta^{\mathcal V,\B}$})$ and 
$(\hat{\mathcal W}^\B\index{W^B^@$\hat{\mathcal W}^\B$},\hat\Delta^{\mathcal W,\B}\index{DeltaWB^@$\hat\Delta^{\mathcal W,\B}$})$, and a topological coalgebra 
$(\hat{\mathcal M}^\B\index{M^B^@$\hat{\mathcal M}^\B$},\hat\Delta^{\mathcal M,\B}\index{DeltaMB^@$\hat\Delta^{\mathcal M,\B}$})$.
It follows from the compatibility of the decomposition \eqref{decomp:kF2} 
with filtrations that $\hat{\mathcal M}^\B$ is free of rank one as a $\hat{\mathcal W}^\B$-module. 

It follows from \S\ref{sect:241:28oct}, Proposition \ref{lem:49:23:08:2017}, and Lemma \ref{lem:sect:243:28oct} that  
the image of the pair $({\mathcal X}^\B,\Delta^{\mathcal X,\B})$ by the functor $\mathrm{gr}:\mathbf k\text{-mod}_{\fil,+}
\to\mathbf k\text{-mod}_{\gr,+}$ (see \S\ref{sect:functors}) is $({\mathcal X}^\DR,\Delta^{\mathcal X,\DR})$ 
for $\mathcal X\in\{\mathcal V,\mathcal W,\mathcal M\}$. 

\section{Fake comparison isomorphisms}\label{sect:fci:24dec2019}

In \S\S\ref{sect:1} and \ref{sect:2:28oct}, we introduced parallel `de Rham' and `Betti' structures, namely an inclusion 
of topological algebras $\hat{\mathcal W}^\DR\hookrightarrow\hat{\mathcal V}^\DR$ (resp. $\hat{\mathcal W}^\B\hookrightarrow
\hat{\mathcal V}^\B$), and a surjective $\hat{\mathcal V}^\DR$-module (resp. $\hat{\mathcal V}^\B$-module) morphism 
$(-)\cdot 1_\DR:\hat{\mathcal V}^\DR\twoheadrightarrow\hat{\mathcal M}^\DR$ (resp. $(-)\cdot 1_\B:\hat{\mathcal V}^\B
\twoheadrightarrow\hat{\mathcal M}^\B$), in which $\hat{\mathcal V}^\DR$ (resp. $\hat{\mathcal V}^\B$) is the left regular module 
over itself. We introduce automorphisms of the de Rham algebra structures in \S\ref{sect:3:1:28oct}, and of the de Rham module 
structures in \S\ref{sect:3:2:28oct}. In \S\ref{sect:3:3:28oct}, we define isomorphisms between the various Betti and 
de Rham structures. In \S\ref{sect:3:4:28oct}, we relate these isomorphisms with the Betti-de Rham comparison isomorphisms, 
either of the fundamental group of $\mathfrak M_{0,4}$ with basepoint $\vec 1$, or of the fundamental groupoid of the same space 
with respect to the pair of basepoints $(\vec 1,\vec 0)$. 

\subsection{The algebra automorphisms $\mathrm{aut}_{(\mu,g)}^{\mathcal V,(1),\DR}$ and 
$\mathrm{aut}_{(\mu,g)}^{\mathcal W,(1),\DR}$}\label{sect:3:1:28oct}

For $(\mu,g)\in\mathbf k^\times\times\mathcal G(\hat{\mathcal V}^\DR)$, we denote by 
$\mathrm{aut}_{(\mu,g)}^{\mathcal V,(1),\DR}$ the automorphism of the topological $\mathbf k$-algebra 
$\hat{\mathcal V}^\DR$ given by 
$$
\mathrm{aut}_{(\mu,g)}^{\mathcal V,(1),\DR} : e_0\mapsto g\cdot \mu e_0\cdot g^{-1},\quad 
e_1\mapsto\mu e_1. \index{autV1@$\mathrm{aut}_{(\mu,g)}^{\mathcal V,(1),\DR}$} 
$$
One checks that this automorphism restricts to an automorphism of the subalgebra 
$\hat{\mathcal W}^\DR\subset\hat{\mathcal V}^\DR$, which we denote 
$\mathrm{aut}_{(\mu,g)}^{\mathcal W,(1),\DR}$\index{autW1@$\mathrm{aut}_{(\mu,g)}^{\mathcal W,(1),\DR}$}. 

\subsection{The module automorphisms $\mathrm{aut}_{(\mu,g)}^{\mathcal V,(10),\DR}$ and 
$\mathrm{aut}_{(\mu,g)}^{\mathcal M,(10),\DR}$}\label{sect:3:2:28oct}

For $(\mu,g)\in\mathbf k^\times\times\mathcal G(\hat{\mathcal V}^\DR)$. Denote by 
$\mathrm{aut}_{(\mu,g)}^{\mathcal V,(10),\DR}$ the automorphism of the topological $\mathbf k$-module 
$\hat{\mathcal V}^\DR$ given by 
\begin{equation}\label{relation:aut1:aut10}
\forall v\in\hat{\mathcal V}^\DR,\quad \mathrm{aut}_{(\mu,g)}^{\mathcal V,(10),\DR}(v):= 
\mathrm{aut}_{(\mu,g)}^{\mathcal V,(1),\DR}(v)\cdot g. 
\index{autV10@$\mathrm{aut}_{(\mu,g)}^{\mathcal V,(10),\DR}$}  
\end{equation}
One checks that $\mathrm{aut}_{(\mu,g)}^{\mathcal V,(10),\DR}$ leaves the $\mathbf k$-submodule 
$\hat{\mathcal V}^\DR e_0\subset\hat{\mathcal V}^\DR$ invariant, and therefore induces an automorphism 
of the topological $\mathbf k$-module $\hat{\mathcal M}^\DR=\hat{\mathcal V}^\DR/\hat{\mathcal V}^\DR e_0$, 
which will be denoted $\mathrm{aut}_{(\mu,g)}^{\mathcal M,(10),\DR}$\index{autM10@$\mathrm{aut}_{
(\mu,g)}^{\mathcal M,(10),\DR}$}. 

One checks that for $\mathcal X\in\{\mathcal V,\mathcal M\}$ and $\mathcal A\in\{\mathcal V,\mathcal W\}$, the $\mathbf k$-module
automorphism $\mathrm{aut}_{(\mu,g)}^{\mathcal X,(10),\DR}$ is compatible (in the sense of \S\ref{conventions:0703}) with the 
module structure of $\hat{\mathcal X}^\DR$ over the algebra $\hat{\mathcal A}^\DR$ and with the algebra automorphism 
$\mathrm{aut}_{(\mu,g)}^{\mathcal A,(1),\DR}$. 

\subsection{The isomorphisms $\mathrm{comp}_{(\mu,\Phi)}^{\mathcal V,(1)}$, $\mathrm{comp}_{(\mu,\Phi)}^{\mathcal W,(1)}$, 
$\mathrm{comp}_{(\mu,\Phi)}^{\mathcal V,(10)}$ and $\mathrm{comp}_{(\mu,\Phi)}^{\mathcal M,(10)}$}\label{sect:3:3:28oct}

There is a unique isomorphism 
$$
\mathrm{iso}^{\mathcal V}:\hat{\mathcal V}^\B\to\hat{\mathcal V}^\DR
\index{isoV@$\mathrm{iso}^{\mathcal V}$}
$$
of topological $\mathbf k$-algebras, defined by the condition $X_i\mapsto\mathrm{exp}(e_i)$ for $i=0,1$. This isomorphism 
restricts to an isomorphism of topological $\mathbf k$-algebras
$$
\mathrm{iso}^{\mathcal W}:\hat{\mathcal W}^\B\to\hat{\mathcal W}^\DR.
\index{isoW@$\mathrm{iso}^{\mathcal W}$}  
$$
There is a commutative diagram 
\begin{equation}\label{diag:M:12032019}
\xymatrix{
\hat{\mathcal V}^\B\ar^{\mathrm{iso}^{\mathcal V}}[r]\ar_{-\cdot(X_0-1)}[d]& 
\hat{\mathcal V}^\DR \ar^{-\cdot(\mathrm{exp}(e_0)-1)}[d]\\
\hat{\mathcal V}^\B \ar^{\mathrm{iso}^{\mathcal V}}[r]& \hat{\mathcal V}^\DR }
\end{equation}
The cokernels of the endomorphisms $-\cdot (\mathrm{exp}(e_0)-1)$ and $-\cdot e_0$ of $\hat{\mathcal V}^\DR$ coincide as 
$-\cdot (\mathrm{exp}(e_0)-1)=(-\cdot e_0)\circ(-\cdot {{\mathrm{exp}(e_0)-1}\over{e_0}})$ and 
$-\cdot {{\mathrm{exp}(e_0)-1}\over{e_0}}$ is a linear automorphism. 
Taking vertical cokernels of the above diagram, we obtain a $\mathbf k$-module isomorphism 
$$
{\mathrm{iso}}^{\mathcal M}:\hat{\mathcal M}^\B\to\hat{\mathcal M}^\DR
\index{isoM@$\mathrm{iso}^{\mathcal M}$} 
$$
which is moreover compatible with the algebra isomorphism ${\mathrm{iso}}^{\mathcal V}$. 

Let $(\mu,g)\in\mathbf k^\times\times\mathcal G(\hat{\mathcal V}^\DR)$. 
Composing the algebra automorphism $\mathrm{aut}^{\mathcal W,(1),\DR}_{(\mu,g)}$ 
(resp. $\mathrm{aut}^{\mathcal V,(1),\DR}_{(\mu,g)}$) of 
$\hat{\mathcal W}^\DR$ (resp. $\hat{\mathcal V}^\DR$) with the algebra
isomorphism $\mathrm{iso}^{\mathcal W}$ (resp. $\mathrm{iso}^{\mathcal V}$), 
we obtain algebra isomorphisms 
\begin{equation}\label{def:comp:1:16dec2019}
\mathrm{comp}_{(\mu,g)}^{\mathcal W,(1)}=\mathrm{aut}^{\mathcal W,(1),\DR}_{(\mu,g)}\circ\mathrm{iso}^{\mathcal W} : 
\hat{\mathcal W}^\B\to\hat{\mathcal W}^\DR,
\index{compW1@$\mathrm{comp}_{(\mu,g)}^{\mathcal W,(1)}$}
\quad
\mathrm{comp}_{(\mu,g)}^{\mathcal V,(1)}=\mathrm{aut}^{\mathcal V,(1),\DR}_{(\mu,g)}\circ\mathrm{iso}^{\mathcal V} : 
\hat{\mathcal V}^\B\to\hat{\mathcal V}^\DR.
\index{compV1@$\mathrm{comp}_{(\mu,g)}^{\mathcal V,(1)}$}
\end{equation}
Composing the $\mathbf k$-module automorphism $\mathrm{aut}_{(\mu,g)}^{\mathcal V,(10),\DR}$ (resp. 
$\mathrm{aut}_{(\mu,g)}^{\mathcal M,(10),\DR}$) of $\hat{\mathcal V}^\DR$ (resp. $\hat{\mathcal M}^\DR$)
with the $\mathbf k$-module isomorphism ${\mathrm{iso}}^{\mathcal V}$ (resp. ${\mathrm{iso}}^{\mathcal M}$), one obtains 
$\mathbf k$-module isomorphisms  
\begin{equation}\label{def:comp:10:16dec2019}
\mathrm{comp}_{(\mu,g)}^{\mathcal V,(10)}:=\mathrm{aut}_{(\mu,g)}^{\mathcal V,(10),\DR}\circ 
{\mathrm{iso}}^{\mathcal V}:\hat{\mathcal V}^\B\to\hat{\mathcal V}^\DR,
\index{compV10@$\mathrm{comp}_{(\mu,g)}^{\mathcal V,(10)}$}
\quad
\mathrm{comp}_{(\mu,g)}^{\mathcal M,(10)}:=\mathrm{aut}_{(\mu,g)}^{\mathcal M,(10),\DR}\circ 
{\mathrm{iso}}^{\mathcal M}:\hat{\mathcal M}^\B\to\hat{\mathcal M}^\DR,
\index{compM10@$\mathrm{comp}_{(\mu,g)}^{\mathcal M,(10)}$}  
\end{equation} 
which are compatible with the previous algebra isomorphisms. 

\begin{lem}\label{lemma:3:1:23dec2019} Let $(\mu,g)\in\mathbf k^\times\times\mathcal G(\hat{\mathcal V}^\DR)$. Then 
\begin{equation}\label{relation:comp10:comp1:16dec2019}
\forall v^\B\in\hat{\mathcal V}^\B,\quad \mathrm{comp}_{(\mu,g)}^{\mathcal V,(10)}(v^\B)=
\mathrm{comp}_{(\mu,g)}^{\mathcal V,(1)}(v^\B)\cdot g. 
\end{equation} 
\end{lem}

\proof This follows from the definitions of $\mathrm{comp}_{(\mu,g)}^{\mathcal V,(1)}$ (see second part of 
\eqref{def:comp:1:16dec2019}), of $\mathrm{comp}_{(\mu,g)}^{\mathcal V,(10)}$ (see first part of 
\eqref{def:comp:10:16dec2019}), and from \eqref{relation:aut1:aut10}, with $v$ being replaced by 
$\mathrm{iso}_{\mathcal V}(v^\B)$. \hfill\qed\medskip 

\begin{lem}\label{lem:com:pi:V:M:comp:17dec2019}
Let $(\mu,g)\in\mathbf k^\times\times\mathcal G(\hat{\mathcal V}^\DR)$. Then the following diagram commutes
$$
\xymatrix{ \hat{\mathcal V}^\B\ar^{\mathrm{comp}_{(\mu,g)}^{\mathcal V,(10)}}[rr]\ar_{(-)\cdot 1_\B}[d]&& \hat{\mathcal V}^\DR
\ar^{(-)\cdot 1_\DR}[d]\\ 
\hat{\mathcal M}^\B\ar_{\mathrm{comp}_{(\mu,g)}^{\mathcal M,(10)}}[rr]&&\hat{\mathcal M}^\DR}
$$
\end{lem}

\proof This follows from $\mathrm{comp}_{(\mu,g)}^{\mathcal X,(10)}=\mathrm{aut}_{(\mu,g)}^{\mathcal X,(10),\DR}\circ
\mathrm{iso}^{\mathcal X}$ for $\mathcal X\in\{\mathcal V,\mathcal M\}$, from $\mathrm{iso}^{\mathcal M}\circ((-)\cdot 1_\B)=
((-)\cdot 1_\DR)\circ\mathrm{iso}^{\mathcal V}$, and from $\mathrm{aut}^{\mathcal M,(10),\DR}_{(\mu,g)}\circ((-)\cdot 1_\DR)
=((-)\cdot 1_\DR)\circ\mathrm{aut}^{\mathcal V,(10),\DR}_{(\mu,g)}$ (see \S\ref{sect:3:2:28oct}). \hfill\qed\medskip 

\subsection{Relation with comparison isomorphisms}\label{sect:3:4:28oct}

\subsubsection{} Recall from \S0.1.1 the notation $\mathcal T=\mathrm{MT}(\mathbb Z)_{\mathbb Q}$ and the motivic Galois group 
schemes $\mathsf{Aut}_{\mathcal T}^\otimes(\omega_?)$  for $?\in\{\B,\DR\}$ (\cite{DG}). By \cite{An}, there is an inclusion of 
$\mathbb Q$-group schemes $\mathsf{Aut}_{\mathcal T}^\otimes(\omega_?)\subset \mathsf{G}^?$, where 
$\mathsf{G}^?$ is the $\mathbb Q$-group scheme defined in \cite{EF2}, \S1.6 for ?=DR and \cite{EF3}, \S2.1 for ?=B.

\subsubsection{}  For $\mathsf H$ a $\mathbb Q$-group scheme, denote by $\mathsf H$-mod the tensor category of its finite-dimensional linear representations. An object of $\mathsf H$-mod is then a pair $(V,\rho)$ where $V$ is a finite dimensional $\mathbb Q$-vector 
space and $\rho : \mathsf H\to\mathsf{GL}(V)$ is a morphism of $\mathbb Q$-group schemes. For $?\in\{\B,\DR\}$, there is a 
sequence of tensor functors $\mathsf{G}^?$-mod$\stackrel{\mathrm{res}_?}{\to}\mathsf{Aut}_{\mathcal T}^\otimes(\omega_?)
$-mod$\simeq\mathcal T$ where the first functor is induced by restriction and the second is the Tannaka equivalence (see \cite{DG}, 
\S2.1). Its composition with $\omega_?$ is the forgetful functor $\mathrm{forget}:\mathsf{G}^?$-mod$\to
\mathrm{Vec}_{\mathbb Q}$.  

\subsubsection{}  
For $?\in\{\B,\DR\}$, set $\hat{\mathcal V}^{?,(1)}_{\mathbb Q}:=(\hat{\mathcal V}^?_{\mathbb Q},(\mu,g)\mapsto
\mathrm{aut}^{\mathcal V,(1),?}_{(\mu,g)})$ and $\hat{\mathcal V}^{?,(10)}_{\mathbb Q}:=
(\hat{\mathcal V}^?_{\mathbb Q},(\mu,g)\mapsto\mathrm{aut}^{\mathcal V,(10),?}_{(\mu,g)})$ (see \S\S\ref{sect:3:1:28oct}
and \ref{sect:3:2:28oct} for ?=DR and \cite{EF3}, \S2.2 for ?=B). By \cite{EF2}, \S1.6 (?=DR), and \cite{EF3}, \S2.2 (?=B), 
these are pro-objects in $\mathsf{G}^?$-mod. Denote by $m_?$ the product of $\hat{\mathcal V}^?_{\mathbb Q}$ and 
recall that $\hat\Delta^{\mathcal V,?}$ denotes its 
coproduct. The category Pro($\mathsf{G}^?$-mod) of pro-objects in $\mathsf{G}^?$-mod is equipped with 
a tensor product, and $(m_?,\hat\Delta^{\mathcal V,?})$ defines a Hopf algebra structure on 
$\hat{\mathcal V}^{?,(1)}_{\mathbb Q}$ in Pro($\mathsf{G}^?$-mod), as well as a module-coalgebra 
structure over it on $\hat{\mathcal V}^{?,(10)}_{\mathbb Q}$ in the same category (this meaning that 
$\hat{\mathcal V}^{?,(10)}_{\mathbb Q}$ is a left $\hat{\mathcal V}^{?,(1)}_{\mathbb Q}$-module and 
a coassociative coalgebra, the left action morphism being a coalgebra morphism), both Hopf algebra and coalgebra structures
being cocommutative. 

\subsubsection{} Denote by $M^\vee$ the dual of an object $M$ of Pro($\mathsf{G}^?$-mod); this is an object in 
the category Ind($\mathsf{G}^?$-mod) of ind-objects in $\mathsf{G}^?$-mod. Then 
$(\hat{\mathcal V}^{?,(1)}_{\mathbb Q})^\vee$ and $(\hat{\mathcal V}^{?,(10)}_{\mathbb Q})^\vee$ 
are objects in Ind($\mathsf{G}^?$-mod); in this category, $(\hat{\mathcal V}^{?,(1)}_{\mathbb Q})^\vee$ is a commutative 
Hopf algebra and $(\hat{\mathcal V}^{?,(10)}_{\mathbb Q})^\vee$ is a commutative algebra-comodule over it. 

\subsubsection{} 
In the rest of \S3.4, we set $X:=\mathbb P^1-\{0,1,\infty\}$ and denote by $0,1$ the vectors $d/dx$ at points $0,1$ of 
$\mathbb P^1$ (with coordinate $z=x+iy$). These are base points at infinity in $X$.  For $x,y\in\{0,1\}$, let $P_{y,x}(X)$ be 
the motivic path space from $x$ to $y$ (see [DG], \S4.3). 

\subsubsection{} 
The images of $(\hat{\mathcal V}^{?,(1)}_{\mathbb Q})^\vee$ and $(\hat{\mathcal V}^{?,(10)}_{\mathbb Q})^\vee$ 
in Ind($\mathcal T$) by $\mathrm{Ind}(\mathrm{res}_?)$ are respectively isomorphic to $\mathcal O(P_{1,1}(X))$ and 
$\mathcal O(P_{1,0}(X))$ in the notation of [DG], Rem. 4.23, the former being equipped with its commutative Hopf algebra 
structure in $\mathrm{Ind}(\mathcal T)$ and the latter  with its commutative algebra-module structure over it, both structures
corresponding to the $\mathcal T$-group scheme structure of $P_{1,1}(X)$ and to the $\mathcal T$-torsor structure over it of 
$P_{1,0}(X)$ (see [DG], \S5). 

\subsubsection{} 
For $x,y\in\{0,1\}$, prounipotent $\mathbb Q$-group schemes $\pi_1^\DR(X,x)$ and 
$\pi_1^{\mathrm{alg,unip}}(X,x)$ are defined in [De], \S10.5, as well as 
torsors $\pi_1^\DR(X,y,x)$ and $\pi_1^{\mathrm{alg,unip}}(X,y,x)$;  
we use the notation $\pi_1^\B$ instead of $\pi_1^{\mathrm{alg,unip}}$. If $?\in\{\B,\DR\}$, then the functor 
$\omega_? :\mathcal T\to\mathrm{Vec}_{\mathbb Q}$ takes the  $\mathcal T$-group scheme (resp. $\mathcal T$-torsor) 
$P_{x,x}(X)$ (resp. $P_{y,x}(X)$) to $\pi_1^?(X,x)$ (resp. $\pi_1^?(X,y,x)$), so that forget takes the Hopf algebra and 
comodule algebra $\hat{\mathcal V}^{?,(1)}$ and $\hat{\mathcal V}^{?,(10)}$ in Pro($\mathsf{G}^?$-mod) to 
the objects with the similar structures $\mathcal O(\pi_1^?(X,1))^\vee$ and $\mathcal O(\pi_1^?(X,1,0))^\vee$ in Pro($\mathrm{Vec}_{\mathbb Q}$), $\mathcal O(-)$ being the ring of regular functions over a scheme. Then 
$\mathrm{comp}^{\mathcal V,(1)}_{(2\pi\mathrm{i},\varphi_\mathrm{KZ})}$ and 
$\mathrm{comp}^{\mathcal V,(10)}_{(2\pi\mathrm{i},\varphi_\mathrm{KZ})}$ can respectively be identified with the Betti-de Rham 
comparison algebra and module isomorphisms  $\mathrm{ULie}\pi_1^\B(X,1)=\mathcal O(\pi_1^\B(X,1))^\vee
\stackrel{\sim}{\to}\mathcal O(\pi_1^\DR(X,1))^\vee=\mathrm{ULie}\pi_1^\DR(X,1)$
and $\mathcal O(\pi_1^\B(X,1,0))^\vee\stackrel{\sim}{\to}\mathcal O(\pi_1^\DR(X,1,0))^\vee$ given by [De], \S12.16
(here $\mathrm{ULie}(-)$ is the cocommutative topological Hopf algebra attached to a $\mathbb Q$-group scheme). The 
isomorphisms $\mathrm{comp}^{\mathcal V,(1)}_{(\mu,g)}$ and $\mathrm{comp}^{\mathcal V,(10)}_{(\mu,g)}$, for
$(\mu,g)\in\mathbf k^\times\times\mathcal G(\hat{\mathcal V}^\DR)$ can therefore be viewed as `fake comparison isomorphisms'.

\part{Geometric interpretation of double shuffle theory}\label{part:geom:int}

The purpose of this part is to construct geometric interpretations of the `de Rham' and `Betti' DSR formalisms introduced in Part 
\ref{part:2:2jan2020}. These interpretations rely on general algebraic constructions which are explained in \S\ref{section:algebraic:lemmas}. 
Then \S\S\ref{sect:ibla:2jan2020} and \ref{section:5:08012018} are devoted to the de Rham DSR formalism: we introduce geometric objects 
related with the de Rham fundamental groups in \S\ref{sect:ibla:2jan2020}, namely infinitesimal braid Lie algebras and morphisms relating 
them; they are used in \S\ref{section:5:08012018} to give a geometric interpretation of the de Rham harmonic algebra coproduct following 
\cite{DT}, as well as of its module counterpart. \S\S\ref{sect:braids} and \ref{sect:geom:betti} follow the same pattern on the Betti side: 
introduction of geometric objects related with the Betti fundamental groups in \S\ref{sect:braids}, namely braid groups and morphisms 
relating them, and geometric interpretations of the  Betti harmonic algebra and module coproducts in \S\ref{sect:geom:betti}. 

\section{Algebraic preliminaries for part \ref{part:geom:int}}\label{section:algebraic:lemmas}

In this section, we prove algebraic results that will be used in \S\S\ref{section:5:08012018} and \ref{sect:geom:betti}. 
The first result allows for the construction of an algebra morphism starting from an algebra containing an ideal with 
freeness properties (\S\ref{subsect:4:1:19032018}, Lemma \ref{algebraic:lemma:1}). The second result describes the behavior of 
this construction under isomorphisms (\S\ref{subsect:4:1:19032018}, Lemma \ref{lemma:ideals:and:morphisms}). The third result 
shows how an algebra morphism may be constructed starting from an algebra morphism to a matrix algebra satisfying some conditions
(\S\ref{section:background:2:08012018}, Lemma \ref{algebraic:lemma:2}). 
 
\subsection{Construction of algebra morphisms based on ideals with freeness properties}
\label{subsect:4:1:19032018}

\begin{lem}\label{algebraic:lemma:1}
Let $R$ be an associative algebra 
and let $J\subset R$ be a two-sided ideal. Assume that
$(j_a)_{a\in[\![1,d]\!]}$ is a family of elements of $J$, which constitutes a basis of $J$ for its 
left $R$-module structure. For $r\in R$, let $(m_{ab}(r))_{a,b\in[\![1,d]\!]}$ 
be the collection of elements of $R$ defined by $j_ar=\sum_{b=1}^d m_{ab}(r)j_b$. 

Then the map $\varpi_{R,J} : R\to M_d(R)$\index{piRJ@$\varpi_{R,J}$}, 
$r\mapsto (m_{ab}(r))_{a,b\in[\![1,d]\!]}$ is an algebra morphism.  
\end{lem}

\proof Let $R^{\oplus d}\to J$ be the map $(r_1,\ldots,r_d)\mapsto \sum_{a=1}^d r_aj_a$. This is an isomorphism of left 
$R$-modules. It sets up an isomorphism of algebras $\mathrm{End}_{R\text{-left}}(J)\simeq \mathrm{End}_{R\text{-left}}(R^{\oplus d})$, 
where the index "$R$-left" means endomorphisms of left $R$-modules. 
The map $R\to\mathrm{End}_{R\text{-left}}(J)$, $r\mapsto(j\mapsto jr)$ is an algebra morphism when $R$ is equipped 
with the opposite algebra structure. Similarly, the map $M_d(R)\to\mathrm{End}_{R\text{-left}}(R^{\oplus d})$ taking $M\in M_d(R)$ 
to the image under the canonical isomorphism $R^{\oplus d}\simeq M_{1\times d}(R)$ of the endomorphism $X\mapsto XM$ 
is an algebra isomorphism when $M_d(R)$ is equipped with the opposite algebra structure. We then obtain a sequence of algebra 
morphisms 
$$
R^{\mathrm{op}}\to\mathrm{End}_{R\text{-left}}(J)\simeq\mathrm{End}_{R\text{-left}}(R^{\oplus d})\simeq M_d(R)^{\mathrm{op}}, 
$$ 
which yields an algebra morphism $R\to M_d(R)$. One checks that this morphism is given by the announced formula. \hfill\qed\medskip 

\begin{lem}\label{lemma:ideals:and:morphisms}
Let $R$ be an associative algebra, 
let $J\subset R$ be a two-sided ideal, free as a left 
$R$-module with basis $(j_a)_{a\in[\![1,d]\!]}$. 
Let $f:R\to R'$ be an algebra isomorphism, and let $(j'_a)_{a\in[\![1,d]\!]}$ be a basis of $f(J)\subset R'$ as a left $R'$-module. 

Then there exists a unique element $P\in\mathrm{GL}_d(R')$, such that 
$\begin{pmatrix} f(j_1) \\\vdots\\f(j_d)\end{pmatrix}=P\begin{pmatrix} j'_1 \\\vdots\\j'_d\end{pmatrix}$
(equality in $M_{d\times 1}(R')$). 
The morphisms $\varpi_{R,J}:R\to M_d(R)$ 
and $\varpi_{R',J'}:R'\to M_d(R')$
respectively attached to the data $(J\subset R,(j_a)_{a\in[\![1,d]\!]})$ and $(J'\subset R',(j'_a)_{a\in[\![1,d]\!]})$
as in Lemma \ref{algebraic:lemma:1} are related by the following diagram 
$$
\xymatrix{ R\ar^{\varpi_{R,J}}[rr]\ar_f[d]& & M_d(R)\ar^{M_d(f)}[d]\\ R'\ar_{\varpi_{R',J'}}[r]& M_d(R')\ar_{\mathrm{Ad}(P)}[r]& M_d(R')}
$$
\end{lem}

\proof For $r\in R$, one has $\begin{pmatrix} j_1 \\\vdots\\j_d\end{pmatrix}r
=\AAA_{R,J}(r)\begin{pmatrix} j_1 \\\vdots\\j_d\end{pmatrix}$, therefore  
$\begin{pmatrix} f(j_1) \\\vdots\\f(j_d)\end{pmatrix}f(r)
=f(\AAA_{R,J}(r))\begin{pmatrix} f(j_1) \\\vdots\\f(j_d)\end{pmatrix}$, that is 
$P\begin{pmatrix} j'_1 \\\vdots\\j'_d\end{pmatrix}f(r)
=f(\AAA_{R,J}(r))P\begin{pmatrix} j'_1 \\\vdots\\j'_d\end{pmatrix}$. Comparing this with the identity 
$\begin{pmatrix} j'_1 \\ \vdots\\ j'_d\end{pmatrix} r'
=\AAA_{R',J'}(r')\begin{pmatrix} j'_1 \\\vdots\\j'_d\end{pmatrix}$ for $r'\in R'$, one obtains $\AAA_{R',J'}(r')=P^{-1}\cdot 
f(\AAA_{R,J}(f^{-1}(r')))\cdot P$. \hfill\qed\medskip 
 
\subsection{Construction of algebra morphisms based on morphisms to matrix algebras}
\label{section:background:2:08012018}

Let $R$ be an associative algebra 
and let $e\in R$. Define $\cdot_e:R\times R\to R$
by $r\cdot_e r':=rer'$. Then $(R,\cdot_e)$ is an associative algebra. 
The subspaces $Re$ and $eR$ of $R$ are subalgebras. There is an algebra morphism 
$\mathrm{mor}_{R,e}:(R,\cdot_e)\to Re$\index{morRe@$\mathrm{mor}_{R,e}$}, 
given by $r\mapsto re$.

\begin{lem}\label{algebraic:lemma:2}
Let $R,S$ be associative algebras, 
let $e\in R$, let $n\geq 1$, and let $f:R\to M_n(S)$ be an algebra 
morphism. 
Assume that there exist elements $\mathrm{row}\in M_{1\times n}(S)\index{row @ $\mathrm{row}$}$ and $\mathrm{col}\in M_{n\times 1}(S)\index{col @ $\mathrm{col}$}$, such that 
$f(e)=\mathrm{col}\cdot \mathrm{row}$. 

 Then the map $\tilde f:(R,\cdot_e)\to S$, defined by $r\mapsto \mathrm{row}\cdot f(r)\cdot \mathrm{col}$
is an algebra morphism. One then has a commutative diagram
\begin{equation}\label{CD:vs:alg}
\xymatrix{ R\ar_{\simeq}[d]\ar^{f}[r]&M_n(S)\ar^{\mathrm{row}\cdot(-)\cdot\mathrm{col}}[d] \\ (R,\cdot_e)\ar_{\tilde f}[r]&S }
\end{equation} 
where 
\begin{equation}\label{general:def:row:col}
\mathrm{row}\cdot(-)\cdot\mathrm{col}: M_n(S)\to S
\end{equation}
is the linear map defined by $m\mapsto \mathrm{row}\cdot m\cdot\mathrm{col}$; in diagram \eqref{CD:vs:alg}, 
the horizontal arrows are algebra morphisms and the vertical arrows are vector space morphisms. 
\end{lem}

\proof For $r,r'\in R$, one has 
\begin{align*}
& \tilde f(r\cdot_e r')=\tilde f(rer')=\mathrm{row}\cdot f(rer')\cdot \mathrm{col}=\mathrm{row}\cdot f(r)f(e)f(r')\cdot \mathrm{col}
\\ & = \mathrm{row}\cdot f(r)\cdot \mathrm{col}\cdot \mathrm{row}\cdot f(r')\cdot \mathrm{col}=\tilde f(r)\tilde f(r').  
\end{align*}
\hfill \qed\medskip 

\section{Infinitesimal braid Lie algebras}\label{sect:ibla:2jan2020}

In \S\ref{sect:material:19032018}, we recall the definition of the infinitesimal braid Lie algebras 
$\mathfrak t_n$, $\mathfrak p_n$ ($n\geq1$), as well as the morphisms $\ell:\mathfrak p_4\to\mathfrak p_5$, 
$\mathrm{pr}_i:\mathfrak p_5\to\mathfrak p_4$ ($i\in[\![1,5]\!]$) and $\mathrm{pr}_{12}:\mathfrak p_5\to(\mathfrak p_4)^{\oplus2}$
relating them. 

In \S\ref{sect:alg:constr:19032018}, we introduce an ideal $J(\mathrm{pr}_5)$ of the universal enveloping algebra $
U\mathfrak p_5$ arising from 
$\mathrm{pr}_5$, and show its freeness as a left $U\mathfrak p_5\index{UP_5@$U(\mathfrak p_5)$}$-module (Lemma \ref{lemma:semidirect:product:LAs}). 
Lemma \ref{algebraic:lemma:1} then gives rise to an algebra morphism $\varpi:U\mathfrak p_5\to M_3(U\mathfrak p_5)$. 
By composing $\varpi$ with the morphisms $\ell$ and $\mathrm{pr}_{12}$, we construct a morphism 
$\rho:\mathcal V^\DR\to M_3((\mathcal V^\DR)^{\otimes2})$. In Lemma \ref{lemma:decomp:A:e_1}, by introducing suitable 
elements $\mathrm{row}_1\in M_{1\times3}((\mathcal V^\DR)^{\otimes2})$ and 
$\mathrm{col}_1\in M_{3\times1}((\mathcal V^\DR)^{\otimes2})$, we show that 
the morphism $\rho$ satisfies the hypothesis of Lemma \ref{algebraic:lemma:2} (the index $1$ in 
the notation $\mathrm{row}_1$ and $\mathrm{col}_1$ is motivated by the fact that these elements should be thought of as associated to the tangential base point $\vec 1$ in $\mathfrak M_{0,4}$). Applying this lemma, we obtain in 
\S\ref{subsection:tilde:B} a morphism $\tilde{\BB}:(\mathcal V^\DR,\cdot_{e_1})\to (\mathcal V^\DR)^{\otimes2}$, which we
compute in Lemma \ref{lemma:53:15:11:2017}. 

\subsection{Material on infinitesimal braid Lie algebras}\label{sect:material:19032018}

\subsubsection{The Lie algebras $\mathfrak t_n,\mathfrak p_n$}\label{sect:the:Lie:algebras:pn}
\label{subsect:def:LA:morphisms}

For $n\geq2$, let $\mathfrak t_n\index{t_n @$\mathfrak t_n$}$ be the graded Lie $\mathbf k$-algebra with generators $t_{ij}\index{t_ij @$t_{ij}$}$, $i\neq j\in[\![1,n]\!]$ of degree 1, 
and relations $t_{ji}=t_{ij}$, $[t_{ij},t_{ik}+t_{jk}]=0$ for $i,j,k$ all different in $[\![1,n]\!]$, and $[t_{ij},t_{kl}]=0$ for $i,j,k,l$ 
all different in $[\![1,n]\!]$. The Lie algebra $\mathfrak t_n$ is called the {\it Drinfeld-Kohno}, or {\it infinitesimal braid, Lie algebra}. 

For $n\geq 4$, we denote by $\mathfrak p_n\index{p_nfrak@$\mathfrak p_n$}$ the graded Lie $\mathbf k$-algebra with generators $e_{ij}\index{e_ij @$e_{ij}$}$ of degree 1, where $i\neq j\in[\![1,n]\!]$
and relations $e_{ji}=e_{ij}$, $\sum_{j\in[\![1,n]\!]-\{i\}}e_{ij}=0$, $[e_{ij},e_{kl}]=0$ for $i,j,k,l$ all distinct in $[\![1,n]\!]$. 
The Lie algebra $\mathfrak p_n$ is called the {\it sphere infinitesimal braid Lie algebra}. 

For $n\geq 3$, there are surjective morphisms of graded Lie algebras $\mathfrak t_n\twoheadrightarrow\mathfrak p_{n+1}$ given by 
$t_{ij}\mapsto e_{ij}$ for $i\neq j\in[\![1,n]\!]$,  $\mathfrak t_{n+1}\twoheadrightarrow\mathfrak p_{n+1}$ given by $t_{ij}\mapsto e_{ij}$ 
for $i\neq j\in[\![1,n+1]\!]$, and an injective morphism $\mathfrak t_n\hookrightarrow\mathfrak t_{n+1}$ given by $t_{ij}\mapsto t_{ij}$ 
for $i\neq j\in[\![1,n]\!]$; they fit in a commutative diagram 
$$
\xymatrix{ \mathfrak t_n\ar@{->>}[r]\ar@{^(->}[d]& \mathfrak p_{n+1}\\ \mathfrak t_{n+1}\ar@{->>}[ur]& }
$$
Moreover, the morphism $\mathfrak t_n\to\mathfrak p_{n+1}$ factorizes as $\mathfrak t_n\twoheadrightarrow 
\mathfrak t_n/Z(\mathfrak t_n)\simeq\mathfrak p_{n+1}$, where $Z(\mathfrak t_n)$ is the 1-dimensional
center of $\mathfrak t_n$ (it is concentrated in degree 1 and spanned by $\sum_{i<j\in[\![1,n]\!]}t_{ij}$). 

\begin{rem}\label{rem:27122017}
Let $P_n\index{P_n@$P_n$}$ (resp.\ $K_n\index{K_n @$K_n$}$) be the {\it pure sphere (resp.\ Artin pure) braid group} with $n$ strands (see \S\ref{subsect:def:gp:morphisms} and 
\cite{Bir}).  Its lower central series defines a descending group filtration. The associated graded $\mathbb Z$-module is a 
$\mathbb Z$-Lie algebra. Then $\mathrm{gr}(P_n)\otimes\mathbf k\simeq\mathfrak p_n$, and 
$\mathrm{gr}(K_n)\otimes\mathbf k\simeq\mathfrak t_n$. 
\end{rem}

\subsubsection{The morphisms $\ell,\mathrm{pr}_i$ and $\mathrm{pr}_{12}$ between infinitesimal braid Lie algebras}
\label{subsect:def:LA:morphisms:2}

There is a graded 
Lie algebra isomorphism
$\mathfrak p_4\simeq\mathfrak f_2$, where $e_0=e_{14}=e_{23}$, $e_1=e_{12}=e_{34}$. One also sets  
$e_\infty:=-e_0-e_1$\index{e_infty@$e_\infty$}, so $e_\infty=e_{13}=e_{24}$. 
By abuse of notation, and following \cite{DT}, we set  
$$
\forall i\in\{0,1\}, e_i:=(e_i,0)\in(\mathfrak f_2)^2,   \quad f_i:=(0,e_i)\in(\mathfrak f_2)^2. \index{f _0, f_1@ $f_0$, $f_1$} 
$$ 
We then have $e_i=e_i\otimes1$, $f_i=1\otimes e_i$ in $(\mathcal V^\DR)^{\otimes2}$. 

One checks that there are Lie algebra morphisms $\mathrm{pr}_i:\mathfrak p_5\to\mathfrak p_4\index{pr 125@$\mathrm{pr}_1$, $\mathrm{pr}_2$, $\mathrm{pr}_5$}$ for $i=1,2,5$, given by 
\begin{center}
\begin{tabular}{|c|c|c|c|c|c|c|c|c|c|c|}
  \hline
  elt $x\in\mathfrak p_5$ & $e_{12}$ & $e_{13}$ & $e_{14}$ & $e_{15}$ & $e_{23}$& $e_{24}$& $e_{25}$& $e_{34}$& $e_{35}$& $e_{45}$ \\
  \hline
  $\mathrm{pr}_1(x)$ & 0 & 0 & 0 & 0 & $e_0$ & $e_\infty$ & $e_1$ & $e_1$ & $e_\infty$ & $e_0$ \\ 
\hline   
$\mathrm{pr}_2(x)$ & 0 & $e_\infty$ & $e_0$ & $e_1$ & 0 & 0 & 0  & $e_1$  & $e_0$ & $e_\infty$ \\
  \hline $\mathrm{pr}_5(x)$ & $e_1$ & $e_\infty$ & $e_0$ & 0 & $e_0$ & $e_\infty$ & 0 & $e_1$ & 0 & 0 \\ \hline
\end{tabular}
\end{center}

\medskip 
These morphisms give rise to the Lie algebra morphism 
$\mathrm{pr}_{12}:\mathfrak p_5\to\mathfrak p_4^{\oplus2}\index{pr 12t@ $\mathrm{pr}_{12}$}$ defined 
by $\mathrm{pr}_{12}(x):=(\mathrm{pr}_1(x),\mathrm{pr}_2(x))$. 

There is a Lie algebra morphism $\ell:\mathfrak p_4\to\mathfrak p_5 \index{l @$\ell$}$, given by 
$e_0\mapsto e_{23}$, $e_1\mapsto e_{12}$. It is such that $\mathrm{pr}_5\circ\ell$
is the identity of $\mathfrak p_4$.

The isomorphism $\mathfrak p_4\simeq\mathfrak f_2$ gives rise to an isomorphism 
$U\mathfrak p_4\simeq U\mathfrak f_2=\mathcal V^\DR$. 

The Lie algebra morphisms $\mathrm{pr}_{12}$, $\mathrm{pr}_i$ and $\ell$ induce algebra morphisms between 
universal enveloping algebras, which will be denoted  
$\mathrm{pr}_{12} :U\mathfrak p_5
\to (\mathcal V^\DR)^{\otimes2}$, 
 $\mathrm{pr}_i : U\mathfrak p_5
 \to \mathcal V^\DR$ and $\ell : \mathcal V^\DR\to U\mathfrak p_5
 $. 

\subsection{Constructions related to an ideal of $U\mathfrak p_5
$}\label{sect:alg:constr:19032018}

\subsubsection{The structure of $J(\mathrm{pr}_5)$}

\begin{defn} {\it We denote by $J(\mathrm{pr}_5)\index{Jpr_5 @ $J(\mathrm{pr}_5)$}$ the kernel 
$\mathrm{Ker}(U\mathfrak p_5
\stackrel{\mathrm{pr}_5}{\to}\mathcal V^\DR)$. 
This is a two-sided ideal of $U\mathfrak p_5
$.}
\end{defn}

In order to study the structure of $J(\mathrm{pr}_5)$, we prove the following Lemma \ref{lemma:semidirect:product:LAs}. 

\begin{lem}\label{lemma:semidirect:product:LAs}
1) The Lie subalgebra of $\mathfrak p_5$ generated by the $e_{i5}$, $i\in[\![1,4]\!]$ is freely generated by the 
$e_{i5}$, $i\in[\![1,3]\!]$, and coincides with the ideal $\mathrm{Ker}(\mathrm{pr}_5)$ of $\mathfrak p_5$; it will be denoted as
$\mathfrak f_3$. 

2) Set $\tilde{\mathfrak p}_4:=\mathrm{im}(\ell:\mathfrak p_4\to\mathfrak p_5)$.
There is a direct sum decomposition $\mathfrak p_5=\tilde{\mathfrak p}_4\oplus\mathfrak f_3$.  
\end{lem}

\begin{proof}  
1) follows from \cite{Ih}, \S1.1 with $n=i=5$ (the notation for 
$\mathfrak f_3$ in {\it loc. cit.} is $N_n$). 
2) follows  from the facts that $\mathfrak f_3=\mathrm{Ker}(\mathrm{pr}_5)$ and that 
$\mathrm{pr}_5\circ\ell=\mathrm{id}_{\mathfrak p_4}$. \end{proof} 

\begin{rem}
Using \cite{Ih}, one can prove that the centralizer of $e_{45}$ in $\mathfrak p_5$
decomposes as a direct sum $\mathbf k e_{45}\oplus \tilde{\mathfrak p}_4$.
\end{rem}

\begin{lem} \label{lemma:decomp:J:pr5}
The map $(U\mathfrak p_5)^{\oplus3}\to J(\mathrm{pr}_5)$, $(p_i)_{i\in[\![1,3]\!]}\mapsto\sum_{i\in[\![1,3]\!]}p_i\cdot e_{i5}$ 
is an isomorphism of left $U\mathfrak p_5$-modules. 
\end{lem}

\proof The equality $\mathfrak p_5=\tilde{\mathfrak p}_4\oplus \mathfrak f_3$ is a decomposition of the Lie algebra $\mathfrak p_5$ 
as a direct sum of two Lie subalgebras. The tensor product of $\ell:U\mathfrak p_4\to U\mathfrak p_5$ with the injection
$U\mathfrak f_3\to U\mathfrak p_5$, followed by the product in $U\mathfrak p_5$, 
induces a linear map 
$$
\mathrm{codec} : U{\mathfrak p}_4\otimes U\mathfrak f_3\to U\mathfrak p_5. 
$$
This map is compatible with the PBW filtrations on both sides, and its associated graded map is the linear map 
$S{\mathfrak p}_4\otimes S\mathfrak f_3\to S\mathfrak p_5$, which is an isomorphism of graded vector spaces, 
so that $\mathrm{codec}$ is an isomorphism of filtered vector spaces. 

The following diagram 
\begin{equation}\label{diagram:Up5:times:Uf3}
\xymatrix{
U{\mathfrak p}_4\otimes U\mathfrak f_3\ar^{\mathrm{codec}}[rr]\ar_{\mathrm{id}\otimes\varepsilon}[rrd] && 
U\mathfrak p_5\ar^{\mathrm{pr}_5}[d]
\\ && U\mathfrak p_4
}
\end{equation}
is commutative, where $\varepsilon$ is the counit of $U\mathfrak f_3$. Indeed, for $p\in U{\mathfrak p}_4$
and $f\in U\mathfrak f_3$, 
$$
\mathrm{pr}_5\circ \mathrm{codec}(p\otimes f)=\mathrm{pr}_5\Big(\ell(p)f\Big)
=\Big(\mathrm{pr}_5( \ell(p))\Big)\cdot \mathrm{pr}_5(f)=p\cdot \varepsilon(f), 
$$
where the first equality follows from the definition of $\mathrm{codec}$, 
the second equality follows from the algebra morphism property of 
$\mathrm{pr}_5$, and the third equality follows from the facts that $\mathrm{pr}_5\circ\ell=\mathrm{id}_{\mathfrak p_4}$ 
and that 
$\mathfrak f_3=\mathrm{Ker}(\mathrm{pr}_5:\mathfrak p_5\to\mathfrak p_4)$. 
This diagram implies that $J(\mathrm{pr}_5)$ is equal to the isomorphic image by $\mathrm{codec}$
of the subspace $U\mathfrak p_4\otimes (U\mathfrak f_3)_+$ of $U\mathfrak p_4\otimes U\mathfrak f_3$, where 
$(U\mathfrak f_3)_+:=\mathrm{Ker}(\varepsilon : U\mathfrak f_3\to\mathbf k)$ is the augmentation ideal of $U\mathfrak f_3$, 
that is 
\begin{equation}\label{equality:J(pr5)}
J(\mathrm{pr}_5)
=\mathrm{im}(U\mathfrak p_4\otimes (U\mathfrak f_3)_+\stackrel{\mathrm{codec}}{\to} U\mathfrak p_5).  
\end{equation}

Since the $e_{i5}$, $i\in[\![1,3]\!]$, belong to $\mathfrak f_3$, the following diagram commutes
\begin{equation}\label{CD:17:11:2017}
\xymatrix{
(U\mathfrak p_4\otimes U\mathfrak f_3)^{\oplus 3} 
\ar_{\mathrm{codec}^{\oplus3}}[d]\ar^{\sim}[r]
&U\mathfrak p_4\otimes (U\mathfrak f_3)^{\oplus 3} \ar[r]& U\mathfrak p_4\otimes U\mathfrak f_3\ar^{\mathrm{codec}}[d]\\ 
(U\mathfrak p_5)^{\oplus 3}\ar[rr]&& U\mathfrak p_5
}
\end{equation} 
where the lower horizontal map is given by $(p_i)_{i\in[\![1,3]\!]}\mapsto\sum_{i\in[\![1,3]\!]}p_ie_{i5}$, and the upper horizontal map
is given by the tensor product of the identity in $U\mathfrak p_4$ with the map $(U\mathfrak f_3)^{\oplus 3}\to U\mathfrak f_3$, 
$(\varphi_i)_{i\in[1,3]}\mapsto\sum_{i\in[\![1,3]\!]}\varphi_ie_{i5}$. 

Since $U\mathfrak f_3$ is freely generated, as an associative algebra, by the $e_{i5}$, $i\in[\![1,3]\!]$, the latter map 
corestricts to a $\mathbf k$-module isomorphism $(U\mathfrak f_3)^{\oplus 3}\to (U\mathfrak f_3)_+$. It follows
that the upper horizontal map of (\ref{CD:17:11:2017}) corestricts to an isomorphism 
$U\mathfrak p_4\otimes (U\mathfrak f_3)^{\oplus 3} \to U\mathfrak p_4\otimes (U\mathfrak f_3)_+$. 

Since the vertical maps of (\ref{CD:17:11:2017}) are isomorphisms, this implies that the lower 
horizontal map of (\ref{CD:17:11:2017}) corestricts to an isomorphism from $(U\mathfrak p_5)^{\oplus 3}$
to the image by $\mathrm{codec}$ of $U\mathfrak p_4\otimes (U\mathfrak f_3)_+$, which is $J(\mathrm{pr}_5)$
according to (\ref{equality:J(pr5)}). \hfill\qed\medskip 

\subsubsection{A morphism $\AAA:U\mathfrak p_5\to M_3(U\mathfrak p_5)$}\label{subsect:def:A}

Lemma \ref{lemma:decomp:J:pr5} says that the hypothesis of Lemma \ref{algebraic:lemma:1} is satisfied 
in the following situation: $R=U\mathfrak p_5$, $J=J(\mathrm{pr}_5)$, $d=3$, $(j_a)_{a\in[\![1,d]\!]}=(e_{i5})_{i\in[\![1,3]\!]}$. 
We denote by 
$$
\AAA:U\mathfrak p_5\to M_3(U\mathfrak p_5)
\index{pivar@$\AAA$}
$$
the algebra morphism given in this situation by Lemma \ref{algebraic:lemma:1}. Then for $p\in U\mathfrak p_5$, 
$\AAA(p)=(a_{ij}(p))_{i,j\in[\![1,3]\!]}$, and 
\begin{equation}\label{def:varpi:16dec2019}
\forall i\in[\![1,3]\!],\quad e_{i5}p=\sum_{j\in[\![1,3]\!]}a_{ij}(p)e_{j5} 
\end{equation}
(equalities in $U\mathfrak p_5$). 

\subsubsection{Construction and properties of a morphism $\BB:\mathcal V^\DR\to M_3((\mathcal V^\DR)^{\otimes2})$}
\label{sect:514:12122017}

Define the algebra morphism 
\begin{equation}\label{equation:B:05012017}
\BB:\mathcal V^\DR\to M_3((\mathcal V^\DR)^{\otimes2})
\index{rhovar@$\BB$}
\end{equation}
to be the composition 
$$
\mathcal V^\DR\simeq U\mathfrak p_4\stackrel{\ell}{\to} U\mathfrak p_5\stackrel{\AAA}{\to}
M_3(U\mathfrak p_5)\stackrel{M_3(\mathrm{pr}_{12})}{\to}
M_3((U\mathfrak p_4)^{\otimes2})\simeq M_3((\mathcal V^\DR)^{\otimes2}), 
$$
where $\ell$ is as in \S\ref{subsect:def:LA:morphisms:2}, $\AAA$ is as in \S\ref{subsect:def:A}, 
and $M_3(\mathrm{pr}_{12})\index{M_3pr_12@$M_3(\mathrm{pr}_{12})$}$ is the  
morphism induced by $\mathrm{pr}_{12}$, i.e., taking $(p_{ij})_{i,j\in[\![1,3]\!]}$ to $(\mathrm{pr}_{12}(p_{ij}))_{i,j\in[\![1,3]\!]}$. 

\begin{lem}\label{lemma:decomp:A:e_1}
Set 
\begin{equation}\label{def:row:col:04012018}
\mathrm{row}_1:=\begin{pmatrix} e_1 & -f_1 & 0\end{pmatrix}\in M_{1\times 3}((\mathcal V^\DR)^{\otimes2}),\quad  
\mathrm{col}_1:=\begin{pmatrix} 1 \\ -1 \\ 0\end{pmatrix}\in M_{3\times 1}((\mathcal V^\DR)^{\otimes2})
\index{row1 @ $\mathrm{row}_1$}\index{col1 @ $\mathrm{col}_1$}
\end{equation}
(recall that 
$e_1,f_1\in (\mathcal V^\DR)^{\otimes2}$ are $e_1\otimes1,1\otimes e_1$), then 
\begin{equation}\label{525bis}
\BB(e_1)=\mathrm{col}_1\cdot\mathrm{row}_1
\end{equation}
(equality in $M_3((\mathcal V^\DR)^{\otimes2})$). 
\end{lem}

\proof One has $\ell(e_1)=e_{12}$. Let us compute $\AAA(e_{12})$. One has 
\begin{align*}
& e_{15}e_{12}=e_{12}e_{15}+[e_{15},e_{12}]=e_{12}e_{15}+[e_{25},e_{15}]=(e_{12}+e_{25})e_{15}-e_{15}e_{25}, \\  
& e_{25}e_{12}=-e_{25}e_{15}+(e_{12}+e_{15})e_{25} 
\intertext{
(applying the permutation of indices 1 and 2 to the previous equality), }
& e_{35}e_{12}=e_{12}e_{35},  
\end{align*}
which implies that 
$$
\AAA(e_{12})=\begin{pmatrix} e_{12}+e_{25}&-e_{15} & 0\\ -e_{25}& e_{12}+e_{15}& 0\\ 0& 0& e_{12}\end{pmatrix}\in M_3(U\mathfrak p_5). 
$$
The image of this matrix in $M_3((\mathcal V^\DR)^{\otimes2})$ is 
$$
\BB(e_1)=\begin{pmatrix} e_1&-f_1 & 0\\ -e_1& f_1& 0\\ 0& 0& 0\end{pmatrix}=\begin{pmatrix} 1 \\ -1 \\ 0\end{pmatrix}
\begin{pmatrix} e_1 & -f_1 & 0\end{pmatrix}=\mathrm{col}_1\cdot\mathrm{row}_1. 
$$
Therefore $\BB(e_1)=\mathrm{col}_1\cdot\mathrm{row}_1$. \hfill\qed\medskip

\subsubsection{Construction and properties of a morphism $\tilde{\BB}:(\mathcal V^\DR,\cdot_{e_1})\to(\mathcal V^\DR)^{\otimes2}$}
\label{subsection:tilde:B}

Lemma \ref{lemma:decomp:A:e_1} shows that the hypothesis of Lemma \ref{algebraic:lemma:2} is satisfied in the following situation: 
$R=\mathcal V^\DR$, $S=(\mathcal V^\DR)^{\otimes2}$, $e=e_1$, $n=3$, $f=\BB$, $\mathrm{row}_1$ and $\mathrm{col}_1$ are as in Lemma 
\ref{lemma:decomp:A:e_1}. We denote by
$$
\tilde{\BB}:(\mathcal V^\DR,\cdot_{e_1})\to (\mathcal V^\DR)^{\otimes2}
\index{rhovar^tilde@$\tilde{\BB}$}
$$
the algebra morphism given in this situation by Lemma \ref{algebraic:lemma:2}. 

Then for any $f\in \mathcal V^\DR$, one has 
\begin{equation}\label{formula:tilde:B:f}
\tilde{\BB}(f)=\mathrm{row}_1\cdot \BB(f)\cdot \mathrm{col}_1=\mathrm{row}_1\cdot 
\{M_3(\mathrm{pr}_{12})\circ \AAA\circ \ell(f)\}\cdot \mathrm{col}_1
\in (\mathcal V^\DR)^{\otimes2}. 
\end{equation}

\begin{lem}\label{lemma:53:15:11:2017}
For any $n\geq 0$, 
$$
\tilde{\BB}(e_0^n)=e_1e_0^n+f_1f_0^n-\sum_{i=0}^{n-1}(e_1e_0^i)\cdot(f_1f_0^{n-1-i})
$$
(equality in $(\mathcal V^\DR)^{\otimes2}=(U\mathfrak f_2)^{\otimes2}$). 
\end{lem}

\proof According to (\ref{formula:tilde:B:f}), $\tilde{\BB}(e_0^n)=\mathrm{row}_1\cdot \BB(e_0^n)\cdot \mathrm{col}_1$. 
As $\BB$ is an algebra morphism, $\BB(e_0^n)=\BB(e_0)^n$. Then $\BB(e_0)=M_3(\mathrm{pr}_{12})\circ \AAA\circ \ell(e_0)
=M_3(\mathrm{pr}_{12})(\AAA(e_{23}))$. 

Let us compute $\AAA(e_{23})$. One has 
\begin{align*}
& e_{15}e_{23}=e_{23}e_{15}, \\  
& e_{25}e_{23}=e_{23}e_{25}+[e_{25},e_{23}]=e_{23}e_{25}+[e_{35},e_{25}]=(e_{23}+e_{35})e_{25}-e_{25}e_{35}, \\
& e_{35}e_{23}=-e_{35}e_{25}+(e_{23}+e_{25})e_{35}\\ \intertext{
(applying the permutation of indices 2 and 3 to the previous equality), }  
\end{align*}
which implies that 
$$
\AAA(e_{23})=\begin{pmatrix}  e_{23}& 0&0  \\   0&e_{23}+e_{35}&-e_{25} \\  0&-e_{35}& e_{23}+e_{25}\end{pmatrix}\in M_3(U\mathfrak p_5). 
$$
Then  
\begin{equation}\label{expression:B(e_0):25oct}
\BB(e_0)=M_3(\mathrm{pr}_{12})(\AAA(e_{23}))=
\begin{pmatrix} e_0&0 & 0\\ 0& -e_1+f_0& -e_1\\ 0& e_0+e_1-f_0& e_0+e_1\end{pmatrix}\in M_3((\mathcal V^\DR)^{\otimes2}). 
\end{equation}
Set $T\index{T@$T$}:=\begin{pmatrix} -e_1+f_0& -e_1\\e_0+e_1-f_0& e_0+e_1\end{pmatrix}\in M_2((\mathcal V^\DR)^{\otimes2})$, 
then $\BB(e_0^n)=\begin{pmatrix} e_0^n& 0\\0& T^n\end{pmatrix}$, therefore 
$$
\tilde{\BB}(e_0^n)=\mathrm{row}_1\cdot\BB(e_0^n)\cdot \mathrm{col}_1=e_1e_0^n+\begin{pmatrix} -f_1&0 \end{pmatrix} T^n
\begin{pmatrix} -1\\0 \end{pmatrix}, 
$$
where the last equality follows from the form of $\mathrm{row}_1$ and $\mathrm{col}_1$. 

One checks that $T=\begin{pmatrix} 1&0\\-1&1 \end{pmatrix} \begin{pmatrix}f_0&-e_1\\0&e_0 \end{pmatrix} 
\begin{pmatrix} 1&0\\-1&1 \end{pmatrix}^{-1}$, therefore 
$$
T^n=\begin{pmatrix} 1&0\\-1&1 \end{pmatrix} \begin{pmatrix}f_0&-e_1\\0&e_0 \end{pmatrix}^n 
\begin{pmatrix} 1&0\\1&1 \end{pmatrix}=\begin{pmatrix} 1&0\\-1&1 \end{pmatrix} \begin{pmatrix}f_0^n&-\sum_{i=0}^{n-1}
f_0^ie_1e_0^{n-1-i}\\0&e_0^n \end{pmatrix} 
\begin{pmatrix} 1&0\\1&1 \end{pmatrix}, 
$$
so 
\begin{align*}
& \tilde{\BB}(e_0^n)=e_1e_0^n+\begin{pmatrix} -f_1&0 \end{pmatrix} \begin{pmatrix}f_0^n&-\sum_{i=0}^{n-1}
f_0^ie_1e_0^{n-1-i}\\0&e_0^n \end{pmatrix} 
\begin{pmatrix} -1\\-1 \end{pmatrix}=e_1e_0^n+f_1f_0^n-\sum_{i=0}^{n-1}
f_1f_0^ie_1e_0^{n-1-i}; 
\end{align*}
the result then follows from the commutativity of $e_s$ with $f_t$ for $s,t\in\{0,1\}$.  
\hfill\qed\medskip 

\begin{rem}\label{remark:alternative:decomp:DR}
If $\mathcal A$ is a unital associative algebra, then the following identity holds in $M_2(\mathcal A)$
\begin{equation}\label{identity:28:11:2017}
\mathrm{Ad}(\begin{pmatrix} 1&0\\a&1 \end{pmatrix})(\begin{pmatrix} u&v\\0&w\end{pmatrix})
=\mathrm{Ad}(\begin{pmatrix} 1&a^{-1}\\0&1 \end{pmatrix})(\begin{pmatrix} a^{-1}wa&0\\au-wa-ava&aua^{-1} \end{pmatrix})
\end{equation}
provided 
$u,v,w\in\mathcal A$ and $a\in\mathcal A^\times$. This implies 
the identity 
$$
T=\begin{pmatrix} 1&-1\\0&1 \end{pmatrix} \begin{pmatrix}e_0&0\\e_0+e_1-f_0&f_0 \end{pmatrix} 
\begin{pmatrix} 1&-1\\0&1 \end{pmatrix}^{-1},
$$ allowing for an alternative computation of $T^n$.  
\end{rem}

\section{Geometric interpretation of the de Rham harmonic coproducts}\label{section:5:08012018}

In this section, we construct commutative diagrams relating the de Rham harmonic algebra and module coproducts
$\Delta^{\mathcal W,\DR}$ and $\Delta^{\mathcal M,\DR}$ with infinitesimal braid Lie algebras (diagrams 
\eqref{diagram:prop:59} and \eqref{diagram:1402}). 

These diagrams involve a localization $\mathcal V^\DR[{1\over e_1}]$ of the algebra $\mathcal V^\DR$, and 
a module $\mathcal M^\DR[{1\over e_1}]$ over this algebra, which are introduced in 
\S\ref{sect:def:loc:VDR}. 

We prove the commutativity of the first diagram in \S\ref{section:5:3:19:03:2018}: we first construct a
diagram relating $\Delta^{\mathcal W,\DR}$ and the morphism $\tilde\rho$ from \S\ref{subsection:tilde:B}
(Lemma \ref{CD:tildeB:Delta*}), and derive from there the commutativity of diagram \eqref{diagram:prop:59} relating 
$\Delta^{\mathcal W,\DR}$ and $\rho$ and $\mathrm{row}_1,\mathrm{col}_1$ (Proposition \ref{prop:CD:DR:28/11/2017}). 
The material in \S\ref{section:5:3:19:03:2018} is inspired by \cite{DT}, more specifically by \S6.3 and Proposition 6.2 in that paper. 

In \S\ref{section:proof:T5}, we introduce a column vector $\mathrm{col}_0$ (the index $0$ indicates that it should be thought of as 
associated to the tangential base point $\vec 0$ of $\mathfrak M_{0,4}$). We then prove the existence of a 
map $\delta:\mathcal M^\DR\to(\mathcal M^\DR[{1\over e_1}])^{\otimes2}$ fitting in a commutative 
diagram involving $\mathrm{col}_0$; using diagram \eqref{diagram:prop:59}, we then identify $\delta$ with 
$\Delta^{\mathcal M,\DR}$, which establishes diagram \eqref{diagram:1402} (Proposition \ref{prop:22:14022019}).  

We finally construct completions of the diagrams \eqref{diagram:prop:59} and  \eqref{diagram:1402} 
in \S\ref{sect:completion:giadrhc}.  

\subsection{Localizations}\label{sect:def:loc:VDR} 

Define ${\mathcal V}^\DR[\frac{1}{e_1}]$\index{V^DR/e_1@${\mathcal V}^\DR[\frac{1}{e_1}]$} to be the localization of ${\mathcal V}^\DR$ with respect to $e_1$, i.e. the unital $\mathbf k$-algebra with generators 
$e_0$ and $e_1^{\pm1}$ and relations $e_1e_1^{-1}=e_1^{-1}e_1=1$; 
${\mathcal V}^\DR[\frac{1}{e_1}]$  is then equipped with a $\mathbb Z$-grading given by $\mathrm{deg}(e_i)=1$ for $i=0,1$. 
It follows that ${\mathcal V}^\DR[\frac{1}{e_1}]$ in an algebra in $\mathbf k\text{-mod}_{\gr}$ (see Definition \ref{def:cat}). 

Set also ${\mathcal M}^\DR[\frac{1}{e_1}]:={\mathcal V}^\DR[\frac{1}{e_1}]/{\mathcal V}^\DR[\frac{1}{e_1}] e_0$\index{M^DR/e_1@${\mathcal M}^\DR[\frac{1}{e_1}]$}. This is a 
$\mathbb Z$-graded ${\mathcal V}^\DR[\frac{1}{e_1}]$-module, therefore ${\mathcal M}^\DR[\frac{1}{e_1}]$ is a 
${\mathcal V}^\DR[\frac{1}{e_1}]$-module in $\mathbf k\text{-mod}_{\gr}$. 

One checks that the natural $\mathbb Z$-graded $\mathbf k$-algebra morphism $\mathcal V^\DR\to{\mathcal V}^\DR[\frac{1}{e_1}]$ 
and $\mathbf k$-module morphism $\mathcal M^\DR\to{\mathcal M}^\DR[\frac{1}{e_1}]$ are injective. 

\subsection{Relationship between infinitesimal braid Lie algebras and $\Delta^{\mathcal W,\DR}$}\label{section:5:3:19:03:2018}

Denote by $\mathbf k[e_0]$ the linear span in $\mathcal V^\DR$ of the elements $e_0^n$, $n\geq0$. 

\begin{lem}\label{lemma:e0n:generating}
$(\mathcal V^\DR,\cdot_{e_1})$
\index{V^DR, e_1@$(\mathcal V^\DR,\cdot_{e_1})$}
is generated, as an associative (non-unital) algebra, by 
$\mathbf k[e_0]$.  
\end{lem}

\proof For $k\geq1$, the $k$-th power of the canonical injection, followed by the $k$-th fold product $\cdot_{e_1}$, sets up a linear
map $\mathbf k[e_0]^{\otimes k}\to \mathcal V^\DR$.  One checks that this map is injective and that its image coincides with 
the part of $\mathcal V^\DR$ of $e_1$-degree equal to $k-1$. So the composition 
$\oplus_{k\geq1}\mathbf k[e_0]^{\otimes k}\to\oplus_{k\geq1}(\mathcal V^\DR)^{\otimes k}\to \mathcal V^\DR$, 
where the first map is the canonical injection and the second map is the iteration of the product $\cdot_{e_1}$, maps 
$\oplus_{k\geq1}\mathbf k[e_0]^{\otimes k}$ injectively (in fact, bijectively) to $\mathcal V^\DR$. \hfill\qed\medskip 

Recall that $\mathcal W^{\DR}$ is the subalgebra of $\mathcal V^\DR$ equal to $\mathbf k\oplus \mathcal V^\DR e_1$
(see \S\ref{sect:1:1:1:crm}). We set 
\begin{equation}\label{del:WlDR:+:04012018}
\mathcal W^{\DR}_+:=\mathcal V^\DR e_1.
\index{W^DR+@$\mathcal W^{\DR}_+$}
\end{equation}
This is a (non-unital) subalgebra of $\mathcal W^{\DR}$. 

Since the right multiplication by $e_1$ is injective in $\mathcal V^\DR$, the algebra morphism 
\begin{equation}\label{def:mor:DR:04012018}
\mathrm{mor}_{\mathcal V^\DR,e_1}:(\mathcal V^\DR,\cdot_{e_1})\to \mathcal W^{\DR}_+
\index{morVDRe1@ $\mathrm{mor}_{\mathcal V^\DR,e_1}$}
\end{equation}
(see \S\ref{section:algebraic:lemmas}) is an algebra isomorphism. 
One also checks that the algebra automorphism $\mathrm{Ad}(e_1^{-1})$ of $\mathcal V^\DR[{1\over e_1}]$ restricts to an algebra morphism 
$\mathcal V^\DR\to\mathcal V^\DR[{1\over e_1}]_{\geq0}$  . 

\begin{lem}\label{CD:tildeB:Delta*}
The following diagram is commutative
$$
\xymatrix{
(\mathcal V^\DR,\cdot_{e_1}) \ar^{\mathrm{Ad}(e_1f_1)^{-1}\circ\tilde{\BB}}[rr]\ar_{\mathrm{mor}_{\mathcal V^\DR,e_1}}^{\simeq}[d]&& 
(\mathcal V^\DR[{1\over e_1}]_{\geq0})^{\otimes2}\\
\mathcal W^{\DR}_+\ar_{\Delta^{\mathcal W,\DR}}[rr]&&  (\mathcal W^{\DR})^{\otimes2}\ar@{^{(}->}[u]
}
$$
where the top horizontal map is the composition of $\tilde\rho$ with the tensor square of the morphism
$\mathrm{Ad}(e_1^{-1}):\mathcal V^\DR\to\mathcal V^\DR[{1\over e_1}]_{\geq0}$  
and the right vertical map is the tensor square of the injection $\mathcal W^{\DR}\hookrightarrow\mathcal V^{\DR}\hookrightarrow
\mathcal V^\DR[{1\over e_1}]_{\geq0}$. 
\end{lem}

\proof For $n\geq1$, 
\begin{align*}
&\Delta^{\mathcal W,\DR}\circ\mathrm{mor}_{\mathcal V^\DR,\cdot_{e_1}}(e_0^n)
=\Delta^{\mathcal W,\DR}(e_0^ne_1)=\Delta^{\mathcal W,\DR}(-y_{n+1})\\ & =-y_{n+1}\otimes1-1\otimes y_{n+1}-\sum_{i=1}^n y_i\otimes y_{n+1-i}
\\ & =e_0^ne_1\otimes1+1\otimes e_0^ne_1-\sum_{i=1}^n e_0^{i-1}e_1\otimes e_0^{n-i}e_1=\mathrm{Ad}(e_1f_1)^{-1}\circ\tilde{\BB}(e_0^n),  
\end{align*}
where the third equality follows from (\ref{image:yn}), 
and the last equality follows from Lemma \ref{lemma:53:15:11:2017}.  

It follows that the two maps of the above diagram agree on $e_0^n$, $n\geq1$. Since these maps are algebra morphisms, 
and since the family $e_0^n$, $n\geq0$ generates $(\mathcal V^\DR,\cdot_{e_1})$ (see Lemma \ref{lemma:e0n:generating}), 
this diagram commutes. \hfill\qed\medskip

\begin{prop}\label{prop:CD:DR:28/11/2017}
The following diagram commutes
\begin{equation}\label{diagram:prop:59}
\xymatrix
{
\mathcal V^\DR \ar_{\simeq}^{\diamond}[d]\ar^\rho[r]& 
M_3((\mathcal V^\DR)^{\otimes2})\ar^{(e_1f_1)^{-1}\mathrm{row}_1\cdot(-)\cdot\mathrm{col}_1e_1f_1}_{\diamond\sharp}[rrrr]&&&&
(\mathcal V^\DR[{1\over e_1}]_{\geq0})^{\otimes2}\\
(\mathcal V^\DR,\cdot_{e_1})\ar_{\mathrm{mor}_{\mathcal V^\DR,e_1}}^{\simeq\sharp}[d]
& &&&&\\ 
\mathcal W^{\DR}_+\ar_{\Delta^{\mathcal W,\DR}}[rrrrr]&& &&& (\mathcal W^{\DR})^{\otimes2}\ar@{^{(}->}[uu]}
\end{equation}
where $\rho$ is as in \eqref{equation:B:05012017}, 
$\mathrm{row}_1,\mathrm{col}_1$ are as in \eqref{def:row:col:04012018}, and $\Delta^{\mathcal W,\DR}$ is as in \S\ref{sect:tcDaD}; 
in this diagram, all the maps are algebra morphisms (resp. degree $0$ maps), except for the maps marked with $\diamond$ (resp. 
$\sharp$), which are only $\mathbf k$-module morphisms (degree $1$ maps).  
\end{prop}

\proof This follows from the combination the commutative diagram from Lemma \ref{CD:tildeB:Delta*} with the commutative diagram
$$
\xymatrix{
\mathcal V^\DR \ar^\rho[rr]\ar_{\simeq}[d]&& M_3((\mathcal V^\DR)^{\otimes2})
\ar^{(e_1f_1)^{-1}\mathrm{row}_1\cdot-\cdot\mathrm{col}_1e_1f_1}[d]\\ (\mathcal V^\DR,\cdot_{e_1})
\ar_{\mathrm{Ad}(e_1f_1)^{-1}\circ\tilde\rho}[rr]
&& (\mathcal V^\DR[{1\over e_1}]_{\geq0})^{\otimes2}
}
$$
which follows from the specialization, based on \eqref{525bis}, of the commutative diagram from Lemma \ref{algebraic:lemma:2} 
to $R=\mathcal V^\DR$, $S=\mathcal V^\DR[{1\over e_1}]^{\otimes2}$, $f=\rho$, $e=e_1$, $\mathrm{row}=(e_1f_1)^{-1}\mathrm{row}_1$, 
$\mathrm{col}=\mathrm{col}_1e_1f_1$. 

The map $\rho$ has degree 0 as it is a composition of maps of degree 0. The other degree statements follow from inspection of the degrees of 
$e_1$, $\mathrm{row}_1$, $\mathrm{col}_1$.  \hfill \qed\medskip 

\subsection{Relationship between infinitesimal braid Lie algebras and $\Delta^{\mathcal M,\DR}$}\label{section:proof:T5} 

\begin{defn}\label{def:col0:21012021}
Set 
$$
\mathrm{col}_0:=\begin{pmatrix} 0 \\ -e_1\cdot 1_\DR^{\otimes2} \\ e_1\cdot 1_\DR^{\otimes2} \end{pmatrix}\in 
M_{3\times1}((\mathcal M^\DR)^{\otimes2}).
\index{col_0@$\mathrm{col}_0$} 
$$
\end{defn}

\begin{lem}
Denote by $(a,x)\mapsto ax$ the action of $M_3((\mathcal V^\DR)^{\otimes2})$ on 
$M_{3\times 1}((\mathcal M^\DR)^{\otimes2})$. Then 
\begin{equation}\label{eq:10022019}
\rho(e_0)\mathrm{col}_0=0.
\end{equation}
\end{lem}

\proof Using \eqref{expression:B(e_0):25oct}, one computes $\rho(e_0)\mathrm{col}_0=
\begin{pmatrix} 0\\ -e_1f_0\cdot 1_\DR^{\otimes2} \\ e_1f_0\cdot 1_\DR^{\otimes2} \end{pmatrix}
=0$ in $M_{3\times 1}((\mathcal V^\DR/\mathcal V^\DR e_0)^{\otimes2})
=M_{3\times 1}((\mathcal M^\DR)^{\otimes2})$. 
\hfill\qed\medskip 

\begin{lem}\label{lemma:6:6}
The map $(e_1f_1)^{-1}\mathrm{row}_1\cdot(-)\cdot\mathrm{col}_0 : M_3((\mathcal V^\DR)^{\otimes2})
\to{\mathcal M}^\DR[\frac{1}{e_1}]^{\otimes2}$ constructed out of matrix multiplication and of 
the tensor square of the map ${\mathcal V}^\DR[\frac{1}{e_1}]\otimes\mathcal V^\DR\otimes\mathcal M^\DR\to
{\mathcal M}^\DR[\frac{1}{e_1}]$ induced by the inclusions $\mathcal V^\DR\subset{\mathcal V}^\DR[\frac{1}{e_1}]$ and 
$\mathcal M^\DR\subset{\mathcal M}^\DR[\frac{1}{e_1}]$, the product on ${\mathcal V}^\DR[\frac{1}{e_1}]$, and the 
action of ${\mathcal V}^\DR[\frac{1}{e_1}]$ on ${\mathcal M}^\DR[\frac{1}{e_1}]$, has image contained in 
$({\mathcal M}^\DR[\frac{1}{e_1}]_{\geq-1})^{\otimes2}$. 
\end{lem}

\proof This map is given by $(m_{ij})_{i,j\in[\![1,3]\!]}\mapsto \{f_1^{-1}(m_{13}-m_{12})e_1+e_1^{-1}(m_{22}-m_{23})e_1\}
\cdot 1_\DR^{\otimes2}$.  One has $f_1^{-1}(m_{13}-m_{12})e_1\cdot 1_\DR^{\otimes2}\in
{\mathcal M}^\DR[\frac{1}{e_1}]_{\geq1}\otimes{\mathcal M}^\DR[\frac{1}{e_1}]_{\geq-1}$ and 
$e_1^{-1}(m_{22}-m_{23})e_1\cdot 1_\DR^{\otimes2}\in({\mathcal M}^\DR[\frac{1}{e_1}]_{\geq0})^{\otimes2}$, 
therefore the sum of these elements belongs to the announced space. \hfill\qed\medskip 

\begin{lem}
There is a unique map $\delta:\mathcal M^\DR\to ({\mathcal M}^\DR[\frac{1}{e_1}]_{\geq-1})^{\otimes2}$\index{delta@$\delta$}, 
such that the following diagram commutes 
\begin{equation}\label{diag:delta:10022019}
\xymatrix{
\mathcal V^\DR\ar^{\rho}[rr]\ar_{(-)\cdot 1_\DR}[d] && M_3((\mathcal V^\DR)^{\otimes2})
\ar^{(e_1f_1)^{-1}\mathrm{row}_1\cdot(-)\cdot\mathrm{col}_0}[d]\\ 
\mathcal M^\DR\ar_{\delta}[rr]&&
({\mathcal M}^\DR[\frac{1}{e_1}]_{\geq-1})^{\otimes2}}
\end{equation}
It is such that 
\begin{equation}\label{delta:1:DR}
\delta(1_\DR)=1_\DR^{\otimes2}.
\end{equation}
\end{lem}

\proof If $x\in \mathcal V^\DR$, then 
$$
(e_1f_1)^{-1}\mathrm{row}_1\cdot\rho(xe_0)\cdot\mathrm{col}_0=
(e_1f_1)^{-1}\mathrm{row}_1\rho(x)\rho(e_0)\mathrm{col}_0=0 
$$
by (\ref{eq:10022019}), therefore $\Big(
(e_1f_1)^{-1}\mathrm{row}_1\cdot(-)\cdot\mathrm{col}_0
\Big)\circ\rho(\mathcal V^\DR e_0)=0$. The existence and uniqueness of $\delta$ follow. 
One then computes 
$$
\delta(1_\DR)=(e_1f_1)^{-1}\mathrm{row}_1\cdot \rho(1)\cdot \mathrm{col}_0=
(e_1f_1)^{-1}\mathrm{row}_1\cdot  \mathrm{col}_0
=(e_1f_1)^{-1}e_1f_1\cdot 1_\DR^{\otimes2}=1_\DR^{\otimes2}.
$$ 
\hfill\qed\medskip 

\begin{lem}
The map $\delta$ satisfies the identity 
\begin{equation}\label{id:delta:module}
\forall x\in\mathcal W^\DR,\quad\forall m\in \mathcal M^\DR, \quad
\delta(x\cdot m)=\Delta^{\mathcal W,\DR}(x)\cdot \delta(m), 
\end{equation}
where the module structure in the left-hand side (resp. right-hand side) is that of $\mathcal M^\DR$ over 
$\mathcal W^\DR$ (resp. ${\mathcal M}^\DR[\frac{1}{e_1}]^{\otimes2}$ over $(\mathcal W^\DR)^{\otimes2}$).  
\end{lem}

\proof The identity is obvious if $x=1$. Assume now that $x=ae_1$ with $a\in\mathcal V^\DR$ and that 
$m\in \mathcal M^\DR=\mathcal V^\DR/\mathcal V^\DR e_0$. Let $\tilde m\in\mathcal V^\DR$ be a lift of $m$. 
Then 
\begin{align*}
& \delta(x\cdot m)=\delta(ae_1\cdot m)=(e_1f_1)^{-1}\mathrm{row}_1\cdot \rho(ae_1\tilde m)\cdot\mathrm{col}_0
\\ & =(e_1f_1)^{-1}\mathrm{row}_1\cdot \rho(a)\rho(e_1)\rho(\tilde m)\cdot\mathrm{col}_0
=(e_1f_1)^{-1}\mathrm{row}_1\cdot \rho(a)\cdot \mathrm{col}_1\cdot\mathrm{row}_1\cdot 
\rho(\tilde m)\cdot\mathrm{col}_0
\\ & = (e_1f_1)^{-1}\mathrm{row}_1\cdot \rho(a)\cdot \mathrm{col}_1\cdot e_1f_1 \cdot 
(e_1f_1)^{-1}\mathrm{row}_1\cdot 
\rho(\tilde m)\cdot\mathrm{col}_0
\\ & = \Delta^{\mathcal W,\DR}(x)\cdot \delta(m), 
\end{align*} 
where $\mathrm{row}_1$ is as in \eqref{def:row:col:04012018}, 
the second equality follows from \eqref{diag:delta:10022019}, the third equality follows from the fact that $\rho$
is an algebra morphism, the fourth equality follows from \eqref{525bis} and the last equality follows from the 
combination of 
\eqref{diag:delta:10022019} and the equality $\Delta^{\mathcal W,\DR}(ae_1)=(e_1f_1)^{-1}\mathrm{row}_1\cdot\rho(a)
\cdot\mathrm{col}_1\cdot(e_1f_1)$ for $a\in\mathcal V^\DR$, which follows from equation \eqref{diagram:prop:59}. 
\hfill\qed\medskip

\begin{prop}\label{prop:22:14022019}
The following diagram commutes 
\begin{equation}\label{diagram:1402}
\xymatrix{
\mathcal V^\DR\ar^{\rho}[rrr]\ar_{(-)\cdot 1_\DR}[d] &&& M_3((\mathcal V^\DR)^{\otimes2})
\ar^{(e_1f_1)^{-1}\mathrm{row}_1\cdot(-)\cdot\mathrm{col}_0}[d]\\ 
\mathcal M^\DR\ar_{\Delta^{\mathcal M,\DR}}[rr]&&(\mathcal M^\DR)^{\otimes2}
\ar@{^(->}[r]&({\mathcal M}^\DR[\frac{1}{e_1}]_{\geq-1})^{\otimes2} 
}
\end{equation}
In this diagram, all the maps have degree $0$. 
\end{prop}

\proof Combining \eqref{delta:1:DR}, \eqref{id:delta:module} and the fact that $\mathcal M^\DR$ is a free $\mathcal W^\DR$-module of 
rank one with generator $1_\DR$, one obtains $\delta=\Delta^{\mathcal M,\DR}$, which one injects in (\ref{diag:delta:10022019}) to get 
the result. The statement on degrees follows from inspection of the degrees of $(e_1f_1)^{-1}\mathrm{row}_1$ and $\mathrm{col}_0$. 
\hfill\qed\medskip 

\subsection{Completions 
(commutativities of (A7) in \eqref{diagg:main:alg} and (M5) in \eqref{diagram:main:mod})}
\label{sect:completion:giadrhc}

The following lemma will be used to prove the commutativities 
mentioned in the title of this subsection. 
\begin{lem}\label{lem:completion:diag:LA}
The commutative diagram \eqref{diagram:prop:59} (resp. \eqref{diagram:1402}) gives rise to a commutative diagram between the
degree completions of its constituents, in which the completion of the map $(\mathcal W^\DR)^{\otimes2}
\hookrightarrow({\mathcal V}^\DR[\frac{1}{e_1}]_{\geq0})^{\otimes2}$ (resp. $(\mathcal M^\DR)^{\otimes2}
\hookrightarrow({\mathcal M}^\DR[\frac{1}{e_1}]_{\geq-1})^{\otimes2} $) is injective. 
\end{lem}

\proof By Proposition \ref{prop:CD:DR:28/11/2017}, if one equips $\mathcal V^\DR$ and $M_3((\mathcal V^\DR)^{\otimes2})$ with the 
shifted gradings $\mathcal V^\DR[1]_n:=\mathcal V^\DR_{n-1}$ and $M_3((\mathcal V^\DR)^{\otimes2})[1]_n:=
M_3(((\mathcal V^\DR)^{\otimes2})_{n-1})$, then \eqref{diagram:prop:59} is a diagram in the category $\mathbf k\text{-mod}_{\gr,+}$. 
By Proposition \ref{prop:22:14022019}, \eqref{diagram:1402} is similarly a diagram in the same category. Applying the functor
$(-)^\wedge\circ\mathrm{taut}_{\gr,+}^{\mathrm{fil},+}$ (see \S\ref{sect:functors}), one obtains commutative diagrams in 
$\mathbf k\text{-mod}_{\topo}$. The announced injectivities follows from Lemma \ref{lemma:injectivity:completion}, 1). 
\hfill\qed\medskip 

\section{Braid groups}\label{sect:braids}

In \S\ref{sect:material:on:P5}, we recall the definition of various families of braid groups (the Artin braid group, the sphere 
braid group and the modular group of the sphere with marked points), of their pure subgroups, and of a diagram of morphisms 
relating them (see \eqref{diag:P:K}). We recall the presentation of these groups and relate various generators by morphisms 
(Lemma \ref{lemma:6:3:19032018}). We then give a presentation of the modular group $P_5^*$ which exhibits an order
5 cyclic symmetry, and may be viewed as an analogue of the presentation of $\mathfrak p_5$ in \cite{Ih}, Proposition 4
(this presentation is not used is the sequel of the paper). In \S\ref{subsect:def:gp:morphisms:bis}, we define morphisms 
$\underline\ell:F_2\to P_5^*$, $\underline{\mathrm{pr}}_i:P_5^*\to F_2$ ($i\in[\![1,5]\!]$),  
$\underline{\mathrm{pr}}_{12}:P_5^*\to(F_2)^2$ relating $P_5^*$ with the free group with two generators $F_2$ or its square. 

In \S\ref{sect:6:2:19032018}, we introduce an ideal $J(\underline{\mathrm{pr}}_5)$ of the group algebra $\mathbf kP_5^*$ arising from $\underline{\mathrm{pr}}_5$, and show its freeness as a left $\mathbf kP_5^*$-module 
(Lemma \ref{lemma:decomp:barJ:pr5}). Lemma \ref{algebraic:lemma:1} then gives rise to an algebra morphism 
$\underline\varpi:\mathbf kP_5^*\to M_3(\mathbf kP_5^*)$. By composing $\underline\varpi$ with the morphisms 
$\underline\ell$ and $\underline{\mathrm{pr}}_{12}$, we construct a morphism $\underline\rho:\mathcal V^\B\to 
M_3((\mathcal V^\B)^{\otimes2})$ (see \eqref{underline:B:05012017}). In Lemma \ref{lemma:def:underline:row:col:04012018}, 
by introducing suitable elements $\underline{\mathrm{row}}_1\in M_{1\times3}((\mathcal V^\B)^{\otimes2})$ and 
$\underline{\mathrm{col}}_1\in M_{3\times1}((\mathcal V^\B)^{\otimes2})$, we show that the morphism $\underline\rho$ 
satisfies the hypothesis of Lemma \ref{algebraic:lemma:2}. Applying this lemma, we obtain in \S\ref{subsection:underline:tilde:B} 
a morphism $\underline{\tilde{\BB}}:(\mathcal V^\B,\cdot_{X_1-1})\to (\mathcal V^\B)^{\otimes2}$, which we compute in 
Lemma \ref{lemma:65:22:11:2017}. 

\subsection{Material on braid groups}\label{sect:material:on:P5}

\subsubsection{Braid groups}\label{subsect:def:gp:morphisms}

For $X$ a topological space, let $C_n(X)\index{C_nX@$C_n(X)$}$ denote its configuration space of $n$ distinct points. 
Let also $\mathfrak M_{0,n+1}\index{M_0,n+1@ $\mathfrak M_{0,n+1}$}$ be the moduli space of smooth complex projective curves of genus 
zero with $n+1$ marked points. Below, we give a list
of topological spaces and simply-connected  subspaces, together with standard names and notation for the 
corresponding  fundamental groups (see \cite{Bir,Ih:G,LS}). 
\begin{center}
\begin{tabular}{|c|c|c|c|c|c|c|c|c|c|c|}
\hline
  range of $n$ & $n\geq1$ 
& $n\geq1$ & $n\geq3$\\ 
 \hline
  space $X$ & $C_n(\mathbb C)$ 
& $C_n(\mathbb P^1_{\mathbb C})$ & $C_n(\mathbb P^1_{\mathbb C})/\mathrm{PGL}_2(\mathbb C)\simeq\mathfrak M_{0,n}$  \\ 
\hline   
  subspace  $b$& $\mathcal U_n$ & $\mathcal U_n$ & $\mathcal B_n$ 
 \\ \hline  
notation for $\pi_1(X,b)$ & $K_n$\index{K_n@$K_n$}  & $P_n$\index{P_n@$P_n$} 
& $P^*_n$\index{P_n*@$P^*_n$}  \\
  \hline name of $\pi_1(X,b)$ & 
$\begin{array}{r}\text{pure Artin} \\ \text{braid group}\end{array}$
& $\begin{array}{r}\text{pure sphere (Hurwitz)} \\ \text{braid group}\end{array}$
& $\begin{array}{r}\text{pure modular group of the} \\ \text{sphere with $n$ marked points} \end{array}$
  \\ \hline
\end{tabular}
\end{center}
\begin{center}
\begin{tabular}{|c|c|c|c|c|c|c|c|c|c|c|}
\hline
  range of $n$ & $n\geq1$ 
& $n\geq1$ & $n\geq3$\\
\hline
 space $X$ & $C_n(\mathbb C)/S_n$ 
& $C_n(\mathbb P^1_{\mathbb C})/S_n$ & 
$\begin{array}{r}C_n(\mathbb P^1_{\mathbb C})/(S_n\times\mathrm{PGL}_2(\mathbb C))
\\ \simeq\mathfrak M_{0,n}/S_n\end{array}$ 
 \\ \hline   
  subspace $b$ & $S_n\cdot\mathcal U_n$ & $S_n\cdot\mathcal U_n$ & $S_n\cdot \mathcal B_n$ 
 \\ \hline  
notation for $\pi_1(X,b)$ & $B_n$\index{B_n@$b_n$}  & $H_n$\index{H_n@$H_n$} & $B^*_n$ \index{B_n*@$B^*_n$} \\  \hline 
name of $\pi_1(X,b)$ & Artin braid group & 
$\begin{array}{r}\text{sphere (Hurwitz)} \\ \text{braid group}\end{array}$
& 
$\begin{array}{r}\text{modular group of the} \\ \text{sphere with $n$ marked points} \end{array}$  \\ \hline
\end{tabular}
\end{center}
Here we set $\mathcal U_n:=\{(x_1,\ldots,x_n)\in \mathbb R^n|x_1<\cdots<x_n\}\subset C_n(\mathbb C)\index{U_n @$\mathcal U_n$}$; 
we define  
$\tilde{\mathcal B}_n\subset C_n(\mathbb P^1)\index{B_ncal^tilde@$\tilde{\mathcal B}_n$}$ to be the set of $n$-tuples $(x_1,\ldots,x_n)$ in $(\mathbb P^1_{\mathbb R})^n$, which lie on $\mathbb P^1_{\mathbb R}$ in the counterclockwise order; one has 
$\tilde{\mathcal B}_n=\mathrm{PGL}_2^+(\mathbb R)\cdot\mathcal U_n$; we define 
$\mathcal B_n\index{B_ncal @${\mathcal B}_n$}$ as the quotient $\tilde{\mathcal B}_n/\mathrm{PGL}_2^+(\mathbb R)$; 
it is a simply-connected subspace of $C_n(\mathbb P^1_{\mathbb C})/\mathrm{PGL}_2(\mathbb C)$. 
 
By homotopy exact sequence, the pure groups appear as the kernels of the natural morphisms of their non-pure counterparts to $S_n$, 
so 
$$
K_n=\mathrm{Ker}(B_n\to S_n), \quad P_n=\mathrm{Ker}(H_n\to S_n), \quad P_n^*=\mathrm{Ker}(B_n^*\to S_n). 
$$

According to \cite{LS}, Appendix, $B_n^*$ is isomorphic to the quotient of 
the Artin braid group $B_n$ by the normal subgroup 
generated by $\eta_n\index{eta_n @ $\eta_n$}$ if $n\geq3$, which corresponds to the winding of $x_n$ around $(x_1,\ldots,x_{n-1})$ (see 
(\ref{formula:etan}) for an expression in terms of standard generators), and by its center $Z(B_n)$, 
which is isomorphic to $\mathbb Z$. One has $Z(B_n)\subset K_n$, and $P_{n+1}^*\simeq K_n/Z(B_n)$ for $n\geq2$.
One also has $P_n\simeq P_n^*\times C_2$, where $C_2$ is the cyclic group of order 2
(see \cite{LS}, Proposition A4 iii) and also \cite{Ih:G}, Corollary 2.1.2).

\begin{rem}
The isomorphism $P_{n+1}^*\simeq K_n/Z(B_n)$ can be interpreted as follows. There is an isomorphism 
$\mathfrak M_{0,n+1}\simeq C_n(\mathbb C)/\mathrm{Aff}$, with $\mathrm{Aff}=
\{x\mapsto ax+b\ |\ a\in\mathbb C^\times,b\in\mathbb C\}$; it gives rise to a homotopy exact sequence 
$\pi_2(\mathfrak M_{0,n+1})\to\pi_1(\mathrm{Aff})\to\pi_1(C_n(\mathbb C))\to\pi_1(\mathfrak M_{0,n+1}) \to 1$. 
The spaces $\mathfrak M_{0,n}$ are $K(\pi,1)$-spaces, as can be seen inductively from the homotopy exact sequences
of the fibrations $\mathfrak M_{0,n+1}\to\mathfrak M_{0,n}$, therefore $\pi_2(\mathfrak M_{0,n+1})=1$. One has 
$K_n=\pi_1(C_n(\mathbb C))$, $P^*_{n+1}=\pi_1(\mathfrak M_{0,n+1})$ and $\pi_1(\mathrm{Aff})=\mathbb Z$, so the above exact sequence 
implies $P^*_{n+1}\simeq K_n/\mathbb Z$. 
\end{rem}

\begin{rem}
The isomorphism $P_n\simeq P_n^*\times C_2$ for $n\geq 2$ implies, in the notation of Remark \ref{rem:27122017}, 
the isomorphism $\mathrm{gr}(P_n)\simeq\mathrm{gr}(P_n^*)\times C_2$, 
therefore $\mathrm{gr}(P_n^*)\otimes\mathbf k\simeq\mathrm{gr}(P_n)\otimes\mathbf k\simeq\mathfrak p_n$ as 
$\mathbb Q\subset\mathbf k$. 
\end{rem}

\underline{A diagram of pure braid groups.}
The canonical projection $C_{n+1}(\mathbb C)\to C_{n+1}(\mathbb P^1_{\mathbb C})/\mathrm{PGL}_2(\mathbb C)$
defines a morphism of topological spaces; this map takes $\mathcal U_{n+1}$ to $\mathcal B_{n+1}$, therefore induces a group 
morphism $K_{n+1}\to P_{n+1}^*$. 

Define a morphism $C_n(\mathbb C)\to C_{n+1}(\mathbb P^1_{\mathbb C})/\mathrm{PGL}_2(\mathbb C)$ to be the map taking 
$(x_1,\ldots,x_n)$ to the class of $(x_1,\ldots,x_n,\infty)$. This map takes $\mathcal U_n$ to $\mathcal B_{n+1}$, therefore induces a group morphism $K_n\to P_{n+1}^*$. 

Let $D$ be the open unit disc in $\mathbb C$, let $C_n^D(\mathbb C)$ (resp.\ $\mathcal U_n^D$) be the intersection of 
$C_n(\mathbb C)$ (resp.\ $\mathcal U_n$) with $D^n$. Then $(C_n^D(\mathbb C),\mathcal U_n^D)$ is a deformation 
retract of $(C_n(\mathbb C),\mathcal U_n)$, therefore $K_n\simeq\pi_1(C_n^D(\mathbb C),\mathcal U_n^D)$. 
The morphism $C_n^D(\mathbb C)\to C_{n+1}(\mathbb C)$ given by $(x_1,\ldots,x_n)\mapsto (x_1,\ldots,x_n,2)$
takes $\mathcal U_n^D$ to $\mathcal U_{n+1}$, and therefore induces a morphism $K_n\to K_{n+1}$. 

Then the diagram 
\begin{equation}\label{diag:P:K}
\xymatrix{ K_n\ar@{->>}[r]\ar@{^(->}[d]& P_{n+1}^*\\ K_{n+1}\ar@{->>}[ur]& }
\end{equation}
commutes. 

\underline{Presentations of braid groups.} \label{ref:26122017}
According to \cite{A}, the group $B_n$ is presented for $n\geq1$ 
by generators $\sigma_1,\ldots,\sigma_{n-1}\index{sigma @$\sigma_1,\ldots,\sigma_{n-1}$}$, 
subject to relations $\sigma_i\sigma_j=\sigma_j\sigma_i$ for $|i-j|>1$, and $\sigma_i\sigma_{i+1}\sigma_i
=\sigma_{i+1}\sigma_i\sigma_{i+1}$ for $i\in[\![1,n-1]\!]$. The morphism $B_n\to S_n$ is then given by 
$\sigma_i\mapsto(i,i+1)$. 

For $i<j\in[\![1,n]\!]$, set 
$$
\tilde x_{ij}:=
(\sigma_{j-2}\cdots\sigma_i)^{-1}\sigma_{j-1}^2(\sigma_{j-2}\cdots\sigma_i)=
(\sigma_{j-1}\cdots\sigma_{i+1})\sigma_i^2(\sigma_{j-1}\cdots\sigma_{i+1})^{-1}\in B_n.
\index{x_ij^tilde @ $\tilde x_{ij}$}
$$
(when $j=i+1$, the products $\sigma_{j-2}\cdots\sigma_i$ and $\sigma_{j-1}\cdots\sigma_{i+1}$ 
are equal to 1 by convention, see \eqref{convention:sums}). 
Then one checks that $\tilde x_{ij}\in K_n$. 
According to \cite{A}, a presentation of $K_n$ is given by 
generators $\tilde x_{ij}$, $i<j\in[\![1,n]\!]$, subject to relations 
\begin{equation}\label{relation:K:1}
(a_{ijk},\tilde x_{ij})=(a_{ijk},\tilde x_{ik})=(a_{ijk},\tilde x_{jk})=1
\end{equation}
for $i<j<k$ and $i,j,k\in[\![1,n]\!]$, where $a_{ijk}=\tilde x_{ij}\tilde x_{ik}\tilde x_{jk} \index{a_ijk @ $a_{ijk}$}$, together with 
\begin{equation}\label{relation:K:2}
(\tilde x_{ij},\tilde x_{kl})=(\tilde x_{ik},\tilde x_{ij}^{-1}\tilde x_{jl}\tilde x_{ij})=(\tilde x_{il},\tilde x_{jk})=1.  
\end{equation} 
for $i<j<k<l$ and $i,j,k,l\in[\![1,n]\!]$. 

Set 
$$
\omega_n:=\tilde x_{12}\cdot (\tilde x_{13}\tilde x_{23})\cdots (\tilde x_{1n}\cdots\tilde x_{n-1,n})\in K_n.
\index{omega_n @ $\omega_n$}
$$ 
Then $\omega_n$ is a generator of $Z(B_n)$, therefore 
\begin{equation}\label{eq:Pn:Kn:crm}
P_{n+1}^*=K_n/\langle \omega_n\rangle.
\end{equation} 

For $n\geq 2$ and $i\in[\![1,n]\!]$, define $\tilde x_{i,n+1}\in K_n$ by 
\begin{equation}\label{def:x:i:n+1}
\tilde x_{i,n+1}:=(\tilde x_{1i}\cdots \tilde x_{i-1,i}\tilde x_{i,i+1}\cdots \tilde x_{in})^{-1}.
\index{x_in+1 @ $\tilde x_{i,n+1}$}
\end{equation}
For $i<j\in[\![1,n+1]\!]$, we denote by $x_{ij}\in P_{n+1}^* \index{x_ij @ $x_{ij}$}$ the image of $\tilde x_{ij}\in K_n$ 
by the projection \eqref{eq:Pn:Kn:crm}. 

\begin{lem}\label{lemma:6:3:19032018}
If $i<j\in[\![1,n+1]\!]$, then $x_{ij}\in P^*_{n+1}$ is the image of $\tilde x_{ij}\in K_{n+1}$ under the 
morphism $K_{n+1}\to P^*_{n+1}$.  
\end{lem}

\proof If $j\leq n$, then the image of $\tilde x_{ij}\in K_n$ under $K_n\to P^*_{n+1}$ is 
$x_{ij}\in P^*_{n+1}$, and the image of the same element under $K_n\to K_{n+1}$
is $\tilde x_{ij}\in K_n$, so the result follows from the commutativity of (\ref{diag:P:K}). 

Assume now that $j=n+1$. According to \cite{LS}, three lines after (A1), one has 
\begin{equation}\label{formula:etan}
\eta_{n+1}=\sigma_n\cdots\sigma_1^2\cdots \sigma_n
\end{equation}
(equality in $B_{n+1}$). Moreover, one checks that for $i\in[\![1,n]\!]$, 
$$
\sigma_i\cdots\sigma_n\cdot \eta_{n+1}^{-1}\cdot (\sigma_i\cdots\sigma_n)^{-1}
=\tilde x_{1i}\cdots\tilde x_{i-1,i}\tilde x_{i,i+1}\cdots\tilde x_{i,n+1}
$$
(equality in $B_{n+1}$). The left-hand side belongs to the kernel of the morphism $B_{n+1}\to B_{n+1}^*$, and the right-hand 
side belongs to the subgroup $K_{n+1}\subset B_{n+1}$, therefore the latter side belongs to the kernel of the 
composed morphism $K_{n+1}\subset B_{n+1}\to B_{n+1}^*$. As this morphism factors as $K_{n+1}\twoheadrightarrow P_{n+1}^*
\hookrightarrow B_{n+1}^*$, the right-hand side belongs to $\mathrm{Ker}(K_{n+1}\twoheadrightarrow P_{n+1}^*)$. 
\hfill \qed\medskip 

\underline{Diagrammatic representation.}
The generators of $B_n$ are depicted as follows when $n=4$,  
\begin{tikzpicture}[xscale=0.5,yscale=0.3]
\draw (3,0) node{$\sigma_1=$} ; 
\braid[number of strands=4]  (braid) at (4,1)  s_1;
\end{tikzpicture} , 
\begin{tikzpicture}[xscale=0.5,yscale=0.3]
\draw (3,0) node{$\sigma_2=$} ; 
\braid[number of strands=4]   (braid) at (4,1)  s_2;
\end{tikzpicture}, 
\begin{tikzpicture}[xscale=0.5,yscale=0.3]
\draw (3,0) node{$\sigma_3=$} ; 
\braid[number of strands=4]    (braid) at (4,1)  s_3;
\end{tikzpicture}
and the convention for the product is 
\begin{tikzpicture}[xscale=0.5,yscale=0.3]
\draw (2.5,-0.7) node{$\sigma_2\sigma_1=$} ; 
\braid[number of strands=4]    (braid) at (4,1)  s_1s_2;
\end{tikzpicture}

The element $\tilde x_{ij}\in K_n$ is then depicted as follows 
\begin{center}
           \begin{tikzpicture}
                              \draw(-1.2,0.5) node{$\tilde x_{ij}=$};
                 \draw[-] (-0.5,0)--(-0.5,1) (-0.2,0)--(-0.2,1)
(0.2,0)--(0.2,1) (0.3,0)--(0.3,1) (0.6,0)--(0.6,1) (1.2,0)--(1.2,1)
(1.5,0)--(1.5,1);
                 \draw[dotted] (-0.5,0.5)--(-0.2,0.5)
(0.3,0.5)--(0.6,0.5) (1.2,0.5)--(1.5,0.5);
                 \draw[-] (0.8,1)--(0.8,0.5);
                 \draw[color=white, line width=7pt](0,0)
..controls(0.1,0.5) and (0.9,0)        ..(1,0.5);
                 \draw[-] (0,0) ..controls(0.1,0.5) and (0.9,0)
..(1,0.5);
                 \draw[color=white, line width=7pt](1,0.5)
..controls(0.9,1.0) and (0.1,0.5)       ..(0,1);
                 \draw[-] (1,0.5) ..controls(0.9,1.0) and (0.1,0.5)
   ..(0,1);
                 \draw[color=white, line width=5pt] (0.8,0.5)--(0.8,0);
                 \draw[-] (0.8,0.5)--(0.8,0);
\draw[decorate,decoration={brace,mirror}] (-0.6,-0.1) -- (-0.1,-0.1)
node[midway,below]{$i-1$};
\draw[decorate,decoration={brace}] (-0.6,1.1) -- (0.6,1.1)
node[midway,above]{$j-1$};
            \end{tikzpicture}
\end{center}
According to these conventions, the diagrammatic representatives of the elements $\tilde x_{i,5}\in B_4$ given by 
\eqref{def:x:i:n+1} are the following 
\begin{figure}[h]
\begin{tabular}{c}
  \begin{minipage}{0.25\hsize}
      \begin{center}
\begin{tikzpicture}
  \draw (-1.5,0.2) node{$\tilde x_{15}=$};
  \draw[-] (0.6,0.25)..controls(0.6,0.75) and (-0.5,0.5)  ..(-0.5,1);
  \draw[color=white, line width=5pt] (-0.2,0)--(-0.2,1) (0.1,0)--(0.1,1) (0.4,0)--(0.4,1);
  \draw[-] (-0.2,-0.5)--(-0.2,1) (0.1,-0.5)--(0.1,1) (0.4,-0.5)--(0.4,1);
  \draw[color=white, line width=5pt] (0.1,0.5)--(0.1,1);
  \draw[-] (0.1,0.5)--(0.1,1);
  \draw[color=white, line width=5pt]  (-0.5,-0.5) ..controls(-0.5,0) and (0.6,-0.25) ..(0.6,0.25);
  \draw[-] (-0.5,-0.5) ..controls(-0.5,0) and (0.6,-0.25) ..(0.6,0.25);
\end{tikzpicture},
     \end{center}
 \end{minipage}
 \begin{minipage}{0.25\hsize}
      \begin{center}
\begin{tikzpicture}
  \draw (-1.5,0.2) node{$\tilde x_{25}=$};
  \draw[-] (-0.5,0.7)--(-0.5,1);
    \draw[color=white, line width=5pt] (-0.2,1)..controls(-0.2,0.5) and (-0.7,1)  ..(-0.7,0.5);
  \draw[-] (-0.2,1)..controls(-0.2,0.5) and (-0.7,1)  ..(-0.7,0.5);
   \draw[-] (-0.7,0.5)..controls(-0.7,0) and (0.6,0.5)  ..(0.6,0);
   \draw[color=white, line width=5pt] (-0.5,0.4)--(-0.5,-0.5) (0.1,0)--(0.1,1) (0.4,0)--(0.4,1);
  \draw[-] (-0.5,0.6)--(-0.5,-0.5)  (0.1,-0.5)--(0.1,1) (0.4,-0.5)--(0.4,1);
     \draw[color=white, line width=5pt](-0.2,-0.5) ..controls(-0.2,0) and (0.6,-0.5) ..(0.6,0);
     \draw[-] (-0.2,-0.5) ..controls(-0.2,0) and (0.6,-0.5) ..(0.6,0);  
\end{tikzpicture},
     \end{center}
 \end{minipage}
 \begin{minipage}{0.25\hsize}
      \begin{center}
\begin{tikzpicture}
  \draw (-1.5,0.2) node{$\tilde x_{35}=$};
    \draw[-] (-0.5,1)--(-0.5,0.4) (-0.2,1)--(-0.2,0.4) ;
    \draw[color=white, line width=5pt]  (0.1,1) ..controls(0.1,0.5) and (-0.7,1) ..(-0.7,0.5);
  \draw[-] (0.1,1) ..controls(0.1,0.5) and (-0.7,1) ..(-0.7,0.5);
(-0.2,1)--(-0.2,-0.5);
   \draw[-] (-0.7,0.5) ..controls(-0.5,0) and (0.6,0.5) ..(0.6,0);
  \draw[color=white, line width=5pt]  (-0.5,0.45)--(-0.5,-0.5)  (-0.2,0.4)--(-0.2,-0.5) (0.4,1)--(0.4,-0.5);
   \draw (-0.5,0.45)--(-0.5,-0.5)  (-0.2,0.45)--(-0.2,-0.5) (0.4,1)--(0.4,-0.5);
  \draw[color=white, line width=5pt] (0.6,0)..controls(0.6,-0.5) and (0.1,0)  ..(0.1,-0.5);
  \draw[-] (0.6,0)..controls(0.6,-0.5) and (0.1,0)  ..(0.1,-0.5);
\end{tikzpicture},
     \end{center}
 \end{minipage}
 \begin{minipage}{0.25\hsize}
      \begin{center}
\begin{tikzpicture}
  \draw (-1.5,0.2) node{$\tilde x_{45}=$};
   \draw[-] (-0.5,0)--(-0.5,1) (-0.2,0)--(-0.2,1) (0.1,0)--(0.1,1);
   \draw[color=white, line width=5pt](-0.7,0.25)..controls(-0.7,0.75) and (0.4,0.5)  ..(0.4,1);
  \draw[-] (-0.7,0.25)..controls(-0.7,0.75) and (0.4,0.5)  ..(0.4,1);
  \draw[-] (-0.7,0.25) ..controls(-0.7,-0.25) and (0.6,0) ..(0.4,-0.5);
  \draw[color=white, line width=5pt] (-0.5,0.2)--(-0.5,-0.5) (-0.2,0)--(-0.2,-0.5) (0.1,0)--(0.1,-0.5);
  \draw[-] (-0.5,0.3)--(-0.5,-0.5) (-0.2,0)--(-0.2,-0.5) (0.1,0)--(0.1,-0.5);
\end{tikzpicture}.
     \end{center}
 \end{minipage}
\end{tabular}
\end{figure}
\par\noindent 

For $a,b\geq 1$, we define 
$\sigma_{a,b}\in B_{a+b}$\index{sigma_ab@$\sigma_{a,b}$} to be the element represented by 
\begin{equation}\label{def:sigma:a:b}
           \begin{tikzpicture}[baseline=(current  bounding  box.center),xscale=0.8, yscale=1]
        \draw (-0.5,0.5) node{$\sigma_{a,b}=$} ; 
        \draw[-] (0,1) --(1.2, 0) (0.2,1)--(1.4,0)  (0.8,1)--(2,0);
        \draw[decorate,decoration={brace}] (-0.1,1.1) -- (0.8,1.1) node[midway,above]{$a$};
        \draw[dotted] (0.3,0.9) --(0.8, 0.9)  ;

         \draw[color=white, line width=5pt] (0,0) --(1.2, 1) (0.2,0)--(1.4,1)  (0.8,0)--(2,1);
         \draw[-] (0,0) --(1.2, 1) (0.2,0)--(1.4,1)  (0.8,0)--(2,1);
         \draw[decorate,decoration={brace}] (1.1,1.1) -- (2.1,1.1) node[midway,above]{$b$};
         \draw[dotted] (0.8,0.45) --(1.3, 0.45)  ;
         \draw (2.7,0.5) node{$\in B_{a+b}$} ; 
        \end{tikzpicture}.
    \end{equation}
Note the equality $\sigma_1=\sigma_{1,1}^{-1} \in B_2$. 

The following result may be viewed as an analogue of Proposition 4 in \cite{Ih}. 
\begin{prop}\label{prop:pres:P5}
For $i\in  C_5\simeq [\![1,5]\!]$, define $g_i\in P_5^*$ by $g_i:=x_{i,i+1}$ (with the convention $x_{\overline 0,\overline 1}:=x_{1,5}$). 
The group $P_5^*$ is presented by generators $g_i\index{gi@ $g_i$ ($i\in C_5$)}$ ($i\in C_5$), subject to the relations
$$
(g_i,g_j)=1 \text{ if } i,j\in C_5 \text{\ and\ }i-j\neq \pm\overline 1, \quad 
(g_{\overline 0},g_{\overline 1})(g_{\overline 1},g_{\overline 2})(g_{\overline 2},g_{\overline 3})(g_{\overline 3},g_{\overline 4})
(g_{\overline 4},g_{\overline 0})=1.
$$
\end{prop}

\proof It follows from \eqref{def:x:i:n+1} that for $i\in[\![1,4]\!]$, one has 
\begin{equation}\label{corr:2f}
x_{i,5}=(x_{1i}\cdots x_{i-1,i}x_{i,i+1}\cdots x_{i4})^{-1}\in P_5^*. 
\end{equation}
Relation $\omega_5=1$, namely 
\begin{equation}\label{relation:center}
x_{12}x_{13}x_{23}x_{14}x_{24}x_{34}=1, 
\end{equation}
together with relation (\ref{corr:2f}) for $i=4$, implies $g_{\overline 4}=x_{45}=x_{12}x_{13}x_{23}$, 
therefore 
\begin{equation}\label{generation:1}
x_{13}=g_{\overline 1}^{-1}g_{\overline 4}g_{\overline 2}^{-1}.
\end{equation} 
Relation (\ref{relation:center})
together with $(x_{14},x_{23})=1$ yields 
$x_{12}x_{13}x_{14}x_{23}x_{24}x_{34}=1$,
which together with 
 (\ref{corr:2f}) for $i=1$, yields $x_{15}=x_{23}x_{24}x_{34}$. This relation yields
\begin{equation}\label{generation:2}
x_{24}=g_{\overline 2}^{-1}g_{\overline 0}g_{\overline 3}^{-1}.
\end{equation}
By the commutation of 
$x_{34}$ with $x_{23}x_{24}x_{34}$, this relation also yields $g_{\overline 0}=x_{15}=x_{24}x_{34}x_{23}=x_{24}x_{34}g_{\overline 2}$, therefore $x_{34}^{-1}x_{24}^{-1}=g_{\overline 2}g_{\overline 0}^{-1}$. 
(\ref{relation:center}) implies $x_{14}=x_{23}^{-1}x_{13}^{-1}x_{12}^{-1}x_{34}^{-1}x_{24}^{-1}$, 
which after combination with the previous equality yields $x_{14}=x_{23}^{-1}x_{13}^{-1}x_{12}^{-1}g_{\overline 2}
g_{\overline 0}^{-1}=
g_{\overline 2}^{-1}x_{13}^{-1}g_{\overline 1}^{-1}g_{\overline 2}g_{\overline 0}^{-1}$. 
Combining with (\ref{generation:1}), we obtain 
\begin{equation}\label{generation:3}
x_{14}=g_{\overline 4}^{-1}g_{\overline 2}g_{\overline 0}^{-1}.
\end{equation} 
As $P_5^*$ is a quotient of $K_4$ it is generated by $\{x_{ij}|i<j\in[\![1,4]\!]]\}$, therefore also by 
the union of this set with $g_{\overline 0}$, which is equal to $\{g_i|i\in C_5\}\cup\{x_{13},x_{24},x_{14}\}$. 
Relations (\ref{generation:1}), (\ref{generation:2}) and (\ref{generation:3}) then imply that a generating set is 
$\{g_i|i\in C_5\}$. 

The group $P_5^*$ may be viewed as generated by $\{g_i|i\in C_5\}\cup\{x_{13},x_{24},x_{14}\}$, subject to 
relations (\ref{relation:K:1}), (\ref{relation:K:2}), $\omega_5=1$, and (\ref{corr:2f}) for $i=1,4$. 
These relations imply expressions (\ref{generation:1}), (\ref{generation:2}) and (\ref{generation:3}) of 
$x_{13},x_{24},x_{14}$ in terms of $\{g_i|i\in C_5\}$. Substituting these expressions in relations 
(\ref{relation:K:1}), (\ref{relation:K:2}), $\omega_5=1$, and (\ref{corr:2f}) for $i=1,4$, we obtain
a presentation of $P_5^*$ in terms of the generators $\{g_i|i\in C_5\}$. The relations obtained in this way are the following. 
Relation $\omega_5=1$ yields the commutation of $g_{\overline 0}$ with $g_{\overline 2}$. 
Relation (\ref{relation:K:1}) for $(i,j,k)=(1,2,3)$ yields the commutation of $g_{\overline 4}$ with $g_{\overline 1}$ and $g_{\overline 2}$.
Relation (\ref{relation:K:1}) for $(i,j,k)=(2,3,4)$ yields the commutation of $g_{\overline 0}$ with $g_{\overline 2}$ and $g_{\overline 3}$.
The first part of relation (\ref{relation:K:2}), namely $(x_{12},x_{34})=1$, yields the commutation of 
$g_{\overline 1}$ with $g_{\overline 3}$. The last part of relation (\ref{relation:K:2}), namely $(x_{14},x_{23})=1$, 
yields a consequence of the already obtained commutations of $g_{\overline 2}$ with $g_{\overline 0}$ and $g_{\overline 4}$.
The middle part of relation (\ref{relation:K:2}), namely $(x_{13},x_{12}^{-1}x_{24}x_{12})=1$, together with the 
commutations $(g_{\overline 1},g_{\overline 3})=(g_{\overline 0},g_{\overline 2})=1$, yields the relation 
$g_{\overline 0}g_{\overline 2}^{-1}g_{\overline 1}g_{\overline 3}^{-1}g_{\overline 2}g_{\overline 4}^{-1}g_3g_{\overline 0}^{-1}g_{\overline 4}g_{\overline 1}^{-1}=1$, 
which by again using the commutation relations yields $g_{\overline 0}g_{\overline 4}^{-1}g_{\overline 1}g_{\overline 0}^{-1}g_{\overline 2}g_{\overline 1}^{-1}g_{\overline 3}g_{\overline 2}^{-1}g_{\overline 4}g_{\overline 3}^{-1}=1$, 
which is equivalent to the cyclic relation $(g_{\overline 0},g_{\overline 1})(g_{\overline 1},g_{\overline 2})(g_{\overline 2},g_{\overline 3})(g_{\overline 3},g_{\overline 4})(g_{\overline 4},g_{\overline 0})=1$. 

Relation (\ref{relation:K:1}) for $(i,j,k)=(1,2,4)$ splits as the conjunction of 
$x_{12}x_{14}x_{24}=x_{14}x_{24}x_{12}$ and $x_{14}x_{24}x_{12}=x_{24}x_{12}x_{14}$. 
The first relation yields a consequence of already obtained relations, namely 
$(g_{\overline 0},g_{\overline 2})=(g_{\overline 1},g_{\overline 3})=(g_{\overline 1},g_{\overline 4})=1$. 
After using $(g_{\overline 0},g_{\overline 2})=(g_{\overline 1},g_{\overline 3})=1$, the second relation yields 
$g_{\overline 0}g_{\overline 3}^{-1}g_{\overline 1}g_{\overline 4}^{-1}g_{\overline 2}g_{\overline 0}^{-1}g_{\overline 3}g_{\overline 1}^{-1}g_{\overline 4}g_{\overline 2}^{-1}=1$ which as above is equivalent to the already obtained cyclic relation.  

Relation (\ref{relation:K:1}) for $(i,j,k)=(1,3,4)$ splits as the conjunction of 
$x_{13}x_{14}x_{34}=x_{34}x_{13}x_{14}$ and $x_{34}x_{13}x_{14}=x_{14}x_{34}x_{13}$. 
The first relation yields a consequence of the already obtained relations 
$(g_{\overline 2},g_{\overline 4})=(g_{\overline 3},g_{\overline 1})=(g_{\overline 3},g_{\overline 0})=1$. After using 
$(g_{\overline 2},g_{\overline 4})=1$ and the commutation of $g_{\overline 1}$ with $g_{\overline 3}^{-1}$, the second relation yields
the already obtained relation $g_{\overline 0}g_{\overline 3}^{-1}g_{\overline 1}g_{\overline 4}^{-1}g_{\overline 2}g_{\overline 0}^{-1}g_{\overline 3}g_{\overline 1}^{-1}g_{\overline 4}g_{\overline 2}^{-1}=1$. 
\hfill\qed\medskip 

\begin{rem}\label{remark:6:5:08012018}
Under the commutation relations, the cyclic relation $\prod_{i\in C_5}(g_i,g_{i+1})=1$ (using the notation 
$\prod_{i\in C_5}a_i:=a_{\overline 0}a_{\overline 1}\cdots a_{\overline 4}$) is equivalent to any of the relations 
$\prod_{i\in C_5}g_ig_{i+j}^{-1}=1$, where $j\in C_5-\{\overline 0,\overline 1\}$. For $j=\overline 2$, this relation 
expresses as $\prod_{i\in C_5}(g_i^{-1},g_{i+1}^{-1})=1$. The cyclic relation is also equivalent to the relation
 $\prod_{i\in C_5}(g_{-i},g_{-i-1})=1$. All this proves that the group $D_5\times C_2$ acts by automorphisms of 
$P_5^*$ as follows: the dihedral group $D_5\index{D_5 @ $D_5$}$ acts by permutation of indices of the generators $(g_i)_{i\in C_5}$
and the cyclic group $C_2$ acts by $g_i\mapsto g_i^{-1}$. 
\end{rem}

\subsubsection{The morphisms $\underline\ell,\underline{\mathrm{pr}}_i$ and $\underline{\mathrm{pr}}_{12}$
between braid groups}\label{subsect:def:gp:morphisms:bis}

Let $F_2\index{F_2@$F_2$}$ be the free group with generators $X_0$, $X_1$
\index{X_0, X_1@$X_0$, $X_1$}.
Similarly to \S\ref{sect:514:12122017}, we will abuse notation by setting for $i=0,1$
\begin{equation}\label{def:of:Xi:Yi}
X_i:=(X_i,1)\in(F_2)^2,\quad  Y_i:=(1,X_i)\in(F_2)^2.
\index{Y_0, Y_1@$Y_0$, $Y_1$}
\end{equation}
We then have $X_i=X_i\otimes1$, $Y_i=1\otimes X_i$ in $(\mathcal V^\B)^{\otimes2}$. 

\begin{lem}\label{lemma:6:6:20190107}
1) There are group morphisms $\underline{\mathrm{pr}}_i:P_5^*\to F_2\index{pr_125_underline@ $\underline{\mathrm{pr}}_1$, $\underline{\mathrm{pr}}_2$, $\underline{\mathrm{pr}}_5$}$ for $i=1,2,5$, given by 

\begin{tabular}{|c|c|c|c|c|c|c|c|c|c|c|}
  \hline
$x\in P_5^*$ & $x_{12}$ & $x_{13}$ & $x_{14}$ & $x_{15}$ & $x_{23}$& $x_{24}$& $x_{25}$& $x_{34}$& $x_{35}$& $x_{45}$ \\
  \hline
  $\underline{\mathrm{pr}}_1(x)$ & $1$ & $1$ & $1$ & $1$ & $X_0$ & $(X_1X_0)^{-1}$ & $X_1$ & $X_1$ & $(X_0X_1)^{-1}$ & $X_0$ \\ 
\hline   
$\underline{\mathrm{pr}}_2(x)$ & $1$ & $(X_0X_1)^{-1}$ & $X_0$ & $X_1$ & $1$ & $1$ & $1$  & $X_1$  & $X_1^{-1}X_0X_1$ & 
$(X_0X_1)^{-1}$ \\
  \hline $\underline{\mathrm{pr}}_5(x)$ & $X_1$ & $(X_0X_1)^{-1}$ & $X_0$ & $1$ & $X_0$ & $(X_1X_0)^{-1}$ & $1$ & $X_1$ & $1$ & $1$ 
\\ \hline
\end{tabular}

2) There is a group morphism $\underline\ell:F_2\to P_5^*$\index{l_underline@$\underline\ell$}, given by $X_0\mapsto x_{23}$, $X_1\mapsto x_{12}$.  
We have $\underline{\mathrm{pr}}_5\circ\underline\ell=\mathrm{id}_{F_2}$. 
\end{lem}

\proof Direct computation. \hfill \qed\medskip 

Define $\underline{\mathrm{pr}}_{12}:P_5^*\to (F_2)^2\index{pr_12t_underline@$\underline{\mathrm{pr}}_{12}$}$ as the morphism $\underline p\mapsto (\underline{\mathrm{pr}}_1(\underline p),
\underline{\mathrm{pr}}_2(\underline p))$. We will denote by $\underline\ell,\underline{\mathrm{pr}}_i$ and $\underline{\mathrm{pr}}_{12}$ 
the Hopf algebra morphisms relating the group algebras $\mathbf k F_2$, $(\mathbf k F_2)^{\otimes2}$ and $\mathbf k P_5^*$
induced by $\underline\ell,\underline{\mathrm{pr}}_i$ and $\underline{\mathrm{pr}}_{12}$. 

\begin{rem}  The pure modular group of the sphere with 4 marked points $P_4^*$ is freely generated by $x_{12},x_{23}$, 
and therefore isomorphic to $F_2$. Composing with this isomorphism 
the morphisms $\underline{\mathrm{pr}}_i$ (resp. $\underline\ell$), one obtains the morphisms $P_5^*\to P_4^*$ (resp. $P_4^*\to P_5^*$) 
induced by the morphisms between moduli spaces consisting in forgetting the $i$-th marked point (resp. doubling the fourth marked
point). 
\end{rem}

\subsection{Algebraic constructions related to an ideal of $\mathbf k P_5^*$}\label{sect:6:2:19032018}

\subsubsection{The structure of $J(\underline{\mathrm{pr}}_5)$}

\begin{defn}
{\it We denote by $J(\underline{\mathrm{pr}}_5)\index{Jpr_5_underline @$J(\underline{\mathrm{pr}}_5)$}$ the kernel 
$\mathrm{Ker}(\mathbf k P_5^*\stackrel{\underline{\mathrm{pr}}_5}{\to}\mathbf k F_2)$.
This is a two-sided ideal of $\mathbf k P_5^*$.} 
\end{defn}

Let $F_3\index{F_3 @ $F_3$}$ be the free group with generators $a_i$,
$i\in[\![1,3]\!]$; there is a unique group morphism $F_3\to P_5^*$, given by 
$a_i\mapsto x_{i5}$ for $i\in[\![1,3]\!]$. 

\begin{lem}\label{lemma:60:16:11:2017}
1) The morphisms $\underline{\mathrm{pr}}_5$ and $F_3\to P_5^*$ fit in an exact sequence $1\to F_3\to P_5^*\to F_2\to1$.
As $F_3\to P_5^*$ is injective, we will identify $F_3$ with its image in $P_5^*$.  

2) The map $F_2\times F_3\to P_5^*$, $(\underline f,\underline f')\mapsto \underline\ell(\underline f)\cdot \underline f'$
is a bijection. 
\end{lem}

\proof 1) follows from \cite{Ih:G}, Proposition 2.1.3, based on \cite{FvB}. 2) then follows from the fact that $\underline\ell$ is a 
section of $\underline{\mathrm{pr}}_5$. \hfill\qed\medskip 

\begin{lem} \label{lemma:decomp:barJ:pr5}
The map $(\mathbf k P_5^*)^{\oplus3}\to J(\underline{\mathrm{pr}}_5)$, $(\underline p_i)_{i\in[\![1,3]\!]}\mapsto\sum_{i\in[\![1,3]\!]}
\underline p_i\cdot (x_{i5}-1)$ 
is an isomorphism of left $\mathbf k P_5^*$-modules. 
\end{lem}

\proof The bijection from Lemma \ref{lemma:60:16:11:2017}, 2) induces a linear isomorphism $\mathbf k F_2\otimes
\mathbf k F_3\stackrel{\simeq}{\to} \mathbf k P_5^*$. For $(\underline f,\underline f')\in F_2\times F_3$, 
$\underline{\mathrm{pr}}_5(\underline\ell(\underline f)\cdot \underline f')=\underline f$ as $\underline{\mathrm{pr}}_5\circ 
\underline\ell=\mathrm{id}_{F_2}$ and $F_3=\mathrm{Ker}(\underline{\mathrm{pr}}_5)$. It follows that the
following diagram is commutative 
$$
\xymatrix{
\mathbf k F_2\otimes
\mathbf k F_3\ar^\simeq[r]\ar_{\mathrm{id}\otimes\varepsilon}[dr] & \mathbf k P_5^* \ar^{\underline{\mathrm{pr}}_5}[d]\\ 
 & \mathbf k F_2}
$$
It follows that $J(\underline{\mathrm{pr}}_5)$ is the image of $\mathbf k F_2\otimes
(\mathbf k F_3)_+$ under the linear isomorphism $\mathbf k F_2\otimes
\mathbf k F_3\stackrel{\simeq}{\to} \mathbf k P_5^*$. 

Since $x_{i5}\in F_3$ for $i\in[\![1,3]\!]$, the diagram 
$$
\xymatrix{
(\mathbf k F_2\otimes
\mathbf k F_3)^{\oplus3}\ar^\simeq[r]\ar[d] & (\mathbf k P_5^*)^{\oplus3}\ar[d]\\ \mathbf k F_2\otimes
\mathbf k F_3\ar^\simeq[r] &\mathbf k P_5^*}
$$
commutes, where the vertical maps are, on the left-hand side, the tensor product of the $\mathrm{id}_{\mathbf k F_2}$
with $\mathbf k F_3^{\oplus3}\to \mathbf k F_3$, $(\underline f_i)_{i\in[\![1,3]\!]}\mapsto \sum_{i\in[\![1,3]\!]}\underline f_i
\cdot(x_{i5}-1)$, and on the right-hand side, the map given by the same formula.  
 
It follows from Lemma \ref{lem:decomp:ZZG} that $\mathbf k F_2\otimes
(\mathbf k F_3)_+$ is the isomorphic image of the left vertical map, therefore $J(\underline{\mathrm{pr}}_5)$
is the isomorphic image of $(\mathbf k P_5^*)^{\oplus3}$ by the right vertical map, which proves the claimed statement. 
\hfill\qed\medskip 

\subsubsection{A morphism $\underline\AAA:\mathbf k P_5^*\to M_3(\mathbf k P_5^*)$}
\label{subsect:def:bar:A}

Lemma \ref{lemma:decomp:barJ:pr5} says that the hypothesis of Lemma \ref{algebraic:lemma:1} is satisfied 
in the following situation: $R=\mathbf k P_5^*$, $J=J(\underline{\mathrm{pr}}_5)$, $d=3$, $(j_a)_{a\in[\![1,d]\!]}
=(x_{i5}-1)_{i\in[\![1,3]\!]}$. We denote by 
$$
\underline\AAA:\mathbf k P_5^*\to M_3(\mathbf k P_5^*)
\index{pivar_underline@$\underline\AAA$}
$$
the algebra morphism given in this situation by Lemma \ref{algebraic:lemma:1}. Then for $\underline p\in \mathbf k P_5^*$, 
$\underline\AAA(\underline p)=(\underline a_{ij}(\underline p))_{i,j\in[\![1,3]\!]}$, and 
$$
\forall i\in[\![1,3]\!],\quad (x_{i5}-1)\underline p=\sum_{j\in[\![1,3]\!]}\underline a_{ij}(\underline p)(x_{j5}-1) 
$$
(equalities in $\mathbf k P_5^*$). 

\subsubsection{Construction and properties of a morphism $\mathcal V^\B\to M_3((\mathcal V^\B)^{\otimes2})$}

Define the algebra morphism 
\begin{equation}\label{underline:B:05012017}
\underline{\BB}:\mathcal V^\B\to M_3((\mathcal V^\B)^{\otimes2})
\index{rhovar_underline@$\underline{\BB}$}
\end{equation}
to be the composition 
$$
\mathcal V^\B\stackrel{\underline \ell}{\to} \mathbf k P_5^*\stackrel{\underline\AAA}{\to}
M_3(\mathbf k P_5^*)\stackrel{M_3(\underline{\mathrm{pr}}_{12})}{\to}
M_3((\mathcal V^\B)^{\otimes2}), 
$$
where $\underline\ell$ is as in 2) of Lemma \ref{lemma:6:6:20190107}, 
$\underline\AAA$ is as in \S\ref{subsect:def:bar:A}, 
and 
$M_3(\underline{\mathrm{pr}}_{12})\index{M_3pr_12 @$M_3(\underline{\mathrm{pr}}_{12})$}$ is the  
morphism induced by $\underline{\mathrm{pr}}_{12}$, i.e., taking $(\underline p_{ij})_{i,j\in[\![1,3]\!]}$ to $(\underline{\mathrm{pr}}_{12}(\underline p_{ij}))_{i,j\in[\![1,3]\!]}$. 

\begin{lem}\label{lemma:def:underline:row:col:04012018}\label{lemma:decomp:underline:A:X_1-1}
Set 
\begin{equation}\label{def:underline:row:col:04012018}
\underline{\mathrm{row}}_1:=\begin{pmatrix} X_1-1& 1-Y_1 & 0 \end{pmatrix}\in M_{1\times 3}((\mathcal V^\B)^{\otimes2}), \quad  
\underline{\mathrm{col}}_1:=\begin{pmatrix} Y_1 \\ -1 \\ 0\end{pmatrix}\in M_{3\times 1}((\mathcal V^\B)^{\otimes2})
\index{row1_underline @ $\underline{\mathrm{row}}_1$}
\index{col1_underline @ $\underline{\mathrm{col}}_1$}
\end{equation} 
(where $X_1,Y_1\in F_2^2\subset (\mathcal V^\B)^{\otimes2}$ are defined by (\ref{def:of:Xi:Yi})), then 
\begin{equation}\label{722bis}
\underline{\BB}(X_1-1)=\underline{\mathrm{col}}_1\cdot\underline{\mathrm{row}}_1
\end{equation}
(equality in $M_3((\mathcal V^\B)^{\otimes2})$). 
\end{lem}

\proof One has $\underline\ell(X_1)=x_{12}$. Let us compute $\underline\AAA(x_{12})$. One has 
\begin{align*}
 (x_{15}-1)x_{12} & =   (x_{15}-1)x_{12}x_{15}x_{25}x_{25}^{-1}x_{15}^{-1}
=x_{12}x_{15}x_{25} (x_{15}-1)x_{25}^{-1}x_{15}^{-1}
\\ & 
=x_{12}x_{15}x_{25}(-x_{15}x_{25}^{-1}x_{15}^{-1}+1+x_{25}^{-1}x_{15}^{-1})(x_{15}-1)
+x_{12}x_{15}x_{25}(1-x_{15})x_{25}^{-1}(x_{25}-1), \\
(x_{25}-1)x_{12} & =(x_{25}-1)x_{12}x_{15}x_{25}x_{25}^{-1}x_{15}^{-1}
=x_{12}x_{15}x_{25}(x_{25}-1)x_{25}^{-1}x_{15}^{-1}
\\ &=x_{12}x_{15}x_{25}(x_{25}^{-1}-1)x_{15}^{-1}(x_{15}-1)+x_{12}x_{15}(x_{25}-1) , \\
(x_{35}-1)x_{12}& =x_{12}(x_{35}-1) , 
\end{align*}
which implies that 
$$
\underline\AAA(x_{12})
=\begin{pmatrix} x_{12}x_{15}x_{25}(-x_{15}x_{25}^{-1}x_{15}^{-1}+1+x_{25}^{-1}x_{15}^{-1})&
x_{12}x_{15}x_{25}(1-x_{15})x_{25}^{-1}& 0\\ 
x_{12}x_{15}x_{25}(x_{25}^{-1}-1)x_{15}^{-1}& x_{12}x_{15}& 0\\ 
0& 0& x_{12}\end{pmatrix}\in M_3(\mathbf k P_5^*). 
$$
The image of $\underline\AAA(x_{12}-1)$ in $M_3((\mathcal V^\B)^{\otimes2})$ by $M_3(\underline{\mathrm{pr}}_{12})$ is therefore 
$$
\begin{pmatrix} (X_1-1)Y_1& Y_1(1-Y_1) & 0\\ 1-X_1& Y_1-1& 0\\ 0& 0& 0\end{pmatrix}
=\begin{pmatrix} Y_1 \\ -1 \\ 0\end{pmatrix}\begin{pmatrix} X_1-1& 1-Y_1 & 0 \end{pmatrix}
=\underline{\mathrm{col}}_1\cdot\underline{\mathrm{row}}_1. 
$$
Therefore $\underline{\BB}(X_1-1)=\underline{\mathrm{col}}_1\cdot\underline{\mathrm{row}}_1$. \hfill\qed\medskip

\subsubsection{Construction and properties of a morphism $\underline{\tilde{\BB}}:(\mathcal V^\B,\cdot_{X_1-1})\to 
(\mathcal V^\B)^{\otimes2}$}\label{subsection:underline:tilde:B}

Lemma \ref{lemma:decomp:underline:A:X_1-1} shows that the hypothesis of Lemma \ref{algebraic:lemma:2} is satisfied in the 
following situation: $R=\mathcal V^\B$, $S=(\mathcal V^\B)^{\otimes2}$, $e=X_1-1$, $n=3$, $f=\underline{\BB}$, $\mathrm{row}$ 
and $\mathrm{col}$ are $\underline{\mathrm{row}}_1$ and $\underline{\mathrm{col}}_1$ from Lemma 
\ref{lemma:decomp:underline:A:X_1-1}. We denote by
$$
\underline{\tilde{\BB}}:(\mathcal V^\B,\cdot_{X_1-1})\to (\mathcal V^\B)^{\otimes2}
\index{rhovar_underline^tilde@$\underline{\tilde{\BB}}$}
$$
the algebra morphism given in this situation by Lemma \ref{algebraic:lemma:2}. 

Then for any $\underline f\in\mathcal V^\B$, one has 
\begin{equation}\label{formula:underline:tilde:B:f}
\underline{\tilde{\BB}}(\underline f)=\underline{\mathrm{row}}_1\cdot \underline{\BB}(\underline f)\cdot \underline{\mathrm{col}}_1
=\underline{\mathrm{row}}_1\cdot 
\{M_3(\underline{\mathrm{pr}}_{12})\circ \underline\AAA\circ \underline \ell(\underline f)\}\cdot \underline{\mathrm{col}}_1
\in (\mathcal V^\B)^{\otimes2}. 
\end{equation}

\begin{lem}\label{lemma:65:22:11:2017}
For any $k\in\mathbb Z$, 
$$
\underline{\tilde{\BB}}(X_0^k) 
=(X_1-1)X_0^k\otimes 1+1\otimes (1-X_1^{-1})X_0^kX_1
-\sum_{i=1}^{k-1}(X_1-1)X_0^{i}\otimes (1-X_1^{-1})X_0^{k-i}X_1, 
$$
$$
\underline{\tilde{\BB}}(X_0^kX_1^{-1})=\underline{\tilde{\BB}}(X_0^k) (X_1^{-1}\otimes X_1^{-1})
$$
(equalities in $(\mathcal V^\B)^{\otimes2}$). 
\end{lem}

\proof According to (\ref{formula:underline:tilde:B:f}), $\underline{\tilde{\BB}}(X_0^k)
=\underline{\mathrm{row}}_1\cdot \underline{\BB}(X_0^k)\cdot \underline{\mathrm{col}}_1$. 
As $\underline{\BB}$ is an algebra morphism, $\underline{\BB}(X_0^k)=\underline{\BB}(X_0)^k$. Then 
$\underline{\BB}(X_0)=M_3(\underline{\mathrm{pr}}_{12})\circ \underline\AAA\circ \underline \ell(X_0)
=M_3(\underline{\mathrm{pr}}_{12})(\underline\AAA(x_{23}))$. 

Let us compute $\underline\AAA(x_{23})$.

As $x_{15}$ commutes with $x_{23}$, one has 
\begin{equation}\label{1-x15-1}
(x_{15}-1)\cdot x_{23}=x_{23}\cdot (x_{15}-1). 
\end{equation}

Since $x_{25}$ commutes with $x_{23}x_{25}x_{35}$, one has 
$$
x_{23}^{-1}x_{25}x_{23}=x_{25}x_{35}x_{25}x_{35}^{-1}x_{25}^{-1}
$$
which implies the first equality in the following chain of equalities 
\begin{align*}
& x_{23}^{-1}(x_{25}-1)x_{23}=x_{25}x_{35}x_{25}x_{35}^{-1}x_{25}^{-1}-1=x_{25}(x_{35}x_{25}x_{35}^{-1}x_{25}^{-1}-1)+x_{25}-1
\\ & = x_{25}\{x_{35}(x_{25}x_{35}^{-1}x_{25}^{-1}-1)+x_{35}-1\}+x_{25}-1
\\ & =  x_{25}[x_{35}\{x_{25}(x_{35}^{-1}x_{25}^{-1}-1)+x_{25}-1\}+x_{35}-1]+x_{25}-1
\\ & = x_{25}\big(x_{35}[x_{25}\{x_{35}^{-1}(x_{25}^{-1}-1)+x_{35}^{-1}-1\}+x_{25}-1]+x_{35}-1\big)+x_{25}-1
\\ & =x_{25}x_{35}x_{25}x_{35}^{-1}(x_{25}^{-1}-1)+x_{25}x_{35}x_{25}(x_{35}^{-1}-1)+x_{25}x_{35}(x_{25}-1)+x_{25}(x_{35}-1)+x_{25}-1
\\ & = (-x_{25}x_{35}x_{25}x_{35}^{-1}x_{25}^{-1}+x_{25}x_{35}+1)(x_{25}-1)+
(-x_{25}x_{35}x_{25}x_{35}^{-1}+x_{25})(x_{35}-1), 
\end{align*}
the following equalities being immediate. Therefore 
\begin{equation}\label{x25-1x23=x23}
(x_{25}-1)\cdot x_{23}=x_{23}(-x_{25}x_{35}x_{25}x_{35}^{-1}x_{25}^{-1}+x_{25}x_{35}+1)\cdot (x_{25}-1)
+x_{23}(-x_{25}x_{35}x_{25}x_{35}^{-1}+x_{25})\cdot (x_{35}-1).
\end{equation}

Since $x_{23}x_{25}x_{35}=x_{35}x_{23}x_{25}$, one has $x_{23}^{-1}x_{35}x_{23}=x_{25}x_{35}x_{25}^{-1}$, 
which implies the first equality in the following chain of equalities 
\begin{align*}
&x_{23}^{-1}(x_{35}-1)x_{23}=x_{25}x_{35}x_{25}^{-1}-1=x_{25}(x_{35}x_{25}^{-1}-1)+x_{25}-1 
\\ & =x_{25}\{x_{35}(x_{25}^{-1}-1)+x_{35}-1\}+x_{25}-1=x_{25}x_{35}(x_{25}^{-1}-1)+x_{25}-1+x_{25}(x_{35}-1)
\\ & =(-x_{25}x_{35}x_{25}^{-1}+1)(x_{25}-1)+x_{25}(x_{35}-1), 
\end{align*}
the following equalities being immediate. Therefore  
\begin{equation}\label{x34-1x23=x23}
(x_{35}-1)\cdot x_{23}=x_{23}(-x_{25}x_{35}x_{25}^{-1}+1)\cdot (x_{25}-1)+x_{23}x_{25}\cdot (x_{35}-1). 
\end{equation}
Equalities (\ref{1-x15-1}), (\ref{x25-1x23=x23}) and (\ref{x34-1x23=x23}) imply that 
$$
\underline\AAA(x_{23})=\begin{pmatrix}  x_{23}& 0&0  \\   0&x_{23}(-x_{25}x_{35}x_{25}x_{35}^{-1}x_{25}^{-1}+x_{25}x_{35}+1)
&x_{23}(-x_{25}x_{35}x_{25}x_{35}^{-1}+x_{25})\\  0&x_{23}(-x_{25}x_{35}x_{25}^{-1}+1)& x_{23}x_{25}\end{pmatrix}\in 
M_3(\mathbf k P_5^*). 
$$
Then  
\begin{equation}\label{eq:rho:X0}
\underline{\BB}(X_0)=M_3(\underline{\mathrm{pr}}_{12})(\underline\AAA(x_{23}))=
\begin{pmatrix} X_0&0 & 0\\ 0& (1-X_1)X_0+Y_1^{-1}Y_0Y_1& (1-X_1)X_0X_1\\ 0& 
X_0-Y_1^{-1}Y_0Y_1X_1^{-1}& X_0X_1\end{pmatrix}\in M_3((\mathcal V^\B)^{\otimes2}). 
\end{equation}
Set $\underline T\index{T_@$\underline T$}:=\begin{pmatrix} (1-X_1)X_0+Y_1^{-1}Y_0Y_1& (1-X_1)X_0X_1\\
X_0-Y_1^{-1}Y_0Y_1X_1^{-1}& X_0X_1\end{pmatrix}\in M_2((\mathcal V^\B)^{\otimes2})$, 
then $\underline{\BB}(X_0^k)=\begin{pmatrix} X_0^k& 0\\0& \underline T^k\end{pmatrix}$, therefore 
$$
\underline{\tilde{\BB}}(X_0^k)=\underline{\mathrm{row}}_1\cdot\underline{\BB}(X_0^k)\cdot\underline{\mathrm{col}}_1
=(X_1-1)X_0^kY_1+\begin{pmatrix} 1-Y_1&0 \end{pmatrix}\underline T^k
\begin{pmatrix} -1\\0 \end{pmatrix}
=(X_1-1)X_0^kY_1-(1-Y_1)(\underline T^k)_{11}, 
$$
where the second equality follows from the form of $\underline{\mathrm{row}}_1$ and $\underline{\mathrm{col}}_1$, and where 
$(\underline T^k)_{11}$ means the $(1,1)$-entry of $\underline T^k$.  

One checks that 
$$
\underline T=\begin{pmatrix} 1&0 \\ -X_1^{-1}&1 \end{pmatrix}
\begin{pmatrix} a& b \\ 0 & c \end{pmatrix}\begin{pmatrix} 1&0 \\ -X_1^{-1}&1 \end{pmatrix}^{-1}
$$
where 
$$
\begin{pmatrix} 1&0 \\ -X_1^{-1}&1 \end{pmatrix}^{-1}=\begin{pmatrix} 1&0 \\ X_1^{-1}&1 \end{pmatrix}
$$
and 
$$
\begin{pmatrix} a& b \\ 0 & c \end{pmatrix}=
\begin{pmatrix} Y_1^{-1}Y_0Y_1& (1-X_1)X_0X_1\\ 0& X_1^{-1}X_0X_1\end{pmatrix}. 
$$

For $k\in\mathbb Z$, one has 
$$
\begin{pmatrix} a& b \\ 0 & c \end{pmatrix}^k=\begin{pmatrix} a^k& \sum_{i=0}^{k-1}a^ibc^{k-i-1}  \\ 0& c^k \end{pmatrix}
$$
using the notation (\ref{convention:sums})
so that 
$$
\underline T^k=\begin{pmatrix} 1&0 \\ -X_1^{-1}&1 \end{pmatrix}
\begin{pmatrix} a^k& \sum_{i=0}^{k-1}a^ibc^{k-i-1}  \\ 0& c^k \end{pmatrix}
\begin{pmatrix} 1&0 \\ X_1^{-1}&1 \end{pmatrix}=\begin{pmatrix} a^k+
\sum_{i=0}^{k-1}a^ibc^{k-i-1} \cdot X_1^{-1}& * \\ * & * \end{pmatrix}
$$
therefore 
\begin{align*}
& (\underline T^k)_{11}=a^k+
\sum_{i=0}^{k-1}a^ibc^{k-i-1} \cdot X_1^{-1}
\\ & = Y_1^{-1}Y_0^kY_1+
\sum_{i=0}^{k-1}Y_1^{-1}Y_0^iY_1  \cdot (1-X_1)X_0X_1\cdot    X_1^{-1}X_0^{k-i-1}X_1 \cdot X_1^{-1}
\\ & = Y_1^{-1}Y_0^kY_1+
\sum_{i=0}^{k-1}Y_1^{-1}Y_0^iY_1  \cdot (1-X_1)X_0^{k-i}
\\ & = Y_1^{-1}Y_0^kY_1+(1-X_1)X_0^k
+\sum_{i=1}^{k-1}Y_1^{-1}Y_0^iY_1\cdot (1-X_1)X_0^{k-i}.  
\end{align*}

It follows that 
\begin{align*}
& \underline{\tilde{\BB}}(X_0^k)=(X_1-1)Y_1X_0^k-(1-Y_1)(\underline T^k)_{11} 
\\ & =
(X_1-1)Y_1X_0^k-(1-Y_1)
\{
(1-X_1)X_0^k+Y_1^{-1}Y_0^kY_1
+\sum_{i=1}^{k-1}Y_1^{-1}Y_0^iY_1\cdot (1-X_1)X_0^{k-i}
\}
\\ &
= (X_1-1)Y_1X_0^k-(1-Y_1)
(1-X_1)X_0^k-(1-Y_1)Y_1^{-1}Y_0^kY_1
-(1-Y_1)\sum_{i=1}^{k-1}Y_1^{-1}Y_0^iY_1\cdot (1-X_1)X_0^{k-i}
\\ & = 
(X_1-1)X_0^k+(Y_1-1)Y_1^{-1}Y_0^kY_1
-\sum_{i=1}^{k-1}(1-X_1)X_0^{k-i}\cdot Y_1^{-1}(1-Y_1)Y_0^iY_1, 
\end{align*} 
which implies the first identity. 

By Lemma \ref{lemma:decomp:underline:A:X_1-1}, one has $\underline{\BB}(X_1)=1+\underline{\mathrm{col}}_1\cdot\underline{\mathrm{row}}_1$. 
As $1+\underline{\mathrm{row}}_1\cdot\underline{\mathrm{col}}_1$ is equal to $X_1Y_1$ and is therefore invertible, one checks that the inverse of $1+\underline{\mathrm{col}}_1\cdot\underline{\mathrm{row}}_1$ is $1-\underline{\mathrm{col}}_1\cdot(1+\underline{\mathrm{row}}_1\cdot\underline{\mathrm{col}}_1)^{-1}\cdot\underline{\mathrm{row}}_1$, therefore 
$$
\underline{\BB}(X_1^{-1})=1-\underline{\mathrm{col}}_1\cdot(1+\underline{\mathrm{row}}_1
\cdot\underline{\mathrm{col}}_1)^{-1}\cdot\underline{\mathrm{row}}_1=1-\underline{\mathrm{col}}_1\cdot(X_1Y_1)^{-1}
\cdot\underline{\mathrm{row}}_1. 
$$
Then 
\begin{align*}
&\underline{\BB}(X_0^kX_1^{-1})=\underline{\mathrm{row}}_1\cdot\underline{\BB}(X_0^k)\cdot\Big(
1-\underline{\mathrm{col}}_1\cdot(X_1Y_1)^{-1}
\cdot\underline{\mathrm{row}}_1\Big)\cdot\underline{\mathrm{col}}_1
\\ & 
=\underline{\mathrm{row}}_1\cdot\underline{\BB}(X_0^k)\cdot\underline{\mathrm{col}}_1\cdot\Big(
1-(X_1Y_1)^{-1}\cdot\underline{\mathrm{row}}_1\cdot\underline{\mathrm{col}}_1\Big)=\underline{\BB}(X_0^k)\cdot(X_1Y_1)^{-1},  
\end{align*}
which proves the second statement. 
\hfill\qed\medskip

\begin{rem}
Identity (\ref{identity:28:11:2017}) from Remark \ref{remark:alternative:decomp:DR} implies another decomposition of $\underline T$, 
namely 
$$
\underline T=
\begin{pmatrix} 1 & -X_1 \\ 0 & 1 \end{pmatrix}
\begin{pmatrix} X_0&  0\\ X_0-Y_1^{-1}Y_0Y_1 X_1^{-1} &  Y_1^{-1}Y_0Y_1\end{pmatrix}
\begin{pmatrix} 1 & -X_1 \\ 0 & 1 \end{pmatrix}^{-1},  
$$
which as in this remark allows for an alternative computation of $\underline T^k$. 
\end{rem}

\section{Geometric interpretation of the Betti harmonic coproducts}\label{sect:geom:betti}

The purpose of this section is to construct a commutative diagram relating the  Betti algebra and module coproducts 
$\Delta^{\mathcal W,\B}$ and $\Delta^{\mathcal M,\B}$ with braid groups (diagrams \eqref{diagram:prop:71} and \eqref{diagram:2202}), 
analogously to the de Rham diagrams \eqref{diagram:prop:59} and \eqref{diagram:1402}. 

This construction involves a localization $\mathcal V^\B[{1\over X_1-1}]$ of the algebra $\mathcal V^\B$, and a module 
$\mathcal M^\B[{1\over X_1-1}]$ over this algebra, which are introduced and studied in \S\ref{sect:loc:Betti}. 
We prove the commutativity of diagram \eqref{diagram:prop:71} in \S\ref{sect:6:3:19032018} (Proposition \ref{prop:comm:DR:05012017}); 
this diagram relates $\Delta^{\mathcal W,\B}$ with $\underline\rho$ and $\underline{\mathrm{row}}_1$, $\underline{\mathrm{col}}_1$, 
and is a Betti analogue of diagram \eqref{diagram:prop:59}. We prove the commutativity of the diagram \eqref{diagram:2202} in 
\S\ref{sect:proof:T1} (Proposition \ref{prop:22022019}); this statement is based on Proposition \ref{prop:comm:DR:05012017} and relates 
$\Delta^{\mathcal M,\B}$ with $\underline\rho$, $\underline{\mathrm{row}}_1$ and the column vector $\underline{\mathrm{col}}_0$ 
(see Definition \ref{def:col:0:Betti:30oct}). We construct completions of the diagrams \eqref{diagram:prop:71} and \eqref{diagram:2202} in 
\S\ref{sect:compl:Betti:30oct}. 

\subsection{Localizations}\label{sect:loc:Betti} 

Define $\mathcal V^\B[{1\over X_1-1}]$\index{V^B/X_1-1@$\mathcal V^\B[{1\over X_1-1}]$} to be the localization of $\mathcal V^\B$ with respect to $X_1-1$, i.e. the $\mathbf k$-algebra
with generators $X_0^{\pm1},X_1^{\pm1},(X_1-1)^{-1}$ and relations expressing that $u^{-1}$ is a left and right inverse of $u$ for 
$u\in\{X_0,X_1,X_1-1\}$. It is equipped with a collection of subspaces indexed by $i\in\mathbb Z$, namely 
\begin{equation}\label{29oct}
F^i(\mathcal V^\B[{1\over{X_1-1}}])
:=\sum_{\substack{
n,i_0,\ldots,i_n\geq 0,\\
i_0+\cdots+i_n-n=i}}F^{i_0}\mathcal V^\B
(X_1-1)^{-1}\cdots(X_1-1)^{-1}F^{i_n}\mathcal V^\B.
\index{F^iVB/X_1-1@$F^i(\mathcal V^\B[{1\over{X_1-1}}])$}
\end{equation}
This collection is decreasing and compatible with the product, and therefore equips $\mathcal V^\B[{1\over{X_1-1}}]$ with the structure of 
an algebra in $\mathbf k\text{-mod}_{\mathrm{fil}}$ (see Definition \ref{def:cat}). 

Set also $\mathcal M^\B[{1\over X_1-1}]:=\mathcal V^\B[{1\over X_1-1}]/\mathcal V^\B[{1\over X_1-1}](X_0-1)$\index{M^B/X_1-1@$\mathcal M^\B[{1\over X_1-1}]$}. This is a module over 
$\mathcal V^\B[{1\over{X_1-1}}]$.
For $i\in\mathbb Z$, let 
$F^i(\mathcal M^\B[{1\over{X_1-1}}])$
\index{F^iMB/X_1-1@$F^i(\mathcal M^\B[{1\over{X_1-1}}])$}
 be the image of 
$F^i(\mathcal V^\B[{1\over{X_1-1}}])$ under the canonical projection $(-)\cdot 1_\B:\mathcal V^\B[{1\over{X_1-1}}]
\to\mathcal M^\B[{1\over{X_1-1}}]$. This collection is decreasing and compatible with the collection of subspaces of 
$\mathcal V^\B[{1\over{X_1-1}}]$ and the action of this algebra, therefore equips $\mathcal M^\B[{1\over X_1-1}]$ with the 
structure of a module over $\mathcal V^\B[{1\over{X_1-1}}]$ in $\mathbf k\text{-mod}_{\mathrm{fil}}$. 

The natural maps $\mathcal V^\B\to\mathcal V^\B[{1\over X_1-1}]$ and $\mathcal M^\B\to\mathcal M^\B[{1\over X_1-1}]$
are compatible algebra and module morphisms in $\mathbf k\text{-mod}_{\mathrm{fil}}$. 

\begin{lem}\label{lem:size:loc}
1) The canonical morphism $\mathcal V^\B\to\mathcal V^\B[{1\over X_1-1}]$ is injective. 

2) The image of $\mathcal V^\B[{1\over X_1-1}]$ under the functor $\mathrm{gr}:\mathbf k\text{-mod}_{\mathrm{fil}}
\to\mathbf k\text{-mod}_{\mathrm{gr}}$ 
(see \S\ref{sect:functors}) is isomorphic to $\mathcal V^\DR[{1\over e_1}]$.  
\end{lem}

\proof 1) Denote by $\mathcal V^\DR[{1\over e_0}]^\wedge$ the completion of $\mathcal V^\DR[{1\over e_1}]$ with respect to large 
degrees (see \S\ref{sect:completions:DR}). It may be viewed as an algebra in $\mathbf k\text{-mod}_{\mathrm{fil}}$ by setting 
$F^\alpha(\mathcal V^\DR[{1\over e_1}]^\wedge):=\prod_{\beta\geq\alpha}\mathcal V^\DR[{1\over e_1}]_\beta$. There is a unique 
algebra morphism $\varphi:\mathcal V^\B[{1\over X_1-1}]\to\mathcal V^\DR[{1\over e_1}]^\wedge$ in $\mathbf k\text{-mod}_{\mathrm{fil}}$
given by $X_i^{\pm1}\mapsto(1+e_i)^{\pm1}$ for $i=0,1$ and $(X_1-1)^{-1}\mapsto e_1^{-1}$. It induces an algebra morphism 
$\mathrm{gr}\varphi:\mathrm{gr}\mathcal V^\B[{1\over X_1-1}]\to\mathrm{gr}\mathcal V^\DR[{1\over e_1}]^\wedge
\simeq\mathcal V^\DR[{1\over e_1}]$ in $\mathbf k\text{-mod}_{\mathrm{gr}}$. Its composition with $\mathcal V^\DR\simeq
\mathrm{gr}\mathcal V^\B\stackrel{\mathrm{gr}(\mathrm{can})}{\to}\mathrm{gr}\mathcal V^\B[{1\over X_1-1}]$, $\mathrm{can}$ 
being the morphism $\mathcal V^\B\to\mathcal V^\B[{1\over X_1-1}]$, is the canonical map 
$\mathcal V^\DR\to\mathcal V^\DR[{1\over e_1}]$, which is injective. It follows that $\mathrm{gr}(\mathrm{can})$ is injective. 
Since $\cap_{k\geq0}F^k\mathcal V^\B=0$, this implies that $\mathrm{can}$ is injective. 

2) There is a unique graded algebra morphism $\psi:\mathcal V^\DR[{1\over e_1}]\to\mathrm{gr}\mathcal V^\B[{1\over X_1-1}]$, induced by 
$e_i\mapsto [X_i-1]\in\mathrm{gr}_1\mathcal V^\B[{1\over X_1-1}]$ for $i=0,1$ and $e_1^{-1}\mapsto[(X_1-1)^{-1}]\in
\mathrm{gr}_{-1}\mathcal V^\B[{1\over X_1-1}]$. One has $\mathrm{gr}\varphi\circ\psi=\mathrm{id}$, which implies that $\psi$
is injective. For $\alpha\in\mathbb Z$,  for $n,\alpha_0,\ldots,\alpha_n$ as in the right-hand side of \eqref{29oct}, and for 
$v_i\in F^{\alpha_i}\mathcal V^\B$ for $i=1,\ldots,n$, the degree $\alpha$ component $\psi_\alpha$ of $\psi$ maps 
$[v_0]e_1^{-1}\cdots e_1^{-1}[v_n]\in\mathcal V^\DR[{1\over e_1}]_\alpha$, where $[v_i]\in \mathrm{gr}_{\alpha_i}\mathcal V^\B
\simeq\mathcal V^\DR_{\alpha_i}$ to $[v_0(X_1-1)^{-1}\cdots(X_1-1)^{-1}v_n]\in\mathrm{gr}_\alpha(\mathcal V^\B[{1\over X_1-1}])$, 
therefore $\psi_\alpha$ is surjective. It follows that $\psi$ is surjective. \hfill\qed\medskip 

\begin{lem}\label{lemma:ass:gr:M:loc:19dec2019}
1) The morphism $\mathcal M^\B\to\mathcal M^\B[{1\over X_1-1}]$ is injective. 

2) The image of $\mathcal M^\B[{1\over X_1-1}]$ under the functor $\mathrm{gr}$ (see \S\ref{sect:functors}) is isomorphic 
to $\mathcal M^\DR[{1\over e_1}]$.  
\end{lem}

\proof 1) Let $\mathcal M^\DR[{1\over e_1}]^\wedge$ the completion of $\mathcal M^\DR[{1\over e_1}]$ with respect to large degrees 
(see \S\ref{sect:completions:DR}); it can be identified with $\mathcal V^\DR[{1\over e_1}]^\wedge/\mathcal V^\DR[{1\over e_1}]^\wedge e_0$. 
The map $\varphi$ in the proof of Lemma \ref{lem:size:loc} maps the left ideal generated by $X_0-1$ 
to the left ideal generated by $e_0$, therefore induces a morphism $\overline\varphi:\mathcal M^\B[{1\over X_1-1}]\to
\mathcal M^\DR[{1\over e_1}]^\wedge$ in $\mathbf k\text{-mod}_{\mathrm{fil}}$, which gives rise to a morphism $\mathrm{gr}\overline\varphi:
\mathrm{gr}\mathcal M^\B[{1\over X_1-1}]\to\mathrm{gr}\mathcal M^\DR[{1\over e_1}]^\wedge$. Its composition with 
$\mathcal M^\DR\simeq\mathrm{gr}\mathcal M^\B\stackrel{\mathrm{gr}(\mathrm{can}_{\mathcal M})}{\to}\mathrm{gr}\mathcal M^\B
[{1\over X_1-1}]$, where $\mathrm{can}_{\mathcal M}$ is the morphism $\mathcal M^\B\to\mathcal M^\B[{1\over X_1-1}]$, is the canonical 
map $\mathcal M^\DR\to\mathcal M^\DR[{1\over e_1}]$, which is injective. This implies the injectivity of $\mathrm{gr}(\mathrm{can}_{\mathcal M})$,
then of $\mathrm{can}_{\mathcal M}$ as in the proof of Lemma \ref{lem:size:loc}.  

2) The composition of $\psi$ from the proof of Lemma \ref{lem:size:loc} with $\mathrm{gr}((-)\cdot 1_\B
)$ takes $e_0$ to $0$, therefore takes 
$\mathcal V^\DR[{1\over e_1}]e_0$ to 0, therefore induces a map $\overline\psi:\mathcal M^\DR[{1\over e_1}]\to
\mathrm{gr}\mathcal M^\B[{1\over X_1-1}]$. One has $\mathrm{gr}\overline\varphi\circ\overline\psi=\mathrm{id}$, which implies 
that $\overline\psi$ is injective. The equality $\mathrm{gr}((-)\cdot 1_\B)
\circ\psi=\overline\psi\circ((-)\cdot 1_\DR)
$, where 
$(-)\cdot 1_\DR:
\mathcal V^\DR[{1\over e_1}] \to\mathcal M^\DR[{1\over e_1}]$ is the canonical projection, the surjectivity of $\psi$
and that of $\mathrm{gr}((-)\cdot 1_\B)
$ (which follows from the construction of the filtration in $\mathcal M^\B[{1\over X_1-1}]$)
imply the surjectivity of $\overline\psi$. \hfill\qed\medskip 

\subsection{Relationship between braid groups and $\Delta^{\mathcal W,\B}$} 
\label{sect:6:3:19032018}

Denote by $\mathbf k[X_0^{\pm1}]$ the linear span in $\mathbf k F_2$ of the elements $X_0^k$, $k\in\mathbb Z$, 
by $\mathbf k[X_0^{\pm1}]X_1^{-1}$ the linear span of the elements $X_0^kX_1^{-1}$, $k\in\mathbb Z$. The sum of 
these submodules of $\mathcal V^\B=\mathbf k F_2$ is direct.  

\begin{lem}\label{lemma:X0k:generating}
$(\mathcal V^\B,\cdot_{X_1-1})$\index{V^B, X_1-1@$(\mathcal V^\B,\cdot_{X_1-1})$} is generated, as an 
associative algebra, by $\mathbf k[X_0^{\pm1}]\oplus\mathbf k[X_0^{\pm1}]X_1^{-1}$.  
\end{lem}

\proof For $s\geq 0$ and for $(k_0,\ldots,k_s)\in\mathbb Z^{s+1}$, $(\epsilon_1,\ldots,\epsilon_s)\in\{\pm1\}^s$, set 
\begin{equation}\label{formula:w:arguments}
w(k_0,\ldots,k_s|\epsilon_1,\ldots,\epsilon_s):=X_0^{k_0}X_1^{\epsilon_1}X_0^{k_1}\cdots X_0^{k_{s-1}}X_1^{\epsilon_s}
X_0^{k_s}\in F_2. 
\end{equation}
If $s\geq 0$, let $(F_2)_s$ be the subset of $F_2$ of all the elements (\ref{formula:w:arguments}), where $(k_0,\ldots,k_s)\in\mathbb Z^{s+1}$, $(\epsilon_1,\ldots,\epsilon_s)\in\{\pm1\}^s$; so when $s=0$, $(F_2)_0$ is the set of all $X_0^k$, $k\in\mathbb Z$. 

One has for $s\geq 0$, $(k_0,\ldots,k_{s+1})\in\mathbb Z^{s+1}$, $(\epsilon_2,\ldots,\epsilon_{s+1})\in\{\pm1\}^s$, 
$$
\begin{array}{r}
w(k_0,\ldots,k_{s+1}|1,\epsilon_2,\ldots,\epsilon_{s+1})=X_0^{k_0}\cdot_{X_1-1}
w(k_1,\ldots,k_{s+1}|\epsilon_2,\ldots,\epsilon_{s+1})\\+w(k_0+k_1,k_2,\ldots,k_{s+1}|\epsilon_2,\ldots,\epsilon_{s+1}), 
\end{array}
$$
$$
\begin{array}{r}
w(k_0,\ldots,k_{s+1}|-1,\epsilon_2,\ldots,\epsilon_{s+1})=-X_0^{k_0}X_1^{-1}\cdot_{X_1-1}
w(k_1,\ldots,k_{s+1}|\epsilon_2,\ldots,\epsilon_{s+1})\\ +w(k_0+k_1,k_2,\ldots,k_{s+1}|\epsilon_2,\ldots,\epsilon_{s+1}). 
\end{array}
$$
These identities enable one to prove by induction on $s\geq0$ that $(F_2)_s$ is contained in the associative 
subalgebra of $(\mathcal V^\B,\cdot_{X_1-1})$ generated by $\mathbf k[X_0^{\pm1}]\oplus\mathbf k[X_0^{\pm1}]X_1^{-1}$.
The statement then follows from the fact that the union for $s\geq0$ of all $(F_2)_s$ is equal to $F_2$.   
\hfill\qed\medskip 

Recall that $\mathcal W^{\B}$ is the subalgebra of $\mathcal V^\B$ equal to $\mathbf k\oplus \mathcal V^\B(X_1-1)$
(see \S\ref{sect:tcDaD}). We set 
$$
\mathcal W^{\B}_+:= \mathcal V^\B(X_1-1).
\index{W^B+@$\mathcal W^{\B}_+$} 
$$
This is a (non-unital) subalgebra of $\mathcal W^{\B}$. It is equipped with the filtration induced by $\mathcal W^{\B}$. 

Since the right multiplication by $X_1-1$ is injective in $\mathcal V^\B$, the algebra morphism 
\begin{equation}\label{def:mor:04012018}
\mathrm{mor}_{\mathcal V^\B,X_1-1}\index{morVBX_1-1@$\mathrm{mor}_{\mathcal V^\B,X_1-1}$}
:(\mathcal V^\B,\cdot_{X_1-1})\to\mathcal W^{\B}_+
\end{equation}
(see \S\ref{section:algebraic:lemmas}) is an algebra isomorphism. 

One checks that the automorphisms 
$\mathrm{Ad}(X_1-1)^{-1}$ and $\mathrm{Ad}(1-X_1^{-1})^{-1}$ of $\mathcal V^\B[{1\over X_1-1}]$ map $\mathcal V^\B$
to $F^0(\mathcal V^\B[{1\over X_1-1}])$. 

\begin{lem}\label{CD:tilde:underlineB:Delta*}
The following diagram is commutative
$$
\xymatrix{
(\mathcal V^\B,\cdot_{X_1-1}) \ar^{\mathrm{Ad}(X_1-1)^{-1}(1-Y_1^{-1})^{-1}\circ\underline{\tilde{\BB}}}[rrrr]
\ar_{\mathrm{mor}_{\mathcal V_\B,X_1-1}}^{\simeq}[d]&&&& F^0(\mathcal V^\B[{1\over{X_1-1}}])^{\otimes2}\\
\mathcal W^{\B}_+\ar_{\Delta^{\mathcal W,\B}}[rrrr]&&&&  (\mathcal W^{\B})^{\otimes2}\ar@{^{(}->}[u]
}
$$
where the top horizontal map is the composition with $\underline{\tilde\rho}$ of tensor product of the maps 
$\mathcal V^\B\to F^0(\mathcal V^\B[{1\over X_1-1}])$ given by $\mathrm{Ad}(X_1-1)^{-1}$ and $\mathrm{Ad}(1-X_1^{-1})^{-1}$
and the right vertical map is the tensor square of $\mathcal W^{\B}\hookrightarrow\mathcal V^\B\hookrightarrow
F^0(\mathcal V^\B[{1\over{X_1-1}}])$. 
\end{lem}

\proof  For $k\in\mathbb Z$, 
\begin{align*}
& \Delta^{\mathcal W,\B}\circ\mathrm{mor}_{\mathcal V^\B,X_1-1}(X_0^k)
=\Delta^{\mathcal W,\B}(X_0^k(X_1-1))= -X_0^k(1-X_1)\otimes 1-1\otimes X_0^k(1-X_1)
\\ & -\sum_{i=1}^{k-1} X_0^{i}(1-X_1) \otimes X_0^{k-i}(1-X_1)
=\mathrm{Ad}(X_1-1)^{-1}(Y_1-1)^{-1}\Big(
(X_1-1)X_0^k\otimes1+1\otimes (X_1-1)X_0^k\\ & -\sum_{i=1}^{k-1} (X_1-1)X_0^{i}\otimes (X_1-1)X_0^{k-i}\Big)
=\mathrm{Ad}((X_1-1)^{-1}(Y_1-1)^{-1}Y_1)(\underline{\tilde{\BB}}(X_0^k)),  
\end{align*}
and similarly 
\begin{align*}
& \Delta^{\mathcal W,\B}\circ\mathrm{mor}_{\mathcal V^\B,X_1-1}(X_0^kX_1^{-1})
=\Delta^{\mathcal W,\B}(X_0^k(1-X_1^{-1}))
=X_0^k(1-X_1^{-1})\otimes 1+1\otimes X_0^k(1-X_1^{-1})
\\ & -\sum_{i=0}^{k} X_0^{i}(1-X_1^{-1}) \otimes X_0^{k-i}(1-X_1^{-1})
=\mathrm{Ad}(X_1-1)^{-1}(Y_1-1)^{-1}\Big(
(X_1-1)X_0^kX_1^{-1}\otimes1 \\ & +1\otimes (X_1-1)X_0^kX_1^{-1}
 -\sum_{i=0}^{k} (X_1-1)X_0^{i}X_1^{-1}\otimes (X_1-1)X_0^{k-i}X_1^{-1}
\Big) 
\\ & =\mathrm{Ad}(X_1-1)^{-1}(Y_1-1)^{-1}\Big((X_1-1)X_0^kX_1^{-1}\otimes X_1^{-1}+X_1^{-1}\otimes (X_1-1)X_0^kX_1^{-1}
\\ & -\sum_{i=1}^{k-1} (X_1-1)X_0^{i}X_1^{-1}\otimes (X_1-1)X_0^{k-i}X_1^{-1}\Big)
=\mathrm{Ad}((X_1-1)^{-1}(Y_1-1)^{-1}Y_1)(\underline{\tilde{\BB}}(X_0^k))(X_1^{-1}\otimes X_1^{-1})
\\ & 
= \mathrm{Ad}((X_1-1)^{-1}(Y_1-1)^{-1}Y_1)(\underline{\tilde{\BB}}(X_0^kX_1^{-1})); 
\end{align*}
in each of these sequence of equalities, the second equality follows from 
\eqref{cop:+} (resp. \eqref{cop:-}), the fourth 
equality follows from Lemma \ref{lemma:65:22:11:2017}. 

It follows that the two maps of the announced diagram agree on the family of elements $X_0^k$, $X_0^kX_1^{-1}$, $k\in\mathbb Z$. 
Since these maps are algebra morphisms, and since this family generates $(\mathcal V^\B,\cdot_{X_1-1})$ (see Lemma 
\ref{lemma:X0k:generating}), this diagram commutes. \hfill\qed\medskip 

\begin{prop}\label{prop:comm:DR:05012017}
The following diagram commutes
\begin{equation}\label{diagram:prop:71}
\xymatrix
{
\mathcal V^\B\ar_{\simeq}^{\diamond}[d]\ar^{\!\!\!\!\!\!\underline\rho}[r]
& 
M_3((\mathcal V^\B)^{\otimes2})\ar^{(X_1-1)^{-1}(1-Y_1^{-1})^{-1}\underline{\mathrm{row}}_1
\cdot(-)\cdot\underline{\mathrm{col}}_1(X_1-1)(1-Y_1^{-1})}_{\diamond\sharp}[rrrrrr]
&&&&&&
F^0({\mathcal V}^\B[\frac{1}{X_1-1}])^{\otimes2}
\\ (\mathcal V^\B,\cdot_{X_1-1})\ar_{\mathrm{mor}_{\mathcal V^\B,X_1-1}}^{\simeq\sharp}[d]&&&&&&
\\ 
\mathcal W^{\B}_+\ar_{\Delta^{\mathcal W,\B}}[rrrrrrr]
&&&&&&& 
(\mathcal W^{\B})^{\otimes2}\ar@{^{(}->}[uu]
}
\end{equation}
where $\underline\rho$ is as in \eqref{underline:B:05012017}, $\underline{\mathrm{row}}_1$, $\underline{\mathrm{col}}_1$, are as in \eqref{def:underline:row:col:04012018}, and $\Delta^{\mathcal W,\DR}$ is as in \S\ref{sect:tcDaD}; in this diagram, all the maps are 
$\mathbf k$-algebra morphisms (resp. compatible with the filtrations), except for the maps marked with $\diamond$ (resp. $\sharp$), which 
are only $\mathbf k$-module morphisms (resp. increase the filtration degrees by $1$).  
\end{prop}

\proof The map marked $\diamond\sharp$ is well-defined as it is the composition of the map $\underline{\mathrm{row}}_1\cdot(-)\cdot
\underline{\mathrm{col}}_1:M_3((\mathcal V^\B)^{\otimes2})\to(\mathcal V^\B)^{\otimes2}$ and of the tensor product of the maps 
$\mathcal V^\B\to F^0(\mathcal V^\B[{1\over X_1-1}])$ given by $\mathrm{Ad}(X_1-1)^{-1}$ and $\mathrm{Ad}(1-X_1^{-1})^{-1}$. 

The commutation of \eqref{diagram:prop:71} follows from the combination the commutative diagram from Lemma 
\ref{CD:tilde:underlineB:Delta*} with the specialization, 
based on using \eqref{722bis}, of the commutative diagram from Lemma \ref{algebraic:lemma:2} to $R=\mathcal V^\B$, 
$S=\mathcal V^\B[{1\over{X_1-1}}]^{\otimes2}$, $e=X_1-1$, $\mathrm{row}=(X_1-1)^{-1}(1-Y_1^{-1})^{-1}\underline{\mathrm{row}}_1$, 
$\mathrm{col}=\underline{\mathrm{col}}_1(X_1-1)(1-Y_1^{-1})$, which yields the commutative diagram 
$$
\xymatrix{\mathcal V^\B \ar^{\underline\rho}[rrr]\ar_\simeq[d]&&& M_3((\mathcal V^\B)^{\otimes2})
\ar^{(X_1-1)^{-1}(1-Y_1^{-1})^{-1}\underline{\mathrm{row}}_1\cdot(-)\cdot
\underline{\mathrm{col}}_1(X_1-1)(1-Y_1^{-1})}[d]\\ 
(\mathcal V^\B,\cdot_{X_1-1})
\ar_{\mathrm{Ad}(X_1-1)^{-1}(1-Y_1^{-1})^{-1}\circ\underline{\tilde\rho}}[rrr]&&& F^0(\mathcal V^\B[{1\over X_1-1}])^{\otimes2}}
$$ 
The filtration statement about $\mathrm{mor}_{\mathcal V^\B,X_1-1}$ follows from: $X_1-1\in F^1\mathcal V^\B$, the filtration of 
$\mathcal V^\B$ is compatible with its algebra structure. The filtration statement about the map marked with $\diamond\sharp$
follows from: same statement regarding $\mathcal V^\B[{1\over X_1-1}]^{\otimes2}$, the components of the row (resp. column) element 
in this map belong to $F^{-1}(\mathcal V^\B[{1\over X_1-1}]^{\otimes2})$ (resp. $F^2(\mathcal V^\B[{1\over X_1-1}]^{\otimes2})$). The 
filtration statement about $\Delta^{\mathcal W,\B}$ is Proposition \ref{lemma:compatibility:Delta:l/r:ideals}. The filtration statement on 
$\underline\rho$ will be proved in Lemma \ref{lemma:compatibility:03012018}. 

\begin{lem}\label{lemma:73:15122017}
Recall that $F_3$ may be viewed as a normal subgroup of $P_5^*$ (see Lemma \ref{lemma:60:16:11:2017}). 
For any $g\in P_5^*$, for any $i\in[\![1,3]\!]$, there exists an element $w(g,i)\in F_3$, such that 
\begin{equation}\label{identity:gxg:gwg}
g^{-1}x_{i5}g=w(g,i)x_{i5}w(g,i)^{-1}. 
\end{equation}
\end{lem}

\proof  As $F_3$ is normal in $P_5^*$, for any $g\in P_5^*$, the inner automorphism of $P_5^*$ given by 
conjugation by $g^{-1}$ restricts to an automorphism $\theta_g$ of $F_3$. Then $g\mapsto\theta_g$ defines a 
group anti-homomorphism $\theta:P_5^*\to\mathrm{Aut}(F_3)$. Let $\mathrm{Aut}^*(F_3)\index{aut^*F_3@$\mathrm{Aut}^*(F_3)$}$ be the subgroup of $\mathrm{Aut}(F_3)$, consisting of all the automorphisms taking each $x_{i5}$, $i\in[\![1,3]\!]$, 
to a conjugate of this element. When $g\in F_3$, $\theta_g\index{theta_g@$\theta_g$}$ is an inner automorphism of $F_3$, therefore belongs to $\mathrm{Aut}^*(F_3)$. One computes 
$$
\theta_{x_{12}}:x_{15}\mapsto\mathrm{Ad}(x_{15}x_{25})(x_{15}), \ x_{25}\mapsto\mathrm{Ad}(x_{15})(x_{25}), \ x_{35}\mapsto x_{35}, 
$$
$$
\theta_{x_{23}}:x_{15}\mapsto x_{15}, \ x_{25}\mapsto\mathrm{Ad}(x_{25}x_{35})(x_{25}), \ x_{35}\mapsto \mathrm{Ad}(x_{25})(x_{35}), 
$$
which implies that the images by $\theta$ of $x_{12}^{-1}$ and $x_{23}^{-1}$, and therefore also of $F_2$, lie in $\mathrm{Aut}^*(F_3)$. 
Together with Lemma \ref{lemma:60:16:11:2017}, this implies that the image of $P_5^*$ is contained in $\mathrm{Aut}^*(F_3)$,
and therefore the announced statement. 
\hfill\qed\medskip 

\begin{lem}\label{prop:image:under:barA}
Equip $\mathbf k P_5^*$ with the adic filtration of its augmentation ideal. The algebra morphism $\underline\AAA:\mathbf k P_5^*\to 
M_3(\mathbf k P_5^*)$ is compatible with the filtrations on both sides. 
\end{lem}

\proof Let $J:=(\mathbf k P_5^*)_+$. The announced statement means that for any $n\geq0$, 
$\underline\AAA(J^n)\subset M_3(J^n)$. This is obvious for $n=0$, let us prove it for $n=1$. 

As $\underline\AAA$ is linear, and as $J$ is spanned by the $g-1$, where $g\in P_5^*$, it suffices to show that 
$\underline\AAA(g)-\mathrm{id}\in M_3(J)$ for any $g\in P_5^*$, where $\mathrm{id}$ is the unit matrix in $M_3(\mathbf k P_5^*)$. 
Let $i\in[\![1,3]\!]$; then 
$$
(x_{i5}-1)g=g\cdot w(g,i)(x_{i5}-1)w(g,i)^{-1}=g\cdot w(g,i)(x_{i5}-1)+g\cdot w(g,i)(x_{i5}-1)(w(g,i)^{-1}-1)
$$ 
where the first equality follows from (\ref{identity:gxg:gwg}). 

By Lemma \ref{lem:decomp:ZZG}, the map $(\mathbf k F_3)^{\oplus 3}\to(\mathbf k F_3)_+$, $(\underline f_i)_{i\in[\![1,3]\!]}
\mapsto \sum_{i\in[\![1,3]\!]}\underline f_i\cdot (x_{i5}-1)$ is bijective. For $w\in F_3$, we denote by 
$(\underline f_i(w))_{i\in[\![1,3]\!]}$ the preimage of $w-1$ in $(\mathbf k F_3)^{\oplus 3}$. Then 
$\sum_{i\in[\![1,3]\!]}\underline f_i(w)\cdot (x_{i5}-1)=w-1$. Substituting $w$ by $w(g,i)^{-1}$ in 
this identity, one obtains  
$$
(x_{i5}-1)g=g\cdot w(g,i)(x_{i5}-1)+\sum_{j\in[\![1,3]\!]}g\cdot w(g,i)(x_{i5}-1)\underline f_j(w(g,i)^{-1})\cdot (x_{j5}-1). 
$$
It follows that 
$$
\underline\AAA(g)_{ij}=g\cdot w(g,i)\delta_{ij}+g\cdot w(g,i)(x_{i5}-1)\underline f_j(w(g,i)^{-1}). 
$$
As $g\cdot w(g,i)\in P_5^*$, $g\cdot w(g,i)\equiv 1$ mod $J$. Moreover, as $J$ is a two-sided ideal, 
$g\cdot w(g,i)(x_{i5}-1)\underline f_j(w(g,i)^{-1})\in J$. It follows that $\underline\AAA(g)_{ij}\equiv \delta_{ij}$
mod $J$, and therefore that $\underline\AAA(g)-\mathrm{id}\in M_3(J)$. 
Therefore $\underline\AAA(J)\subset M_3(J)$. 

Then for $n\geq1$, $\underline\AAA(J^n)\subset\underline\AAA(J)^n\subset M_3(J)^n\subset M_3(J^n)$. 
\hfill\qed\medskip 

\begin{lem}\label{lemma:compatibility:03012018}
The algebra morphism $\underline\rho$ is compatible with the filtrations of its source and target algebras.  
\end{lem}

\proof The morphisms $\underline\ell$ and $\underline{\mathrm{pr}}_{12}$ are compatible with the filtrations as they can be identified with 
algebra morphisms between group algebras induced by group morphisms, the group algebras being equipped with the adic filtrations of their augmentation ideals. Together with Lemma \ref{prop:image:under:barA}, this implies the result. 
\hfill \qed\qed\medskip 

\subsection{Relationship between braid groups and $\Delta^{\mathcal M,\B}$} \label{sect:proof:T1} 

\begin{defn}\label{def:col:0:Betti:30oct}
Set 
$$
\underline{\mathrm{col}}_0:=\begin{pmatrix} 0 \\ (1-X_1)Y_1^{-1}\cdot 1_\B^{\otimes2} \\ (1-X_1^{-1})Y_1^{-1}
\cdot 1_\B^{\otimes2} \end{pmatrix}\in 
M_{3\times1}((\mathcal M^\B)^{\otimes2}).
\index{col_0_underline@$\underline{\mathrm{col}}_0$} 
$$
\end{defn}

\begin{lem}
Denote by $(a,x)\mapsto ax$ the action of $M_3((\mathcal V^\B)^{\otimes2})$ on 
$M_{3\times 1}((\mathcal M^\B)^{\otimes2})$. Then 
\begin{equation}\label{eq:22022019}
\underline\rho(X_0-1)\underline{\mathrm{col}}_0=0.
\end{equation}
\end{lem}

\proof \eqref{eq:rho:X0} implies the equalities
$$
\underline\rho(X_0)\underline{\mathrm{col}}_0=
\begin{pmatrix} 0\\ Y_1^{-1}Y_0Y_1(1-X_1)Y_1^{-1}\cdot 1_\B^{\otimes2} \\ 
-Y_1^{-1}Y_0Y_1X_1^{-1}(1-X_1)Y_1^{-1}
\cdot 1_\B^{\otimes2} \end{pmatrix}
=\begin{pmatrix} 0\\ (1-X_1)Y_1^{-1}Y_0\cdot 1_\B^{\otimes2} \\ 
(1-X_1^{-1})Y_1^{-1}Y_0\cdot 1_\B^{\otimes2} \end{pmatrix}
=\underline{\mathrm{col}}_0
$$ 
in $M_{3\times 1}((\mathcal V^\B/\mathcal V^\B(X_0-1))^{\otimes2})
=M_{3\times 1}((\mathcal M^\B)^{\otimes2})$. 
\hfill\qed\medskip

\begin{lem}\label{prop:8:11}
The map $(X_1-1)^{-1}(1-Y_1^{-1})^{-1}\underline{\mathrm{row}}_1\cdot(-)\cdot\underline{\mathrm{col}}_0 : M_3((\mathcal V^\B)^{\otimes2})
\to{\mathcal M}^\B[\frac{1}{X_1-1}]^{\otimes2}$ 
defined similarly to the map in Lemma \ref{lemma:6:6}, replacing exponents $\mathrm{DR}$ by $\mathrm{B}$ and $e_1$ by $X_1-1$ 
in localizations, 
has image contained in $F^{-1}({\mathcal M}^\B[\frac{1}{X_1-1}])^{\otimes2}$. 
\end{lem}

\proof This map is given by $(m_{ij})_{i,j\in[\![1,3]\!]}\mapsto \{(Y_1-1)^{-1}Y_1(m_{13}X_1^{-1}-m_{12})Y_1^{-1}(X_1-1)
\}\cdot 1_\B^{\otimes2}
+\{(X_1-1)^{-1}Y_1(m_{22}-m_{23}X_1^{-1})Y_1^{-1}(X_1-1)\}
\cdot 1_\B^{\otimes2}$.  The first (resp. second) summand belongs to 
$F^1({\mathcal M}^\B[\frac{1}{X_1-1}])\otimes F^{-1}({\mathcal M}^\B[\frac{1}{X_1-1}])$ (resp. 
$F^0({\mathcal M}^\B[\frac{1}{X_1-1}])^{\otimes2}$), 
therefore their sum belongs to the announced space. \hfill\qed\medskip 

\begin{lem}
There is a unique map $\underline\delta:\mathcal M^\B\to F^{-1}({\mathcal M}^\B[\frac{1}{X_1-1}])^{\otimes2}$\index{delta_underline@$\underline\delta$}, 
such that the following diagram commutes 
\begin{equation}\label{diag:delta:22022019}
\xymatrix{
\mathcal V^\B\ar^{\underline\rho}[rr]\ar_{(-)\cdot 1_\B}[d] && M_3((\mathcal V^\B)^{\otimes2})
\ar^{
(X_1-1)^{-1}(1-Y_1^{-1})^{-1}\underline{\mathrm{row}}_1
\cdot(-)\cdot\underline{\mathrm{col}}_0}[d]\\ 
\mathcal M^\B\ar_{\!\!\!\!\!\!\!\!\!\!\!\!\!\!\underline\delta}[rr]&&
F^{-1}({\mathcal M}^\B[\frac{1}{X_1-1}])^{\otimes2}}
\end{equation}
It is such that 
\begin{equation}\label{delta:1:B}
\underline\delta(1_\B)=1_\B^{\otimes2}.
\end{equation}
\end{lem}

\proof If $x\in\mathcal V^\B$, then 
$$
(X_1-1)^{-1}(1-Y_1^{-1})^{-1}\underline{\mathrm{row}}_1\cdot\underline\rho(x(X_0-1))\cdot
\underline{\mathrm{col}}_0=
(X_1-1)^{-1}(1-Y_1^{-1})^{-1}\underline{\mathrm{row}}_1\cdot\underline\rho(x)
\underline\rho(X_0-1)\underline{\mathrm{col}}_0=0 
$$
by (\ref{eq:22022019}), therefore $\Big(
(X_1-1)^{-1}(1-Y_1^{-1})^{-1}\underline{\mathrm{row}}_1\cdot(-)\cdot\underline{\mathrm{col}}_0
\Big)\circ\rho(\mathcal V^\B(X_0-1))=0$. The existence and uniqueness of $\underline\delta$ follow. 
One then computes 
\begin{align*}
& \underline\delta(1_\B)=(X_1-1)^{-1}(1-Y_1^{-1})^{-1}\underline{\mathrm{row}}_1
\cdot \underline\rho(1)\cdot \underline{\mathrm{col}}_0
=(X_1-1)^{-1}(1-Y_1^{-1})^{-1}\underline{\mathrm{row}}_1\cdot
\underline{\mathrm{col}}_0
\\ & =
(X_1-1)^{-1}(1-Y_1^{-1})^{-1}\cdot
(1-Y_1)(1-X_1)Y_1^{-1}\cdot 1_\B^{\otimes2}=1_\B^{\otimes2}.
\end{align*}
\hfill\qed\medskip 

\begin{lem}
The map $\underline\delta$ satisfies the identity 
\begin{equation}\label{id:underline:delta:module}
\forall x\in\mathcal W^\B,\quad\forall m\in \mathcal M^\B, \quad
\underline\delta(x\cdot m)=\Delta^{\mathcal W,\B}(x)\cdot \underline\delta(m), 
\end{equation}
where the module structure in the left-hand side (resp. right-hand side) is that of $\mathcal M^\B$ over 
$\mathcal W^\B$ (resp. ${\mathcal M}^\B[\frac{1}{X_1-1}]^{\otimes2}$ over 
$(\mathcal W^\B)^{\otimes2}$).  
\end{lem}

\proof The identity is obvious if $x=1$. Assume now that $x=a(X_1-1)$ with $a\in\mathcal V^\B$ and that 
$m\in \mathcal M^\B=\mathcal V^\B/\mathcal V^\B(X_0-1)$. Let $\tilde m\in\mathcal V^\B$ be a lift of $m$. 
Then 
\begin{align*}
& \underline\delta(x\cdot m)=\underline\delta(a(X_1-1)\cdot m)
=
(X_1-1)^{-1}(1-Y_1^{-1})^{-1}\underline{\mathrm{row}}_1
\cdot \underline\rho(a(X_1-1)\tilde m)\cdot\underline{\mathrm{col}}_0
\\ & =
(X_1-1)^{-1}(1-Y_1^{-1})^{-1}\underline{\mathrm{row}}_1
\cdot \underline\rho(a)\underline\rho(X_1-1)\underline\rho(\tilde m)\cdot\underline{\mathrm{col}}_0
\\ & =
(X_1-1)^{-1}(1-Y_1^{-1})^{-1}\underline{\mathrm{row}}_1
\cdot\underline\rho(a)\cdot \underline{\mathrm{col}}_1\cdot\underline{\mathrm{row}}_1\cdot 
\underline\rho(\tilde m)\cdot\underline{\mathrm{col}}_0
\\ & 
\scriptstyle{= 
(X_1-1)^{-1}(1-Y_1^{-1})^{-1}\underline{\mathrm{row}}_1
\cdot \underline\rho(a)\cdot \underline{\mathrm{col}}_1\cdot (X_1-1)(1-Y_1^{-1})\cdot 
(X_1-1)^{-1}(1-Y_1^{-1})^{-1}\underline{\mathrm{row}}_1\cdot 
\underline\rho(\tilde m)\cdot\underline{\mathrm{col}}_0}
\\ & = \Delta^{\mathcal W,\B}(x)\cdot \underline\delta(m), 
\end{align*} 
where $\underline{\mathrm{row}}_1$ is as in \eqref{def:underline:row:col:04012018}, 
the second equality follows from \eqref{diag:delta:22022019}, the third equality follows from the fact that
$\underline\rho$ is an algebra morphism, the fourth equality follows from \eqref{722bis}, 
and the last equality follows from the combination of 
\eqref{diag:delta:22022019} and the equality 
$$
\Delta^{\mathcal W,\B}(a(X_1-1))=
(X_1-1)^{-1}(1-Y_1^{-1})^{-1}\underline{\mathrm{row}}_1\cdot\underline\rho(a)
\cdot\underline{\mathrm{col}}_1\cdot(X_1-1)(1-Y_1^{-1})
$$
for $a\in\mathcal V^\B$, which follows from \eqref{diagram:prop:71}. 
\hfill\qed\medskip

\begin{prop}\label{prop:22022019}
The following diagram commutes 
\begin{equation}\label{diagram:2202}
\xymatrix{
\mathcal V^\B\ar^{\underline\rho}[rrr]\ar_{(-)\cdot 1_\B}[d] &&& M_3((\mathcal V^\B)^{\otimes2})
\ar^{(X_1-1)^{-1}(1-Y_1^{-1})^{-1}\underline{\mathrm{row}}_1\cdot(-)\cdot\underline{\mathrm{col}}_0}[d]\\ 
\mathcal M^\B\ar_{\Delta^{\mathcal M,\B}}[rr]&&(\mathcal M^\B)^{\otimes2}
\ar@{^(->}[r]&F^{-1}({\mathcal M}^\B[\frac{1}{X_1-1}])^{\otimes2} 
}
\end{equation}
where the right vertical map is as in Proposition \ref{prop:8:11}.  All the maps in this diagram are compatible with the filtrations.  
\end{prop}

\proof 
Combining \eqref{delta:1:B}, \eqref{id:underline:delta:module} and the fact that $\mathcal M^\B$ is a free $\mathcal W^\B$-module 
of rank one with generator $1_\B$, one obtains $\underline\delta=\Delta^{\mathcal W,\B}$, which we inject in 
(\ref{diag:delta:22022019}) to get the result. 

The filtration statement on $\underline\rho$ is Lemma \ref{lemma:compatibility:03012018}. The filtration statement on 
$\Delta^{\mathcal M,\B}$ follows from Lemma \ref{lem:sect:243:28oct}. The filtration statement on the right vertical map 
follows from: the components of the row vector belong to $F^{-1}(\mathcal V^\B[{1\over{X_1-1}}]^{\otimes2})$, the components 
of the column vector belong to $F^1(\mathcal M^\B[{1\over{X_1-1}}]^{\otimes2})$, the action of 
$\mathcal V^\B[{1\over{X_1-1}}]$ on $\mathcal M^\B[{1\over{X_1-1}}]$ is compatible with the filtrations. \hfill\qed\medskip 

\subsection{Completions 
(commutativities of (A1) in \eqref{diagg:main:alg} and (M1) in \eqref{diagram:main:mod})}\label{sect:compl:Betti:30oct}

The following lemma will be used to prove the commutativities 
mentioned in the title of this subsection. 
\begin{lem}\label{lemma:completion:10dec2019}
The commutative diagram \eqref{diagram:prop:71} (resp. \eqref{diagram:2202}) gives rise to a commutative diagram between the
completions of its constituents with respect to the filtrations, in which the completion of the map $(\mathcal W^\B)^{\otimes2}
\hookrightarrow F^0({\mathcal V}^\B[\frac{1}{X_1-1}])^{\otimes2}$ (resp. $(\mathcal M^\B)^{\otimes2}
\hookrightarrow F^{-1}({\mathcal M}^\B[\frac{1}{X_1-1}])^{\otimes2}$) is injective.  
\end{lem}

\proof By Proposition \ref{prop:comm:DR:05012017}, if one equips $\mathcal V^\B$ and $M_3((\mathcal V^\B)^{\otimes2})$ with the 
shifted filtration $F^n(\mathcal V^\B[1]):=F^{n-1}(\mathcal V^\B)$ and $F^n(M_3((\mathcal V^\B)^{\otimes2})[1]):=
M_3(F^{n-1}(\mathcal V^\B)^{\otimes2})$, then \eqref{diagram:prop:71} is a diagram in the category $\mathbf k\text{-mod}_{\fil,+}$. 
By Proposition \ref{prop:22022019}, \eqref{diagram:2202} is similarly a diagram in the same category. Applying the functor
$(-)^\wedge$ (see \S\ref{sect:functors}), one obtains commutative diagrams in $\mathbf k\text{-mod}_{\topo}$. The injectivity of the
completions of the said maps follows from Lemma \ref{lemma:injectivity:completion}, 2), together with the identification of the associated 
graded maps with the maps $(\mathcal W^\DR)^{\otimes2}\to({\mathcal V}^\DR[\frac{1}{e_1}]_{\geq0})^{\otimes2}$ and 
$(\mathcal M^\DR)^{\otimes2}\to{\mathcal M}^\DR[\frac{1}{e_1}]_{\geq-1}^{\otimes2}$ (see Proposition \ref{prop:computation:grVB:grWB}, 
Lemma \ref{lem:size:loc}, 2), Lemma \ref{lem:sect:243:28oct}, Lemma \ref{lemma:ass:gr:M:loc:19dec2019}, 2)), which are known to be injective. 
\hfill\qed\medskip 

\part{Associators and harmonic coproducts}

In this part, we recall the formalism of associators and its interpretation in terms of categories related to braids, and do some 
related computations of matrices (\S\ref{sect:associators:2jan2020}). We combine these computations with the geometric 
interpretations of the Betti and de Rham algebra coproducts obtained in Part \ref{part:geom:int} to show Theorem 
\ref{thm:compat:assoc:algebra:24dec2019}, according to which any associator relates these two algebra coproducts 
(\S\ref{sect:8:19032018}). We follow a similar approach in \S\ref{sect:aamhc:2jan2020} to show Theorem 
\ref{thm:new:crm}, according to which any associator relates the Betti and de Rham module coproducts.  

\section{Associators and comparison isomorphisms}\label{sect:associators:2jan2020}

The purpose of this section is to recall some facts on associators, and in particular how these objects
relate braid groups to infinitesimal braid Lie algebras.

In \S\ref{sect:assoc:20032018}, we recall the definition of the set of associators over $\mathbf k$ with a given parameter 
$\mu\in\mathbf k^\times$ (Definition \ref{def:assoc:20032018}). In \S\ref{subsect:Gamma:fun:20032018}, we recall the construction and properties of $\Gamma$-functions of associators (Lemma \ref{lemma:Gamma:associators}). In 
\S\ref{isos:associators:20033018}, we recall from \cite{BN} that the choice of an associator gives rise to a functor 
from a category $\mathbf{PaB}$ of parenthesized braids to a category $\widehat{\mathbf{PaCD}}$ of parenthesized 
chord diagrams; we compute images of particular morphisms of $\mathbf{PaB}$. This leads, for each associator $(\mu,\Phi)$, to a collection $\mathrm{comp}_{(\mu,\Phi)}^P$ of morphisms from algebras associated to braid groups to algebras associated to 
infinitesimal braid Lie algebras, indexed by parenthesized words $P$ in one letter $\bullet$. In \S\ref{sect:rels:B4}, we express particular elements of the pure braid group $K_4$  in terms of elements 
$\sigma_{a,b}\in B_4$. We apply these results to explicitly compute the images of these elements in the modular group 
$P_5^*$ in the group of invertible elements of the completion $(U\mathfrak p_5)^\wedge$ of the enveloping algebra of 
the infinitesimal braid Lie algebra $\mathfrak p_5$, first under the functor 
$\mathrm{comp}^{(\bullet(\bullet\bullet))\bullet}_{(\mu,\Phi)}$ (\S\ref{comp:of:elements:20032018}, Proposition 
\ref{prop:comp:images:20190728}), then under the functor $\mathrm{comp}^{((\bullet\bullet)\bullet)\bullet}_{(\mu,\Phi)}$ 
(\S\ref{sect:9:6:12nov2019}, Proposition \ref{prop:comp:images:brd:grp:elts}). This leads to the definition of 
matrices $P_{(\mu,\Phi)}$, $R_{(\mu,\Phi)}$ and $\overline P_{(\mu,\Phi)}$, $\overline R_{(\mu,\Phi)}$ (\S\ref{sect:def:matrices:P:R}) 
and to their computation in \S\S\ref{sect:comp:P} and \ref{sect:comp:R}. 

\subsection{The set $\mathsf M(\mathbf k)$ of associators}\label{sect:assoc:20032018}


Let $\mathcal A=F^0\mathcal A\supset F^1\mathcal A\supset\cdots$ be a filtered algebra over $\mathbf k$, complete 
with respect to its filtration. Let $a_0,a_1\in F^1\mathcal A$. Then there is a morphism 
$\mathrm{ev}_{a_0,a_1}\index{ev_a@$\mathrm{ev}_{a_0,a_1}$}:\hat{\mathcal V}^\DR\to\mathcal A$, uniquely determined by the conditions $e_i\mapsto a_i$ for $i=0,1$. 

\begin{defn}
{\it For $\Phi\in \hat{\mathcal V}^\DR$, we set $\Phi(a_0,a_1):=\mathrm{ev}_{a_0,a_1}(\Phi)\index{Phi(a_0,a_1) @ $\Phi(a_0,a_1)$}$
(this is an element of $\mathcal A$).} 
\end{defn}

\begin{rem}
Assume that $\mathcal A=\hat{\mathcal V}^\DR$, $(a_0,a_1):=(e_0,e_1)$, then 
$\mathrm{ev}_{a_0,a_1}=\mathrm{id}_{\hat{\mathcal V}^\DR}$, therefore $\Phi=\Phi(e_0,e_1)\index{Phi(e_0,e_1) @ $\Phi(e_0,e_1)$}$. 
\end{rem}

Recall the graded Lie algebra $\mathfrak t_4$ from \S\ref{subsect:def:LA:morphisms}. 

\begin{defn} \label{def:assoc:20032018} (\cite{Dr})
{\it If $\mu\in\mathbf k$, one defines the set $\mathsf M_\mu(\mathbf k)\index{M_muk @ $\mathsf M_\mu(\mathbf k)$}$ 
to be the set of elements $\Phi\in\hat{\mathcal V}^\DR$, which are group-like for $\hat{\Delta}^{\mathcal V,\DR}$ and such that  
\begin{equation}\label{duality:rel}
\Phi(b,a)=\Phi(a,b)^{-1}, 
\end{equation}
\begin{equation}\label{hexagon+}
e^{(\mu/2)b}\Phi(c,b)e^{(\mu/2)c}\Phi(a,c)e^{(\mu/2)a}\Phi(b,a)=1 
\end{equation}
(equalities in $\mathbf k\langle\langle a,b\rangle\rangle$, called the 2-cycle and hexagon conditions; here $c:=-a-b$), 
$$
\Phi(t_{12},t_{23}+t_{24})\cdot\Phi(t_{13}+t_{23},t_{34})=\Phi(t_{23},t_{34})\cdot\Phi(t_{12}+t_{13},t_{24}+t_{34})\cdot\Phi(t_{12},t_{23}) 
$$
(equality in $(U\mathfrak t_4)^\wedge$, called the pentagon condition).

We call $\mathsf M(\mathbf k):=\{(\mu,\Phi)|\mu\in\mathbf k^\times,\Phi\in
\mathsf M_\mu(\mathbf k)\}$\index{M_k @ $\mathsf M(\mathbf k)$} the set of associators over $\mathbf k$. 
} 
\end{defn}
In \cite{F:pentagon}, it was proved  that the first two equalities are consequences of the latter one, 
where $\mu$ is obtained from the expansion $\Phi=1+{\mu^2\over 24}[e_0,e_1]+$ terms of degree $\geq3$.   
These two equalities also imply the relation 
\begin{equation}\label{hexagon-}
e^{-(\mu/2)b}\Phi(c,b)e^{-(\mu/2)c}\Phi(a,c)e^{-(\mu/2)a}\Phi(b,a)=1 
\end{equation}
in $\mathbf k\langle\langle a,b\rangle\rangle$. 

\begin{rem}
In \cite{Dr}, \S2, it is proved that $\varphi_{\mathrm{KZ}}\in(\hat{\mathcal V}_{\mathbb C}^\DR)^\times$
(see \S\ref{sect:rel:Rac}) belongs to $\mathsf M_1(\mathbb C)$, and that $\mathsf M_\mu(\mathbb Q)$ is nonempty, from 
which one derives that $\mathsf M_\mu(\mathbf k)$ is nonempty for $\mathbf k$ any $\mathbb Q$-algebra and any 
$\mu\in\mathbf k^\times$. 
\end{rem}

\subsection{$\Gamma$-functions}\label{subsect:Gamma:fun:20032018}

Let $\Phi\in\mathsf M_\mu(\mathbf k)$. Let $\varphi_0,\varphi_1$ be the elements of $\hat{\mathcal V}^\DR$
defined by the equality $\Phi=1+\varphi_0 e_0+\varphi_1 e_1$ (equality in $\hat{\mathcal V}^\DR$). 
Let $\overline e_0,\overline e_1\index{e_0, e_1^bar @ $\overline e_0,\overline e_1$}$ be free commutative formal variables; there is a unique continuous $\mathbf k$-algebra morphism 
$\hat{\mathcal V}^\DR\to\mathbf k[[\overline e_0,\overline e_1]]
\index{k[[e_0,e_1]]@$\mathbf k[[\overline e_0,\overline e_1]]$}$, denoted $f\mapsto f^{\mathrm{ab}}\index{f^ab@$f^{\mathrm{ab}}$}$, such that $e_i\mapsto
\overline e_i$ for $i=0,1$.  

\begin{lem}\label{lemma:Gamma:associators} 
Let $\mu\in\mathbf k$ and $\Phi\in\mathsf  M_\mu(\mathbf k)$. Let $\Gamma_\Phi(t)\in\mathbf k[\![t]\!]^\times$ be as in 
\eqref{def:Gamma:fun:of:NC:series}. 
\begin{enumerate}
\item One has the identity 
\begin{equation}\label{identity:Gamma:1/12/2017}
(1+\varphi_1 e_1)^{\mathrm{ab}}={\Gamma_{\Phi}(-\overline e_0)\Gamma_{\Phi}(-\overline e_1)\over
\Gamma_{\Phi}(-\overline e_0-\overline e_1)}
\end{equation}
in $\mathbf k[[\overline e_0,\overline e_1]]$. 

\item $\Gamma_{\Phi}$ satisfies the identity
\begin{equation}\label{funct:id:Gamma:Phi}
\Gamma_{\Phi}(t)\Gamma_{\Phi}(-t)={\mu t\over e^{\mu t/2}-e^{-\mu t/2}}
\end{equation}
in $1+t^2\mathbf k[[t]]$.  
\end{enumerate}
\end{lem}

\proof In \cite{E}, one attaches to $\Phi$ a collection $(\zeta_\Phi(n))_{n\geq2}\index{zeta_\Phi@$\zeta_\Phi(n)$}$ of elements of $\mathbf k$
with the following properties: (a) for $n$ even $\geq 2$, one has $\zeta_\Phi(n)=\mu^n\cdot \zeta(n)/(2\pi \mathrm{i})^n$; 
(b) the series $\tilde\Gamma_\Phi(t)$ (denoted $\Gamma_\Phi(t)$ in \cite{E}) defined by 
$\tilde\Gamma_\Phi(t):=\mathrm{exp}(-\sum_{n\geq2}\zeta_\Phi(n)t^n/n)$ is such that 
\begin{equation}\label{interm:eq}
(1+\varphi_1 e_1)^{\mathrm{ab}}=
{\tilde\Gamma_\Phi(\overline e_0+\overline e_1)\over\tilde\Gamma_\Phi(\overline e_0)\tilde\Gamma_\Phi(\overline e_1)}
\end{equation}
(identity in $\mathbf k[[\overline e_0,\overline e_1]]$). 
Then 
\begin{align*}
& \text{(r.h.s. of (\ref{interm:eq}))}=\mathrm{exp}(-\sum_{n\geq 2}{\zeta_\Phi(n)\over n} 
\{(\overline e_0+\overline e_1)^n-\overline e_0^n-\overline e_1^n\})
\\ 
& =\mathrm{exp}(-\sum_{n\geq 2}{\zeta_\Phi(n)\over n} 
\{n\overline e_0^{n-1}\overline e_1+O(\overline e_1^2)\})  
= 1-\sum_{n\geq 2}\zeta_\Phi(n)\overline e_0^{n-1}\overline e_1+O(\overline e_1^2),   
\end{align*}
and 
$$
\text{(l.h.s. of (\ref{interm:eq}))}=1+\sum_{n\geq 2}(\Phi|e_0^{n-1}e_1)\overline e_0^{n-1}\overline e_1+O(\overline e_1^2). 
$$
Then (\ref{interm:eq}) implies $\zeta_\Phi(n)=-(\Phi|e_0^{n-1}e_1)$ for $n\geq2$. It follows that 
$\Gamma_{\Phi}(t)=1/\tilde\Gamma_\Phi(-t)$, 
where $\Gamma_{\Phi}$ is as in \eqref{def:Gamma:fun:of:NC:series}. Plugging this equality in (\ref{interm:eq}), we obtain 
(\ref{identity:Gamma:1/12/2017}). This proves 1). 

Let us prove 2). Since $\mathrm{log}\Gamma(1-t)=\gamma t+\sum_{n\geq 2}\zeta(n) t^n/n$, 
the series $\mathrm{exp}(2\sum_{n\mathrm{\ even},n\geq 2}{\zeta(n)t^n\over n})$ is equal to 
$\Gamma(1+t)\Gamma(1-t)$. The identities $\Gamma(t+1)=t\Gamma(t)$ and $\Gamma(t)\Gamma(-t)
={t\over{\mathrm{sin}(\pi t)}}$ imply that the latter series equals $(2\pi \mathrm{i} t) / (e^{(2\pi \mathrm{i} t)/2}-
e^{-(2\pi \mathrm{i} t)/2})$, therefore 
$$
\mathrm{exp}(2\sum_{n\mathrm{\ even},n\geq 2}{\zeta(n)t^n\over n})={2\pi \mathrm{i} t\over{e^{(2\pi \mathrm{i} t)/2}-
e^{-(2\pi \mathrm{i} t)/2}}}. 
$$
Then 
\begin{align*}
& \Gamma_{\Phi}(t)\Gamma_{\Phi}(-t)=(\tilde\Gamma_\Phi(t)\tilde\Gamma_\Phi(-t))^{-1}
=\mathrm{exp}(2\sum_{n\mathrm{\ even},n\geq 2}{\zeta_\Phi(n)t^n\over n})
=\mathrm{exp}(2\sum_{n\mathrm{\ even},n\geq 2}{\zeta(n) (\mu t/(2\pi \mathrm{i}))^n\over n})
\\ & ={\mu t\over e^{\mu t/2}-e^{-\mu t/2}},  
\end{align*}
which proves 2). \hfill \qed\medskip 

\begin{rem}
Lemma \ref{lemma:Gamma:associators} can also be derived by combining two results from \cite{F}, namely 
the inclusion result $\mathsf M_\mu(\mathbf k)\subset\mathsf{DMR}_\mu(\mathbf k)$ 
and the result on $\Gamma$-functions for elements of $\mathsf{DMR}_\mu(\mathbf k)$.  
\end{rem}

\begin{lem}
The following identities 
\begin{equation}\label{ids:phi:ab}
\varphi_1(\alpha,\beta)={1\over\beta}\Big({{\Gamma_\Phi(-\alpha)\Gamma_\Phi(-\beta)}\over{\Gamma_\Phi(-\alpha-\beta)}}-1\Big), 
\quad
\varphi_0(\alpha,\beta)=-{1\over\alpha}\Big({{\Gamma_\Phi(-\alpha)\Gamma_\Phi(-\beta)}\over{\Gamma_\Phi(-\alpha-\beta)}}-1\Big), 
\end{equation}
hold in the commutative formal series ring $\mathbf k[[\alpha,\beta]]$. 
\end{lem}

\proof The first identity follows from the image of \eqref{identity:Gamma:1/12/2017} by the isomorphism of formal series rings taking 
$(\overline e_0,\overline e_1)$ to $(\alpha,\beta)$. One has $\Phi(e_0,e_1)\in\mathcal G(\hat{\mathcal V}^\DR)$, which implies that 
$\mathrm{log}\Phi(e_0,e_1)\in\hat{\mathfrak f}_2$. One also knows that the degree 1 component of the series 
$\mathrm{log}\Phi(e_0,e_1)$ in $e_0,e_1$ is zero. This implies $(\mathrm{log}\Phi(e_0,e_1))^{\mathrm{ab}}=0$, therefore 
$\Phi(e_0,e_1)^{\mathrm{ab}}=1$. This implies the equality $1+\alpha\varphi_0(\alpha,\beta)+\beta\varphi_1(\alpha,\beta)=1$. 
Together with the first identity of \eqref{ids:phi:ab}, this implies the second identity of \eqref{ids:phi:ab}. \hfill\qed\medskip 

\subsection{Functors arising from associators}\label{isos:associators:20033018} 

\subsubsection{The category $\mathbf{PaB}$}\label{sect:9.1.3.0306}

For $n\geq0$, let $\mathrm{Par}_n$\index{Parn@$\mathrm{Par}_n$} be the set of parenthesizations of the word $\underbrace{\bullet\bullet\cdots\bullet}_{n\text{ letters}}$. 
Set $\mathrm{Par}:=\sqcup_{n\geq 0}\mathrm{Par}_n$\index{Par@$\mathrm{Par}$}. The set Par is a monoid with product denoted $(P,Q)\mapsto PQ$. 
For $P\in\mathrm{Par}$, we denote by $|P|$ the integer $n$ such that $P\in\mathrm{Par}_n$. 

We denote by $\mathbf{PaB}$\index{PaB@$\mathbf{PaB}$} the category with set of objects given by $\mathrm{Ob}(\mathbf{PaB}):=\mathrm{Par}$
and set of morphisms given by $\mathbf{PaB}(P,Q):=\emptyset$ if $|P|\neq|Q|$, $\mathbf{PaB}(P,Q):=B_{|P|}$ if $|P|=|Q|$, where 
$B_{|P|}$ is the Artin braid group with $|P|$ strands. 

The composition $\mathbf{PaB}(P,Q)\times\mathbf{PaB}(Q,R)\to\mathbf{PaB}(P,R)$ is given by the product in $B_{|P|}$
if $|P|=|Q|=|R|$, more precisely $((P\stackrel{f}{\to}Q),(Q\stackrel{g}{\to}R))\mapsto
(P\stackrel{gf}{\to}R)$ for $f,g\in B_{|P|}$. 

We now introduce particular morphisms of $\mathbf{PaB}$: 

(a) for $P,Q,R\in\mathrm{Par}$, we denote by 
$$
a_{P,Q,R}\in\mathbf{PaB}((PQ)R,P(QR))
\index{a_PQR@$a_{P,Q,R}$}
$$
the image of $1\in B_{|P|+|Q|+|R|}\simeq\mathbf{PaB}((PQ)R,P(QR))$;   

(b) for $P\in\mathrm{Par}$ and $b\in B_{|P|}$, we denote by 
$$
b_P\in\mathbf{PaB}(P)
\index{b_p@$b_P$}
$$ 
the image of $b\in B_{|P|}\simeq\mathbf{PaB}(P)$;  

(c) for $P,Q\in\mathrm{Par}$, we denote by 
$$
\sigma_{P,Q}\in\mathbf{PaB}(PQ,QP)
\index{sigma_PQ@$\sigma_{P,Q}$}
$$
the image of $\sigma_{|P|,|Q|}\in B_{|P|+|Q|}\simeq\mathbf{PaB}(PQ,QP)$, where $\sigma_{|P|,|Q|}$ is as in \eqref{def:sigma:a:b}.  

\subsubsection{The category $\widehat{\mathbf{PaCD}}$}

For $n\geq1$, the permutation group $S_n\index{S_n@$S_n$}$ acts by automorphisms of $\mathfrak t_n$ by permutation of indices via $\sigma\cdot t_{ij}
=t_{\sigma(i)\sigma(j)}$. 
This action gives rise to an algebra structure on the tensor product 
$(U\mathfrak t_n)^\wedge\otimes\mathbf k S_n$, and to a topological Hopf algebra on it, defined by the conditions that 
the elements of $S_n$ be group-like and the elements of $\mathfrak t_n$ be primitive; the resulting topological Hopf algebra structure
is denoted by $(U\mathfrak t_n)^\wedge\rtimes S_n\index{UT_nkS_n@ $U(\mathfrak t_n)^\wedge\rtimes\mathbf k S_n$}$.

Following \cite{BN}, we denote by $\widehat{\mathbf{PaCD}}$\index{PaCD^@$\widehat{\mathbf{PaCD}}$} the category with set of objects $\mathrm{Par}$ and sets of morphisms given by 
$\widehat{\mathbf{PaCD}}(P,Q):=(U\mathfrak t_{|P|})^\wedge\rtimes S_{|P|}$ if $|P|=|Q|$ and $\widehat{\mathbf{PaCD}}(P,Q):=0$ 
otherwise, and with composition induced by the product in $(U\mathfrak t_{|P|})^\wedge\rtimes S_{|P|}$ in the same way as the composition 
in $\mathbf{PaB}$ is induced by the product in braid groups. 

We also denote by $\widehat{\mathbf{PaCD}^*}$\index{PaCD^ast@$\widehat{\mathbf{PaCD}^*}$} the analogue of $\widehat{\mathbf{PaCD}}$, where the Lie algebra $\mathfrak t_{|P|}$
is replaced by its quotient $\mathfrak p_{|P|+1}$. One then has a natural functor $\widehat{\mathbf{PaCD}}\to\widehat{\mathbf{PaCD}^*}$. 

\subsubsection{The functor $\mathrm{comp}_{(\mu,\Phi)}:\mathbf{PaB}\to\widehat{\mathbf{PaCD}}$}\label{sect:9:3:3:20191203}

It follows from \cite{BN}, Theorem 1, that each 
pair $(\mu,\Phi)$, with $\mu\in\mathbf k^\times$ and $\Phi\in\mathsf M_\mu(\mathbf k)$, gives rise to a functor 
$$
\mathrm{comp}_{(\mu,\Phi)}:\mathbf{PaB}\to\widehat{\mathbf{PaCD}}
\index{comp@$\mathrm{comp}_{(\mu,\Phi)}$}
$$
determined by 
\begin{equation}\label{image:sigma:11}
\mathrm{comp}_{(\mu,\Phi)}(\sigma_{\bullet,\bullet})=\mathrm{exp}(-{\mu\over 2}t_{12})\cdot \bigl(\begin{smallmatrix}
    1 & 2 \\
    2 & 1
  \end{smallmatrix}\bigr)\in\widehat{\mathbf{PaCD}}(\bullet\bullet), 
\end{equation}
\begin{equation}\label{image:iso}
\mathrm{comp}_{(\mu,\Phi)}(a_{\bullet,\bullet,\bullet})=\Phi(t_{12},t_{23})
\in\widehat{\mathbf{PaCD}}((\bullet\bullet)\bullet,\bullet(\bullet\bullet)), 
\end{equation}
and its compatibility with operations of extensions, cabling and strand removal from \cite{BN} 
in both the source and target categories.  

Let us again denote by 
$$
\mathrm{comp}_{(\mu,\Phi)}:\mathbf{PaB}\to\widehat{\mathbf{PaCD}^*}
$$
the composition of functors $\mathbf{PaB}\stackrel{\mathrm{comp}_{(\mu,\Phi)}}{\to}\widehat{\mathbf{PaCD}}\to\widehat{\mathbf{PaCD}^*}$.

For $P,Q\in\mathrm{Par}$ with $|P|=|Q|=n$, we will denote by 
$$
\mathrm{comp}_{(\mu,\Phi)}^{P,Q}:B_n\to(U\mathfrak p_{n+1})^\wedge\rtimes S_n
\index{comp^PQ@$\mathrm{comp}_{(\mu,\Phi)}^{P,Q}$}
$$
the composed map 
$$
B_n\simeq \mathbf{PaB}(P,Q)
\stackrel{\mathrm{comp}^{P,Q}_{(\mu,\Phi)}}{\to}\widehat{\mathbf{PaCD}^*}(P,Q)
\simeq(U\mathfrak p_{n+1})^\wedge\rtimes S_n. 
$$
It corestricts to a map $B_n\to((U\mathfrak p_{n+1})^\wedge)^\times\rtimes S_n$, which is a group morphism 
denoted $\mathrm{comp}_{(\mu,\Phi)}^P:=\mathrm{comp}_{(\mu,\Phi)}^{P,P}$
\index{comp^P@$\mathrm{comp}_{(\mu,\Phi)}^{P}$}
 when $Q=P$.  

The map $\mathrm{comp}_{(\mu,\Phi)}^{P,Q}$ restricts and corestricts to a map $K_n\to
(U\mathfrak p_{n+1})^\wedge$, which factors through a map 
$P_{n+1}^*\to(U\mathfrak p_{n+1})^\wedge$. The latter map induces an isomorphism 
$\hat{\mathcal V}^\B(\mathfrak M_{0,{n+1}})\stackrel{\sim}{\to}\hat{\mathcal V}^\DR(\mathfrak M_{0,{n+1}})$, 
where $\hat{\mathcal V}^\B(\mathfrak M_{0,{n+1}}):=(\mathbf kP_{n+1}^*)^\wedge$, 
$\hat{\mathcal V}^\DR(\mathfrak M_{0,{n+1}}):=(U\mathfrak p_{n+1})^\wedge$. 

It follows from the categorical formalism that the algebra isomorphisms 
$\mathrm{comp}_{(\mu,\Phi)}^{(\bullet(\bullet\bullet))\bullet}$ and 
$\mathrm{comp}_{(\mu,\Phi)}^{((\bullet\bullet)\bullet)\bullet} : 
\hat{\mathcal V}^\B(\mathfrak M_{0,5})\to\hat{\mathcal V}^\DR(\mathfrak M_{0,5})$ are related by 
\begin{equation}\label{relation:comp:comp}
\mathrm{comp}_{(\mu,\Phi)}^{(\bullet(\bullet\bullet))\bullet}=\mathrm{Ad}(\Phi(e_{12},e_{23}))\circ
\mathrm{comp}_{(\mu,\Phi)}^{((\bullet\bullet)\bullet)\bullet}
\end{equation}

\begin{rem}
There are natural identifications $\mathrm{comp}_{(\mu,\Phi)}^{\mathcal V,(1)}\simeq\mathrm{comp}_{(\mu,\Phi)}^{(\bullet\bullet)\bullet}$ and $\mathrm{comp}_{(\mu,\Phi)}^{\mathcal V,(10)}\simeq
\mathrm{comp}_{(\mu,\Phi)}^{\bullet(\bullet\bullet),(\bullet\bullet)\bullet}$.
 \end{rem}

\subsubsection{Images by $\mathrm{comp}_{(\mu,\Phi)}$ of particular morphisms}

\begin{lem}\label{lemma:images}

Let $\mu\in\mathbf k^\times$ and $\Phi\in\mathsf M_\mu(\mathbf k)$. 

1) If $P,Q$ are in $\mathrm{Par}$, with $|P|=a$, $|Q|=b$, then 
\begin{align}\label{image:sigma}
& \mathrm{comp}_{(\mu,\Phi)}(\sigma_{P,Q}) =
 \bigl(\begin{smallmatrix}    1 & \cdots & a& a+1 & \cdots & a+b \\    b+1 & \cdots & a+b& 1 & \cdots & b\end{smallmatrix}\bigr) 
 \cdot \mathrm{exp}({\mu\over 2}\sum_{\beta\in a+[\![1,b]\!]}e_{\beta,a+b+1}
 +\mu\sum_{\alpha<\beta\in a+[\![1,b]\!]}e_{\alpha,\beta})
  \\ & \nonumber
  = \bigl(\begin{smallmatrix}
    1 & \cdots & a& a+1 & \cdots & a+b \\
    b+1 & \cdots & a+b& 1 & \cdots & b
  \end{smallmatrix}\bigr) \cdot 
  \mathrm{exp}({\mu\over 2}\sum_{\alpha\in [\![1,a]\!]}e_{\alpha,a+b+1}
 +\mu\sum_{\alpha<\beta\in [\![1,a]\!]}e_{\alpha,\beta})
   \in\widehat{\mathbf{PaCD}^*}(PQ,QP). 
\end{align}

2) If $P,Q,R$ are in $\mathrm{Par}$, with $|P|=a$, $|Q|=b$, $|R|=c$, then 
\begin{equation}\label{image:a}
\mathrm{comp}_{(\mu,\Phi)}(a_{P,Q,R})=\Phi(\sum_{\gamma\in a+b+[\![1,c]\!]}e_{\gamma,a+b+1},
\sum_{\alpha\in[\![1,a]\!]}e_{\alpha,a+b+1})
\in\widehat{\mathbf{PaCD}^*}((PQ)R,P(QR)). 
\end{equation}
\end{lem}

\proof
1) It follows from \eqref{image:sigma:11} that if $P,Q$ are in $\mathrm{Par}$, with $|P|=a$, $|Q|=b$, then 
$$
\mathrm{comp}_{(\mu,\Phi)}(\sigma_{P,Q})=
\bigl(\begin{smallmatrix}    1 & \cdots & a& a+1 & \cdots & a+b \\    b+1 & \cdots & a+b& 1 & \cdots & b\end{smallmatrix}\bigr) 
 \cdot \mathrm{exp}(-{\mu\over 2}\sum_{\alpha\in A,\beta\in B}t_{\alpha,\beta})
\in\widehat{\mathbf{PaCD}}(PQ,QP),  
$$
where $A:=[\![1,a]\!]$, $B:=a+[\![1,b]\!]$. 
The image of $-{1\over 2}\sum_{\alpha\in A,\beta\in B}t_{\alpha,\beta}$ under $\mathfrak t_{a+b}\to\mathfrak p_{a+b+1}$ is 
$$ -{1\over2}\sum_{\alpha\in A,\beta\in B}e_{\alpha,\beta} 
={1\over 2}\sum_{\beta\in B}(-\sum_{\alpha\in A}e_{\alpha,\beta}-\sum_{\beta'\in B}e_{\beta,\beta'})
 +\sum_{\alpha<\beta\in B}e_{\alpha,\beta}= {1\over 2}\sum_{\beta\in B}e_{\beta,a+b+1}
 +\sum_{\alpha<\beta\in B}e_{\alpha,\beta}. 
$$
By symmetry, one also has 
$$ 
-{1\over2}\sum_{\alpha\in A,\beta\in B}e_{\alpha,\beta} 
={1\over 2}\sum_{\alpha\in A}e_{\alpha,a+b+1}+\sum_{\alpha<\beta\in B}e_{\alpha,\beta}. 
$$
This implies that the image of $\mathrm{comp}_{(\mu,\Phi)}(\sigma_{P,Q})$ in $\widehat{\mathbf{PaCD}^*}(PQ,QP)$
is the announced value. 

2) 
It follows from \eqref{image:iso} that if $P,Q,R$ are in $\mathrm{Par}$, with $|P|=a$, $|Q|=b$, $|R|=c$, then 
$$
\mathrm{comp}_{(\mu,\Phi)}(a_{P,Q,R})=\Phi(\sum_{\alpha\in A,\beta\in B}t_{\alpha,\beta},
\sum_{\beta\in B,\gamma\in C}t_{\beta,\gamma})
\in\widehat{\mathbf{PaCD}}((PQ)R,P(QR)), 
$$
where $A:=[\![1,a]\!]$, $B:=a+[\![1,b]\!]$, $C:=a+b+[\![1,c]\!]$. 

For $I\subset[\![1,a+b+c]\!]$, set $e_I:=\sum_{i,i'\in I}e_{i,i'}$, and for $I,J\subset[\![1,a+b+c]\!]$ disjoint, set $e_{IJ}:=\sum_{i\in I,j\in J}e_{i,j}$. 
Then 
$$ 
\sum_{\gamma\in C}e_{\gamma,a+b+c+1}=-\sum_{\gamma\in C,\alpha\in A}e_{\alpha,\gamma}
-\sum_{\gamma\in C,\beta\in B}e_{\beta,\gamma}-\sum_{\gamma,\gamma'\in C}e_{\gamma,\gamma'}
=z+e_{AB}, 
$$
where 
$$
z=-e_{AB}-e_{AC}-e_{BC}-e_C.
$$
Similarly,
$$ 
\sum_{\alpha\in A}e_{\alpha,a+b+c+1}=z'+e_{BC},
$$
where 
$$
z'=-e_{AB}-e_{AC}-e_{BC}-e_A, 
$$
One immediately has $[e_{AB},e_C]=[e_{BC},e_A]=0$. Moreover, the relations $[e_{\alpha,\beta}+e_{\alpha',\beta},e_{\alpha,\alpha'}]=0$
for $\alpha,\alpha'\in A$ and $\beta\in B$,  
and $[e_{\alpha'',\beta},e_{\alpha,\alpha'}]=0$ for $\beta\in B$ and $\alpha,\alpha',\alpha''\in A$ all distinct, 
imply that $[e_{AB},e_A]=0$. Similarly, $[e_{AC},e_A]=0$ and $[e_{AC},e_C]=[e_{BC},e_C]=0$. Finally, 
\begin{align*}
&
\sum_{\alpha,\alpha'\in A,\beta\in B,\gamma\in C}[e_{\alpha,\gamma},e_{\alpha',\beta}]
+\sum_{\alpha\in A,\beta,\beta'\in B,\gamma\in C}[e_{\beta,\gamma},e_{\alpha,\beta'}]
=\sum_{\alpha\in A,\beta\in B,\gamma\in C}[e_{\alpha,\gamma},e_{\alpha,\beta}]
+\sum_{\alpha\in A,\beta\in B,\gamma\in C}[e_{\beta,\gamma},e_{\alpha,\beta}]
\\&=\sum_{\alpha\in A,\beta\in B,\gamma\in C}[e_{\alpha,\gamma}+e_{\beta,\gamma},e_{\alpha,\beta}]=0
\end{align*} 
therefore $[e_{AC}+e_{BC},e_{AB}]=0$. 

All this implies
$$
[z,z']=[z,e_{AB}]=[z',e_{BC}]=[z,e_{BC}]=[z',e_{AB}]=0, 
$$
which, together with the fact that $\Phi$ is the exponential of a Lie series without linear terms, implies that 
$\mathrm{comp}_{(\mu,\Phi)}(a_{P,Q,R})\in\widehat{\mathbf{PaCD}^*}((PQ)R,P(QR))$ is the announced value. 
\hfill \qed\medskip 

\subsection{Relations in $B_4$}\label{sect:rels:B4}

\begin{lem}\label{lemma: elements and equalities in Bn}
One has the following equalities in $B_4$:  
\begin{equation}\label{equality:x15}
\tilde x_{1,5}=\sigma_{3,1}\sigma_{1,3}, 
\end{equation}
\begin{equation}\label{equality:x25:bis}
\tilde x_{2,5}=\sigma_{2,2}\cdot\tilde x_{4,5}\cdot\sigma_{2,2}^{-1}. 
\end{equation}
\begin{equation}\label{equality:x35}
\tilde x_{3,5}=\sigma_{1,3}\cdot\tilde x_{4,5}\cdot\sigma_{1,3}^{-1}.
\end{equation}
\begin{equation}\label{equality:x45}
\tilde x_{4,5}=\sigma_{1,3}\sigma_{3,1}, 
\end{equation}
\begin{equation}\label{equality:x25}
\tilde x_{2,5}=\sigma_{1,3}\sigma_{2,2}^{-1}\cdot\tilde x_{1,5}\cdot\sigma_{2,2}\sigma_{1,3}^{-1}, 
\end{equation}
\end{lem}

\proof
These equalities follow from the fact their left-hand sides can be represented as indicated in Figure \ref{picture for lemma1.1}.
\begin{figure}[h]
\begin{tabular}{c}
  \begin{minipage}{0.2\hsize}
      \begin{center} 
  \begin{tikzpicture}
      \draw[-] (0,2)--(1.5,1);
      \draw[color=white, line width=3pt] (0,1)--(0.5,2) (0.5,1)--(1,2) (1,1)--(1.5,2) ; 
      \draw[-] (0,1)--(0.5,2) (0.5,1)--(1,2) (1,1)--(1.5,2) ; 
      \draw[-] (0,1)--(0.5,0) (0.5,1)--(1,0) (1,1)--(1.5,0) ;
      \draw[color=white, line width=3pt] (0,0)--(1.5,1) ; 
      \draw[-] (0,0)--(1.5,1) ; 
      \draw[color=white, line width=3pt] (1,3.75)--(1,3.75) (1,-1.35)--(1,-1.35) ;
  \end{tikzpicture}    \\ $\tilde x_{1,5}$
     \end{center}
 \end{minipage}
 \begin{minipage}{0.2\hsize}
      \begin{center}
  \begin{tikzpicture}
      \draw[-] (0,3)--(1,4) (0.5,3)--(1.5,4) ; 
      \draw[color=white, line width=3pt] (0,4)--(1,3) (0.5,4)--(1.5,3) ; 
      \draw[-] (0,4)--(1,3) (0.5,4)--(1.5,3) ;     

      \draw[-](-0.5,2.5)..controls(-0.75,2.25) and (1.7,2.25)  ..(1.5,2);
      \draw[color=white, line width=3pt] (0,3)--(0,2) (0.5,3)--(0.5,2) (1,3)--(1,2); 
      \draw[-] (0,3)--(0,2) (0.5,3)--(0.5,2) (1,3)--(1,2); 
      \draw[color=white, line width=3pt](1.5,3)..controls(1.7,2.75) and (-0.75,2.75)  ..(-0.5,2.5);
      \draw[-](1.5,3)..controls(1.7,2.75) and (-0.75,2.75)  ..(-0.5,2.5);
            
      \draw[-] (-0,2)--(1,1) (0.5,2)--(1.5,1) ;           
      \draw[color=white, line width=3pt] (1,2)--(0,1) (1.5,2)--(0.5,1);  
      \draw[-] (1,2)--(0,1) (1.5,2)--(0.5,1);     
      \draw[color=white, line width=3pt] (1.2,5)--(1.2,5) (0,0)--(0,0) ;
  \end{tikzpicture}  \\ $\tilde x_{2,5}$
      \end{center}
 \end{minipage}
 \begin{minipage}{0.2\hsize}
      \begin{center}
  \begin{tikzpicture}
      \draw[-] (0.5,3)--(2,4) ; 
      \draw[color=white, line width=3pt] (0.5,4)--(1,3) (1,4)--(1.5,3) (1.5,4)--(2,3) ;
      \draw[-] (0.5,4)--(1,3) (1,4)--(1.5,3) (1.5,4)--(2,3) ; 
            
      \draw[-](0,2.5)..controls(-0.25,2.25) and (2.2,2.25)  ..(2,2);
      \draw[color=white, line width=3pt] (0.5,3)--(0.5,2) (1,3)--(1,2) (1.5,3)--(1.5,2); 
      \draw[-] (0.5,3)--(0.5,2) (1,3)--(1,2) (1.5,3)--(1.5,2); 
      \draw[color=white, line width=3pt](2,3)..controls(2.2,2.75) and (-0.25,2.75)  ..(0,2.5);
      \draw[-](2,3)..controls(2.2,2.75) and (-0.25,2.75)  ..(0,2.5);      
      
      \draw[-](0.5,2)--(2,1) ;
      \draw[color=white, line width=3pt] (1,2)--(0.5,1) (1.5,2)--(1,1) (2,2)--(1.5,1) ;    
      \draw[-] (1,2)--(0.5,1) (1.5,2)--(1,1) (2,2)--(1.5,1) ; 
      
      \draw[color=white, line width=3pt] (1.2,5.2)--(1.2,5.2) (0,0.)--(0,0.) ;
  \end{tikzpicture}  \\ $\tilde x_{3,5}$
    \end{center}
 \end{minipage}
 \begin{minipage}{0.2\hsize}
      \begin{center}
  \begin{tikzpicture}
        \draw[-](0,2.5)..controls(-0.25,2.25) and (2.2,2.25)  ..(2,2);
      \draw[color=white, line width=3pt] (0.5,3)--(0.5,2) (1,3)--(1,2) (1.5,3)--(1.5,2); 
      \draw[-] (0.5,3)--(0.5,2) (1,3)--(1,2) (1.5,3)--(1.5,2); 
      \draw[color=white, line width=3pt](2,3)..controls(2.2,2.75) and (-0.25,2.75)  ..(0,2.5);
      \draw[-](2,3)..controls(2.2,2.75) and (-0.25,2.75)  ..(0,2.5);

      \draw[color=white, line width=3pt] (1.2,5.)--(1.2,5.) (0,0)--(0,0) ;

  \end{tikzpicture}  \\ $\tilde x_{4,5}$
       \end{center}
 \end{minipage}
 \begin{minipage}{0.2\hsize}
      \begin{center}
  \begin{tikzpicture}
      \draw[-] (0,4)--(1.5,5) ; 
      \draw[color=white, line width=3pt] (0,5)--(0.5,4) (0.5,5)--(1,4) (1,5)--(1.5,4) ;  
      \draw[-] (0,5)--(0.5,4) (0.5,5)--(1,4) (1,5)--(1.5,4) ; 

      \draw[-] (0,4)--(1,3) (0.5,4)--(1.5,3) ;     
      \draw[color=white, line width=3pt](0,3)--(1,4) (0.5,3)--(1.5,4) ;    
      \draw[-] (0,3)--(1,4) (0.5,3)--(1.5,4) ; 
      
      \draw[-](0,3)..controls(-0.25,2.75) and (2.2,2.75)  ..(2,2.5);
      \draw[color=white, line width=3pt] (0.5,3)--(0.5,2) (1,3)--(1,2) (1.5,3)--(1.5,2); 
      \draw[-] (0.5,3)--(0.5,2) (1,3)--(1,2) (1.5,3)--(1.5,2); 
      \draw[color=white, line width=3pt](2,2.5)..controls(2.2,2.25) and (-0.25,2.25)  ..(0,2);
      \draw[-](2,2.5)..controls(2.2,2.25) and (-0.25,2.25)  ..(0,2);

      \draw[-] (1,2)--(0,1) (1.5,2)--(0.5,1);     
      \draw[color=white, line width=3pt](0,2)--(1,1) (0.5,2)--(1.5,1) ; 
      \draw[-] (0,2)--(1,1) (0.5,2)--(1.5,1) ;     
      
      \draw[-](0,1)--(1.5,0) ;     
      \draw[color=white, line width=3pt] (0.5,1)--(0,0) (1,1)--(0.5,0) (1.5,1)--(1,0) ;     
      \draw[-] (0.5,1)--(0,0) (1,1)--(0.5,0) (1.5,1)--(1,0) ;    
  \end{tikzpicture}  \\ $\tilde x_{2,5}$
     \end{center}
 \end{minipage}

\end{tabular}
  \caption{Proof of Lemma \ref{lemma: elements and equalities in Bn}}
   \label{picture for lemma1.1}
\end{figure}
\hfill\qed\medskip 

\subsection{Computations of images of $\tilde x_{i5}$ under $\mathrm{comp}_{(\mu,\Phi)}^{((\bullet\bullet)\bullet)\bullet}$}\label{comp:of:elements:20032018} 

\subsubsection{Relations in $\mathbf{PaB}$}
 
\begin{lem}\label{lemma:dnb:2}
1) The automorphism $(\tilde x_{1,5})_{((\bullet\bullet)\bullet)\bullet}$ of the object $((\bullet\bullet)\bullet)\bullet$ of the category $\mathbf{PaB}$ is equal to the composition
$$
((\bullet\bullet)\bullet)\bullet\stackrel{a_{\bullet\bullet,\bullet,\bullet}}{\to}
(\bullet\bullet)(\bullet\bullet)\stackrel{a_{\bullet,\bullet,\bullet\bullet}}{\to}
\bullet(\bullet(\bullet\bullet))\stackrel{(\tilde x_{15})_{\bullet(\bullet(\bullet\bullet))}}{\to}
\bullet(\bullet(\bullet\bullet))\stackrel{a_{\bullet,\bullet,\bullet\bullet}^{-1}}{\to}
(\bullet\bullet)(\bullet\bullet)\stackrel{a_{\bullet\bullet,\bullet,\bullet}^{-1}}{\to}
((\bullet\bullet)\bullet)\bullet
$$
of isomorphisms of this category. 

2) The automorphism $(\tilde x_{2,5})_{((\bullet\bullet)\bullet)\bullet}$ of the object $((\bullet\bullet)\bullet)\bullet$ of the category $\mathbf{PaB}$ is equal to the composition
\begin{align*}
& ((\bullet\bullet)\bullet)\bullet\stackrel{a_{\bullet\bullet,\bullet,\bullet}}{\to}
(\bullet\bullet)(\bullet\bullet)\stackrel{\sigma^{-1}_{\bullet\bullet,\bullet\bullet}}{\to}
(\bullet\bullet)(\bullet\bullet)\stackrel{a^{-1}_{\bullet\bullet,\bullet,\bullet}}{\to}
((\bullet\bullet)\bullet)\bullet\stackrel{(\tilde x_{45})_{((\bullet\bullet)\bullet)\bullet}}{\to}
\\ & 
((\bullet\bullet)\bullet)\bullet\stackrel{a_{\bullet\bullet,\bullet,\bullet}}{\to}
(\bullet\bullet)(\bullet\bullet)\stackrel{\sigma_{\bullet\bullet,\bullet\bullet}}{\to}
(\bullet\bullet)(\bullet\bullet)\stackrel{a_{\bullet\bullet,\bullet,\bullet}^{-1}}{\to}
((\bullet\bullet)\bullet)\bullet
\end{align*}
of isomorphisms of this category. 

3) The automorphism $(\tilde x_{3,5})_{((\bullet\bullet)\bullet)\bullet}$ of the object $((\bullet\bullet)\bullet)\bullet$ of the category $\mathbf{PaB}$ is equal to the composition
$$
((\bullet\bullet)\bullet)\bullet\stackrel{\sigma^{-1}_{\bullet,(\bullet\bullet)\bullet}}{\to}
\bullet((\bullet\bullet)\bullet)\stackrel{a^{-1}_{\bullet,\bullet\bullet,\bullet}}{\to}
(\bullet(\bullet\bullet))\bullet\stackrel{(\tilde x_{45})_{(\bullet(\bullet\bullet))\bullet}}{\to}
(\bullet(\bullet\bullet))\bullet\stackrel{a_{\bullet,\bullet\bullet,\bullet}}{\to}
\bullet((\bullet\bullet)\bullet)\stackrel{\sigma_{\bullet,(\bullet\bullet)\bullet}}{\to}
((\bullet\bullet)\bullet)\bullet
$$
of isomorphisms of this category. 

4) The automorphism $(\tilde x_{4,5})_{((\bullet\bullet)\bullet)\bullet}$ of the object $((\bullet\bullet)\bullet)\bullet$ of the category $\mathbf{PaB}$ is equal to the composition
$$
((\bullet\bullet)\bullet)\bullet\stackrel{\sigma_{(\bullet\bullet)\bullet,\bullet}}{\to}
\bullet((\bullet\bullet)\bullet)\stackrel{\sigma_{\bullet,(\bullet\bullet)\bullet}}{\to}
((\bullet\bullet)\bullet)\bullet
$$
of isomorphisms of this category. 
\end{lem}

\proof Let $\mathbf B$ be the groupoid with set of objets equal to $\mathbb Z_{\geq 0}$ and sets of morphisms given by 
$\mathbf B(n):=B_n$ if $n\geq0$ (with the convention $B_0=1$) and $\mathbf B(n,m):=\emptyset$ if $m\neq n$. 

There is a functor
$\mathrm{forget}:\mathbf{PaB}\to\mathbf{B}$ given by $\mathrm{forget}(P):=|P|$ at the level of objects and defined at the level of 
morphisms by the condition that the map $\mathrm{forget}:\mathbf{PaB}(P,Q)\to \mathbf{B}(|P|,|Q|)$ is given by the composition 
$\mathbf{PaB}(P,Q)\simeq B_{|P|}\simeq\mathbf{B}(|P|,|Q|)$ when $|Q|=|P|$, and is the only self-map of the empty set if $|Q|\neq|P|$. 

Then 
\begin{equation}\label{to:be:used}
\forall P,Q\in\mathrm{Ob}(\mathbf{PaB}), \quad \mathbf{PaB}(P,Q)\to\mathbf B(|P|,|Q|) \text{ is an isomorphism.} 
\end{equation}

Each of the items of the lemma states the equality of two elements of $\mathbf{PaB}((\bullet(\bullet\bullet))\bullet)$. In view 
of \eqref{to:be:used}, this equality is a consequence of the equality of their images in $\mathbf B(4)=B_4$. The latter equality is itself a consequence of Lemma \ref{lemma: elements and equalities in Bn}: \eqref{equality:x25:bis} (resp. 
\eqref{equality:x35}, \eqref{equality:x45}) implies the equality of elements of $B_4$ relevant to 2) (resp. 3), 4)).  
\hfill\qed\medskip 

\begin{rem}
The equalities from Lemma \ref{lemma:dnb:2} can be translated in the graphical language of \cite{BN} as
the statement that for $i\in[\![1,4]\!]$, the isomorphism $(\tilde x_{i5})_{((\bullet\bullet)\bullet)\bullet}$ of the object 
$((\bullet\bullet)\bullet)\bullet$ of the category $\mathbf{PaB}$ is equal to the composition of morphisms 
indicated in Figure \ref{picture for lemma3.3}. 
\begin{figure}[h]
\begin{tabular}{c}
  \begin{minipage}{0.25\hsize}
      \begin{center} 
  \begin{tikzpicture}

      \foreach \x in {0,0.5,1, 1.5} \foreach \y in {0,1,2,3,4,5}
      \draw[-,thin] (\x,5)--(\x,3)  (\x,2)--(\x,0);  

      \draw[-](0,3)..controls(-0.25,2.75) and (2.2,2.75)  ..(2,2.5);
      \draw[color=white, line width=3pt] (0.5,3)--(0.5,2) (1,3)--(1,2) (1.5,3)--(1.5,2); 
      \draw[-] (0.5,3)--(0.5,2) (1,3)--(1,2) (1.5,3)--(1.5,2); 
      \draw[color=white, line width=3pt](2,2.5)..controls(2.2,2.25) and (-0.25,2.25)  ..(0,2);
      \draw[-](2,2.5)..controls(2.2,2.25) and (-0.25,2.25)  ..(0,2); 
      
      \foreach \x in {0,0.5,1, 1.5} \foreach \y in {0,1,2,3,4,5}
      \draw[fill=black](\x,\y) circle(2pt);

\draw(-0.2,5) node{$(($};
\draw(0.7,5)node{$)$};
\draw(1.2,5)node{$)$};

\draw(-0.2,4) node{$($};
\draw(0.6,4)node{$)$};
\draw(0.9,4)node{$($};
\draw(1.7,4)node{$)$};

\draw(0.4,3) node{$($};
\draw(0.9,3)node{$($};
\draw(1.7,3)node{$))$};

\draw(0.4,2) node{$($};
\draw(0.9,2)node{$($};
\draw(1.7,2)node{$))$};

\draw(-0.2,1) node{$($};
\draw(0.6,1)node{$)$};
\draw(0.9,1)node{$($};
\draw(1.7,1)node{$)$};

\draw(-0.2,0) node{$(($};
\draw(0.7,0)node{$)$};
\draw(1.2,0)node{$)$};

 \draw[color=white, line width=3pt] (0.5,-1)--(0.5,-1) (0.5,6)--(0.5,6);
      
 \end{tikzpicture}    \\ $(\tilde x_{1,5})_{((\bullet\bullet)\bullet)\bullet}$
     \end{center}
 \end{minipage}
 \begin{minipage}{0.25\hsize}
      \begin{center}
  \begin{tikzpicture}
      \foreach \x in {0.5,1, 1.5, 2 } 
      \draw[-] (\x,6)--(\x,5) ;     
      
      \draw[-] (0.5,4)--(1.5,5) (1,4)--(2,5); 
      \draw[color=white, line width=3pt] (0.5,5)--(1.5,4) (1,5)--(2,4) ; 
      \draw[-] (0.5,5)--(1.5,4) (1,5)--(2,4) ; 
 
      \foreach \x in {0.5,1, 1.5, 2 } 
      \draw[-] (\x,4)--(\x,3) ;
                     
      \draw[-](0,2.5)..controls(-0.25,2.25) and (2.2,2.25)  ..(2,2);
      \draw[color=white, line width=3pt] (0.5,3)--(0.5,2) (1,3)--(1,2) (1.5,3)--(1.5,2); 
      \draw[-] (0.5,3)--(0.5,2) (1,3)--(1,2) (1.5,3)--(1.5,2); 
      \draw[color=white, line width=3pt](2,3)..controls(2.2,2.75) and (-0.25,2.75)  ..(0,2.5);
      \draw[-](2,3)..controls(2.2,2.75) and (-0.25,2.75)  ..(0,2.5);      

      \foreach \x in {0.5,1, 1.5, 2 } 
      \draw[-] (\x,2)--(\x,1) ; 
      
      \draw[-] (0.5,1)--(1.5,0) (1,1)--(2,0); 
      \draw[color=white, line width=3pt] (0.5,0)--(1.5,1) (1,0)--(2,1) ; 
      \draw[-](0.5,0)--(1.5,1) (1,0)--(2,1) ; 
 
      \foreach \x in {0.5,1, 1.5, 2 } 
      \draw[-] (\x,0)--(\x,-1) ; 
                  
      \foreach \x in {0.5,1, 1.5, 2} \foreach \y in {-1,0,1,2,3,4,5,6}
      \draw[fill=black](\x,\y) circle(2pt);

\draw(0.3,6) node{$(($};
\draw(1.2,6)node{$)$};
\draw(1.7,6)node{$)$};

\draw(0.3,5) node{$($};
\draw(1.1,5)node{$)$};
\draw(1.4,5)node{$($};
\draw(2.2,5)node{$)$};

\draw(0.3,4) node{$($};
\draw(1.1,4)node{$)$};
\draw(1.4,4)node{$($};
\draw(2.2,4)node{$)$};

\draw(0.3,3) node{$(($};
\draw(1.1,3)node{$)$};
\draw(1.7,3)node{$)$};

\draw(0.3,2) node{$(($};
\draw(1.1,2)node{$)$};
\draw(1.7,2)node{$)$};

\draw(0.3,1) node{$($};
\draw(1.1,1)node{$)$};
\draw(1.4,1)node{$($};
\draw(2.2,1)node{$)$};

\draw(0.3,0) node{$($};
\draw(1.1,0)node{$)$};
\draw(1.4,0)node{$($};
\draw(2.2,0)node{$)$};

\draw(0.3,-1) node{$(($};
\draw(1.2,-1)node{$)$};
\draw(1.7,-1)node{$)$}; 
  \end{tikzpicture}  \\ $(\tilde x_{2,5})_{((\bullet\bullet)\bullet)\bullet}$
     \end{center}
 \end{minipage}
 \begin{minipage}{0.25\hsize}
      \begin{center}
  \begin{tikzpicture}
        
      \draw[-] (0.5,4)--(2,5); 
      \draw[color=white, line width=3pt](0.5,5)--(1,4) (1,5)--(1.5,4) (1.5,5)--(2,4);
      \draw[-] (0.5,5)--(1,4) (1,5)--(1.5,4) (1.5,5)--(2,4); 
      
      \foreach \x in {0.5,1, 1.5, 2 } 
      \draw[-] (\x,4)--(\x,3) ;   
            
      \draw[-](0,2.5)..controls(-0.25,2.25) and (2.2,2.25)  ..(2,2);
      \draw[color=white, line width=3pt] (0.5,3)--(0.5,2) (1,3)--(1,2) (1.5,3)--(1.5,2); 
      \draw[-] (0.5,3)--(0.5,2) (1,3)--(1,2) (1.5,3)--(1.5,2); 
      \draw[color=white, line width=3pt](2,3)..controls(2.2,2.75) and (-0.25,2.75)  ..(0,2.5);
      \draw[-](2,3)..controls(2.2,2.75) and (-0.25,2.75)  ..(0,2.5);      

       \foreach \x in {0.5,1, 1.5, 2 } 
      \draw[-] (\x,2)--(\x,1) ; 
      
        \draw[-] (0.5,1)--(2,0); 
      \draw[color=white, line width=3pt](0.5,0)--(1,1) (1,0)--(1.5,1) (1.5,0)--(2,1);
      \draw[-] (0.5,0)--(1,1) (1,0)--(1.5,1) (1.5,0)--(2,1); 
                  
      \foreach \x in {0.5,1, 1.5, 2} \foreach \y in {0,1,2,3,4,5}
      \draw[fill=black](\x,\y) circle(2pt);

\draw(0.3,5) node{$(($};
\draw(1.2,5)node{$)$};
\draw(1.7,5)node{$)$};

\draw(0.8,4) node{$(($};
\draw(1.6,4)node{$)$};
\draw(2.2,4)node{$)$};

\draw(0.3,3) node{$($};
\draw(0.8,3)node{$($};
\draw(1.7,3)node{$))$};

\draw(0.3,2) node{$($};
\draw(0.8,2)node{$($};
\draw(1.7,2)node{$))$};

\draw(0.8,1) node{$(($};
\draw(1.6,1)node{$)$};
\draw(2.2,1)node{$)$};

\draw(0.3,0) node{$(($};
\draw(1.2,0)node{$)$};
\draw(1.7,0)node{$)$};

 \draw[color=white, line width=3pt] (0.5,-1.2)--(0.5,-1.2) (0.5,6.3)--(0.5,6.3);

        \end{tikzpicture}  \\ $(\tilde x_{3,5})_{((\bullet\bullet)\bullet)\bullet}$
    \end{center}
 \end{minipage}
 \begin{minipage}{0.25\hsize}
      \begin{center}
  \begin{tikzpicture}
        \draw[-](0,2.5)..controls(-0.25,2.25) and (2.2,2.25)  ..(2,2);
      \draw[color=white, line width=3pt] (0.5,3)--(0.5,2) (1,3)--(1,2) (1.5,3)--(1.5,2); 
      \draw[-] (0.5,3)--(0.5,2) (1,3)--(1,2) (1.5,3)--(1.5,2); 
      \draw[color=white, line width=3pt](2,3)..controls(2.2,2.75) and (-0.25,2.75)  ..(0,2.5);
      \draw[-](2,3)..controls(2.2,2.75) and (-0.25,2.75)  ..(0,2.5);
      \draw[color=white, line width=3pt] (1.2,5)--(1.2,5.01) (0,0.5)--(0,0.51) ;

    \foreach \x in {0.5,1, 1.5, 2} \foreach \y in {2,3}
      \draw[fill=black](\x,\y) circle(2pt);

\draw(0.2,3) node{$(($};
\draw(1.2,3)node{$)$};
\draw(1.7,3)node{$)$};

\draw(0.2,2) node{$(($};
\draw(1.2,2)node{$)$};
\draw(1.7,2)node{$)$};

 \draw[color=white, line width=3pt] (0.5,-1.25)--(0.5,-1.25) (0.5,6.25)--(0.5,6.25);

  \end{tikzpicture}  \\ $(\tilde x_{4,5})_{((\bullet\bullet)\bullet)\bullet}$
       \end{center}
 \end{minipage}
\end{tabular}
  \caption{$(\tilde x_{i,5})_{((\bullet\bullet)\bullet)\bullet}$}
            \label{picture for lemma3.3}
\end{figure}
\end{rem}

\subsubsection{Computation of images of morphisms in $\mathbf{PaB}$}

\begin{prop}\label{prop:comp:images:20190728}
Let $\mu\in\mathbf k^\times$ and $\Phi\in \mathsf M_\mu(\mathbf k)$. One has
\begin{equation}\label{image:partI:x15}
\mathrm{comp}_{(\mu,\Phi)}^{{((\bullet\bullet)\bullet)\bullet}}(\tilde x_{1,5})=
\mathrm{Ad}\Big(\Phi(e_{4,5},e_{12,5})^{-1}\Phi(e_{34,5},e_{1,5})^{-1}
\Big)(\mathrm{exp}(\mu e_{1,5})), 
\end{equation}
\begin{equation}\label{image:partI:x25}
\mathrm{comp}_{(\mu,\Phi)}^{{((\bullet\bullet)\bullet)\bullet}}(\tilde x_{2,5})
=\mathrm{Ad}\Big(
\Phi(e_{4,5},e_{12,5})^{-1}\mathrm{exp}({\mu\over 2}e_{34,5})\Phi(e_{34,5},e_{2,5})^{-1}
\Big)
(\mathrm{exp}(\mu e_{2,5})), 
\end{equation}
\begin{equation}\label{image:partI:x35}
\mathrm{comp}_{(\mu,\Phi)}^{((\bullet\bullet)\bullet)\bullet}(\tilde x_{3,5})
= \mathrm{Ad}\Big( \mathrm{exp}({\mu\over 2}e_{4,5})\Phi(e_{3,5},e_{4,5})\Big)( \mathrm{exp}(\mu e_{3,5})), 
\end{equation}
\begin{equation}\label{image:partI:x45}
\mathrm{comp}_{(\mu,\Phi)}^{((\bullet\bullet)\bullet)\bullet}(\tilde x_{4,5})
=\mathrm{exp}(\mu e_{4,5}). 
\end{equation}
\end{prop}

\proof For $i\in[\![1,4]\!]$, $\mathrm{comp}_{(\mu,\Phi)}^{((\bullet\bullet)\bullet)\bullet}(\tilde x_{i,5})$ is equal to $\mathrm{comp}_{(\mu,\Phi)}((\tilde x_{i,5})_{((\bullet\bullet)\bullet)\bullet})$.  One computes
\begin{align}\label{09072019toto}
& \nonumber
\mathrm{comp}_{(\mu,\Phi)}((\tilde x_{1,5})_{\bullet(\bullet(\bullet\bullet))})=
\mathrm{comp}_{(\mu,\Phi)}(\sigma_{\bullet(\bullet\bullet),\bullet})\circ
\mathrm{comp}_{(\mu,\Phi)}(\sigma_{\bullet,\bullet(\bullet\bullet)}) 
 \\ & 
=\bigl(\begin{smallmatrix} 1 & 2 &3&4 \\ 2&3&4&1\end{smallmatrix}\bigr) 
 \cdot \mathrm{exp}({\mu\over 2}e_{45})\cdot
 \bigl(\begin{smallmatrix} 1 & 2 &3&4 \\ 4&1&2&3\end{smallmatrix}\bigr) 
 \cdot \mathrm{exp}({\mu\over 2}e_{15}) =\mathrm{exp}(\mu e_{15}),   
\end{align}
where the second equality follows from the first equality in \eqref{image:sigma} for 
$(P,Q)=(\bullet(\bullet\bullet),\bullet)$ and from the second equality in \eqref{image:sigma} for 
$(P,Q)=(\bullet,\bullet(\bullet\bullet))$, and the last equality is a computation in $(U\mathfrak p_5)^\wedge\rtimes S_4$. 

Then 
\begin{align*}
& 
\mathrm{comp}_{(\mu,\Phi)}((\tilde x_{1,5})_{((\bullet\bullet)\bullet)\bullet})=
\mathrm{comp}_{(\mu,\Phi)}(a_{\bullet\bullet,\bullet,\bullet})^{-1}\circ 
\mathrm{comp}_{(\mu,\Phi)}(a_{\bullet,\bullet,\bullet\bullet})^{-1}
\\ & \circ 
\mathrm{comp}_{(\mu,\Phi)}((\tilde x_{1,5})_{\bullet(\bullet(\bullet\bullet))})
\circ\mathrm{comp}_{(\mu,\Phi)}(a_{\bullet,\bullet,\bullet\bullet})\circ
\mathrm{comp}_{(\mu,\Phi)}(a_{\bullet\bullet,\bullet,\bullet})
\\&=\Phi(e_{4,5},e_{12,5})^{-1}\cdot\Phi(e_{34,5},e_{1,5})^{-1}\cdot \mathrm{exp}(\mu e_{15})\cdot\Phi(e_{34,5},e_{1,5})\cdot\Phi(e_{4,5},e_{12,5}),  
\end{align*}
where the first equality follows from Lemma \ref{lemma:dnb:2}, 1), and the second equality follows from \eqref{image:a} for 
$(P,Q,R)=(\bullet\bullet,\bullet,\bullet)$ and $(\bullet,\bullet,\bullet\bullet)$ and from \eqref{09072019toto}. 
This proves \eqref{image:partI:x15}. 

One computes
\begin{align*}
& \nonumber
\mathrm{comp}_{(\mu,\Phi)}((\tilde x_{4,5})_{((\bullet\bullet)\bullet)\bullet})=
\mathrm{comp}_{(\mu,\Phi)}(\sigma_{\bullet,(\bullet\bullet)\bullet})\circ 
\mathrm{comp}_{(\mu,\Phi)}(\sigma_{(\bullet\bullet)\bullet,\bullet})
 \\ & 
=\bigl(\begin{smallmatrix} 1 & 2 &3&4 \\ 4&1&2&3\end{smallmatrix}\bigr) 
 \cdot \mathrm{exp}({\mu\over 2}e_{15}) \cdot
\bigl(\begin{smallmatrix} 1 & 2 &3&4 \\ 2&3&4&1\end{smallmatrix}\bigr) 
 \cdot \mathrm{exp}({\mu\over 2}e_{4,5})=\mathrm{exp}(\mu e_{4,5}),     
\end{align*}
where the first equality follows from Lemma \ref{lemma:dnb:2}, 2), and the second 
equality follows from the second equality in \eqref{image:sigma} for $(P,Q)=(\bullet,(\bullet\bullet)\bullet)$
and from the first equality in \eqref{image:sigma} for $(P,Q)=(\bullet,(\bullet\bullet)\bullet)$. This proves \eqref{image:partI:x45}. 

Then 
\begin{align*}
& 
\mathrm{comp}_{(\mu,\Phi)}((\tilde x_{2,5})_{((\bullet\bullet)\bullet)\bullet})=
\mathrm{comp}_{(\mu,\Phi)}(a_{\bullet\bullet,\bullet,\bullet})^{-1}\circ
\mathrm{comp}_{(\mu,\Phi)}(\sigma_{\bullet\bullet,\bullet\bullet})\circ
\mathrm{comp}_{(\mu,\Phi)}(a_{\bullet\bullet,\bullet,\bullet})
\\ & 
\circ\mathrm{comp}_{(\mu,\Phi)}((\tilde x_{4,5})_{((\bullet\bullet)\bullet)\bullet})\circ
\mathrm{comp}_{(\mu,\Phi)}(a_{\bullet\bullet,\bullet,\bullet})^{-1}\circ
\mathrm{comp}_{(\mu,\Phi)}(\sigma_{\bullet\bullet,\bullet\bullet})^{-1}\circ
\mathrm{comp}_{(\mu,\Phi)}(a_{\bullet\bullet,\bullet,\bullet})
\\ & 
=\Phi(e_{4,5},e_{12,5})^{-1}
\cdot\mathrm{exp}({\mu\over 2}e_{34,5}+\mu e_{3,4})\cdot 
\bigl(\begin{smallmatrix} 1 & 2 &3&4 \\ 3&4&1&2\end{smallmatrix}\bigr) 
\cdot \Phi(e_{4,5},e_{12,5})
\\ & 
\cdot e^{\mu e_{4,5}}\cdot\Phi(e_{4,5},e_{12,5})^{-1} 
\cdot \bigl(\begin{smallmatrix} 1 & 2 &3&4 \\ 3&4&1&2\end{smallmatrix}\bigr) 
 \cdot \mathrm{exp}(-{\mu\over 2}e_{34,5}-\mu e_{3,4})\cdot\Phi(e_{4,5},e_{12,5})
 \\ & = \Phi(e_{4,5},e_{12,5})^{-1}\cdot\mathrm{exp}({\mu\over 2}e_{34,5})\cdot  \mathrm{exp}(\mu e_{3,4})\cdot \Phi(e_{2,5},e_{34,5})\cdot
 e^{\mu e_{2,5}}\cdot\Phi(e_{2,5},e_{34,5})^{-1} \\ & 
  \cdot \mathrm{exp}(-\mu e_{3,4})\cdot \mathrm{exp}(-{\mu\over 2}e_{34,5})\cdot\Phi(e_{4,5},e_{12,5})
  \\ & =\Phi(e_{4,5},e_{12,5})^{-1}\cdot\mathrm{exp}({\mu\over 2}e_{34,5})\cdot\Phi(e_{2,5},e_{34,5}) \cdot 
 e^{\mu e_{2,5}}\cdot\Phi(e_{2,5},e_{34,5})^{-1}
 \cdot \mathrm{exp}(-{\mu\over 2}e_{34,5})\cdot\Phi(e_{4,5},e_{12,5}) 
 \\ & =\Phi(e_{4,5},e_{12,5})^{-1}\cdot\mathrm{exp}({\mu\over 2}e_{34,5})\cdot\Phi(e_{34,5},e_{2,5})^{-1} \cdot 
 e^{\mu e_{2,5}}\cdot\Phi(e_{34,5},e_{2,5})
 \cdot \mathrm{exp}(-{\mu\over 2}e_{34,5})\cdot\Phi(e_{4,5},e_{12,5}), 
\end{align*}
where the first equality follows from Lemma \ref{lemma:dnb:2}, 3), the second equality follows from 
\eqref{image:a} for $(P,Q,R)=(\bullet\bullet,\bullet,\bullet)$, from the second equality in \eqref{image:sigma} 
for $(P,Q)=(\bullet\bullet,\bullet\bullet)$ and from \eqref{image:partI:x45}, the third equality follows from computation in  
$(U\mathfrak p_5)^\wedge\rtimes S_4$, in particular the fact that $e_{3,4}$ commutes with $e_{34,5}$, the fourth equality follows from the
fact the $e_{3,4}$ commutes with $e_{2,5}$ and $e_{34,5}$, and the last equality follows from the 2-cycle identity 
\eqref{duality:rel}. This proves \eqref{image:partI:x25}. 

Finally
\begin{align*}
& 
\mathrm{comp}_{(\mu,\Phi)}((\tilde x_{3,5})_{((\bullet\bullet)\bullet)\bullet})=
\mathrm{comp}_{(\mu,\Phi)}(\sigma_{\bullet,(\bullet\bullet)\bullet})\circ
\mathrm{comp}_{(\mu,\Phi)}(a_{\bullet,\bullet\bullet,\bullet})
\\ & \circ
\mathrm{comp}_{(\mu,\Phi)}((\tilde x_{4,5})_{(\bullet(\bullet\bullet))\bullet})
\circ\mathrm{comp}_{(\mu,\Phi)}(a_{\bullet,\bullet\bullet,\bullet})^{-1}\circ
\mathrm{comp}_{(\mu,\Phi)}(\sigma_{\bullet,(\bullet\bullet)\bullet})^{-1}
\\ & 
=  \bigl(\begin{smallmatrix} 1 & 2 &3&4 \\ 4&1&2&3\end{smallmatrix}\bigr) 
 \cdot \mathrm{exp}({\mu\over 2}e_{1,5})  \cdot\Phi(e_{4,5},e_{1,5})\cdot e^{\mu e_{4,5}}\cdot \Phi(e_{4,5},e_{1,5})^{-1}
  \cdot \mathrm{exp}(-{\mu\over 2}e_{1,5})\cdot
\bigl(\begin{smallmatrix} 1 & 2 &3&4 \\ 2&3&4&1\end{smallmatrix}\bigr) 
\\ & =  \mathrm{exp}({\mu\over 2}e_{4,5})  \cdot\Phi(e_{3,5},e_{4,5})\cdot e^{\mu e_{3,5}}\cdot \Phi(e_{3,5},e_{4,5})^{-1}
  \cdot \mathrm{exp}(-{\mu\over 2}e_{4,5}), 
\end{align*}
where the first equality follows from Lemma \ref{lemma:dnb:2}, 4), the second equality follows from \eqref{image:a} for 
$(P,Q,R)=(\bullet,\bullet\bullet,\bullet)$, from the second equality in \eqref{image:sigma} for $(P,Q)=(\bullet,(\bullet\bullet)\bullet)$
and from \eqref{image:partI:x45}, and the third equality follows from computation in  
$(U\mathfrak p_5)^\wedge\rtimes S_4$. This proves \eqref{image:partI:x35}. 
\hfill\qed\medskip

\subsection{Computations of images of $\tilde x_{i5}$ under $\mathrm{comp}_{(\mu,\Phi)}^{(\bullet(\bullet\bullet))\bullet}$
} \label{sect:9:6:12nov2019}

\subsubsection{Relations in $\mathbf{PaB}$}

\begin{lem}\label{lemma:1:1:dnb}
1) The automorphism $(\tilde x_{1,5})_{(\bullet(\bullet\bullet))\bullet}$ of the object $(\bullet(\bullet\bullet))\bullet$ of the category $\mathbf{PaB}$ is equal to the composition
$$
(\bullet(\bullet\bullet))\bullet\stackrel{a_{\bullet,\bullet\bullet,\bullet}}{\to}
\bullet((\bullet\bullet)\bullet)\stackrel{\sigma_{\bullet,(\bullet\bullet)\bullet}}{\to}
((\bullet\bullet)\bullet)\bullet\stackrel{\sigma_{(\bullet\bullet)\bullet,\bullet}}{\to}
\bullet((\bullet\bullet)\bullet)\stackrel{a_{\bullet,\bullet\bullet,\bullet}^{-1}}{\to}
(\bullet(\bullet\bullet))\bullet
$$
of isomorphisms of this category. 

2) The automorphism $(\tilde x_{2,5})_{(\bullet(\bullet\bullet))\bullet}$ of the object $(\bullet(\bullet\bullet))\bullet$ of the category $\mathbf{PaB}$ is equal to the composition
\begin{align*}
&(\bullet(\bullet\bullet))\bullet\stackrel{\sigma_{\bullet,\bullet(\bullet\bullet)}^{-1}}{\to}
\bullet(\bullet(\bullet\bullet))\stackrel{a^{-1}_{\bullet,\bullet,\bullet\bullet}}{\to}
(\bullet\bullet)(\bullet\bullet)\stackrel{\sigma_{\bullet\bullet,\bullet\bullet}}{\to}
(\bullet\bullet)(\bullet\bullet)\stackrel{a_{\bullet,\bullet,\bullet\bullet}}{\to}
\bullet(\bullet(\bullet\bullet))\stackrel{(\tilde x_{1,5})_{\bullet(\bullet(\bullet\bullet))}}{\to}
\bullet(\bullet(\bullet\bullet))
\\ & 
\stackrel{a_{\bullet,\bullet,\bullet\bullet}^{-1}}{\to}
(\bullet\bullet)(\bullet\bullet)\stackrel{\sigma_{\bullet\bullet,\bullet\bullet}^{-1}}{\to}
(\bullet\bullet)(\bullet\bullet)\stackrel{a_{\bullet,\bullet,\bullet\bullet}}{\to}
\bullet(\bullet(\bullet\bullet))\stackrel{\sigma_{\bullet,\bullet(\bullet\bullet)}}{\to}
(\bullet(\bullet\bullet))\bullet
\end{align*}
of isomorphisms of this category. 

3) The automorphism $(\tilde x_{3,5})_{(\bullet(\bullet\bullet))\bullet}$ of the object 
$(\bullet(\bullet\bullet))\bullet$ of the category $\mathbf{PaB}$ is equal to the composition
\begin{align*}
&(\bullet(\bullet\bullet))\bullet
\stackrel{\sigma_{\bullet,\bullet(\bullet\bullet)}^{-1}}{\to}\bullet(\bullet(\bullet\bullet))
\stackrel{a^{-1}_{\bullet,\bullet,\bullet\bullet}}{\to}(\bullet\bullet)(\bullet\bullet)
\stackrel{a^{-1}_{\bullet\bullet,\bullet,\bullet}}{\to}((\bullet\bullet)\bullet)\bullet
\stackrel{(\tilde x_{4,5})_{((\bullet\bullet)\bullet)\bullet}}{\to}((\bullet\bullet)\bullet)\bullet
\\ & \stackrel{a_{\bullet\bullet,\bullet,\bullet}}{\to}(\bullet\bullet)(\bullet\bullet)
\stackrel{a_{\bullet,\bullet,\bullet\bullet}}{\to}\bullet(\bullet(\bullet\bullet))
\stackrel{\sigma_{\bullet,\bullet(\bullet\bullet)}}{\to}
(\bullet(\bullet\bullet))\bullet
\end{align*}
of isomorphisms of this category. 

4) The automorphism $(\tilde x_{4,5})_{(\bullet(\bullet\bullet))\bullet}$ of the object $(\bullet(\bullet\bullet))\bullet$ of the category $\mathbf{PaB}$ is equal to the composition
$$
(\bullet(\bullet\bullet))\bullet\stackrel{\sigma_{\bullet(\bullet\bullet),\bullet}}{\to}
\bullet(\bullet(\bullet\bullet))\stackrel{\sigma_{\bullet,\bullet(\bullet\bullet)}}{\to}
(\bullet(\bullet\bullet))\bullet
$$
of isomorphisms of this category. 
\end{lem}

\proof Each of the items of the lemma states the equality of two elements of $\mathbf{PaB}((\bullet(\bullet\bullet))\bullet)$. In view 
of \eqref{to:be:used}, this equality is a consequence of the equality of their images in $\mathbf B(4)=B_4$ under the 
morphism of groupoids $\mathrm{forget}:\mathbf{PaB}\to\mathbf{B}$ from the proof of Lemma \ref{lemma:dnb:2}. 
The latter equality is itself a consequence of Lemma \ref{lemma: elements and equalities in Bn}: \eqref{equality:x15} 
(resp., \eqref{equality:x25}, \eqref{equality:x35}, \eqref{equality:x45}) implies the equality of elements of $B_4$ relevant to 1) 
(resp., 2), 3), 4)). \hfill\qed\medskip 

\begin{rem}
The equalities from Lemma \ref{lemma:1:1:dnb} can be translated in the graphical language of \cite{BN} as
the statement that for $i\in[\![1,4]\!]$, the isomorphism $(\tilde x_{i5})_{(\bullet(\bullet\bullet))\bullet}$ of the object 
$(\bullet(\bullet\bullet))\bullet$ of the category $\mathbf{PaB}$ is equal to the composition of morphisms 
indicated in Figure \ref{picture for lemma2.1}.
\begin{figure}[h]
\begin{tabular}{c}
  \begin{minipage}{0.25\hsize}
      \begin{center} 
  \begin{tikzpicture}

      \foreach \x in {0,0.5,1, 1.5} \foreach \y in {0,1,2,3,4,5}
      \draw[-] (\x,5)--(\x,4)  (\x,2)--(\x,1);  

      \draw[-] (0,4)--(1.5,3);
      \draw[color=white, line width=3pt] (0,3)--(0.5,4) (0.5,3)--(1,4) (1,3)--(1.5,4) ; 
      \draw[-] (0,3)--(0.5,4) (0.5,3)--(1,4) (1,3)--(1.5,4) ; 
      \draw[-] (0,3)--(0.5,2) (0.5,3)--(1,2) (1,3)--(1.5,2) ;
      \draw[color=white, line width=3pt] (0,2)--(1.5,3) ; 
      \draw[-] (0,2)--(1.5,3) ; 
      
      \foreach \x in {0,0.5,1, 1.5} \foreach \y in {1,2,3,4,5}
      \draw[fill=black](\x,\y) circle(2pt);

\draw(-0.2,5) node{$($};
\draw(0.4,5)node{$($};
\draw(1.2,5)node{$))$};

\draw(0.4,4) node{$(($};
\draw(1.1,4)node{$)$};
\draw(1.7,4)node{$)$};

\draw(-0.2,3) node{$(($};
\draw(0.7,3)node{$)$};
\draw(1.2,3)node{$)$};

\draw(0.3,2) node{$(($};
\draw(1.1,2)node{$)$};
\draw(1.7,2)node{$)$};

\draw(-0.2,1) node{$($};
\draw(0.4,1)node{$($};
\draw(1.2,1)node{$))$};
 \draw[color=white, line width=3pt] (0.5,-1.85)--(0.5,-1.85) (0.5,8)--(0.5,8);
      
 \end{tikzpicture}    \\ $(\tilde x_{1,5})_{(\bullet(\bullet\bullet))\bullet}$
     \end{center}
 \end{minipage}
 \begin{minipage}{0.25\hsize}
      \begin{center}
  \begin{tikzpicture}

      \draw[-] (0.5,6)--(2,7); 
      \draw[color=white, line width=3pt](0.5,7)--(1,6) (1,7)--(1.5,6) (1.5,7)--(2,6);
      \draw[-] (0.5,7)--(1,6) (1,7)--(1.5,6) (1.5,7)--(2,6); 
      
      \foreach \x in {0.5,1, 1.5, 2 } 
      \draw[-] (\x,6)--(\x,5) ;     
      
      \foreach \x in {0.5,1, 1.5, 2 } 
      \draw[-] (\x,6)--(\x,5) ;

     \draw[-] (0.5,5)--(1.5,4) (1,5)--(2,4); 
      \draw[color=white, line width=3pt] (0.5,4)--(1.5,5) (1,4)--(2,5) ; 
      \draw[-](0.5,4)--(1.5,5) (1,4)--(2,5) ; 

    \foreach \x in {0.5,1, 1.5, 2 } 
      \draw[-] (\x,4)--(\x,3) ;

      \draw[-](0.5,3)..controls(0.25,2.75) and (2.7,2.75)  ..(2.5,2.5);
      \draw[color=white, line width=3pt] (1,3)--(1,2) (1.5,3)--(1.5,2) (2,3)--(2,2); 
      \draw[-](1,3)--(1,2) (1.5,3)--(1.5,2) (2,3)--(2,2); 
      \draw[color=white, line width=3pt](2.5,2.5)..controls(2.7,2.25) and (0.25,2.25)  ..(0.5,2);
      \draw[-](2.5,2.5)..controls(2.7,2.25) and (0.25,2.25)  ..(0.5,2); 
                          
      \foreach \x in {0.5,1, 1.5, 2 } 
      \draw[-] (\x,2)--(\x,1) ; 
      
      \draw[-] (0.5,0)--(1.5,1) (1,0)--(2,1); 
      \draw[color=white, line width=3pt] (0.5,1)--(1.5,0) (1,1)--(2,0) ; 
      \draw[-] (0.5,1)--(1.5,0) (1,1)--(2,0) ; 
      
      \foreach \x in {0.5,1, 1.5, 2 } 
      \draw[-] (\x,0)--(\x,-1) ; 

       \draw[-] (0.5,-1)--(2,-2); 
      \draw[color=white, line width=3pt](0.5,-2)--(1,-1) (1,-2)--(1.5,-1) (1.5,-2)--(2,-1);
      \draw[-] (0.5,-2)--(1,-1) (1,-2)--(1.5,-1) (1.5,-2)--(2,-1); 
                  
      \foreach \x in {0.5,1, 1.5, 2} \foreach \y in {-2,-1,0,1,2,3,4,5,6,7}
      \draw[fill=black](\x,\y) circle(2pt);

\draw(0.3,7) node{$($};
\draw(0.8,7)node{$($};
\draw(1.7,7)node{$))$};

\draw(0.9,6) node{$($};
\draw(1.4,6)node{$($};
\draw(2.2,6)node{$))$};

\draw(0.3,5) node{$($};
\draw(1.1,5)node{$)$};
\draw(1.4,5)node{$($};
\draw(2.2,5)node{$)$};

\draw(0.3,4) node{$($};
\draw(1.1,4)node{$)$};
\draw(1.4,4)node{$($};
\draw(2.2,4)node{$)$};

\draw(0.9,3) node{$($};
\draw(1.4,3)node{$($};
\draw(2.2,3)node{$))$};

\draw(0.9,2) node{$($};
\draw(1.4,2)node{$($};
\draw(2.2,2)node{$))$};

\draw(0.3,1) node{$($};
\draw(1.1,1)node{$)$};
\draw(1.4,1)node{$($};
\draw(2.2,1)node{$)$};

\draw(0.3,0) node{$($};
\draw(1.1,0)node{$)$};
\draw(1.4,0)node{$($};
\draw(2.2,0)node{$)$};

\draw(0.9,-1) node{$($};
\draw(1.4,-1)node{$($};
\draw(2.2,-1)node{$))$};

\draw(0.3,-2) node{$($};
\draw(0.8,-2)node{$($};
\draw(1.7,-2)node{$))$}; 

\draw[color=white, line width=3pt] (0.5,-2.4)--(0.5,-2.4) (0.5,7.3)--(0.5,7.3);
  \end{tikzpicture}  \\ $(\tilde x_{2,5})_{(\bullet(\bullet\bullet))\bullet}$
     \end{center}
 \end{minipage}
 \begin{minipage}{0.25\hsize}
      \begin{center}
  \begin{tikzpicture}
        
      \draw[-] (0.5,5)--(2,6); 
      \draw[color=white, line width=3pt](0.5,6)--(1,5) (1,6)--(1.5,5) (1.5,6)--(2,5);
      \draw[-] (0.5,6)--(1,5) (1,6)--(1.5,5) (1.5,6)--(2,5); 
      
      \foreach \x in {0.5,1, 1.5, 2 } 
      \draw[-] (\x,5)--(\x,3) ;   
            
      \draw[-](0,2.5)..controls(-0.25,2.25) and (2.2,2.25)  ..(2,2);
      \draw[color=white, line width=3pt] (0.5,3)--(0.5,2) (1,3)--(1,2) (1.5,3)--(1.5,2); 
      \draw[-] (0.5,3)--(0.5,2) (1,3)--(1,2) (1.5,3)--(1.5,2); 
      \draw[color=white, line width=3pt](2,3)..controls(2.2,2.75) and (-0.25,2.75)  ..(0,2.5);
      \draw[-](2,3)..controls(2.2,2.75) and (-0.25,2.75)  ..(0,2.5);      

       \foreach \x in {0.5,1, 1.5, 2 } 
      \draw[-] (\x,2)--(\x,0) ; 
      
        \draw[-] (0.5,0)--(2,-1); 
      \draw[color=white, line width=3pt](0.5,-1)--(1,0) (1,-1)--(1.5,0) (1.5,-1)--(2,0);
      \draw[-] (0.5,-1)--(1,0) (1,-1)--(1.5,0) (1.5,-1)--(2,0); 
                  
      \foreach \x in {0.5,1, 1.5, 2} \foreach \y in {-1,0,1,2,3,4,5,6}
      \draw[fill=black](\x,\y) circle(2pt);

\draw(0.3,6) node{$($};
\draw(0.8,6)node{$($};
\draw(1.7,6)node{$))$};

\draw(0.9,5) node{$($};
\draw(1.4,5)node{$($};
\draw(2.2,5)node{$))$};

\draw(0.3,4) node{$($};
\draw(1.1,4)node{$)$};
\draw(1.4,4)node{$($};
\draw(2.2,4)node{$)$};

\draw(0.3,3) node{$(($};
\draw(1.2,3)node{$)$};
\draw(1.7,3)node{$)$};

\draw(0.3,2) node{$(($};
\draw(1.2,2)node{$)$};
\draw(1.7,2)node{$)$};

\draw(0.3,1) node{$($};
\draw(1.1,1)node{$)$};
\draw(1.4,1)node{$($};
\draw(2.2,1)node{$)$};

\draw(0.8,0) node{$($};
\draw(1.4,0)node{$($};
\draw(2.2,0)node{$))$};

\draw(0.3,-1) node{$($};
\draw(0.8,-1)node{$($};
\draw(1.7,-1)node{$))$}; 

 \draw[color=white, line width=3pt] (0.5,-2.35)--(0.5,-2.35) (0.5,7.30)--(0.5,7.30);

        \end{tikzpicture}  \\ $(\tilde x_{3,5})_{(\bullet(\bullet\bullet))\bullet}$
    \end{center}
 \end{minipage}
 \begin{minipage}{0.25\hsize}
      \begin{center}
  \begin{tikzpicture}      
      \draw[-] (0,4)--(0.5,3) (0.5,4)--(1,3) (1,4)--(1.5,3) ;
      \draw[color=white, line width=3pt] (0,3)--(1.5,4) ; 
      \draw[-] (0,3)--(1.5,4) ; 

      \draw[-] (0,3)--(1.5,2);
      \draw[color=white, line width=3pt] (0,2)--(0.5,3) (0.5,2)--(1,3) (1,2)--(1.5,3) ; 
      \draw[-] (0,2)--(0.5,3) (0.5,2)--(1,3) (1,2)--(1.5,3) ; 

      \foreach \x in {0, 0.5,1, 1.5} \foreach \y in {2,3,4}
      \draw[fill=black](\x,\y) circle(2pt);

\draw(-0.2,4) node{$($};
\draw(0.3,4)node{$($};
\draw(1.2,4)node{$))$};

\draw(0.3,3) node{$($};
\draw(0.8,3)node{$($};
\draw(1.7,3)node{$))$};

\draw(-0.2,2) node{$($};
\draw(0.3,2)node{$($};
\draw(1.2,2)node{$))$};

 \draw[color=white, line width=3pt] (0.5,-1.9)--(0.5,-1.9) (0.5,7.9)--(0.5,7.9);

  \end{tikzpicture}  \\ $(\tilde x_{4,5})_{(\bullet(\bullet\bullet))\bullet}$
       \end{center}
 \end{minipage}
\end{tabular}
  \caption{$(\tilde x_{i,5})_{(\bullet(\bullet\bullet))\bullet}$}
            \label{picture for lemma2.1}
\end{figure}
\end{rem}

\subsubsection{Computation of images of morphisms in $\mathbf{PaB}$}

\begin{prop}\label{prop:comp:images:brd:grp:elts}
Let $\mu\in\mathbf k^\times$ and $\Phi\in\mathsf M_\mu(\mathbf k)$. One has
\begin{equation}\label{image:partIII:x15}
\mathrm{comp}_{(\mu,\Phi)}^{(\bullet(\bullet\bullet))\bullet}(\tilde x_{1,5})=
\mathrm{Ad}\Big(\Phi(e_{1,5},e_{4,5})\Big)(\mathrm{exp}(\mu e_{15})),
\end{equation}
\begin{equation}\label{image:partIII:x25}
\mathrm{comp}_{(\mu,\Phi)}^{(\bullet(\bullet\bullet))\bullet}(\tilde x_{2,5})
=\mathrm{Ad}\Big( \mathrm{exp}({\mu\over 2}e_{4,5})
  \Phi(e_{23,5},e_{4,5}) \mathrm{exp}({\mu\over 2}e_{23,5}) \Phi(e_{2,5},e_{14,5})\Big)(\mathrm{exp}(\mu e_{2,5})),
\end{equation}
\begin{equation}\label{image:partIII:x35}
\mathrm{comp}_{(\mu,\Phi)}^{(\bullet(\bullet\bullet))\bullet}(\tilde x_{3,5})
=\mathrm{Ad}\Big(  \mathrm{exp}({\mu\over 2}e_{4,5})
 \cdot \Phi(e_{23,5},e_{4,5})\cdot \Phi(e_{3,5},e_{14,5})
\Big)(\mathrm{exp}(\mu e_{3,5})), 
\end{equation}
\begin{equation}\label{image:partIII:x45}
\mathrm{comp}_{(\mu,\Phi)}^{(\bullet(\bullet\bullet))\bullet}(\tilde x_{4,5})
=\mathrm{exp}(\mu e_{4,5}). 
\end{equation}
\end{prop}

\proof For $i\in[\![1,4]\!]$, $\mathrm{comp}_{(\mu,\Phi)}^{(\bullet(\bullet\bullet))\bullet}(\tilde x_{i,5})$ is equal to $\mathrm{comp}_{(\mu,\Phi)}((\tilde x_{i,5})_{(\bullet(\bullet\bullet))\bullet})$. One computes
\begin{align*}
&\mathrm{comp}_{(\mu,\Phi)}((\tilde x_{1,5})_{(\bullet(\bullet\bullet))\bullet})
=\mathrm{comp}_{(\mu,\Phi)}(a_{\bullet,\bullet\bullet,\bullet}^{-1}\circ
\sigma_{(\bullet\bullet)\bullet,\bullet}\circ
\sigma_{\bullet,(\bullet\bullet)\bullet}\circ
a_{\bullet,\bullet\bullet,\bullet})
\\ & 
=\mathrm{comp}_{(\mu,\Phi)}(a_{\bullet,\bullet\bullet,\bullet})^{-1}\circ
\mathrm{comp}_{(\mu,\Phi)}(\sigma_{(\bullet\bullet)\bullet,\bullet})\circ
\mathrm{comp}_{(\mu,\Phi)}(\sigma_{\bullet,(\bullet\bullet)\bullet})\circ\mathrm{comp}_{(\mu,\Phi)}(a_{\bullet,\bullet\bullet,\bullet})
\\ & = \Phi(e_{4,5},e_{1,5})^{-1}\cdot
\bigl(\begin{smallmatrix} 1 & 2 &3&4 \\ 2&3&4&1\end{smallmatrix}\bigr) 
 \cdot \mathrm{exp}({\mu\over 2}e_{4,5})
 \cdot
\bigl(\begin{smallmatrix} 1 & 2 &3&4 \\ 4&1&2&3\end{smallmatrix}\bigr) 
 \cdot \mathrm{exp}({\mu\over 2}e_{1,5})
 \cdot\Phi(e_{4,5},e_{1,5})
 \\ & =\Phi(e_{1,5},e_{4,5})\cdot \mathrm{exp}(\mu e_{15})\cdot\Phi(e_{1,5},e_{4,5})^{-1},   
\end{align*}
where the first equality follows from Lemma \ref{lemma:1:1:dnb}, 1), the third equality from the relation between composition and product, 
from \eqref{image:a} for $(P,Q,R)=(\bullet,\bullet\bullet,\bullet)$, from the first equality in \eqref{image:sigma} for 
$(P,Q)=((\bullet\bullet)\bullet,\bullet)$ and from the second equality in \eqref{image:sigma} for 
$(P,Q)=(\bullet,(\bullet\bullet)\bullet)$, and the last equality follows 
from the duality relation \eqref{duality:rel}. This proves \eqref{image:partIII:x15}. 

Similarly, 
\begin{align*}
&\mathrm{comp}_{(\mu,\Phi)}((\tilde x_{2,5})_{(\bullet(\bullet\bullet))\bullet})
\\ & 
=\mathrm{comp}_{(\mu,\Phi)}(\sigma_{\bullet,\bullet(\bullet\bullet)}\circ
a_{\bullet,\bullet,\bullet\bullet}\circ
\sigma_{\bullet\bullet,\bullet\bullet}^{-1}\circ
a_{\bullet,\bullet,\bullet\bullet}^{-1}\circ
(\tilde x_{1,5})_{\bullet(\bullet(\bullet\bullet))}\circ
a_{\bullet,\bullet,\bullet\bullet}\circ
\sigma_{\bullet\bullet,\bullet\bullet}\circ
a_{\bullet,\bullet\bullet,\bullet}^{-1}\circ
\sigma_{\bullet,\bullet(\bullet\bullet)}^{-1})
\\ & 
=\mathrm{comp}_{(\mu,\Phi)}(\sigma_{\bullet,\bullet(\bullet\bullet)})\circ
\mathrm{comp}_{(\mu,\Phi)}(a_{\bullet,\bullet,\bullet\bullet})\circ
\mathrm{comp}_{(\mu,\Phi)}(\sigma_{\bullet\bullet,\bullet\bullet})^{-1}\circ
\mathrm{comp}_{(\mu,\Phi)}(a_{\bullet,\bullet,\bullet\bullet})^{-1}
\\ & 
\scriptstyle{\circ\mathrm{comp}_{(\mu,\Phi)}(\tilde x_{1,5})_{\bullet(\bullet(\bullet\bullet))}
\circ\mathrm{comp}_{(\mu,\Phi)}(a_{\bullet,\bullet,\bullet\bullet})\circ
\mathrm{comp}_{(\mu,\Phi)}(\sigma_{\bullet\bullet,\bullet\bullet})\circ
\mathrm{comp}_{(\mu,\Phi)}(a_{\bullet,\bullet\bullet,\bullet})^{-1}
\mathrm{comp}_{(\mu,\Phi)}(\sigma_{\bullet,\bullet(\bullet\bullet)})^{-1}}
\\ & = 
\scriptstyle{\mathrm{Ad}\Big(
\bigl(\begin{smallmatrix} 1 & 2 &3&4 \\ 4&1&2&3\end{smallmatrix}\bigr) 
 \cdot \mathrm{exp}({\mu\over 2}e_{1,5})
 \cdot \Phi(e_{34,5},e_{1,5})\cdot \bigl(\begin{smallmatrix} 1 & 2 &3&4 \\ 3&4&1&2\end{smallmatrix}\bigr) 
 \cdot \mathrm{exp}(-{\mu\over 2}e_{34,5}-\mu e_{3,4})\cdot \Phi(e_{34,5},e_{1,5})^{-1}
 \Big)(\mathrm{exp}(\mu e_{1,5}))}
 \\ &= 
 \mathrm{Ad}\Big(
  \mathrm{exp}({\mu\over 2}e_{4,5})
 \cdot \Phi(e_{23,5},e_{4,5})
 \cdot \bigl(\begin{smallmatrix} 1 & 2 &3&4 \\ 2&3&4&1\end{smallmatrix}\bigr) 
 \cdot \mathrm{exp}(-{\mu\over 2}e_{34,5})\cdot \Phi(e_{34,5},e_{1,5})^{-1} \Big) 
(\mathrm{exp}(\mu e_{15}))
 \\ & 
 = \mathrm{Ad}\Big( \mathrm{exp}({\mu\over 2}e_{4,5})
 \cdot \Phi(e_{23,5},e_{4,5})
 \cdot \mathrm{exp}(-{\mu\over 2}e_{14,5})\cdot \Phi(e_{14,5},e_{2,5})^{-1}\Big) (\mathrm{exp}(\mu e_{2,5}))
 \\ & 
 = \mathrm{Ad}\Big( \mathrm{exp}({\mu\over 2}e_{4,5})
 \cdot \Phi(e_{23,5},e_{4,5})
 \cdot \mathrm{exp}({\mu\over 2}e_{23,5})\cdot \Phi(e_{2,5},e_{14,5})\Big)(\mathrm{exp}(\mu e_{2,5})) ;   
 \end{align*}
the first equality follows from Lemma \ref{lemma:1:1:dnb}, 2), the third equality follows from the relation between composition 
and product, from \eqref{image:a} for $(P,Q,R)=(\bullet,\bullet,\bullet\bullet)$ and from the second equality in \eqref{image:sigma} 
for $(P,Q)=(\bullet,\bullet(\bullet\bullet))$ and $(\bullet\bullet,\bullet\bullet)$. The two next equalities are computations in 
$U(\mathfrak p_5)^\wedge\rtimes S_4$, more precisely both are based on the effect of conjugation by an element of $S_4$ on 
an element of $\mathfrak p_5$  is this algebra, and the fourth equality also uses that $\mathbf kS_4\to 
U(\mathfrak p_5)^\wedge\rtimes S_4$ is an algebra morphism, 
and the fact that $e_{34}$ commutes with both $e_{34,5}$ and $e_{15}$. The sixth equality follows from 
$e_{1234,5}=0$ and from the duality identity \eqref{duality:rel}. This proves \eqref{image:partIII:x25}. 

One computes 
\begin{align*}
&\mathrm{comp}_{(\mu,\Phi)}((\tilde x_{3,5})_{(\bullet(\bullet\bullet))\bullet})
\\ & 
=\mathrm{comp}_{(\mu,\Phi)}(\sigma_{\bullet,\bullet(\bullet\bullet)}\circ
a_{\bullet,\bullet,\bullet\bullet}\circ
a_{\bullet\bullet,\bullet,\bullet}\circ
(\tilde x_{4,5})_{((\bullet\bullet)\bullet)\bullet}\circ
a_{\bullet\bullet,\bullet,\bullet}^{-1}\circ
a_{\bullet,\bullet,\bullet\bullet}^{-1}\circ
\sigma_{\bullet,\bullet(\bullet\bullet)}^{-1})
\\ & 
=\mathrm{comp}_{(\mu,\Phi)}(\sigma_{\bullet,\bullet(\bullet\bullet)})\circ
\mathrm{comp}_{(\mu,\Phi)}(a_{\bullet,\bullet,\bullet\bullet})\circ
\mathrm{comp}_{(\mu,\Phi)}(a_{\bullet\bullet,\bullet,\bullet})\circ
\mathrm{comp}_{(\mu,\Phi)}((\tilde x_{4,5})_{((\bullet\bullet)\bullet)\bullet})\circ\\ & 
\circ
\mathrm{comp}_{(\mu,\Phi)}(a_{\bullet\bullet,\bullet,\bullet})^{-1}
\circ\mathrm{comp}_{(\mu,\Phi)}(a_{\bullet,\bullet,\bullet\bullet})^{-1}
\circ\mathrm{comp}_{(\mu,\Phi)}(\sigma_{\bullet,\bullet(\bullet\bullet)})^{-1}
\\ & =
\mathrm{Ad}\Big(
 \bigl(\begin{smallmatrix} 1 & 2 &3&4 \\ 4&1&2&3\end{smallmatrix}\bigr) 
 \cdot \mathrm{exp}({\mu\over 2}e_{1,5})
 \cdot \Phi(e_{34,5},e_{1,5})\cdot \Phi(e_{4,5},e_{12,5})\Big) 
(\mathrm{exp}(\mu e_{45}))
 \\ &=  \mathrm{Ad}\Big(\mathrm{exp}({\mu\over 2}e_{4,5})
 \cdot \Phi(e_{23,5},e_{4,5})\cdot \Phi(e_{3,5},e_{14,5})\Big) 
(\mathrm{exp}(\mu e_{3,5})), 
 \end{align*}
where the first equality follows from Lemma \ref{lemma:1:1:dnb}, 3), the third equality follows from 
\eqref{image:a} for $(P,Q,R)=(\bullet\bullet,\bullet,\bullet)$ and $(\bullet,\bullet,\bullet\bullet)$,  
from the second equality in \eqref{image:sigma} for $(P,Q)=(\bullet,(\bullet\bullet)\bullet)$, and the last equality is a  
computation in $(U\mathfrak p_5)^\wedge\rtimes S_4$. This proves \eqref{image:partIII:x35}. 

One computes
\begin{align*}
&\mathrm{comp}_{(\mu,\Phi)}((\tilde x_{4,5})_{(\bullet(\bullet\bullet))\bullet})
=\mathrm{comp}_{(\mu,\Phi)}(
\sigma_{\bullet,\bullet(\bullet\bullet)}\circ
\sigma_{\bullet(\bullet\bullet),\bullet})
=\mathrm{comp}_{(\mu,\Phi)}(
\sigma_{\bullet,\bullet(\bullet\bullet)})\circ
\mathrm{comp}_{(\mu,\Phi)}(\sigma_{\bullet(\bullet\bullet),\bullet})
\\ & 
=\bigl(\begin{smallmatrix} 1 & 2 &3&4 \\ 4&1&2&3\end{smallmatrix}\bigr) 
 \cdot \mathrm{exp}({\mu\over 2}e_{15}) \cdot
\bigl(\begin{smallmatrix} 1 & 2 &3&4 \\ 2&3&4&1\end{smallmatrix}\bigr) 
 \cdot \mathrm{exp}({\mu\over 2}e_{45})=\mathrm{exp}(\mu e_{45}),   
\end{align*}
where the first equality follows from Lemma \ref{lemma:1:1:dnb}, 4), and 
the third equality follows from the second equality in \eqref{image:sigma} for $(P,Q)=(\bullet,\bullet(\bullet\bullet))$
and from the first equality in \eqref{image:sigma} for $(P,Q)=(\bullet,\bullet(\bullet\bullet))$. This proves \eqref{image:partIII:x45}. 
\hfill\qed\medskip 

\subsection{Definition of the matrices $P_{(\mu,\Phi)}$, $Q_{(\mu,\Phi)}$ and $\overline P_{(\mu,\Phi)}$, 
$\overline Q_{(\mu,\Phi)}$}\label{sect:def:matrices:P:R}

Recall from Lemma \ref{lemma:semidirect:product:LAs}, 1) the graded Lie subalgebra $\mathfrak f_3$ of $\mathfrak p_5$ generated 
by the 
$e_{i5}$, $i\in[\![1,4]\!]$. It follows from Propositions \ref{prop:comp:images:20190728} and \ref{prop:comp:images:brd:grp:elts} that 
for $i\in[\![1,4]\!]$, both $\mathrm{comp}^{((\bullet\bullet)\bullet)\bullet}_{(\mu,\Phi)}(\tilde x_{i5}-1)$ and 
$\mathrm{comp}^{(\bullet(\bullet\bullet))\bullet}_{(\mu,\Phi)}(\tilde x_{i5}-1)$  belong to the augmentation ideal of the degree 
completion $(U\mathfrak f_3)^\wedge$ of its universal enveloping algebra.  This augmentation ideal is freely generated, as a left 
$(U\mathfrak f_3)^\wedge$-module, by the elements $e_{i5}$, $i\in[\![1,3]\!]$. It follows that there are uniquely defined collections 
$(p_{ij})_{i,j\in[\![1,3]\!]}$ and $(q_{ij})_{i,j\in[\![1,3]\!]}$ of elements of $(U\mathfrak f_3)^\wedge$, such that 
$$
\forall i,j\in[\![1,3]\!],\quad \mathrm{comp}^{((\bullet\bullet)\bullet)\bullet}_{(\mu,\Phi)}(\tilde x_{i5}-1)
=\sum_{i\in[\![1,3]\!]}p_{ij}e_{j5},\quad \mathrm{comp}^{(\bullet(\bullet\bullet))\bullet}_{(\mu,\Phi)}(\tilde x_{i5}-1)
=\sum_{i\in[\![1,3]\!]}q_{ij}e_{j5}
$$
(equalities in $(U\mathfrak f_3)^\wedge$). 

\begin{defn}\label{def:P:R:20191203}
Let $\mu\in\mathbf k^\times$ and $\Phi\in\mathsf M_\mu(\mathbf k)$. One sets $P_{(\mu,\Phi)}:=(p_{ij})_{i,j\in[\![1,3]\!]}\in M_3((U\mathfrak f_3)^\wedge)$, 
$Q_{(\mu,\Phi)}:=(q_{ij})_{i,j\in[\![1,3]\!]}\in 
M_3((U\mathfrak f_3)^\wedge)$.
\index{P_mu,Phi@$P_{(\mu,\Phi)}$}
\index{Q_mu,Phi@$Q_{(\mu,\Phi)}$} 
\end{defn}
One then has 
\begin{equation}\label{matrices:P:R:16dec2019}
M_{3\times 1}(\mathrm{comp}^{((\bullet\bullet)\bullet)\bullet}_{(\mu,\Phi)})
(\begin{pmatrix} \tilde x_{15}-1\\  \tilde x_{25}-1\\  \tilde x_{35}-1\end{pmatrix})
=P_{(\mu,\Phi)}\begin{pmatrix} e_{15} \\ e_{25} \\ e_{35} \end{pmatrix}, \quad 
M_{3\times 1}(\mathrm{comp}^{(\bullet(\bullet\bullet))\bullet}_{(\mu,\Phi)})
(\begin{pmatrix} \tilde x_{15}-1\\  \tilde x_{25}-1\\  \tilde x_{35}-1\end{pmatrix})
=Q_{(\mu,\Phi)}\begin{pmatrix} e_{15} \\ e_{25} \\ e_{35} \end{pmatrix} 
\end{equation}
(equalities in $M_{3\times 1}((U\mathfrak f_3)^\wedge)$). 

\begin{lem} If $\mu\in\mathbf k^\times$ and $\Phi\in\mathsf M_\mu(\mathbf k)$, then the matrices $P_{(\mu,\Phi)}$ and $Q_{(\mu,\Phi)}$ belong to $\mathrm{GL}_3((U\mathfrak f_3)^\wedge)$. 
\end{lem}

\proof This follows from the fact that that both $p_{ij}$ and $q_{ij}$ are equal to $\mu\delta_{ij}$ modulo the augmentation ideal of 
$(U\mathfrak f_3)^\wedge$, which is equal to the direct product of the components of $U\mathfrak f_3$ of degrees $>0$. 
\hfill\qed\medskip 

\begin{defn}
For $i,j\in[\![1,3]\!]$, one sets $\overline p_{ij}:=\mathrm{pr}_{12}^\wedge(p_{ij})\in (\mathcal V^\DR)^{\otimes 2,\wedge}$, 
$\overline q_{ij}:=\mathrm{pr}_{12}^\wedge(q_{ij})\in (\mathcal V^\DR)^{\otimes 2,\wedge}$. 
\end{defn}

\begin{defn}\label{def:barP:barR}
Let $\mu\in\mathbf k^\times$ and $\Phi\in\mathsf M_\mu(\mathbf k)$. One sets $\overline P_{(\mu,\Phi)}:=\mathrm{pr}_{12}^\wedge(P_{(\mu,\Phi)})
=(\overline p_{ij})_{i,j\in[\![1,3]\!]}\in \mathrm{GL}_3((\mathcal V^\DR)^{\otimes 2,\wedge})$, 
$\overline Q_{(\mu,\Phi)}:=\mathrm{pr}_{12}^\wedge(Q_{(\mu,\Phi)})
=(\overline q_{ij})_{i,j\in[\![1,3]\!]}\in \mathrm{GL}_3((\mathcal V^\DR)^{\otimes 2,\wedge})$.
\index{P_mu,Phi^bar@$\overline P_{(\mu,\Phi)}$}
\index{Q_mu,Phi^bar@$\overline Q_{(\mu,\Phi)}$}  
\end{defn}

\subsection{Computation of $P_{(\mu,\Phi)}$ and $\overline P_{(\mu,\Phi)}$} 
\label{sect:comp:P}

Recall the elements $\varphi_0,\varphi_1\in\hat{\mathcal V}^\DR$ associated with $\Phi$ (\S\ref{subsect:Gamma:fun:20032018}). 
\begin{prop}\label{prop:comput:P:muPhi}
Let $\mu\in\mathbf k^\times$ and $\Phi\in\mathsf M_\mu(\mathbf k)$. One has
\begin{align*}
& P_{(\mu,\Phi)}=& \\
&\mathrm{diag}\left(\Phi(e_{4,5},e_{12,5})^{-1}\Phi(e_{34,5},e_{1,5})^{-1},\quad 
\Phi(e_{4,5},e_{12,5})^{-1}e^{{\mu\over 2}e_{34,5}}\Phi(e_{34,5},e_{2,5})^{-1}
,\quad e^{{\mu\over 2}e_{4,5}}\Phi(e_{4,5},e_{3,5})^{-1} \right)\cdot 
\\ & \cdot 
\left[
\begin{array}{c|c|c} 
\begin{matrix} \scriptstyle{
(e^{\mu e_{1,5}}-1)\Phi(e_{34,5},e_{1,5})} \\ 
\scriptstyle{\cdot(\varphi_1-\varphi_0)(e_{4,5},e_{12,5})} \\ 
\scriptstyle{+(e^{\mu e_{1,5}}-1)(\varphi_1-\varphi_0)(e_{34,5},e_{1,5})} 
\\ \scriptstyle{+{e^{\mu e_{1,5}}-1 \over e_{1,5}}} 
\end{matrix}
& 
\begin{matrix} 
\scriptstyle{(e^{\mu e_{1,5}}-1)\Phi(e_{34,5},e_{1,5})} \\ 
\scriptstyle{\cdot (\varphi_1-\varphi_0)(e_{4,5},e_{12,5})} \\ 
\scriptstyle{-(e^{\mu e_{1,5}}-1)\varphi_0(e_{34,5},e_{1,5})} 
\end{matrix}
& 
\begin{matrix}
\scriptstyle{-(e^{\mu e_{1,5}}-1)\Phi(e_{34,5},e_{1,5})} 
\\ \scriptstyle{\cdot\varphi_0(e_{4,5},e_{12,5})}
\end{matrix}
\\
\hline
\begin{matrix}
\scriptstyle{(e^{\mu e_{2,5}}-1)\Phi(e_{34,5},e_{2,5})} 
\\ \scriptstyle{\cdot e^{-{\mu\over 2}e_{34,5}}(\varphi_1-\varphi_0)(e_{4,5},e_{12,5})} \\ 
\scriptstyle{-(e^{\mu e_{2,5}}-1)\Phi(e_{34,5},e_{2,5}){{e^{-{\mu\over 2}e_{34,5}}-1}\over{e_{34,5}}} } 
\\ \scriptstyle{- (e^{\mu e_{2,5}}-1)\varphi_0(e_{34,5},e_{2,5})}
\end{matrix}
&
\begin{matrix}
\scriptstyle{(e^{\mu e_{2,5}}-1)\Phi(e_{34,5},e_{2,5})} \\ 
\scriptstyle{\cdot e^{-{\mu\over 2}e_{34,5}}(\varphi_1-\varphi_0)(e_{4,5},e_{12,5})} \\ 
\scriptstyle{-(e^{\mu e_{2,5}}-1)\Phi(e_{34,5},e_{2,5}){{e^{-{\mu\over 2}e_{34,5}}-1}\over{e_{34,5}}} }\\ 
\scriptstyle{+(e^{\mu e_{2,5}}-1)(\varphi_1-\varphi_0)(e_{34,5},e_{2,5})} \\ \scriptstyle{+{{e^{\mu e_{2,5}}-1}\over{e_{2,5}}}}
\end{matrix}
& 
\begin{matrix}
\scriptstyle{-(e^{\mu e_{2,5}}-1)\Phi(e_{34,5},e_{2,5}) } \\ 
\scriptstyle{\cdot e^{-{\mu\over 2}e_{34,5}}\varphi_0(e_{4,5},e_{12,5})}
\end{matrix}
\\ 
\hline
\begin{matrix}
\scriptstyle{-(e^{\mu e_{3,5}}-1)\Phi(e_{4,5},e_{3,5}){{e^{-{\mu\over 2}e_{4,5}}-1}\over{e_{4,5}}}} \\ 
\scriptstyle{-(e^{\mu e_{3,5}}-1)\varphi_0(e_{4,5},e_{3,5})}
\end{matrix}
&
\begin{matrix}
\scriptstyle{-(e^{\mu e_{3,5}}-1)\Phi(e_{4,5},e_{3,5}){{e^{-{\mu\over 2}e_{4,5}}-1}\over{e_{4,5}}}} \\ 
\scriptstyle{-(e^{\mu e_{3,5}}-1)\varphi_0(e_{4,5},e_{3,5})}
\end{matrix}
&
\begin{matrix}
\scriptstyle{-(e^{\mu e_{3,5}}-1)\Phi(e_{4,5},e_{3,5}){{e^{-{\mu\over 2}e_{4,5}}-1}\over{e_{4,5}}}} \\ 
\scriptstyle{+(e^{\mu e_{3,5}}-1)(\varphi_1-\varphi_0)(e_{4,5},e_{3,5})} \\ 
\scriptstyle{+{{e^{\mu e_{3,5}}-1}\over{e_{3,5}}}}
\end{matrix} 
\end{array}
\right]
\end{align*}
where $\mathrm{diag}(d_1,d_2,d_3)$ denotes the diagonal matrix with diagonal elements $d_1,d_2,d_3$.   
\end{prop}

\proof The computation of the first line of the matrix follows from (\ref{image:partI:x15}) and from
\begin{align*}
&
(e^{\mu e_{1,5}}-1)\Phi(e_{34,5},e_{1,5})\Phi(e_{4,5},e_{12,5})
\\ & 
=(e^{\mu e_{1,5}}-1)\Phi(e_{34,5},e_{1,5})\varphi_0(e_{4,5},e_{12,5})e_{4,5}
+(e^{\mu e_{1,5}}-1)\Phi(e_{34,5},e_{1,5})\varphi_1(e_{4,5},e_{12,5})e_{12,5}
\\ & 
+(e^{\mu e_{1,5}}-1)\varphi_0(e_{34,5},e_{1,5})e_{34,5}
+(e^{\mu e_{1,5}}-1)\varphi_1(e_{34,5},e_{1,5})e_{1,5}+(e^{\mu e_{1,5}}-1)
\\ &
=\Big((e^{\mu e_{1,5}}-1)\Phi(e_{34,5},e_{1,5})(\varphi_1-\varphi_0)(e_{4,5},e_{12,5})
+(e^{\mu e_{1,5}}-1)(\varphi_1-\varphi_0)(e_{34,5},e_{1,5})
+{e^{\mu e_{1,5}}-1 \over e_{1,5}}\Big)e_{1,5}
\\ & +\Big((e^{\mu e_{1,5}}-1)\Phi(e_{34,5},e_{1,5})(\varphi_1-\varphi_0)(e_{4,5},e_{12,5})
+(e^{\mu e_{1,5}}-1)(-\varphi_0)(e_{34,5},e_{1,5})
\Big)e_{2,5}
\\ & +\Big((e^{\mu e_{1,5}}-1)\Phi(e_{34,5},e_{1,5})(-\varphi_0)(e_{4,5},e_{12,5})
\Big)e_{3,5}. 
\end{align*}
The computation of the second line follows from \eqref{image:partI:x25} and from 
\begin{align*}
& (e^{\mu e_{2,5}}-1)\Phi(e_{34,5},e_{2,5})e^{-{\mu\over 2}e_{34,5}}\Phi(e_{4,5},e_{12,5}) 
=(e^{\mu e_{2,5}}-1)\Phi(e_{34,5},e_{2,5})e^{-{\mu\over 2}e_{34,5}}\varphi_0(e_{4,5},e_{12,5})e_{4,5}
\\ & 
+(e^{\mu e_{2,5}}-1)\Phi(e_{34,5},e_{2,5})e^{-{\mu\over 2}e_{34,5}}\varphi_1(e_{4,5},e_{12,5})e_{12,5}
+(e^{\mu e_{2,5}}-1)\Phi(e_{34,5},e_{2,5})(e^{-{\mu\over 2}e_{34,5}}-1)
\\ & +(e^{\mu e_{2,5}}-1)\varphi_0(e_{34,5},e_{2,5})e_{34,5}+(e^{\mu e_{2,5}}-1)\varphi_1(e_{34,5},e_{2,5})e_{2,5}+(e^{\mu e_{2,5}}-1)
\\ & 
= \Big(
(e^{\mu e_{2,5}}-1)\Phi(e_{34,5},e_{2,5})e^{-{\mu\over 2}e_{34,5}}(\varphi_1-\varphi_0)(e_{4,5},e_{12,5})
-(e^{\mu e_{2,5}}-1)\Phi(e_{34,5},e_{2,5})
{{e^{-{\mu\over 2}e_{34,5}}-1}\over{e_{34,5}}} \\ & 
+(e^{\mu e_{2,5}}-1)(-\varphi_0)(e_{34,5},e_{2,5})\Big)e_{1,5}
+ \Big(
(e^{\mu e_{2,5}}-1)\Phi(e_{34,5},e_{2,5})e^{-{\mu\over 2}e_{34,5}}(\varphi_1-\varphi_0)(e_{4,5},e_{12,5})
\\ & -(e^{\mu e_{2,5}}-1)\Phi(e_{34,5},e_{2,5}){{e^{-{\mu\over 2}e_{34,5}}-1}\over{e_{34,5}}}  
+(e^{\mu e_{2,5}}-1)(\varphi_1-\varphi_0)(e_{34,5},e_{2,5})+{{e^{\mu e_{2,5}}-1}\over{e_{2,5}}}\Big)e_{2,5}
\\ & + \Big(
(e^{\mu e_{2,5}}-1)\Phi(e_{34,5},e_{2,5})e^{-{\mu\over 2}e_{34,5}}(-\varphi_0)(e_{4,5},e_{12,5})\Big)e_{3,5}. 
\end{align*}
The computation of the third line follows from \eqref{image:partI:x35}, from the 2-cycle identity \eqref{duality:rel}, and from 
\begin{align*}
& (e^{\mu e_{3,5}}-1)\Phi(e_{4,5},e_{3,5})e^{-{\mu\over 2}e_{4,5}}=
(e^{\mu e_{3,5}}-1)\Phi(e_{4,5},e_{3,5})(e^{-{\mu\over 2}e_{4,5}}-1)+(e^{\mu e_{3,5}}-1)\varphi_0(e_{4,5},e_{3,5})e_{4,5}
\\ & +(e^{\mu e_{3,5}}-1)\varphi_1(e_{4,5},e_{3,5})e_{3,5}+(e^{\mu e_{3,5}}-1)
=\Big(-(e^{\mu e_{3,5}}-1)\Phi(e_{4,5},e_{3,5}){{e^{-{\mu\over 2}e_{4,5}}-1}\over{e_{4,5}}}
\\ & -(e^{\mu e_{3,5}}-1)\varphi_0(e_{4,5},e_{3,5})\Big)e_{1,5}
+ \Big(-(e^{\mu e_{3,5}}-1)\Phi(e_{4,5},e_{3,5}){{e^{-{\mu\over 2}e_{4,5}}-1}\over{e_{4,5}}}
\\ & -(e^{\mu e_{3,5}}-1)\varphi_0(e_{4,5},e_{3,5})\Big)e_{2,5}
+ \Big(-(e^{\mu e_{3,5}}-1)\Phi(e_{4,5},e_{3,5}){{e^{-{\mu\over 2}e_{4,5}}-1}\over{e_{4,5}}}
\\ & +(e^{\mu e_{3,5}}-1)(\varphi_1-\varphi_0)(e_{4,5},e_{3,5})
+{{e^{\mu e_{3,5}}-1}\over{e_{3,5}}}\Big)e_{3,5}. 
\end{align*}
\hfill \qed\medskip 

\begin{cor}\label{cor:equation:barP}
Let $\mu\in\mathbf k^\times$ and $\Phi\in\mathsf M_\mu(\mathbf k)$. One has 
\begin{align}\nonumber
& \overline P_{(\mu,\Phi)}
=  & \\ \nonumber
&
\mathrm{diag}\left(\Phi(e_0,e_1)^{-1}\Phi(f_\infty,f_1)^{-1}, 
\Phi(e_0,e_1)^{-1}e^{-{\mu\over 2}e_1}\Phi(f_\infty,f_1)^{-1}e^{-{\mu\over 2}f_1}, 
e^{{\mu\over 2}e_0}\Phi^{-1}(e_0,e_\infty)e^{{\mu\over 2}f_\infty}\Phi^{-1}(f_\infty,f_0) \right)
\\ & \nonumber\cdot 
\left[
\begin{array}{c|c|c} 
\begin{matrix} 
\scriptstyle{(e^{\mu f_1}-1)(\varphi_1-\varphi_0)(e_0+f_\infty,e_1+f_1)} \\ 
\scriptstyle{+(e^{\mu f_1}-1)(\varphi_1-\varphi_0)(-(e_1+f_1),f_1)} \\ 
\scriptstyle{+{e^{\mu f_1}-1 \over f_1}} 
\end{matrix}
& 
\begin{matrix} 
\scriptstyle{(e^{\mu f_1}-1)\cdot} \\ 
\scriptstyle{\cdot(\varphi_1-\varphi_0)(e_0+f_\infty,e_1+f_1)} \\ 
\scriptstyle{-(e^{\mu f_1}-1)\varphi_0(-(e_1+f_1),f_1)} \\ \ 
\end{matrix}
& 
\begin{matrix}
\scriptstyle{-(e^{\mu f_1}-1)\cdot} \\ \scriptstyle{\cdot\varphi_0(e_0+f_\infty,e_1+f_1)}
\end{matrix}
\\
\hline
\begin{matrix}
\scriptstyle{(e^{\mu e_1}-1)e^{{\mu\over 2}(e_1+f_1)}\cdot} \\ 
\scriptstyle{\cdot (\varphi_1-\varphi_0)(e_0+f_\infty,e_1+f_1)} \\ 
\scriptstyle{-(e^{\mu e_1}-1){{e^{{\mu\over 2}(e_1+f_1)}-1}\over{-(e_1+f_1)}} } \\ 
\scriptstyle{- (e^{\mu e_1}-1)\varphi_0(-(e_1+f_1),e_1)}
\end{matrix}
&
\begin{matrix}
\scriptstyle{(e^{\mu e_1}-1)e^{{\mu\over 2}(e_1+f_1)}\cdot} \\ 
\scriptstyle{\cdot(\varphi_1-\varphi_0)(e_0+f_\infty,e_1+f_1)} \\ 
\scriptstyle{-(e^{\mu e_1}-1){{e^{{\mu\over 2}(e_1+f_1)}-1}\over{-(e_1+f_1)}} }\\ 
\scriptstyle{+(e^{\mu e_1}-1)(\varphi_1-\varphi_0)(-(e_1+f_1),e_1)} \\ 
\scriptstyle{+{{e^{\mu e_1}-1}\over{e_1}}}
\end{matrix}
& 
\begin{matrix}
\scriptstyle{-(e^{\mu e_1}-1)e^{{\mu\over 2}(e_1+f_1)} \cdot} \\ 
\scriptstyle{\cdot \varphi_0(e_0+f_\infty,e_1+f_1)}
\end{matrix}
\\ 
\hline
\begin{matrix}
\scriptstyle{-(e^{\mu (e_\infty+f_0)}-1)\Phi(e_0,e_\infty)\cdot}\\ 
\scriptstyle{\cdot\Phi(f_\infty,f_0){{e^{-{\mu\over 2}(e_0+f_\infty)}-1}\over{e_0+f_\infty}}} \\ 
\scriptstyle{-(e^{\mu (e_\infty+f_0)}-1)\varphi_0(e_0+f_\infty,e_\infty+f_0)}
\end{matrix}
&
\begin{matrix}
\scriptstyle{-(e^{\mu (e_\infty+f_0)}-1)\Phi(e_0,e_\infty)\cdot}\\ 
\scriptstyle{\cdot\Phi(f_\infty,f_0){{e^{-{\mu\over 2}(e_0+f_\infty)}-1}\over{e_0+f_\infty}}} \\ 
\scriptstyle{-(e^{\mu (e_\infty+f_0)}-1)\varphi_0(e_0+f_\infty,e_\infty+f_0)}
\end{matrix}
&
\begin{matrix}
\scriptstyle{-(e^{\mu (e_\infty+f_0)}-1)\Phi(e_0,e_\infty)\cdot}\\ 
\scriptstyle{\cdot \Phi(f_\infty,f_0){{e^{-{\mu\over 2}(e_0+f_\infty)}-1}\over{e_0+f_\infty}}} \\ 
\scriptstyle{+(e^{\mu(e_\infty+f_0)}-1)\cdot} \\ 
\scriptstyle{\cdot(\varphi_1-\varphi_0)(e_0+f_\infty,e_\infty+f_0)}\\ 
\scriptstyle{+{{e^{\mu (e_\infty+f_0)}-1}\over{e_\infty+f_0}}}
\end{matrix} 
\end{array}
\right]
\end{align}
\end{cor}

\proof One transforms the expression obtained from Proposition \ref{prop:comput:P:muPhi} using identities $\Phi(e_0+f_\infty,e_1+f_1)=\Phi(e_0,e_1)\Phi(f_\infty,f_1)$, $\Phi(e_0+f_\infty,e_\infty+f_0)=\Phi(e_0,e_\infty)\Phi(f_\infty,f_0)$, 
$\Phi(-(e_1+f_1),f_1)=\Phi(-(e_1+f_1),e_1)=1$, which follow from the fact $\Phi$ is a group-like element with vanishing linear terms 
in $e_0,e_1$. \hfill\qed\medskip 

\subsection{Computation of $Q_{(\mu,\Phi)}$ and $\overline Q_{(\mu,\Phi)}$}\label{sect:comp:R}

\begin{prop}\label{prop:comput:R:muPhi}
Let $\mu\in\mathbf k^\times$ and $\Phi\in\mathsf M_\mu(\mathbf k)$. One has
\begin{align}\nonumber
& Q_{(\mu,\Phi)}
=  & \\ \nonumber
&
\mathrm{diag}\left(\Phi(e_{15},e_{45}), 
e^{{\mu\over 2}e_{45}}
  \Phi(e_{23,5},e_{45}) e^{{\mu\over 2}e_{23,5}} \Phi(e_{25},e_{14,5}), 
e^{{\mu\over 2}e_{45}} 
  \Phi(e_{23,5},e_{45}) \Phi(e_{35},e_{14,5}) \right)\cdot
\\ & \nonumber\cdot 
\left[
\begin{array}{c|c|c} 
\begin{matrix} 
\scriptstyle{{{e^{\mu e_{15}}-1}\over{e_{15}}}} \\ 
\scriptstyle{+(e^{\mu e_{15}}-1)(\varphi_1-\varphi_0)(e_{45},e_{15})}
\end{matrix}
 & 
\scriptstyle{-(e^{\mu e_{15}}-1)\varphi_0(e_{45},e_{15})}& 
\scriptstyle{-(e^{\mu e_{15}}-1)\varphi_0(e_{45},e_{15})}
\\
\hline
\begin{matrix} 
\scriptstyle{-(e^{\mu e_{25}}-1)\Phi(e_{14,5},e_{25})e^{-{\mu\over 2}e_{23,5}}\cdot} \\ 
\scriptstyle{\cdot\Phi(e_{45},e_{23,5})
{{e^{-{\mu\over 2}e_{45}}-1}\over{e_{45}}}}  \\ 
\scriptstyle{-(e^{\mu e_{25}}-1)\Phi(e_{14,5},e_{25})\cdot} \\ 
\scriptstyle{\cdot e^{-{\mu\over 2}e_{23,5}} 
\varphi_0(e_{45},e_{23,5})}
\end{matrix}
&
\begin{matrix} 
\scriptstyle{-(e^{\mu e_{25}}-1)\Phi(e_{14,5},e_{25})e^{-{\mu\over 2}e_{23,5}}\cdot} \\ 
\scriptstyle{\cdot\Phi(e_{45},e_{23,5}){{e^{-{\mu\over 2}e_{45}}-1}\over{e_{45}}}} \\ 
\scriptstyle{+(e^{\mu e_{25}}-1)\Phi(e_{14,5},e_{25})e^{-{\mu\over 2}e_{23,5}}\cdot}  \\ 
\scriptstyle{\cdot (\varphi_1-\varphi_0)(e_{45},e_{23,5})} \\ 
\scriptstyle{+ (e^{\mu e_{25}}-1)\Phi(e_{14,5},e_{25}){{e^{-{\mu\over 2}e_{23,5}}-1}\over{e_{23,5}}}} \\ 
\scriptstyle{ +(e^{\mu e_{25}}-1)(\varphi_1-\varphi_0)(e_{14,5},e_{25})}
\\ \scriptstyle{+{{e^{\mu e_{25}}-1}\over{e_{25}}}}
\end{matrix}
& 
\begin{matrix} 
\scriptstyle{-(e^{\mu e_{25}}-1)\Phi(e_{14,5},e_{25})e^{-{\mu\over 2}e_{23,5}}\cdot} \\ 
\scriptstyle{\cdot\Phi(e_{45},e_{23,5}){{e^{-{\mu\over 2}e_{45}}-1}\over{e_{45}}}} \\ 
\scriptstyle{+(e^{\mu e_{25}}-1)\Phi(e_{14,5},e_{25})e^{-{\mu\over 2}e_{23,5}} \cdot}  \\ 
\scriptstyle{\cdot(\varphi_1-\varphi_0)(e_{45},e_{23,5})}  \\ 
\scriptstyle{+ (e^{\mu e_{25}}-1)\Phi(e_{14,5},e_{25}){{e^{-{\mu\over 2}e_{23,5}}-1}\over{e_{23,5}}}} \\ 
\scriptstyle{-(e^{\mu e_{25}}-1)\varphi_0(e_{14,5},e_{25})} 
\end{matrix}
\\ 
\hline
\begin{matrix} 
\scriptstyle{-(e^{\mu e_{35}}-1)\Phi(e_{14,5},e_{35})\cdot
} \\ 
\scriptstyle{\cdot\Phi(e_{45},e_{23,5}){{e^{-{\mu\over 2}e_{45}}-1}\over{e_{45}}}}  \\ 
\scriptstyle{- (e^{\mu e_{35}}-1)\Phi(e_{14,5},e_{35})\varphi_0(e_{45},e_{23,5})} 
\end{matrix}
&
\begin{matrix} 
\scriptstyle{ -(e^{\mu e_{35}}-1)\Phi(e_{14,5},e_{35})
\cdot} \\ \scriptstyle{\cdot
\Phi(e_{45},e_{23,5}) {{e^{-{\mu\over 2}e_{45}}-1}\over{e_{45}}}}  \\ 
\scriptstyle{+ (e^{\mu e_{35}}-1)\Phi(e_{14,5},e_{35})
\cdot} \\ \scriptstyle{\cdot
(\varphi_1-\varphi_0)(e_{45},e_{23,5})}  \\ 
\scriptstyle{- (e^{\mu e_{35}}-1)\varphi_0(e_{14,5},e_{35})}  
\end{matrix}
&
\begin{matrix} 
\scriptstyle{-(e^{\mu e_{35}}-1)\Phi(e_{14,5},e_{35})\cdot} \\ 
\scriptstyle{\cdot\Phi(e_{45},e_{23,5}) {{e^{-{\mu\over 2}e_{45}}-1}\over{e_{45}}}} \\ 
\scriptstyle{+ (e^{\mu e_{35}}-1)\Phi(e_{14,5},e_{35})\cdot} \\ 
\scriptstyle{\cdot(\varphi_1-\varphi_0)(e_{45},e_{23,5})}  \\ 
\scriptstyle{+(e^{\mu e_{35}}-1)(\varphi_1-\varphi_0)(e_{14,5},e_{35})}  \\ 
\scriptstyle{+ {{e^{\mu e_{35}}-1}\over{e_{35}}}} 
\end{matrix} 
\end{array}
\right]
\end{align}
\end{prop}

\proof The computation of the first line of the matrix follows from \eqref{image:partIII:x15}, from the 2-cycle identity \eqref{duality:rel}, and from 
\begin{align*}
&(e^{\mu e_{15}}-1)\Phi(e_{45},e_{15})=(e^{\mu e_{15}}-1)(1+\varphi_0(e_{45},e_{15})e_{45}+\varphi_1(e_{45},e_{15})e_{15})
\\ & =\Big({{e^{\mu e_{15}}-1}\over{e_{15}}}
+(e^{\mu e_{15}}-1)\varphi_1(e_{45},e_{15})\Big)e_{15}+(e^{\mu e_{15}}-1)\varphi_0(e_{45},e_{15})e_{45}
\\ & 
\scriptstyle{=\Big({{e^{\mu e_{15}}-1}\over{e_{15}}}+(e^{\mu e_{15}}-1)(\varphi_1-\varphi_0)(e_{45},e_{15})\Big)e_{15}
-(e^{\mu e_{15}}-1)\varphi_0(e_{45},e_{15})e_{25}-(e^{\mu e_{15}}-1)\varphi_0(e_{45},e_{15})e_{35}.} 
\end{align*}

The computation of the second line of the matrix follows from \eqref{image:partIII:x25}, from the 2-cycle identity \eqref{duality:rel}, and from
\begin{align*}
&
(e^{\mu e_{25}}-1)\Phi(e_{14,5},e_{25})e^{-{\mu\over 2}e_{23,5}} 
\Phi(e_{45},e_{23,5})e^{-{\mu\over 2}e_{45}}
=(e^{\mu e_{25}}-1)\Phi(e_{14,5},e_{25})e^{-{\mu\over 2}e_{23,5}} 
\Phi(e_{45},e_{23,5})\cdot
\\ & \cdot(e^{-{\mu\over 2}e_{45}}-1)
+(e^{\mu e_{25}}-1)\Phi(e_{14,5},e_{25})e^{-{\mu\over 2}e_{23,5}} 
(\varphi_0(e_{45},e_{23,5})e_{45}
+\varphi_1(e_{45},e_{23,5})e_{23,5}) 
\\ & 
\scriptstyle{+ (e^{\mu e_{25}}-1)\Phi(e_{14,5},e_{25})(e^{-{\mu\over 2}e_{23,5}}-1)
+ (e^{\mu e_{25}}-1)(\varphi_0(e_{14,5},e_{25})e_{14,5}+\varphi_1(e_{14,5},e_{25})e_{25})
 +(e^{\mu e_{25}}-1)}
\\ &
=\scriptstyle{\Big(
-(e^{\mu e_{25}}-1)\Phi(e_{14,5},e_{25})e^{-{\mu\over 2}e_{23,5}}\Phi(e_{45},e_{23,5})
{{e^{-{\mu\over 2}e_{45}}-1}\over{e_{45}}}
-(e^{\mu e_{25}}-1)\Phi(e_{14,5},e_{25})e^{-{\mu\over 2}e_{23,5}} 
\varphi_0(e_{45},e_{23,5})\Big) e_{15}}
\\&
\scriptstyle{+\Big(
-(e^{\mu e_{25}}-1)\Phi(e_{14,5},e_{25})e^{-{\mu\over 2}e_{23,5}}\Phi(e_{45},e_{23,5})
{{e^{-{\mu\over 2}e_{45}}-1}\over{e_{45}}}
+(e^{\mu e_{25}}-1)\Phi(e_{14,5},e_{25})e^{-{\mu\over 2}e_{23,5}} 
(\varphi_1-\varphi_0)(e_{45},e_{23,5})}
\\ & 
\scriptstyle{
+ (e^{\mu e_{25}}-1)\Phi(e_{14,5},e_{25}){{e^{-{\mu\over 2}e_{23,5}}-1}\over{e_{23,5}}}
 +(e^{\mu e_{25}}-1)(\varphi_1-\varphi_0)(e_{14,5},e_{25})
+{{e^{\mu e_{25}}-1}\over{e_{25}}}\Big) e_{25}}
\\ & 
\scriptstyle{+\Big(
-(e^{\mu e_{25}}-1)\Phi(e_{14,5},e_{25})e^{-{\mu\over 2}e_{23,5}}\Phi(e_{45},e_{23,5})
{{e^{-{\mu\over 2}e_{45}}-1}\over{e_{45}}}
+(e^{\mu e_{25}}-1)\Phi(e_{14,5},e_{25})e^{-{\mu\over 2}e_{23,5}} 
(\varphi_1-\varphi_0)(e_{45},e_{23,5})}
 \\ & 
\scriptstyle{
+ (e^{\mu e_{25}}-1)\Phi(e_{14,5},e_{25}){{e^{-{\mu\over 2}e_{23,5}}-1}\over{e_{23,5}}}
 -(e^{\mu e_{25}}-1)\varphi_0(e_{14,5},e_{25})\Big)e_{35}}
\end{align*}

The computation of the third line of the matrix follows from \eqref{image:partIII:x35}, from the 2-cycle identity \eqref{duality:rel}, 
and from
\begin{align*}
& (e^{\mu e_{35}}-1)\Phi(e_{14,5},e_{35})\Phi(e_{45},e_{23,5}) e^{-{\mu\over 2}e_{45}} 
=(e^{\mu e_{35}}-1)\Phi(e_{14,5},e_{35})\Phi(e_{45},e_{23,5}) (e^{-{\mu\over 2}e_{45}}-1)
\\ &+ (e^{\mu e_{35}}-1)\Phi(e_{14,5},e_{35})(\varphi_0(e_{45},e_{23,5})e_{45}+\varphi_1(e_{45},e_{23,5})e_{23,5})
\\ &+ (e^{\mu e_{35}}-1)(\varphi_0(e_{14,5},e_{35})e_{14,5}+\varphi_1(e_{14,5},e_{35})e_{35})+ (e^{\mu e_{35}}-1)
\\ &
=\Big( -(e^{\mu e_{35}}-1)\Phi(e_{14,5},e_{35})\Phi(e_{45},e_{23,5}) {{e^{-{\mu\over 2}e_{45}}-1}\over{e_{45}}}
- (e^{\mu e_{35}}-1)\Phi(e_{14,5},e_{35})\varphi_0(e_{45},e_{23,5})\Big) e_{15}
\\ &
+\Big( -(e^{\mu e_{35}}-1)\Phi(e_{14,5},e_{35})\Phi(e_{45},e_{23,5}) {{e^{-{\mu\over 2}e_{45}}-1}\over{e_{45}}}
+(e^{\mu e_{35}}-1)\Phi(e_{14,5},e_{35})(\varphi_1-\varphi_0)(e_{45},e_{23,5})
\\ &
- (e^{\mu e_{35}}-1)\varphi_0(e_{14,5},e_{35})\Big)e_{25}
+\Big( -(e^{\mu e_{35}}-1)\Phi(e_{14,5},e_{35})\Phi(e_{45},e_{23,5}) {{e^{-{\mu\over 2}e_{45}}-1}\over{e_{45}}}
\\ &+(e^{\mu e_{35}}-1)\Phi(e_{14,5},e_{35})(\varphi_1-\varphi_0)(e_{45},e_{23,5})
+ (e^{\mu e_{35}}-1)(\varphi_1-\varphi_0)(e_{14,5},e_{35})
+ {{e^{\mu e_{35}}-1}\over{e_{35}}}\Big)e_{35}. 
\end{align*}
\hfill\qed\medskip 

\begin{cor}\label{cor:comput:bar:R:muPhi}
Let $\mu\in\mathbf k^\times$ and $\Phi\in\mathsf M_\mu(\mathbf k)$. One has
\begin{align}\label{equation:bar:R}
& \overline Q_{(\mu,\Phi)}= \nonumber  \\ 
\nonumber
&
\mathrm{diag}\left(\Phi(f_1,f_\infty), e^{{\mu\over 2}f_\infty}
 \Phi(f_0,f_\infty)e^{{\mu\over 2}f_0} \Phi(e_1,e_0), 
 e^{{\mu\over 2}(e_0+f_\infty)} 
\Phi(f_0,f_\infty)\Phi(e_\infty,e_0)\right)\cdot 
\\ & \nonumber\cdot 
\left[
\begin{array}{c|c|c} 
\begin{matrix}
\scriptstyle{{{e^{\mu f_1}-1}\over{f_1}}} \\ 
\scriptstyle{+(e^{\mu f_1}-1)(\varphi_1-\varphi_0)(e_0+f_\infty,f_1)} 
\end{matrix}
& 
\scriptstyle{-(e^{\mu f_1}-1)\varphi_0(e_0+f_\infty,f_1)}& 
\scriptstyle{-(e^{\mu f_1}-1)\varphi_0(e_0+f_\infty,f_1)}
\\
\hline
\begin{matrix} 
\scriptstyle{-(e^{\mu e_1}-1)\Phi(e_0,e_1)e^{-{\mu\over 2}(-e_0+f_0)}\cdot} \\ 
\scriptstyle{\cdot\Phi(f_\infty,f_0)
{{e^{-{\mu\over 2}(e_0+f_\infty)}-1}\over{e_0+f_\infty}}}  \\ 
\scriptstyle{-(e^{\mu e_1}-1)\Phi(e_0,e_1) e^{-{\mu\over 2}(-e_0+f_0)}\cdot} \\ 
\scriptstyle{\cdot  
\varphi_0(e_0+f_\infty,-e_0+f_0)}
\end{matrix}
&
\begin{matrix} 
\scriptstyle{-(e^{\mu e_1}-1)\Phi(e_0,e_1)e^{-{\mu\over 2}(-e_0+f_0)}\cdot} \\ 
\scriptstyle{\cdot\Phi(f_\infty,f_0){{e^{-{\mu\over 2}(e_0+f_\infty)}-1}\over{e_0+f_\infty}}} \\ 
\scriptstyle{+(e^{\mu e_1}-1)\Phi(e_0,e_1)e^{-{\mu\over 2}(-e_0+f_0)}\cdot}  \\ 
\scriptstyle{\cdot (\varphi_1-\varphi_0)(e_0+f_\infty,-e_0+f_0)} \\ 
\scriptstyle{+ (e^{\mu e_1}-1)\Phi(e_0,e_1){{e^{-{\mu\over 2}(-e_0+f_0)}-1}\over{-e_0+f_0}}} \\ 
\scriptstyle{ +(e^{\mu e_1}-1)(\varphi_1-\varphi_0)(e_0-f_0,e_1)} \\
\scriptstyle{+{{e^{\mu e_1}-1}\over{e_1}}}
\end{matrix}
& 
\begin{matrix} 
\scriptstyle{-(e^{\mu e_1}-1)\Phi(e_0,e_1)e^{-{\mu\over 2}(-e_0+f_0)}\cdot} \\ 
\scriptstyle{\cdot\Phi(f_\infty,f_0){{e^{-{\mu\over 2}(e_0+f_\infty)}-1}\over{e_0+f_\infty}}} \\ 
\scriptstyle{+(e^{\mu e_1}-1)\Phi(e_0,e_1)e^{-{\mu\over 2}(-e_0+f_0)} \cdot}  \\ 
\scriptstyle{\cdot(\varphi_1-\varphi_0)(e_0+f_\infty,-e_0+f_0)}  \\ 
\scriptstyle{+ (e^{\mu e_1}-1)\Phi(e_0,e_1){{e^{-{\mu\over 2}(-e_0+f_0)}-1}\over{-e_0+f_0}}} \\ 
\scriptstyle{-(e^{\mu e_1}-1)\varphi_0(e_0-f_0,e_1)} 
\end{matrix}
\\ 
\hline
\begin{matrix} 
\scriptstyle{-(e^{\mu (e_\infty+f_0)}-1)\Phi(e_0,e_\infty)\cdot} \\ 
\scriptstyle{\cdot\Phi(f_\infty,f_0){{e^{-{\mu\over 2}(e_0+f_\infty)}-1}\over{e_0+f_\infty}}}  \\ 
\scriptstyle{- (e^{\mu (e_\infty+f_0)}-1)\Phi(e_0,e_\infty)\cdot} \\
\scriptstyle{\cdot\varphi_0(e_0+f_\infty,-e_0+f_0)}
\end{matrix}
&
\begin{matrix} 
\scriptstyle{ -(e^{\mu (e_\infty+f_0)}-1)\Phi(e_0,e_\infty)
\cdot} \\ \scriptstyle{\cdot
\Phi(f_\infty,f_0) {{e^{-{\mu\over 2}(e_0+f_\infty)}-1}\over{e_0+f_\infty}}}  \\ 
\scriptstyle{+ (e^{\mu (e_\infty+f_0)}-1)\Phi(e_0,e_\infty)
\cdot} \\ \scriptstyle{\cdot
(\varphi_1-\varphi_0)(e_0+f_\infty,-e_0+f_0)}  \\ 
\scriptstyle{- (e^{\mu (e_\infty+f_0)}-1)\varphi_0(e_0-f_0,e_\infty+f_0)}  
\end{matrix}
&
\begin{matrix} 
\scriptstyle{-(e^{\mu (e_\infty+f_0)}-1)\Phi(e_0,e_\infty)\cdot} \\ 
\scriptstyle{\cdot\Phi(f_\infty,f_0) {{e^{-{\mu\over 2}(e_0+f_\infty)}-1}\over{e_0+f_\infty}}} \\ 
\scriptstyle{+ (e^{\mu (e_\infty+f_0)}-1)\Phi(e_0,e_\infty)\cdot} \\ 
\scriptstyle{\cdot(\varphi_1-\varphi_0)(e_0+f_\infty,-e_0+f_0)}  \\ 
\scriptstyle{+(e^{\mu (e_\infty+f_0)}-1)(\varphi_1-\varphi_0)(e_0-f_0,e_\infty+f_0)}  \\ 
\scriptstyle{+ {{e^{\mu (e_\infty+f_0)}-1}\over{e_\infty+f_0}}} 
\end{matrix} 
\end{array}
\right]
\end{align}
\end{cor}

\proof 
One transforms the expression obtained from Proposition \ref{prop:comput:R:muPhi} using the commutativity of $e_0$ with $f_0,f_\infty$, 
as well as the identities $\Phi(e_0+f_\infty,e_1+f_1)=\Phi(e_0,e_1)\Phi(f_\infty,f_1)$, 
$\Phi(e_0+f_\infty,e_\infty+f_0)=\Phi(e_0,e_\infty)\Phi(f_\infty,f_0)$, 
$\Phi(-(e_1+f_1),f_1)=\Phi(-(e_1+f_1),e_1)=1$, which follow from the fact $\Phi$ is a group-like element with vanishing linear terms 
in $e_0,e_1$. \hfill\qed\medskip 

\section{Associators and harmonic algebra coproducts}\label{sect:8:19032018}

The purpose of this section is the proof of the first main result of this paper, namely the compatibility of associators with 
harmonic algebra coproducts (Theorem \ref{thm:compat:assoc:algebra:24dec2019}). This proof relies on the commutativity of 
diagram \eqref{diagg:main:alg}, which is divided into subdiagrams (A1) to (A8). Among them, subdiagram (A1) (resp. (A7)) 
expresses the interpretation of the Betti (resp. de Rham) algebra coproduct and was proved in Lemma \ref{lemma:completion:10dec2019} 
(resp. Lemma \ref{lem:completion:diag:LA}). The main remaining diagrams are (A3) and (A4), both of which relate, by means of the 
matrix $\overline P_{(\mu,\Phi)}$ from \S\ref{sect:def:matrices:P:R} and the algebra comparison isomorphisms, 
`de Rham' and `Betti' objects: namely, the morphisms $\hat\rho$ and 
$\underline{\hat\rho}$ in the case of (A3), and the morphisms $\mathrm{row}_1\cdot(-)\cdot\mathrm{col}_1$ and 
$\underline{\mathrm{row}}_1\cdot(-)\cdot\underline{\mathrm{col}}_1$ in the case of (A4). \S\S\ref{sect:8:1:21032018} to 
\ref{sect:8:5:12122017} are devoted to the commutativity of (A3); more precisely, \S\S\ref{sect:8:1:21032018} to 
\ref{subsection8:3:21032018} are devoted to the commutativity of diagrams relating constituents of $\hat\rho$ and 
$\underline{\hat\rho}$, and these results are gathered in \S\ref{sect:8:5:12122017} for proving the commutativity of (A3). 
\S\S\ref{sect:rel:col:barcol:21032018} to \ref{CD:row:col:barrow:barcol:21032018} are devoted to the commutativity of (A4); 
more precisely, \S\ref{sect:rel:col:barcol:21032018} (resp. \S\ref{sect:8:6:row:bar:row:21032018}) is devoted to the proof of 
equalities relating $\overline P_{(\mu,\Phi)}$, $\mathrm{col}_1$ and $\underline{\mathrm{col}}_1$ (resp. $\overline P_{(\mu,\Phi)}$, 
$\mathrm{row}_1$ and $\underline{\mathrm{row}}_1$), based on the explicit computation of $\overline P_{(\mu,\Phi)}$ in 
\S\ref{sect:comp:P}, and these results are combined in \S\ref{CD:row:col:barrow:barcol:21032018} for proving the commutativity of (A4).  
Theorem \ref{thm:compat:assoc:algebra:24dec2019} is formulated and proved in \S\ref{sect:8:8:21:03:2018}. 

\subsection{Commutative diagram relating $\hat\ell$ and \underline{$\hat{\ell}$}}
\label{sect:8:1:21032018}

One checks that for $(\mu,\Phi)\in\mathbf k^\times\times\mathcal G(\hat{\mathcal V}^\DR)$, the morphism  
\begin{equation}\label{def:underline:alpha:Phi:group}
\mathrm{comp}_{(\mu,\Phi)}^{\mathcal V,(1)}:\hat{\mathcal V}^\B\to \hat{\mathcal V}^\DR
\end{equation}
defined in \S\ref{sect:3:3:28oct} is an isomorphism of topological Hopf algebras, where the two sides are equipped with the coproducts
$\Delta^{\mathcal V,\B}$ and $\Delta^{\mathcal V,\DR}$, and is given by  
\begin{equation}\label{formulas:a:Phi:12122017}
X_0\mapsto\Phi(e_0,e_1)e^{\mu e_0}\Phi(e_0,e_1)^{-1}, \quad X_1\mapsto e^{\mu e_1}. 
\end{equation}
Then: 
\begin{lem}\label{lemma:52:18:08:2017}
Let $\mu\in\mathbf k^\times$ and $\Phi\in\mathsf M_\mu(\mathbf k)$. The following diagram of topological Hopf algebras 
$$
\xymatrix{ \hat{\mathcal V}^\B\ar^{\underline{\hat\ell}}[r]\ar_{\mathrm{comp}_{(\mu,\Phi)}^{\mathcal V,(1)}}[d]& 
(\mathbf kP_5^*)^\wedge \ar^{\mathrm{comp}_{(\mu,\Phi)}^{((\bullet\bullet)\bullet)\bullet}}[d]
\\ \hat{\mathcal V}^\DR
\ar_{\hat\ell}[r]
& (U\mathfrak p_5)^\wedge}
$$
is commutative, where $\mathrm{comp}_{(\mu,\Phi)}^{\mathcal V,(1)}$ is as in \eqref{def:underline:alpha:Phi:group}, 
$\underline{\hat\ell}$\index{l_underline^hat@$\underline{\hat\ell}$} is the completed version of $\underline{\ell}$ from \S\ref{subsect:def:gp:morphisms:bis}, 
$\mathrm{comp}_{(\mu,\Phi)}^{((\bullet\bullet)\bullet)\bullet}$ is as in \S\ref{sect:9:3:3:20191203}, and 
$\hat\ell$\index{l^hat@$\hat\ell$} is the completed version of $\ell$ from 
\S\ref{subsect:def:LA:morphisms:2}.
\end{lem}

\proof The following equalities $x_{12}=\sigma_1^2=\sigma_{1,1}^{-2}$ hold in $B_2$ (see 
\S\ref{subsect:def:gp:morphisms}) and in view of \S\ref{sect:9.1.3.0306}, (c), 
imply the equality $(x_{12})_{\bullet\bullet}=\sigma_{\bullet,\bullet}^{-2}$ in 
$\mathbf{PaB}(\bullet\bullet)$. The following equalities 
$(x_{12})_{((\bullet\bullet)\bullet)\bullet}=(\sigma_{\bullet,\bullet}^{-2}\otimes\mathrm{id}_\bullet)\otimes\mathrm{id}_\bullet$
and $(x_{23})_{(\bullet(\bullet\bullet))\bullet}=(\mathrm{id}_\bullet\otimes\sigma_{\bullet,\bullet}^{-2})\otimes
\mathrm{id}_\bullet$ similarly hold in $\mathbf{PaB}(((\bullet\bullet)\bullet)\bullet)$ and 
$\mathbf{PaB}((\bullet(\bullet\bullet))\bullet)$. \S\ref{sect:9.1.3.0306} then implies the equalitites  
\begin{equation}\label{comp:image:x12:5dec}
\mathrm{comp}_{(\mu,\Phi)}^{((\bullet\bullet)\bullet)\bullet}(x_{12})
=e^{\mu e_{12}}, \quad
\mathrm{comp}_{(\mu,\Phi)}^{(\bullet(\bullet\bullet))\bullet}(x_{23})
=e^{\mu e_{23}}.  
\end{equation}
Then  
\begin{align}\label{comp:image:x23:5dec}
&\mathrm{comp}_{(\mu,\Phi)}^{((\bullet\bullet)\bullet)\bullet}(x_{23})=
\mathrm{comp}_{(\mu,\Phi)}((x_{23})_{((\bullet\bullet)\bullet)\bullet})
\\ & \nonumber=\mathrm{comp}_{(\mu,\Phi)}((a_{\bullet,\bullet,\bullet}\otimes\mathrm{id}_\bullet)^{-1}
\circ(x_{23})_{(\bullet(\bullet\bullet))\bullet}\circ(a_{\bullet,\bullet,\bullet}\otimes\mathrm{id}_\bullet))
=\Phi(e_{12},e_{23})^{-1}e^{\mu e_{23}}\Phi(e_{12},e_{23}),
\end{align}
where the last identity uses the second part of \eqref{comp:image:x12:5dec}. 

The first part of \eqref{comp:image:x12:5dec} implies that $\mathrm{comp}_{(\mu,\Phi)}^{((\bullet\bullet)\bullet)\bullet}\circ
\underline{\hat\ell}(X_1)=e^{\mu e_{12}}$, which is equal to $\hat\ell\circ\mathrm{comp}_{(\mu,\Phi)}^{\mathcal V,(1)}(X_1)$. 
\eqref{comp:image:x23:5dec} implies $\mathrm{comp}_{(\mu,\Phi)}^{((\bullet\bullet)\bullet)\bullet}\circ\underline{\hat\ell}(X_0)=
\Phi(e_{12},e_{23})^{-1}e^{\mu e_{23}}\Phi(e_{12},e_{23})$. One has 
$\hat\ell\circ\mathrm{comp}_{(\mu,\Phi)}^{\mathcal V,(1)}(X_0)=\Phi(e_{23},e_{12})e^{\mu e_{23}}\Phi(e_{23},e_{12})^{-1}$. 
The 2-cycle identity \eqref{duality:rel} implies that these images are equal. Therefore the images of both $X_0$ and $X_1$
under $\mathrm{comp}_{(\mu,\Phi)}^{((\bullet\bullet)\bullet)\bullet}\circ\underline{\hat\ell}$ and 
$\hat\ell\circ\mathrm{comp}_{(\mu,\Phi)}^{\mathcal V,(1)}$ are the same. 
\hfill\qed\medskip

\subsection{Commutative diagram relating $M_3(\mathrm{pr}_{12}^\wedge)$ and $M_3($\underline{$\mathrm{pr}$}$^\wedge_{12})$}
\label{sect:8:2:21032018}

\begin{lem}\label{lemma:CD:29122917}
Let $\mu\in\mathbf k^\times$ and $\Phi\in\mathsf M_\mu(\mathbf k)$. The following 
$$
\xymatrix{(\mathbf k P_5^*)^\wedge\ar_{\mathrm{comp}_{(\mu,\Phi)}^{((\bullet\bullet)\bullet)\bullet}}[d]
\ar^{\underline{\mathrm{pr}}_1^\wedge}[rrr]& & & 
\hat{\mathcal V}^\B \ar^{\mathrm{comp}_{(\mu,\Phi)}^{\mathcal V,(1)}}[d]\\ (U{\mathfrak p}_5)^\wedge
\ar_{\mathrm{pr}_1^\wedge}[r]& \hat{\mathcal V}^\DR && \hat{\mathcal V}^\DR \ar^{\mathrm{Ad}(\Phi(e_0,e_1)^{-1})}[ll]} 
\quad
\xymatrix{ (\mathbf kP_5^*)^\wedge\ar_{\mathrm{comp}_{(\mu,\Phi)}^{((\bullet\bullet)\bullet)\bullet}}[d]
\ar^{\underline{\mathrm{pr}}_2^\wedge}[rrr]&& & \hat{\mathcal V}^\B\ar^{\mathrm{comp}_{(\mu,\Phi)}^{\mathcal V,(1)}}[d]\\ (U{\mathfrak p}_5)^\wedge
\ar_{\mathrm{pr}_2^\wedge}[r]& \hat{\mathcal V}^\DR
& &\hat{\mathcal V}^\DR\ar^{\mathrm{Ad}(\Phi(e_\infty,e_1)^{-1}e^{(\mu/2)e_1})}[ll] 
} 
$$
are commutative diagrams of topological Hopf algebras.  
\end{lem}

\proof One computes 
\begin{align}\label{comp:x34:20191203}
& \mathrm{comp}_{(\mu,\Phi)}^{((\bullet\bullet)\bullet)\bullet}(x_{34})=
\mathrm{comp}_{(\mu,\Phi)}((x_{34})_{((\bullet\bullet)\bullet)\bullet})
\\ & \nonumber =\mathrm{comp}_{(\mu,\Phi)}((a_{\bullet\bullet,\bullet,\bullet})^{-1}
\circ(x_{34})_{(\bullet\bullet)(\bullet\bullet)}\circ a_{\bullet\bullet,\bullet,\bullet})
=\Phi(e_{12,3},e_{34})^{-1}e^{\mu e_{34}}\Phi(e_{12,3},e_{34}). 
\end{align}

It follows from Proposition \ref{prop:pres:P5} that the elements $x_{i,i+1}$, $i\in C_5$ generate $P_5^*$. So it suffices to check 
the commutativity of the diagrams on these elements. Using \eqref{image:partI:x15} and \eqref{image:partI:x45}, 
\eqref{comp:image:x12:5dec},  \eqref{comp:image:x23:5dec} and \eqref{comp:x34:20191203}, the images of these generators 
under the two maps can be shown to be given by the following table. 
\begin{center} 
\begin{tabular}{|c|c|c|c|c|c|c|c|c|c|c|}
 \hline
generator $x$ of  $P_5^*$ & $x_{12}$ &  $x_{23}$&   $x_{34}$&   $x_{45}$&   $x_{15}$
\\  \hline
$\begin{array}{r}\text{common value of} \\ \mathrm{Ad}(\Phi(e_0,e_1)^{-1})\circ\mathrm{comp}_{(\mu,\Phi)}^{\mathcal V,(1)}
\circ\underline{\mathrm{pr}}_1^\wedge(x) \\ \text{and } 
\mathrm{pr}_1^\wedge\circ\mathrm{comp}_{(\mu,\Phi)}^{((\bullet\bullet)\bullet)\bullet}(x) \end{array}$
& $1$ & $e^{\mu e_0}$ & $\Phi(e_0,e_1)^{-1}e^{\mu e_1}\Phi(e_0,e_1)$ & $e^{\mu e_0}$ & 1
\\ \hline   
\end{tabular}
\end{center} 
This implies that the first diagram commutes. 

The situation in the case of the second diagram is given by the following table  
\begin{center} 
\begin{tabular}{|c|c|c|c|c|c|c|c|c|c|c|}
 \hline
generator $x$ of  $P_5^*$ & $x_{12}$ &  $x_{23}$&   $x_{34}$&   $x_{45}$&   $x_{15}$
\\  \hline
$\begin{array}{r}\mathrm{Ad}(\Phi(e_\infty,e_1)^{-1}e^{(\mu/2)e_1})\\ 
\circ\mathrm{comp}_{(\mu,\Phi)}^{\mathcal V,(1)}\circ\underline{\mathrm{pr}}_2^\wedge(x)\end{array}$
& $1$ & $1$ & $\begin{array}{r}\Phi(e_\infty,e_1)^{-1}e^{\mu e_1}\cdot \\ \cdot \Phi(e_\infty,e_1)\end{array}$ & $\star$ & 
$\begin{array}{r}\Phi(e_\infty,e_1)^{-1}e^{\mu e_1}\cdot \\ \cdot \Phi(e_\infty,e_1)\end{array}$
\\ \hline   
$\mathrm{pr}_2^\wedge\circ\mathrm{comp}_{(\mu,\Phi)}^{((\bullet\bullet)\bullet)\bullet}(x)$
& $1$ & $1$ & $\begin{array}{r}\Phi(e_\infty,e_1)^{-1}e^{\mu e_1}\cdot \\ \cdot \Phi(e_\infty,e_1)\end{array}$ & $e^{\mu e_\infty}$ & 
$\begin{array}{r}\Phi(e_\infty,e_1)^{-1}e^{\mu e_1}\cdot \\ \cdot \Phi(e_\infty,e_1)\end{array}$
\\ \hline  
\end{tabular}
\end{center} 
where $\star=\Phi(e_\infty,e_1)^{-1}e^{-(\mu/2)e_1}\Phi(e_0,e_1)e^{-\mu e_0}\Phi(e_0,e_1)^{-1}e^{-(\mu/2)e_1}\Phi(e_\infty,e_1)$. 

By the hexagon identity \eqref{hexagon+}, one has 
$$
e^{(\mu/2)t_{23}}\Phi(t_{12},t_{23})e^{(\mu/2)t_{12}}=\Phi(t_{13},t_{23})e^{(\mu/2)(t_{12}+t_{23})}\Phi(t_{12},t_{13})
$$ 
and by the exchange of $t_{12}$ and $t_{23}$, also 
$$
e^{(\mu/2)t_{12}}\Phi(t_{23},t_{12})e^{(\mu/2)t_{23}}=\Phi(t_{13},t_{12})e^{(\mu/2)(t_{12}+t_{23})}\Phi(t_{23},t_{13}). 
$$
Taking the product of the last equality with the previous one and using the 2-cycle relation, we get 
$$
e^{(\mu/2)t_{12}}\Phi(t_{23},t_{12})e^{\mu t_{23}}
\Phi(t_{12},t_{23})e^{(\mu/2)t_{12}}
=\Phi(t_{13},t_{12})e^{\mu(t_{12}+t_{23})}
\Phi(t_{12},t_{13})
$$
Inverting, using the 2-cycle identity and conjugating by $\Phi(t_{13},t_{12})^{-1}e^{-(\mu/2)t_{12}}e^{(\mu/2)t_{12}}$, 
we get 
\begin{equation}\label{lunettes}
\Phi(t_{13},t_{12})^{-1}e^{-(\mu/2)t_{12}}\Phi(t_{12},t_{23})^{-1}e^{-\mu t_{23}}\Phi(t_{23},t_{12})^{-1}e^{-\mu t_{12}}
(\Phi(t_{13},t_{12})^{-1}e^{-(\mu/2)t_{12}})^{-1}=e^{-\mu t_{2,13}}. 
\end{equation}
Taking the image of the resulting identity by the 
morphism $(U\mathfrak t_3)^\wedge\to\hat{\mathcal V}^\DR$ given by $t_{12}\mapsto e_1$, $t_{23}\mapsto e_0$, 
$t_{13}\mapsto e_\infty$ and using the 2-cycle identity, one obtains the equality 
$\star=e^{\mu e_\infty}$. It follows that the second diagram commutes. 
\hfill \qed\medskip 

\begin{lem}\label{lemma:83:12122017} 
Let $\mu\in\mathbf k^\times$ and $\Phi\in\mathsf M_\mu(\mathbf k)$. The following diagram commutes 
$$
\xymatrix{M_3((\mathbf k P_5^*)^\wedge)\ar_{M_3(\mathrm{comp}_{(\mu,\Phi)}^{((\bullet\bullet)\bullet)\bullet})}[d]
\ar^{M_3(\underline{\mathrm{pr}}_{12}^\wedge)}[rr]
& & M_3((\mathcal V^\B)^{\otimes2\wedge})\ar^{M_3((\mathrm{comp}_{(\mu,\Phi)}^{\mathcal V,(1)})^{\otimes2})}[d]\\ 
M_3((U{\mathfrak p}_5)^\wedge)
\ar_{M_3(\mathrm{pr}_{12}^\wedge)}[r]& M_3((\mathcal V^\DR)^{\otimes2\wedge}) \ar_{\mathrm{Ad}(\kappa_{(\mu,\Phi)})}[r]& 
M_3((\mathcal V^\DR)^{\otimes2\wedge}) 
} 
$$
where 
\begin{equation}\label{definition:of:Xi}
\kappa_{(\mu,\Phi)}:=\Phi(e_0,e_1)\cdot e^{-(\mu/2)f_1}\Phi(f_\infty,f_1)\in(\mathcal V^\DR)^{\otimes2\wedge},  
\index{kappa_mu,Phi@$\kappa_{(\mu,\Phi)}$}
\end{equation} 
and $\mathrm{Ad}(\kappa_{(\mu,\Phi)})$ denotes the automorphism taking each entry of a matrix to its image by the automorphism 
$\mathrm{Ad}(\kappa_{(\mu,\Phi)})$ of $(\mathcal V^\DR)^{\otimes2\wedge}$. 
\end{lem}

\proof Combining the tensor product of the two diagrams from Lemma \ref{lemma:CD:29122917} with the diagram expressing the 
compatibility of $\mathrm{comp}_{(\mu,\Phi)}^{((\bullet\bullet)\bullet)\bullet}$ with the coproducts of its source and 
target, one obtains the following commutative diagram
$$
\xymatrix{(\mathbf k P_5^*)^\wedge\ar_{\mathrm{comp}_{(\mu,\Phi)}^{((\bullet\bullet)\bullet)\bullet}}[d]
\ar^{\underline{\mathrm{pr}}_{12}^\wedge}[rrrrr]& & &&& (\mathcal V^\B)^{\otimes2\wedge}
\ar^{(\mathrm{comp}_{(\mu,\Phi)}^{\mathcal V,(1)})^{\otimes2}}[d]\\ (U{\mathfrak p}_5)^\wedge
\ar_{\mathrm{pr}_{12}^\wedge}[r]& (\mathcal V^\DR)^{\otimes2\wedge} &&&& (\mathcal V^\DR)^{\otimes2\wedge}
\ar^{\mathrm{Ad}(\Phi(e_0,e_1)^{-1}\Phi(f_\infty,f_1)^{-1}e^{(\mu/2)f_1})}[llll]} 
$$
from where we derive the announced commutative diagram.  \hfill\qed\medskip 

\subsection{Commutative diagram relating $\hat\AAA$ and \underline{$\hat\AAA$}}\label{subsection8:3:21032018}

\begin{lem}\label{12122017Lemma84}
Let $\mu\in\mathbf k^\times$ and $\Phi\in\mathsf M_\mu(\mathbf k)$. The following diagram is commutative 
\begin{equation}
\xymatrix{
(\mathbf k P_5^*)^\wedge \ar^{\underline{\hat\AAA}}[rr]\ar_{\mathrm{comp}_{(\mu,\Phi)}^{((\bullet\bullet)\bullet)\bullet}}[d]& & 
M_3((\mathbf k P_5^*)^\wedge)\ar^{M_3(\mathrm{comp}_{(\mu,\Phi)}^{((\bullet\bullet)\bullet)\bullet})}[d]\\ 
(U\mathfrak p_5)^\wedge\ar_{\hat\AAA}[r]& M_3((U\mathfrak p_5)^\wedge)\ar_{\mathrm{Ad}(P_{(\mu,\Phi)})}[r]
&M_3((U\mathfrak p_5)^\wedge)}
\end{equation}
where $\hat\AAA\index{pivar^hat@$\hat\AAA$},\underline{\hat\AAA}\index{pivar_underline^hat@$\underline{\hat\AAA}$}$ are the 
completions of $\AAA,\underline\AAA$ (see \S\S\ref{sect:completion:giadrhc}, \ref{sect:compl:Betti:30oct}) and 
$P_{(\mu,\Phi)}\in\mathrm{GL}_3((U\mathfrak f_3)^\wedge)\subset\mathrm{GL}_3((U\mathfrak p_5)^\wedge)$ is 
given by Definition \ref{def:P:R:20191203}. 
\end{lem}

\proof This follows from Lemma \ref{lemma:ideals:and:morphisms}, with $R=(\mathbf kP_5^*)^\wedge$, 
$J=J(\underline{\mathrm{pr}}_5)^\wedge$, $(j_a)_{a\in[\![1,d]\!]}=(x_{i5}-1)_{i\in[\![1,3]\!]}$, 
$R'=(U\mathfrak p_5)^\wedge$, $J'=J(\mathrm{pr}_5)^\wedge$, $(j'_a)_{a\in[\![1,d]\!]}=(e_{i5})_{i\in[\![1,3]\!]}$, 
$f=\mathrm{comp}_{(\mu,\Phi)}^{((\bullet\bullet)\bullet)\bullet}$. \hfill\qed\medskip 

\subsection{Commutative diagram relating $\hat{\BB}$ and \underline{$\hat{\BB}$} ((A3) in \eqref{diagg:main:alg})}\label{sect:8:5:12122017}

\begin{lem}\label{lem:CD:relating:B:underlineB}
Let $\mu\in\mathbf k^\times$ and $\Phi\in\mathsf M_\mu(\mathbf k)$ and let $\kappa_{(\mu,\Phi)}\cdot 
\overline P_{(\mu,\Phi)}\in\mathrm{GL}_3((\mathcal V^\DR)^{\otimes2\wedge})$ be the matrix obtained 
by the left multiplication by the scalar $\kappa_{(\mu,\Phi)}\in((\mathcal V^\DR)^{\otimes2\wedge})^\times$ 
(see \eqref{definition:of:Xi}) of the entries of the matrix $\overline P_{(\mu,\Phi)}\in\mathrm{GL}_3((\mathcal V^\DR)^{\otimes2\wedge})$ 
(see Definition \ref{def:barP:barR}). The following diagram commutes 
\begin{equation}
\xymatrix{ \hat{\mathcal V}^\B\ar^{\underline{\hat{\BB}}}[rr]\ar_{\mathrm{comp}_{(\mu,\Phi)}^{\mathcal V,(1)}}[d]& & 
M_3((\mathcal V^\B)^{\otimes2\wedge})\ar^{M_3((\mathrm{comp}_{(\mu,\Phi)}^{\mathcal V,(1)})^{\otimes2})}[d]\\ 
\hat{\mathcal V}^\DR\ar_{\!\!\!\!\!\!\!\!\!\!\!\!\!\!\!\!\!\!\hat{\BB}}[r]& M_3((\mathcal V^\DR)^{\otimes2\wedge})
\ar_{\mathrm{Ad}(\kappa_{(\mu,\Phi)}\cdot \overline P_{(\mu,\Phi)})}[r] 
& M_3((\mathcal V^\DR)^{\otimes2\wedge})}
\end{equation} 
where $\hat\rho$, $\underline{\hat\rho}$ are as in  \S\S\ref{sect:completion:giadrhc}, \ref{sect:compl:Betti:30oct}. 
\end{lem}

\proof One checks that the following diagram commutes 
$$
\xymatrix{
M_3((U\mathfrak p_5)^\wedge)\ar^{\mathrm{Ad}(P_{(\mu,\Phi)})}[r]\ar_{M_3(\mathrm{pr}_{12}^\wedge)}[rrrd]& 
M_3((U\mathfrak p_5)^\wedge)
\ar^{M_3(\mathrm{pr}_{12}^\wedge)}[r]
& M_3((\mathcal V^\DR)^{\otimes2\wedge})\ar^{\mathrm{Ad}(\kappa_{(\mu,\Phi)})}[r]& 
M_3((\mathcal V^\DR)^{\otimes2\wedge})\\ 
& & & M_3((\mathcal V^\DR)^{\otimes2\wedge})\ar_{\mathrm{Ad}(\kappa_{(\mu,\Phi)}\cdot\overline P_{(\mu,\Phi)})}[u]}.
$$
The result follows from the juxtaposition of this diagram and of the diagrams from Lemmas \ref{lemma:52:18:08:2017}, 
\ref{lemma:83:12122017} and \ref{12122017Lemma84}. \hfill \qed\medskip 

\subsection{Relationship between $\mathrm{col}_1$, \underline{$\mathrm{col}$}$_1$ and $\overline P_{(\mu,\Phi)}$}
\label{sect:rel:col:barcol:21032018}

\begin{lem}\label{lemma:P:col}
Let $\mu\in\mathbf k^\times$ and $\Phi\in\mathsf M_\mu(\mathbf k)$. One has 
\begin{equation}\label{equality:13122017}
(\mathrm{comp}_{(\mu,\Phi)}^{\mathcal V,(1)})^{\otimes2}(\underline{\mathrm{col}}_1)=\kappa_{(\mu,\Phi)}\cdot\overline P_{(\mu,\Phi)}
\cdot\mathrm{col}_1\cdot v_{(\mu,\Phi)}
\end{equation}
(equality in $M_{3\times1}((\mathcal V^\DR)^{\otimes2\wedge})$), where $\mathrm{col}_1$, $\underline{\mathrm{col}}_1$ 
are as in \eqref{def:row:col:04012018}, \eqref{def:underline:row:col:04012018}, 
$\kappa_{(\mu,\Phi)}$, $\overline P_{(\mu,\Phi)}$
are as in \eqref{definition:of:Xi}, Definition \ref{def:barP:barR} and where  
\begin{equation}\label{def:v:mu:Phi}
v_{(\mu,\Phi)}:={1\over\mu}e^{\mu f_1}{{\Gamma_\Phi(e_1)\Gamma_\Phi(f_1)}\over{\Gamma_\Phi(e_1+f_1)}}\in(\mathcal V^\DR)^{\otimes2\wedge}. 
\index{v_mu,Phi@$v_{(\mu,\Phi)}$}
\end{equation}
\end{lem}

\proof Let $d_{(\mu,\Phi)}$ and $\tilde P_{(\mu,\Phi)}$ be the matrices of the right-hand side of equality of Corollary 
\ref{cor:equation:barP}, so that $\overline P_{(\mu,\Phi)}=d_{(\mu,\Phi)}\cdot\tilde P_{(\mu,\Phi)}$. One computes 
\begin{align*}
&\tilde P_{(\mu,\Phi)}\cdot\mathrm{col}_1=\begin{pmatrix} (e^{\mu f_1}-1)\Big(\varphi_1(-(e_1+f_1),f_1)+{1\over f_1}\Big)\\
-(e^{\mu e_1}-1)\Big(\varphi_1(-(e_1+f_1),e_1)+{1\over e_1}\Big) \\ 0 \end{pmatrix}
=\begin{pmatrix} {{e^{\mu f_1}-1}\over{f_1}}{{\Gamma_\Phi(e_1+f_1)\Gamma_\Phi(-f_1)}\over{\Gamma_\Phi(e_1)}}\\
-{{e^{\mu e_1}-1}\over{e_1}}{{\Gamma_\Phi(e_1+f_1)\Gamma_\Phi(-e_1)}\over{\Gamma_\Phi(f_1)}} \\ 0 \end{pmatrix}
\\ & 
=\begin{pmatrix} \mu e^{(\mu/2)f_1}\\-\mu e^{(\mu/2)e_1}\\ 0 \end{pmatrix}
{{\Gamma_\Phi(e_1+f_1)}\over{\Gamma_\Phi(e_1)\Gamma_\Phi(f_1)}}, 
\end{align*}
where the second equality follows from $\varphi_1(\alpha,\beta)={1\over\beta}\Big({{\Gamma_\Phi(-\alpha)\Gamma_\Phi(-\beta)}\over{\Gamma_\Phi(-\alpha-\beta)}}-1\Big)$ for $\alpha,\beta$ commutative formal variables (see \eqref{identity:Gamma:1/12/2017}), the third equality follows from the functional equation \eqref{funct:id:Gamma:Phi} for $\Gamma_\Phi$. Therefore 
$$
\tilde P_{(\mu,\Phi)}\cdot\mathrm{col}_1\cdot v_{(\mu,\Phi)}=
\begin{pmatrix}  e^{(3\mu/2)f_1}\\- e^{(\mu/2)e_1+\mu f_1}\\ 0 \end{pmatrix}.
$$ 
On the other hand, $\kappa_{(\mu,\Phi)}d_{(\mu,\Phi)}$ is a diagonal matrix of the form $\mathrm{diag}(e^{-(\mu/2)f_1},e^{-(\mu/2)e_1-\mu f_1},*)$, 
with $*\in(\mathcal V^{\DR})^{\otimes2\wedge}$. Then 
$$
\kappa_{(\mu,\Phi)}\overline P_{(\mu,\Phi)}\cdot\mathrm{col}_1\cdot v_{(\mu,\Phi)}=\kappa_{(\mu,\Phi)}d_{(\mu,\Phi)}\cdot
\tilde P_{(\mu,\Phi)}\cdot\mathrm{col}_1\cdot v_{(\mu,\Phi)}=\begin{pmatrix}  e^{\mu f_1}\\-1\\ 0 \end{pmatrix}
=(\mathrm{comp}_{(\mu,\Phi)}^{\mathcal V,(1)})^{\otimes2}(\underline{\mathrm{col}}_1).
$$
\hfill\qed\medskip 

\subsection{Relationship between $\mathrm{row}_1$, \underline{$\mathrm{row}$}$_1$ and $\overline P_{(\mu,\Phi)}$}
\label{sect:8:6:row:bar:row:21032018}

\begin{lem}
Let $\mu\in\mathbf k^\times$ and $\Phi\in\mathsf M_\mu(\mathbf k)$. One has 
\begin{equation}\label{newequality:13122017}
M_{1\times3}((\mathrm{comp}_{(\mu,\Phi)}^{\mathcal V,(1)})^{\otimes2})(\underline{\mathrm{row}}_1)\cdot\kappa_{(\mu,\Phi)}\cdot\overline P_{(\mu,\Phi)}
=u_{(\mu,\Phi)}\cdot\mathrm{row}_1 
\end{equation}
(equality in $M_{1\times3}((\mathcal V^\DR)^{\otimes2\wedge})$), where 
$\mathrm{row}_1$, $\underline{\mathrm{row}}_1$ 
are as in \eqref{def:row:col:04012018}, \eqref{def:underline:row:col:04012018}, 
$\kappa_{(\mu,\Phi)}$, $\overline P_{(\mu,\Phi)}$
are as in \eqref{definition:of:Xi}, Definition \ref{def:barP:barR}, and where
\begin{equation}\label{def:u:mu:Phi}
u_{(\mu,\Phi)}:=
\mu e^{-\mu f_1} {e^{\mu(e_1+f_1)}-1 \over e_1+f_1}
{\Gamma_\Phi(e_1+f_1)\over\Gamma_\Phi(e_1)\Gamma_\Phi(f_1)}
\in(\mathcal V^\DR)^{\otimes2\wedge}. 
\index{u_mu,Phi@$u_{(\mu,\Phi)}$}
\end{equation}
\end{lem}

\proof One has $M_{1\times3}((\mathrm{comp}_{(\mu,\Phi)}^{\mathcal V,(1)})^{\otimes2})(\underline{\mathrm{row}}_1)=
\begin{pmatrix} e^{\mu e_1}-1 & 1-e^{\mu f_1} & 0 \end{pmatrix}$. With $d_{(\mu,\Phi)}$ as in the proof of 
Lemma \ref{lemma:P:col}, this implies  
$$
M_{1\times3}((\mathrm{comp}_{(\mu,\Phi)}^{\mathcal V,(1)})^{\otimes2})
(\underline{\mathrm{row}}_1)\cdot \kappa_{(\mu,\Phi)}d_{(\mu,\Phi)}=
e^{-(\mu/2)f_1}\begin{pmatrix} e^{\mu e_1}-1 & -e^{-(\mu/2)(e_1+f_1)}(e^{\mu f_1}-1) & 0 \end{pmatrix}. 
$$
Then with $\tilde P_{(\mu,\Phi)}$ as in the proof of Lemma \ref{lemma:P:col}, 
\begin{align}\label{eq:interm:row}
& 
M_{1\times3}((\mathrm{comp}_{(\mu,\Phi)}^{\mathcal V,(1)})^{\otimes2})
(\underline{\mathrm{row}}_1)\cdot \kappa_{(\mu,\Phi)}\cdot\overline P_{(\mu,\Phi)} 
=M_{1\times3}((\mathrm{comp}_{(\mu,\Phi)}^{\mathcal V,(1)})^{\otimes2})
(\underline{\mathrm{row}}_1)\cdot \kappa_{(\mu,\Phi)}d_{(\mu,\Phi)}\cdot\tilde P_{(\mu,\Phi)}
\\ & \nonumber
=e^{-(\mu/2)f_1}\begin{pmatrix} e^{\mu e_1}-1 & -e^{-(\mu/2)(e_1+f_1)}(e^{\mu f_1}-1) & 0 \end{pmatrix}\cdot\tilde P_{(\mu,\Phi)}. 
\\ & \nonumber
=e^{-(\mu/2)f_1}(e^{\mu e_1}-1)(e^{\mu f_1}-1)\cdot
\\ & \nonumber
\cdot\begin{pmatrix}
\begin{matrix}
\scriptstyle{(\varphi_1-\varphi_0)(-(e_1+f_1),f_1)+e^{-(\mu/2)(e_1+f_1)}\varphi_0(-(e_1+f_1),e_1)} \\ 
\scriptstyle{+{1\over f_1}+e^{-(\mu/2)(e_1+f_1)}{{e^{(\mu/2)(e_1+f_1)}-1}\over{-e_1-f_1}}}
\end{matrix}
& 
\begin{matrix}
\scriptstyle{-\varphi_0(-(e_1+f_1),f_1)-e^{-(\mu/2)(e_1+f_1)}(\varphi_1-\varphi_0)(-(e_1+f_1),e_1)} \\ 
\scriptstyle{-{e^{-(\mu/2)(e_1+f_1)}\over e_1}+e^{-(\mu/2)(e_1+f_1)}{{e^{(\mu/2)(e_1+f_1)}-1}\over{-e_1-f_1}}}
\end{matrix}& 0 \end{pmatrix}
\end{align}

%

Then 
\begin{align*}
& (\varphi_1-\varphi_0)(-(e_1+f_1),f_1)+e^{-(\mu/2)(e_1+f_1)}\varphi_0(-(e_1+f_1),e_1)
+{1\over f_1}+e^{-(\mu/2)(e_1+f_2)}{{e^{(\mu/2)(e_1+f_2)}-1}\over{-e_1-f_1}}
\\ & =
({1\over-(e_1+f_1)}+{1\over f_1})\Big({{\Gamma_\Phi(e_1+f_1)\Gamma_\Phi(-f_1)}\over{\Gamma_\Phi(e_1)}}\Big)
+e^{-(\mu/2)(e_1+f_1)}\cdot{1\over e_1+f_1}\Big({{\Gamma_\Phi(e_1+f_1)\Gamma_\Phi(-e_1)}\over{\Gamma_\Phi(f_1)}}\Big)
\\ &= 
{\Gamma_\Phi(e_1+f_1)\over\Gamma_\Phi(e_1)\Gamma_\Phi(f_1)}\cdot\Big(
({1\over-(e_1+f_1)}+{1\over f_1})\Gamma_\Phi(f_1)\Gamma_\Phi(-f_1)
+e^{-(\mu/2)(e_1+f_1)}\cdot{1\over e_1+f_1}\Gamma_\Phi(e_1)\Gamma_\Phi(-e_1)
\Big)
\\ &= 
{\Gamma_\Phi(e_1+f_1)\over\Gamma_\Phi(e_1)\Gamma_\Phi(f_1)}\cdot\Big(
({1\over-(e_1+f_1)}+{1\over f_1}){\mu f_1 e^{(\mu/2)f_1}\over{e^{\mu f_1}-1}}
+e^{-(\mu/2)(e_1+f_1)}\cdot{1\over e_1+f_1}{\mu e_1 e^{(\mu/2)e_1}\over{e^{\mu e_1}-1}}
\Big)
\\ & = 
{\Gamma_\Phi(e_1+f_1)\over\Gamma_\Phi(e_1)\Gamma_\Phi(f_1)}\cdot {e^{-(\mu/2)f_1}\over{(e^{\mu e_1}-1)(e^{\mu f_1}-1)}}\cdot
{\mu e_1\over e_1+f_1}(e^{\mu(e_1+f_1)}-1). 
\end{align*}
where the first equality follows from \eqref{ids:phi:ab} and the third equality follows from \eqref{funct:id:Gamma:Phi}; 
similarly 
 \begin{align*}
& -\varphi_0(-(e_1+f_1),f_1)-e^{-(\mu/2)(e_1+f_1)}(\varphi_1-\varphi_0)(-(e_1+f_1),e_1)
-{e^{-(\mu/2)(e_1+f_1)}\over e_1}\\ & +e^{-(\mu/2)(e_1+f_1)}{{e^{(\mu/2)(e_1+f_1)}-1}\over{-e_1-f_1}}
\\ & ={1\over-(e_1+f_1)}{{\Gamma_\Phi(e_1+f_1)\Gamma_\Phi(-f_1)}\over{\Gamma_\Phi(e_1)}}
-e^{-(\mu/2)(e_1+f_1)}({1\over-(e_1+f_1)}+{1\over e_1}){{\Gamma_\Phi(e_1+f_1)\Gamma_\Phi(-e_1)}\over{\Gamma_\Phi(f_1)}}
\\ &=
{\Gamma_\Phi(e_1+f_1)\over\Gamma_\Phi(e_1)\Gamma_\Phi(f_1)}\cdot\Big(
{1\over-(e_1+f_1)}{\mu f_1\over{e^{(\mu/2)f_1}-e^{-(\mu/2)f_1}}}
-e^{-(\mu/2)(e_1+f_1)}({1\over-(e_1+f_1)}+{1\over e_1}){\mu e_1\over{e^{(\mu/2)e_1}-e^{-(\mu/2)e_1}}}\Big)
\\ &=
{\Gamma_\Phi(e_1+f_1)\over\Gamma_\Phi(e_1)\Gamma_\Phi(f_1)}\cdot {e^{-(\mu/2)f_1}\over{(e^{\mu e_1}-1)(e^{\mu f_1}-1)}}\cdot 
{-\mu f_1\over e_1+f_1} (e^{\mu(e_1+f_1)}-1),  
\end{align*}
where the first equality follows from \eqref{ids:phi:ab}, the second equality follows from  \eqref{funct:id:Gamma:Phi}.  
It follows that the right-hand side of \eqref{eq:interm:row} is equal to 
$$
e^{-\mu f_1}\cdot
{\Gamma_\Phi(e_1+f_1)\over\Gamma_\Phi(e_1)\Gamma_\Phi(f_1)}
{e^{\mu(e_1+f_1)}-1 \over e_1+f_1}
\begin{pmatrix}
\mu e_1& {-\mu f_1} & 0 
\end{pmatrix}=u_{(\mu,\Phi)}\cdot\mathrm{row}_1. 
$$
\hfill\qed\medskip 

\subsection{Commutative diagram relating $\mathrm{row}_1\cdot(-)\cdot\mathrm{col}_1$ and \underline{$\mathrm{row}$}$_1\cdot(-)\cdot$\underline{$\mathrm{col}$}$_1$ ((A4) in \eqref{diagg:main:alg})}\label{CD:row:col:barrow:barcol:21032018}

\begin{lem}\label{lemma:CD:04012018}
Let $\mu\in\mathbf k^\times$ and $\Phi\in\mathsf M_\mu(\mathbf k)$. The following diagram of $\mathbf k$-module morphisms is commutative, 
where $u\cdot(-)\cdot v$ is the linear map $x\mapsto uxv$ 
$$
\xymatrix{M_3((\mathcal V^\B)^{\otimes2\wedge}) \ar^{\underline{\mathrm{row}}_1\cdot(-)\cdot\underline{\mathrm{col}}_1}[rr]
\ar^{\simeq}_{M_3((\mathrm{comp}_{(\mu,\Phi)}^{\mathcal V,(1)})^{\otimes2})}[d]&& (\mathcal V^\B)^{\otimes2\wedge}
\ar_{\simeq}^{(\mathrm{comp}_{(\mu,\Phi)}^{\mathcal V,(1)})^{\otimes2}}[d]\\ 
M_3((\mathcal V^\DR)^{\otimes2\wedge})& &(\mathcal V^\DR)^{\otimes2\wedge}\\ 
M_3((\mathcal V^\DR)^{\otimes2\wedge})\ar_{\mathrm{row}_1\cdot(-)\cdot\mathrm{col}_1}[rr]
\ar^{\mathrm{Ad}(\kappa_{(\mu,\Phi)}\cdot\overline P_{(\mu,\Phi)})}_{\simeq}[u]&& 
(\mathcal V^\DR)^{\otimes2\wedge}\ar^{\simeq}_{u_{(\mu,\Phi)}\cdot(-)\cdot v_{(\mu,\Phi)}}[u] }
$$
\end{lem}

\proof This follows from the juxtaposition of the obviously commutative diagram  
$$
\xymatrix{M_3((\mathcal V^\B)^{\otimes2\wedge}) \ar^{\underline{\mathrm{row}}_1\cdot(-)\cdot\underline{\mathrm{col}}_1}[rrrrrrr]
\ar^{\simeq}_{M_3((\mathrm{comp}_{(\mu,\Phi)}^{\mathcal V,(1)})^{\otimes2})}[d]&&&&&&& (\mathcal V^\B)^{\otimes2\wedge}
\ar_{\simeq}^{(\mathrm{comp}_{(\mu,\Phi)}^{\mathcal V,(1)})^{\otimes2}}[d]\\ 
M_3((\mathcal V^\DR)^{\otimes2\wedge})\ar_{M_{1\times3}((\mathrm{comp}_{(\mu,\Phi)}^{\mathcal V,(1)})^{\otimes2})
(\underline{\mathrm{row}}_1)\cdot(-)\cdot
M_{3\times1}((\mathrm{comp}_{(\mu,\Phi)}^{\mathcal V,(1)})^{\otimes2})(\underline{\mathrm{col}}_1)}[rrrrrrr]&&&&& &&(\mathcal V^\DR)^{\otimes2\wedge}}
$$
with the following diagram
$$
\xymatrix{
M_3((\mathcal V^\DR)^{\otimes2\wedge})\ar^{M_{1\times3}((\mathrm{comp}_{(\mu,\Phi)}^{\mathcal V,(1)})^{\otimes2})
(\underline{\mathrm{row}}_1)\cdot(-)\cdot
M_{3\times1}((\mathrm{comp}_{(\mu,\Phi)}^{\mathcal V,(1)})^{\otimes2})(\underline{\mathrm{col}}_1)}[rrrrrrr]&&&&&& &(\mathcal V^\DR)^{\otimes2\wedge}
\\ 
M_3((\mathcal V^\DR)^{\otimes2\wedge})\ar_{\mathrm{row}_1\cdot(-)\cdot\mathrm{col}_1}[rrrrrrr]
\ar^{\mathrm{Ad}(\kappa_{(\mu,\Phi)}\cdot\overline P_{(\mu,\Phi)})}_{\simeq}[u]&&&&&&& (\mathcal V^\DR)^{\otimes2\wedge}
\ar^{\simeq}_{u_{(\mu,\Phi)}\cdot(-)\cdot v_{(\mu,\Phi)}}[u]
}
$$
whose commutativity follows from (\ref{equality:13122017}) and (\ref{newequality:13122017}). 
\hfill\qed\medskip

\subsection{Associators and harmonic algebra coproducts
}\label{sect:8:8:21:03:2018}

\begin{thm}\label{thm:compat:assoc:algebra:24dec2019}
Let $\mu\in\mathbf k^\times$ and $\Phi\in\mathsf M_\mu(\mathbf k)$. Set 
\begin{equation}\label{def:B:Phi}
B_\Phi:={{\Gamma_\Phi(-e_1)\Gamma_\Phi(-f_1)}\over{\Gamma_\Phi(-e_1-f_1)}}
\in((\mathcal W^\DR)^{\otimes2\wedge})^\times.
\index{B_Phi@$B_\Phi$} 
\end{equation}
The following diagram commutes
\begin{equation}\label{main:diag:alg:10dec2019}
\xymatrix{ \hat{\mathcal W}^\B\ar^{\hat\Delta^{\mathcal W,\B}}[rr]
\ar_{\mathrm{comp}_{(\mu,\Phi)}^{\mathcal W,(1)}}[d]& & (\mathcal W^\B)^{\otimes2\wedge}
\ar^{(\mathrm{comp}_{(\mu,\Phi)}^{\mathcal W,(1)})^{\otimes2}}[d]\\ 
 \hat{\mathcal W}^\DR\ar_{\hat\Delta^{\mathcal W,\DR}}[r]& (\mathcal W^\DR)^{\otimes2\wedge}
 \ar_{\mathrm{Ad}(B_\Phi)}[r]&
 (\mathcal W^\DR)^{\otimes2\wedge} }
\end{equation} 
\end{thm}

\proof

If $(\mu,g)\in\mathbf k^\times\times(\hat{\mathcal V}^\DR)^\times$, there is a unique extension 
of the algebra automorphism $\mathrm{aut}^{\mathcal V,(1),\DR}_{(\mu,g)}$ of $\hat{\mathcal V}^\DR$ 
(see \S\ref{sect:3:1:28oct}) to an algebra automorphism $\mathrm{aut}^{\mathcal V_{\mathrm{loc}},(1)}_{(\mu,g)}$\index{autVloc1@$\mathrm{aut}^{\mathcal V_{\mathrm{loc}},(1)}_{(\mu,g)}$} of its localization 
$\hat{\mathcal V}^\DR[{1\over e_1}]$ (see \S\ref{sect:def:loc:VDR}). The algebra isomorphism $\mathrm{iso}^{\mathcal V}:
\hat{\mathcal V}^\B\to\hat{\mathcal V}^\DR$ from \S\ref{sect:3:3:28oct} also extends uniquely to an algebra isomorphism 
$\mathrm{iso}^{\mathcal V_{\mathrm{loc}}}:\hat{\mathcal V}^\B[{1\over{X_1-1}}]\to\hat{\mathcal V}^\DR[{1\over e_1}]$\index{isoVloc@$\mathrm{iso}^{\mathcal V_{\mathrm{loc}}}$}.
We denote by 
\begin{equation}\label{def:comp:V:loc}
\mathrm{comp}^{\mathcal V_{\mathrm{loc}},(1)}_{(\mu,g)}:\hat{\mathcal V}^\B[{1\over{X_1-1}}]\to\hat{\mathcal V}^\DR[{1\over e_1}]
\index{compVloc1@$\mathrm{comp}^{\mathcal V_{\mathrm{loc}},(1)}_{(\mu,g)}$}
\end{equation} 
the algebra isomorphism given by the composition $\mathrm{aut}^{\mathcal V_{\mathrm{loc}},(1)}_{(\mu,g)}\circ
\mathrm{iso}^{\mathcal V_{\mathrm{loc}}}$. 
It is compatible with the filtrations on both sides, and in particular it induces an algebra isomorphism between 
$F^0\hat{\mathcal V}^\B[{1\over{X_1-1}}]$ and $\hat{\mathcal V}^\DR[{1\over e_1}]_{\geq0}$. 

Let $\mu\in\mathbf k^\times$ and $\Phi\in\mathsf M_\mu(\mathbf k)$. Consider the diagram 
\begin{equation}\label{diagg:main:alg}
{\tiny\xymatrix{
\hat{\mathcal W}^\B_+
\ar^{{\hat\Delta}^{\mathcal W,\B}}[rrrrrr]\ar^*[@]{\hbox to 15pt {$\mathrm{comp}_{(\mu,\Phi)}^{\mathcal W,(1)}$}}[ddddd]
\ar@{}^*[@]{{\hbox to -10pt{$\!\!\!\!\!\!\!\!\!\!\!\!\!\!\mathrm{mor}_{\hat{\mathcal V}^\B,X_1-1}$}}}[rd] 
\ar@{}[dddddr]|{\mathrm{(A2)}}
&&&&& 
\ar@{}[dllll]|{\mathrm{(A1)}}
& 
(\mathcal W^\B)^{\otimes2\wedge}
\ar^*[@]{\hbox to 15pt {$(\mathrm{comp}_{(\mu,\Phi)}^{\mathcal W,(1)})^{\otimes2}$}}[ddd]\ar@{^(->}[ld] \\
& 
\hat{\mathcal V}^\B
\ar^{\hbox to 20pt{$\underline{\hat\rho}$}}[r]
\ar^*[@]{\vbox to 4pt{\hbox to 4pt {$\!\!\!\!\!\!\mathrm{comp}_{(\mu,\Phi)}^{\mathcal V,(1)}$}}}[dd]
\ar[lu]& 
{\scriptstyle M_3((\mathcal V^\B)^{\otimes2\wedge})}
\ar^{\ \ \ \ \underline{\mathrm{row}}_1\cdot(-)\cdot\underline{\mathrm{col}}_1}[rr]
\ar^{M_3((\mathrm{comp}_{(\mu,\Phi)}^{\mathcal V,(1)})^{\otimes2})}[d] && 
(\mathcal V^\B)^{\otimes2\wedge}
\ar^{\begin{matrix} \mathrm{Ad}((X_1-1)^{-1}\cdot\\ \cdot(1-Y_1^{-1})^{-1})\end{matrix}}[r]
\ar^{(\mathrm{comp}_{(\mu,\Phi)}^{\mathcal V,(1)})^{\otimes2}}[d]
&  
F^0\mathcal V^\B[{1\over{X_1-1}}]^{\otimes2\wedge}
\ar^{(\mathrm{comp}_{(\mu,\Phi)}^{\mathcal V_{\mathrm{loc}},(1)})^{\otimes2}}[d] & \\
&\ar@{}[r]|{\mathrm{\!\!\!\!\!\!(A3)}}& 
{\scriptstyle M_3((\mathcal V^\DR)^{\otimes2\wedge})} 
\ar^{\mathrm{Ad}((\kappa_{(\mu,\Phi)}\cdot\overline P_{(\mu,\Phi)})^{-1})}[d]
\ar@{}[rr]|{\mathrm{(A4)}}
&& 
(\mathcal V^\DR)^{\otimes2\wedge}
\ar^{u_{(\mu,\Phi)}^{-1}\cdot(-)\cdot v_{(\mu,\Phi)}^{-1}}[d]
\ar@{}[r]|{\mathrm{(A5)}}
& 
(\mathcal V^\DR[{1\over{e_1}}]_{\geq0})^{\otimes2\wedge} 
\ar^{\begin{matrix} B_\Phi^{-1}\cdot(-)\cdot B_\Phi\cdot \\ \cdot {{e_1+f_1}\over{{e^{\mu(e_1+f_1)}}-1}}\end{matrix}}[d] 
& \\
& 
\hat{\mathcal V}^\DR
\ar_{\hbox to 20pt{$\hat\rho$}}[r]\ar^{\mathrm{mor}_{\hat{\mathcal V}^\DR,e_1}}[d] & 
{\scriptstyle M_3((\mathcal V^\DR)^{\otimes2\wedge})}
\ar_{\mathrm{row}_1\cdot(-)\cdot\mathrm{col}_1}[rr] && 
(\mathcal V^\DR)^{\otimes2\wedge}
\ar_{\mathrm{Ad}((e_1f_1)^{-1})}[r] 
\ar@{}[dll]|{\mathrm{(A7)}} & 
(\mathcal V^\DR[{1\over e_1}]_{\geq0})^{\otimes2\wedge} 
\ar@{}[dr]|{\mathrm{(A6)}}
&
({\mathcal W}^{\DR})^{\otimes2\wedge}
\ar^*[@]{\hbox to 10pt{$\mathrm{Ad}(B_\Phi^{-1})$}}[dd] 
\\
& 
\hat{\mathcal W}^{\DR}_+
\ar_{\hat\Delta^{\mathcal W,\DR}}[rrrr]
\ar[ld] 
&&&
\ar@{}[dll]|{\mathrm{(A8)}}
& 
(\mathcal W^\DR)^{\otimes2\wedge}
\ar@{^(->}[u]
\ar@{}^*[@]{\hbox to 23pt{\vbox to 0pt {$\!\!\!\!\!\!\!\!\!\!\!\!\!\!\!\!\!\!\!\!\!\!\!(-)\cdot{{e_1+f_1}\over{e^{\mu(e_1+f_1)}-1}}$}}}[rd]  
& \\
\hat{\mathcal W}^{\DR}_+
\ar_{\hat\Delta^{\mathcal W,\DR}}[rrrrrr]
\ar@{}^*[@]{{\hbox to 5pt{$\!\!\!\!\!\!\!\!\!\!\!(-)\cdot{e^{\mu e_1}-1}\over{e_1}$}}}[ru] 
&&&&&&(\mathcal W^{\DR})^{\otimes 2\wedge}\ar[ul]
}}
\end{equation}
It is divided into subdiagrams (A1) to (A8). 
In this diagram: 
\begin{itemize}
\item the commutativity of (A1) (resp. (A3), (A4), (A7)) follows from Lemma \ref{lemma:completion:10dec2019} (resp. 
Lemma \ref{lem:CD:relating:B:underlineB}, Lemma \ref{lemma:CD:04012018}, Lemma \ref{lem:completion:diag:LA}); 

\item the commutativity of (A2) follows from the definition of $\mathrm{mor}_{-,-}$, from the compatibilities of 
$\mathrm{comp}_{(\mu,\Phi)}^{\mathcal V,(1)}$ and $\mathrm{comp}_{(\mu,\Phi)}^{\mathcal W,(1)}$, and from 
the fact that $\mathrm{comp}_{(\mu,\Phi)}^{\mathcal V,(1)}$ is an algebra morphism taking
$X_1-1$ to $e^{\mu e_1}-1$; 

\item the commutativity of (A5) follows from the equalities 
\begin{equation}\label{relation:u:B:17dec2019}
B_\Phi^{-1}\cdot (e^{\mu e_1}-1)^{-1}(1-e^{-\mu f_1})^{-1}
=(e_1f_1)^{-1}u_{(\mu,\Phi)}^{-1}, 
\end{equation}
$$
(e^{\mu e_1}-1)(1-e^{-\mu f_1})\cdot B_\Phi\cdot{{e_1+f_1}\over{e^{\mu(e_1+f_1)}-1}}
=v_{\mu,\Phi}^{-1}e_1f_1
$$
\begin{equation}\label{relation:XY:ef:17dec2019}
(\mathrm{comp}_{(\mu,\Phi)}^{\mathcal V_{\mathrm{loc}},(1)})^{\otimes2}((X_1-1)^{\pm1}(1-Y_1^{-1})^{\pm1})=
(e^{\mu e_1}-1)^{\pm1}(1-e^{-\mu f_1})^{\pm1}
\end{equation} 
in $(\mathcal V^\DR[{1\over e_1}])^{\otimes2\wedge}$, where the two first equalities follow from \eqref{funct:id:Gamma:Phi}, 
and the last equality is immediate;  

\item the commutativity of (A6) follows from the compatibility of 
$\mathrm{comp}_{(\mu,\Phi)}^{\mathcal V_{\mathrm{loc}},(1)}$ and $\mathrm{comp}_{(\mu,\Phi)}^{\mathcal W,(1)}$, 
and the commutativity of (A8) follows from the fact that $\hat{\Delta}^{\mathcal W,\DR}$ 
is an algebra morphism, such that $e_1\mapsto e_1+f_1$. 
\end{itemize}
Using successively the commutativities of (A6), (A1), (A5), (A4), (A3), (A7), (A8), (A2), and denoting by $\mathrm{map}_i$ ($i=1,2$) 
the two maps $\hat{\mathcal W}^\B\to(\mathcal W^\DR)^{\otimes2\wedge}$ which can be derived from diagram \eqref{main:diag:alg:10dec2019}, 
one sees that the two maps 
$$
\xymatrix{
\hat{\mathcal V}^\B\ar^{\mathrm{mor}_{\mathcal V^\B,X_1-1}}[rr]&&
\hat{\mathcal W}^\B_+\ar^{\!\!\!\!\!\!\!\!\!\!\!\!\!\!\mathrm{map}_i}[r]&
(\mathcal W^\DR)^{\otimes2\wedge}\ar^{\!\!\!\!\!\!\!\!\!\!\!\!\!\!\!\!\!\!\!\!\!\!\!\!\!\!\!\!\!\!\!\!\!\!B_\Phi^{-1}\cdot(-)\cdot B_\Phi {{e_1+f_1}\over{e^{\mu(e_1+f_1)}-1}}}[rrr]&&&
(\mathcal W^\DR)^{\otimes2\wedge}\hookrightarrow
(\mathcal V^\DR[{1\over e_1}]_{\geq0})^{\otimes2\wedge}
}
$$
coincide ($i=1,2$). 
Since the map $\xymatrix{(\mathcal W^\DR)^{\otimes2\wedge}\ar^{\!\!\!\!\!\!\!\!\!\!\!\!\!\!\!\!\!\!\!\!\!\!\!\!\!\!\!\!\!\!\!\!\!\!B_\Phi^{-1}\cdot(-)
\cdot B_\Phi {{e_1+f_1}\over{e^{\mu(e_1+f_1)}-1}}}[rrr]&&&
(\mathcal W^\DR)^{\otimes2\wedge}\hookrightarrow
(\mathcal V^\DR[{1\over e_1}]_{\geq0})^{\otimes2\wedge}}$ is injective, and since $\mathrm{mor}_{\mathcal V^\B,X_1-1}:
\hat{\mathcal V}^\B\to\hat{\mathcal W}^\B_+$ is a bijection, it follows that the restrictions of $\mathrm{map}_1$ and $\mathrm{map}_2$ 
to $\hat{\mathcal W}^\B_+$ coincide.  One has $\mathrm{map}_1(1)=1^{\otimes2}=\mathrm{map}_2(1)$, which implies that 
the restrictions of $\mathrm{map}_1$ and $\mathrm{map}_2$ to $\mathbf k1\subset\hat{\mathcal W}^\B$ also coincide. 
Since $\hat{\mathcal W}^\B$ is the direct sum of $\mathbf k1$ and $\hat{\mathcal W}^\B_+$, it follows that the maps 
$\mathrm{map}_1$ and $\mathrm{map}_2$ coincide, therefore that \eqref{main:diag:alg:10dec2019} commutes. \hfill\qed\medskip 

\section{Associators and harmonic module coproducts}\label{sect:aamhc:2jan2020}

The purpose of this section is to prove the second main result of this paper, namely the compatibility of associators with harmonic module
coproducts (Theorem \ref{thm:new:crm}). This proof relies on the commutativity of diagram \eqref{diagram:main:mod}. The main new 
ingredients in this diagram are: a family of localized versions of the comparison isomorphisms from \S\ref{sect:fci:24dec2019}, 
which are constructed in \S\ref{sect:lvoci:24dec2019}; and a matrix $R_{(\mu,\Phi)}$, which is defined in \S\ref{sect:3:2:crm}. 
Diagram \eqref{diagram:main:mod} is divided into subdiagrams (M1) to (M5). Among them, subdiagram (M1) 
(resp. (M5)) expresses the geometric interpretation of the Betti (resp. de Rham) module coproduct and was proved in Lemma 
\ref{lemma:completion:10dec2019} (resp. Lemma \ref{lem:completion:diag:LA}). The main remaining diagrams are (M3) and (M4), which relate, 
by means of the module comparison isomorphisms and the matrices $R_{(\mu,\Phi)}$ from \S\ref{sect:3:2:crm} and 
$\overline P_{(\mu,\Phi)}$ from 
\S\ref{sect:def:matrices:P:R}, `de Rham' and `Betti' objects: namely, the morphisms $\hat\rho$ and $\underline{\hat\rho}$ in the 
case of (M3), and the morphisms $\mathrm{row}_1\cdot(-)\cdot\mathrm{col}_0$ and $\underline{\mathrm{row}}_1\cdot(-)
\cdot\underline{\mathrm{col}}_0$ in the case of (M4). The commutativity of (M3) is established in \S\ref{sect:3:2:crm} using the module 
morphism properties of its constituents over the constituents of (A3) in \eqref{diagg:main:alg}, the commutativity of (A3), and the 
definition of $R_{(\mu,\Phi)}$. The commutativity of (M4) is established in \S\S\ref{sect:4:1:1/2} to \ref{sect:4:4:crm};  more precisely, 
\S\ref{sect:4:1:1/2} is devoted to the expression of $R_{(\mu,\Phi)}$ in terms of the matrix $\overline Q_{(\mu,\Phi)}$ from 
\S\ref{sect:def:matrices:P:R}; \S\ref{sect:4:3} is devoted to the proof of an equality relating $R_{(\mu,\Phi)}$, $\mathrm{col}_0$ and 
$\underline{\mathrm{col}}_0$, based in \S\ref{sect:4:1:1/2} and the computation of $\overline Q_{(\mu,\Phi)}$ in \S\ref{sect:comp:R}; 
this result is combined in \S\ref{sect:4:4:crm} with the equality relating $\overline P_{(\mu,\Phi)}$, $\mathrm{row}_1$ and 
$\underline{\mathrm{row}}_1$ from \S\ref{sect:rel:col:barcol:21032018} to prove the commutativity of (M4). Theorem \ref{thm:new:crm} 
is formulated and proved in \S\ref{sect:aamhc:24dec2019}. 

\subsection{Localized versions of comparison isomorphisms}\label{sect:lvoci:24dec2019}

For $(\mu,g)\in\mathbf k^\times\times(\hat{\mathcal V}^\DR)^\times$, let 
$$
\mathrm{comp}_{(\mu,g)}^{\mathcal V_{\mathrm{loc}},(10)}:\mathcal V^\B[{1\over X_1-1}]^\wedge\to\mathcal V^\DR[{1\over e_1}]^\wedge
\index{compVloc10@$\mathrm{comp}^{\mathcal V_{\mathrm{loc}},(10)}_{(\mu,g)}$}
$$
be the $\mathbf k$-module isomorphism given by 
\begin{equation}\label{def:comp:V:loc:10}
\forall v\in\mathcal V^\B[{1\over X_1-1}]^\wedge,\quad \mathrm{comp}_{(\mu,g)}^{\mathcal V_{\mathrm{loc}},(10)}(v):=
\mathrm{comp}_{(\mu,g)}^{\mathcal V_{\mathrm{loc}},(1)}(v)\cdot g 
\end{equation}
(see \eqref{def:comp:V:loc}). One checks that it restricts to an isomorphism 
$\mathcal V^\B[{1\over X_1-1}]^\wedge(X_0-1)\to\mathcal V^\DR[{1\over e_1}]^\wedge e_0$. 

\begin{defn}
For $(\mu,g)\in\mathbf k^\times\times(\hat{\mathcal V}^\DR)^\times$, 
 $\mathrm{comp}_{(\mu,g)}^{\mathcal M_{\mathrm{loc}},(10)}:\mathcal M^\B[{1\over X_1-1}]^\wedge\to\mathcal M^\DR[{1\over e_1}]^\wedge$
 \index{compM10loc@$\mathrm{comp}^{\mathcal M_{\mathrm{loc}},(10)}_{(\mu,g)}$}
is the $\mathbf k$-module isomorphism induced by $\mathrm{comp}_{(\mu,g)}^{\mathcal V_{\mathrm{loc}},(10)}$. 
\end{defn}

\begin{lem}\label{compat:comM:compMloc:17dec2019}
For $(\mu,g)\in\mathbf k^\times\times(\hat{\mathcal V}^\DR)^\times$, the isomorphism 
$\mathrm{comp}_{(\mu,g)}^{\mathcal M_{\mathrm{loc}},(10)}$
is compatible with the filtrations on both sides. Moreover, the following diagram commutes
$$
\xymatrix{
\hat{\mathcal M}^\B \ar[d]\ar^{\mathrm{comp}^{\mathcal M,(10)}_{(\mu,g)}}[rr]&& \hat{\mathcal M}^\DR\ar[d]
\\
\mathcal M^\B[{1\over X_1-1}]^\wedge\ar_{\mathrm{comp}^{\mathcal M_{\mathrm{loc}},(10)}_{(\mu,g)}}[rr]&& 
\mathcal M^\DR[{1\over e_1}]^\wedge}
$$
\end{lem}

\proof The first statement follows from the fact that $\mathrm{comp}_{(\mu,g)}^{\mathcal V_{\mathrm{loc}},(10)}$ is compatible 
with filtrations. It follows from its definition (see \eqref{def:comp:V:loc}) that the algebra isomorphisms 
$\mathrm{comp}^{\mathcal V,(1)}_{(\mu,g)}$ and $\mathrm{comp}^{\mathcal V_{\mathrm{loc}},(1)}_{(\mu,g)}$ intertwine the algebra morphisms 
$\hat{\mathcal V}^\B\to\mathcal V^\B[{1\over{X_1-1}}]^\wedge$ and $\hat{\mathcal V}^\DR\to\mathcal V^\DR[{1\over{e_1}}]^\wedge$. 
Expressing the resulting identity and multiplying it by $g$, one sees that the $\mathbf k$-module morphisms 
$\mathrm{comp}^{\mathcal V,(10)}_{(\mu,g)}$ and $\mathrm{comp}^{\mathcal V_{\mathrm{loc}},(10)}_{(\mu,g)}$ intertwine the same 
algebra morphisms. The announced diagram follows by projection. \hfill\qed\medskip 

The compatibility of the algebra morphism $\mathrm{comp}_{(\mu,g)}^{\mathcal V,(1)}$ and the $\mathbf k$-module morphism 
$\mathrm{comp}_{(\mu,g)}^{\mathcal M,(10)}$ with the module structures (see \S\ref{sect:3:3:28oct}) extends to the localized 
situation as follows. 

\begin{lem}\label{module:structure:compMloc:compVloc:17dec2019}
For $(\mu,g)\in\mathbf k^\times\times(\hat{\mathcal V}^\DR)^\times$, the $\mathbf k$-module morphism 
$\mathrm{comp}_{(\mu,g)}^{\mathcal M_{\mathrm{loc}},(10)}$ is compatible (in the sense of \S\ref{conventions:0703})
with the $\mathbf k$-algebra morphism $\mathrm{comp}_{(\mu,g)}^{\mathcal V_{\mathrm{loc}},(1)}$ and with the module 
structure of $\mathcal M^\B[{1\over X_1-1}]^\wedge$ (resp. $\mathcal M^\DR[{1\over e_1}]^\wedge$) over 
$\mathcal V^\B[{1\over X_1-1}]^\wedge$ (resp. $\mathcal V^\DR[{1\over e_1}]^\wedge$). 
\end{lem}

\proof It follows from \eqref{def:comp:V:loc:10} and from the algebra morphism property of 
$\mathrm{comp}_{(\mu,g)}^{\mathcal V_{\mathrm{loc}},(1)}$
that 
$$
\forall v,v'\in\mathcal V^\B[{1\over X_1-1}]^\wedge,\quad 
\mathrm{comp}_{(\mu,g)}^{\mathcal V_{\mathrm{loc}},(10)}(vv')=
\mathrm{comp}_{(\mu,g)}^{\mathcal V_{\mathrm{loc}},(1)}(v)
\mathrm{comp}_{(\mu,g)}^{\mathcal V_{\mathrm{loc}},(10)}(v'). 
$$
Applying to this identity the projection $\mathcal V^\DR[{1\over e_1}]^\wedge\to\mathcal M^\DR[{1\over e_1}]^\wedge$, one obtains: 
$$
\forall v\in\mathcal V^\B[{1\over X_1-1}]^\wedge,\forall m\in\mathcal M^\B[{1\over X_1-1}]^\wedge,\quad 
\mathrm{comp}_{(\mu,g)}^{\mathcal M_{\mathrm{loc}},(10)}(vm)=
\mathrm{comp}_{(\mu,g)}^{\mathcal V_{\mathrm{loc}},(1)}(v)
\mathrm{comp}_{(\mu,g)}^{\mathcal M_{\mathrm{loc}},(10)}(m). 
$$
\hfill\qed\medskip 

\subsection{Definition of the matrix $R_{(\mu,\Phi)}$ and commutativity of (M3) in \eqref{diagram:main:mod}}\label{sect:3:2:crm} 

Recall that $\Phi\in(\hat{\mathcal V}^\DR)^\times$, and that $\hat\rho:\hat{\mathcal V}^\DR\to M_3((\mathcal V^\DR)^{\otimes2\wedge})$
is an algebra morphism. It follows that $\hat\rho(\Phi)\in\mathrm{GL}_3((\mathcal V^\DR)^{\otimes2\wedge})$. 

\begin{defn}\label{def:of:Q:23dec2019}
Let $\mu\in\mathbf k^\times$ and $\Phi\in\mathsf M_\mu(\mathbf k)$. We set
\begin{equation}\label{def:Q:mu:Phi}
R_{(\mu,\Phi)}:=\hat\rho(\Phi)^{-1}\overline P_{(\mu,\Phi)}^{-1}\kappa_{(\mu,\Phi)}^{-1}\Phi(e_0,e_1)\Phi(f_0,f_1)\in
\mathrm{GL}_3((\mathcal V^\DR)^{\otimes2\wedge}).
\index{R_mu,Phi@$R_{(\mu,\Phi)}$} 
\end{equation} 
\end{defn}

\begin{lem}\label{lem:S3:27032019}
Let $\mu\in\mathbf k^\times$ and $\Phi\in\mathsf M_\mu(\mathbf k)$. The diagram 
\begin{equation}\label{diagram:S3}
\xymatrix{
\hat{\mathcal V}^\B\ar^{\!\!\!\!\!\!\!\!\!\!\!\!\!\!\!\underline{\hat\rho}}[r]\ar_{\mathrm{comp}^{\mathcal V,(10)}_{(\mu,\Phi)}}[dd]
& M_3((\mathcal V^\B)^{\otimes2\wedge})
\ar^{M_3((\mathrm{comp}^{\mathcal V,(10)}_{(\mu,\Phi)})^{\otimes2})}[d]\\
& M_3((\mathcal V^\DR)^{\otimes2\wedge})\\
\hat{\mathcal V}^\DR\ar^{\!\!\!\!\!\!\!\!\!\!\!\!\!\!\!\hat\rho}[r]&
M_3((\mathcal V^\DR)^{\otimes2\wedge})\ar_{\kappa_{(\mu,\Phi)}\overline P_{(\mu,\Phi)}\cdot(-)\cdot R_{(\mu,\Phi)}}[u] }
\end{equation}
((M3) in \eqref{diagram:main:mod}) commutes. 
\end{lem}

\proof
Let $v\in\hat{\mathcal V}^\B$. Then 
\begin{align*}
& M_3((\mathrm{comp}^{\mathcal V,(10)}_{(\mu,\Phi)})^{\otimes2})(\underline{\hat\rho}(v))
=M_3((\mathrm{comp}^{\mathcal V,(1)}_{(\mu,\Phi)})^{\otimes2})(\underline{\hat\rho}(v))\cdot 
\Phi(e_0,e_1)\Phi(f_0,f_1)
\\ & 
=\mathrm{Ad}(\kappa_{(\mu,\Phi)}\overline P_{(\mu,\Phi)})\Big(\hat\rho\big(\mathrm{comp}_{(\mu,\Phi)}^{\mathcal V,(1)}(v)\big)\Big)\cdot 
\Phi(e_0,e_1)\Phi(f_0,f_1)
\\ & 
=\mathrm{Ad}(\kappa_{(\mu,\Phi)}\overline P_{(\mu,\Phi)})\Big(\hat\rho\big(\mathrm{comp}_{(\mu,\Phi)}^{\mathcal V,(10)}(v)\cdot\Phi^{-1}\big)\Big)\cdot 
\Phi(e_0,e_1)\Phi(f_0,f_1)
\\ & 
=\mathrm{Ad}(\kappa_{(\mu,\Phi)}\overline P_{(\mu,\Phi)})\Big(\hat\rho\big(\mathrm{comp}_{(\mu,\Phi)}^{\mathcal V,(10)}(v)\big)\cdot
\hat\rho(\Phi)^{-1}\Big)\cdot 
\Phi(e_0,e_1)\Phi(f_0,f_1)
\\ & 
=\kappa_{(\mu,\Phi)}\overline P_{(\mu,\Phi)}\cdot
\hat\rho\big(\mathrm{comp}_{(\mu,\Phi)}^{\mathcal V,(10)}(v)\big)\cdot
\hat\rho(\Phi)^{-1}\cdot
(\kappa_{(\mu,\Phi)}\overline P_{(\mu,\Phi)})^{-1}
\cdot 
\Phi(e_0,e_1)\Phi(f_0,f_1)
\\ & 
=\kappa_{(\mu,\Phi)}\overline P_{(\mu,\Phi)}\cdot
\hat\rho\big(\mathrm{comp}_{(\mu,\Phi)}^{\mathcal V,(10)}(v)\big)\cdot
R_{(\mu,\Phi)}, 
\end{align*}
where the first and third equalities follow from \eqref{relation:comp10:comp1:16dec2019}, the second equality follows from
Lemma \ref{lem:CD:relating:B:underlineB}, the fourth equality follows from the fact that $\hat\rho$ is an algebra morphism, 
and the last equality follows from \eqref{def:Q:mu:Phi}. 
\hfill\qed\medskip 

\subsection{Computation of  $R_{(\mu,\Phi)}$ 
in terms of $\overline Q_{(\mu,\Phi)}$}
\label{sect:4:1:1/2} 

Let us first relate the elements $P_{(\mu,\Phi)}$ and $Q_{(\mu,\Phi)}$ from Definition \ref{def:P:R:20191203}. 

\begin{lem}\label{lemma:30072019}
Let $\mu\in\mathbf k^\times$ and $\Phi\in\mathsf M_\mu(\mathbf k)$. One has 
\begin{equation}\label{formula:R}
Q_{(\mu,\Phi)}=\hat\ell(\Phi)^{-1}\cdot P_{(\mu,\Phi)}\cdot\hat \varpi(\hat\ell(\Phi)) 
\end{equation}
(equality in $\mathrm{GL}_3((U\mathfrak p_5)^{\otimes2\wedge})$). 
\end{lem}

\proof The following sequence of equalities holds in $M_{3\times1}((U\mathfrak p_5)^{\otimes2\wedge})$
\begin{align*}
& 
Q_{(\mu,\Phi)}\cdot \begin{pmatrix}e_{15} \\ e_{25} \\e_{35} \end{pmatrix}=
M_{3\times1}(\mathrm{comp}_{(\mu,\Phi)}^{(\bullet(\bullet\bullet))\bullet})
(\begin{pmatrix}x_{15}-1 \\ x_{25}-1 \\x_{35}-1 \end{pmatrix})
= \hat\ell(\Phi)^{-1}\cdot
M_{3\times1}(\mathrm{comp}_{(\mu,\Phi)}^{((\bullet\bullet)\bullet)\bullet})
(\begin{pmatrix}x_{15}-1 \\ x_{25}-1 \\x_{35}-1 \end{pmatrix})\cdot \hat\ell(\Phi)
\\ & = \hat\ell(\Phi)^{-1}\cdot P_{(\mu,\Phi)}\cdot\begin{pmatrix}e_{15} \\ e_{25} \\e_{35} \end{pmatrix}\cdot \hat\ell(\Phi)
= \hat\ell(\Phi)^{-1}\cdot P_{(\mu,\Phi)}\cdot \hat\varpi(\hat\ell(\Phi))\cdot\begin{pmatrix}e_{15} \\ e_{25} \\e_{35} \end{pmatrix}, 
\end{align*}
since: the first (resp. third) equality follows from the second (resp. first) part of \eqref{matrices:P:R:16dec2019};  
the second equality follows from \eqref{relation:comp:comp} together with $\Phi(e_{12},e_{23})=\hat\ell(\Phi)^{-1}$, which follows from 
the definition of $\hat\ell$ (see \S\ref{subsect:def:LA:morphisms:2}) and \eqref{duality:rel}; 
the last equality follows from \eqref{def:varpi:16dec2019}. 

The result then follows from the injectivity of the map $((U\mathfrak p_5)^\wedge)^{\oplus3}\to (U\mathfrak p_5)^\wedge$, 
$(p_i)_{i\in[\![1,3]\!]}\mapsto\sum_{i\in[\![1,3]\!]}p_i\cdot e_{i5}$ (see Lemma \ref{lemma:decomp:J:pr5}). 
\hfill\qed\medskip 

\begin{lem}\label{lem:31052019}
Let $\mu\in\mathbf k^\times$ and $\Phi\in \mathsf M_\mu(\mathbf k)$. One has 
$$
\overline P_{(\mu,\Phi)}\hat\rho(\Phi)=\overline Q_{(\mu,\Phi)}
$$
(equality in $\mathrm{GL}_3((\mathcal V^\DR)^{\otimes2,\wedge})$). 
\end{lem}

\proof One has 
\begin{align*}
&\overline Q_{(\mu,\Phi)}=\mathrm{pr}_{12}^\wedge(Q_{(\mu,\Phi)})
=\mathrm{pr}_{12}^\wedge(\hat\ell(\Phi)^{-1}\cdot P_{(\mu,\Phi)}\cdot \hat\varpi(\hat\ell(\Phi)))
\\ &=\mathrm{pr}_{12}^\wedge(\hat\ell(\Phi))^{-1}\cdot\mathrm{pr}_{12}^\wedge(P_{(\mu,\Phi)})\cdot
\mathrm{pr}_{12}^\wedge(\hat\varpi(\hat\ell(\Phi)))
=1\cdot \overline P_{(\mu,\Phi)}\cdot\hat\rho(\Phi)
=\overline P_{(\mu,\Phi)}\hat\rho(\Phi), 
\end{align*}
where the first equality follows from Definition \ref{def:barP:barR}, the second equality follows from Lemma \ref{lemma:30072019}, 
the third equality from the fact that $\mathrm{pr}_{12}^\wedge$ is an algebra morphism. 

One has $\mathrm{pr}_{12}^\wedge\circ\hat\ell(e_1)=0$. As $\mathrm{pr}_{12}^\wedge$ is an algebra morphism and as the 
logarithm of $\Phi$ is a Lie series in $e_0,e_1$ without degree 1 terms, this implies $\mathrm{pr}_{12}^\wedge\circ\hat\ell(\Phi)=1$. 
Together with Definition \ref{def:barP:barR} and the definition of $\hat\rho$ (see \S\ref{sect:514:12122017}), this implies the fourth 
equality, hence the result.  \hfill\qed\medskip 

\begin{cor}\label{cor:23dec2019}
Let $\mu\in\mathbf k^\times$ and $\Phi\in\mathsf M_\mu(\mathbf k)$. One has 
$$
R_{(\mu,\Phi)}=\overline Q_{(\mu,\Phi)}^{-1}\kappa_{(\mu,\Phi)}^{-1}\Phi(e_0,e_1)\Phi(f_0,f_1)
$$
(equality in $\mathrm{GL}_3((\mathcal V^\DR)^{\otimes2\wedge})$).
\end{cor}

\proof This follows from Definition \ref{def:of:Q:23dec2019} and Lemma \ref{lem:31052019}.  \hfill\qed\medskip

\subsection{Relationship between $\mathrm{col}_0$, \underline{$\mathrm{col}$}$_0$ and $R_{(\mu,\Phi)}$}\label{sect:4:3} 

\begin{lem}
Let $\mu\in\mathbf k^\times$ and $\Phi\in\mathsf M_\mu(\mathbf k)$. The following equalities hold in $(\mathcal M^\DR)^{\otimes2\wedge}$: 
\begin{equation}\label{interm:1:23dec2019}
\Big(-(e^{\mu e_1}-1)\varphi_1(e_0-f_0,e_1)-{{e^{\mu e_1}-1}\over{e_1}}\Big) e_1 1_\DR^{\otimes2}
=(1-e^{\mu e_1})\Phi(e_0,e_1)1_\DR^{\otimes2}
\end{equation}
\begin{equation}\label{interm:2:23dec2019}
\Big((e^{\mu(e_\infty+f_0)}-1)\varphi_1(e_0-f_0,e_\infty+f_0)+{{e^{\mu(e_\infty+f_0)}-1}\over{e_\infty+f_0}}\Big) e_1 1_\DR^{\otimes2}
= (1-e^{\mu e_\infty})\Phi(e_0,e_\infty) 1_\DR^{\otimes2}
\end{equation}
\begin{equation}\label{A:A':A''}
\Phi(e_0,e_1)e^{(\mu/2)e_0}\Phi(e_\infty,e_0)(1-e^{\mu e_\infty}) \Phi(e_0,e_\infty) 1_\DR^{\otimes2}
=
 (1-e^{-\mu e_1}) \Phi(e_0,e_1)1_\DR^{\otimes2}. 
\end{equation}
\end{lem}

\proof
The following holds in $(\mathcal M^\DR)^{\otimes2\wedge}$: 
\begin{align*}
&\nonumber \Big(-(e^{\mu e_1}-1)\varphi_1(e_0-f_0,e_1)-{{e^{\mu e_1}-1}\over{e_1}}\Big) e_1 1_\DR^{\otimes2}
=-(e^{\mu e_1}-1)\Big(\varphi_1(e_0-f_0,e_1)e_1+1\Big)  1_\DR^{\otimes2}
\\ & =(1-e^{\mu e_1})\Big(\varphi_0(e_0-f_0,e_1)(e_0-f_0)+\varphi_1(e_0-f_0,e_1)e_1+1\Big)  1_\DR^{\otimes2}
=(1-e^{\mu e_1})\Phi(e_0-f_0,e_1)1_\DR^{\otimes2}
\\ & =(1-e^{\mu e_1})\Phi(e_0,e_1)1_\DR^{\otimes2}
\end{align*}
as the second equality follows from $e_0 1_\DR^{\otimes2}=f_0 1_\DR^{\otimes2}=0$, the third equality from the relation 
between $\varphi_0,\varphi_1$ and $\Phi$ (see \S\ref{subsect:Gamma:fun:20032018}), and the last equality from the 
commutation of $f_0$ with $e_0$ and $e_1$. This proves \eqref{interm:1:23dec2019}. 

The following holds in $(\mathcal M^\DR)^{\otimes2\wedge}$: 
\begin{align*}
&\nonumber
\Big((e^{\mu(e_\infty+f_0)}-1)\varphi_1(e_0-f_0,e_\infty+f_0)+{{e^{\mu(e_\infty+f_0)}-1}\over{e_\infty+f_0}}\Big) e_1 1_\DR^{\otimes2}
\\ & 
=-\Big((e^{\mu(e_\infty+f_0)}-1)\varphi_1(e_0-f_0,e_\infty+f_0)+{{e^{\mu(e_\infty+f_0)}-1}\over{e_\infty+f_0}}\Big) (e_\infty+f_0)1_\DR^{\otimes2}
\\ & =\nonumber
-(e^{\mu(e_\infty+f_0)}-1)\Big(\varphi_1(e_0-f_0,e_\infty+f_0)(e_\infty+f_0)+
1\Big) 1_\DR^{\otimes2}
\\ &
=-(e^{\mu(e_\infty+f_0)}-1)\Big(\varphi_0(e_0-f_0,e_\infty+f_0)(e_0-f_0)+\varphi_1(e_0-f_0,e_\infty+f_0)(e_\infty+f_0)+
1\Big) 1_\DR^{\otimes2}
\\ &  
= \nonumber
-(e^{\mu(e_\infty+f_0)}-1)\Phi(e_0-f_0,e_\infty+f_0) 1_\DR^{\otimes2}
\\ & 
=(1-e^{\mu(e_\infty+f_0)})\Phi(e_0,e_\infty) 1_\DR^{\otimes2}= 
(1-e^{\mu e_\infty})\Phi(e_0,e_\infty) 1_\DR^{\otimes2}
\end{align*}
as the first and third equalities follow from $e_0 1_\DR^{\otimes2}=f_0 1_\DR^{\otimes2}=0$, the fourth from 
the relation between $\varphi_0,\varphi_1$ and $\Phi$ (see \S\ref{subsect:Gamma:fun:20032018}), the fifth from 
the commutation of $f_0$ with $e_0$ and $e_\infty$, and the last from the same fact, together with $f_0 1_\DR^{\otimes2}=0$. This 
proves \eqref{interm:2:23dec2019}. 

The following holds in $(\mathcal V^\DR)^{\otimes2\wedge}$: 
\begin{align*}
&\Phi(e_0,e_1)e^{(\mu/2)e_0}\Phi(e_\infty,e_0)(1-e^{\mu e_\infty}) \Phi(e_0,e_\infty) 
\\  &
=e^{-(\mu/2)e_1}\Phi(e_\infty,e_1)(e^{-(\mu/2)e_\infty}- e^{(\mu/2)e_\infty})\Phi(e_0,e_\infty)
\\ &
= e^{-(\mu/2)e_1}\cdot \Big(e^{(\mu/2)e_1}\Phi(e_0,e_1) e^{(\mu/2)e_0}- e^{-(\mu/2)e_1} \Phi(e_0,e_1) e^{-(\mu/2)e_0}\Big)
\\ & 
=(1-e^{-\mu e_1})\Phi(e_0,e_1)+\Phi(e_0,e_1)(e^{(\mu/2)e_0}-1)-e^{-(\mu/2)e_1}\Phi(e_0,e_1)(e^{-(\mu/2)e_0}-1)
\end{align*}
as the first equality follows from the 2-cycle identity \eqref{duality:rel} and the hexagon identity \eqref{hexagon+} for 
$(a,b,c)=(e_0,e_\infty,e_1)$ and the second equality follows from the hexagon identities \eqref{hexagon+} and \eqref{hexagon-} 
for $(a,b,c)=(e_0,e_1,e_\infty)$. Applying this equality to $1_\DR^{\otimes2}$, and using $e_0 1_\DR^{\otimes2}=0$, one 
obtains \eqref{A:A':A''}. \hfill\qed\medskip 

\begin{lem}
Let $\mu\in\mathbf k^\times$ and $\Phi\in\mathsf M_\mu(\mathbf k)$. One has 
\begin{equation}\label{equality:col0:17dec2019}
\Phi(e_0,e_1)^{-1} \Phi(f_0,f_1)^{-1}M_{3\times 1}((\mathrm{comp}_{(\mu,\Phi)}^{\mathcal M,(10)})^{\otimes2})\big(\underline{\mathrm{col}}_0 \big) 
=R_{(\mu,\Phi)}^{-1}\mathrm{col}_0
\end{equation}
(equality in $M_{3\times 1}((\mathcal M^\DR)^{\otimes2\wedge})$), where $\mathrm{col}_0$, $\underline{\mathrm{col}}_0$ 
are as in Definitions \ref{def:col0:21012021} and \ref{def:col:0:Betti:30oct}.
\end{lem}

\proof Using Lemmas \ref{lemma:3:1:23dec2019} and \ref{lem:com:pi:V:M:comp:17dec2019}, one obtains
\begin{align}\label{value:LHS}
M_{3\times1}((\mathrm{comp}_{(\mu,\Phi)}^{{\mathcal M},(10)})^{\otimes2})(\underline{\mathrm{col}}_0)
=
e^{-\mu f_1}\Phi(f_0,f_1)\begin{pmatrix} 0 \\ (1-e^{\mu e_1})\Phi(e_0,e_1) 1_\DR^{\otimes2}\\ 
(1-e^{-\mu e_1})\Phi(e_0,e_1) 1_\DR^{\otimes2}
\end{pmatrix}. 
\end{align}

On the other hand, 
\begin{align}\label{value:RHS:24dec2019}
&  \Phi(e_0,e_1) \Phi(f_0,f_1)R_{(\mu,\Phi)}^{-1}\cdot \mathrm{col}_0
=\kappa_{(\mu,\Phi)}\cdot \overline Q_{(\mu,\Phi)}\cdot \mathrm{col}_0 
\\
& \nonumber=\kappa_{(\mu,\Phi)}\mathrm{diag}\left(\Phi(f_1,f_\infty), 
e^{{\mu\over 2}f_\infty} \Phi(f_0,f_\infty)e^{{\mu\over 2}f_0} \Phi(e_1,e_0), 
e^{{\mu\over 2}(e_0+f_\infty)} \Phi(f_0,f_\infty)\Phi(e_\infty,e_0)\right)\cdot 
\\ & \nonumber
\cdot \begin{pmatrix} 0 \\ \Big(-(e^{\mu e_1}-1)\varphi_1(e_0-f_0,e_1)-{{e^{\mu e_1}-1}\over{e_1}}\Big) e_1 
1_\DR^{\otimes2}
\\ 
\Big((e^{\mu(e_\infty+f_0)}-1)\varphi_1(e_0-f_0,e_\infty+f_0)+{{e^{\mu(e_\infty+f_0)}-1}\over{e_\infty+f_0}}\Big) e_1 1_\DR^{\otimes2}\end{pmatrix}
\\ & \nonumber
=\begin{pmatrix} 0 \\ 
e^{-(\mu/2)f_1}\Phi(f_\infty,f_1)
e^{{\mu\over 2}f_\infty} \Phi(f_0,f_\infty)e^{{\mu\over 2}f_0} 
\Big(-(e^{\mu e_1}-1)\varphi_1(e_0-f_0,e_1)-{{e^{\mu e_1}-1}\over{e_1}}\Big) e_1 1_\DR^{\otimes2}\\ 
\scriptstyle{\Phi(e_0,e_1)\cdot e^{-(\mu/2)f_1}\Phi(f_\infty,f_1)
e^{{\mu\over 2}(e_0+f_\infty)} \Phi(f_0,f_\infty)\Phi(e_\infty,e_0) 
\Big((e^{\mu(e_\infty+f_0)}-1)\varphi_1(e_0-f_0,e_\infty+f_0)+{{e^{\mu(e_\infty+f_0)}-1}\over{e_\infty+f_0}}\Big) e_1 1_\DR^{\otimes2}}
\end{pmatrix}
\\ & \nonumber
=\begin{pmatrix} 0 \\ 
e^{-(\mu/2)f_1}\Phi(f_\infty,f_1)
e^{{\mu\over 2}f_\infty} \Phi(f_0,f_\infty)e^{{\mu\over 2}f_0} 
(1-e^{\mu e_1})\Phi(e_0,e_1)1_\DR^{\otimes2}
\\ 
\scriptstyle{\Phi(e_0,e_1)\cdot e^{-(\mu/2)f_1}\Phi(f_\infty,f_1)
e^{{\mu\over 2}(e_0+f_\infty)} \Phi(f_0,f_\infty)\Phi(e_\infty,e_0) 
(1-e^{\mu e_\infty})\Phi(e_0,e_\infty) 1_\DR^{\otimes2}}
\end{pmatrix}
\\ & \nonumber
=e^{-(\mu/2)f_1}\Phi(f_\infty,f_1)e^{{\mu\over 2}f_\infty} \Phi(f_0,f_\infty)e^{{\mu\over 2}f_0}
\begin{pmatrix} 0 \\ 
(1-e^{\mu e_1})\Phi(e_0,e_1)1_\DR^{\otimes2}
\\ 
\Phi(e_0,e_1)e^{{\mu\over 2}e_0} \Phi(e_\infty,e_0) 
(1-e^{\mu e_\infty})\Phi(e_0,e_\infty) 1_\DR^{\otimes2}
\end{pmatrix}
\\ & \nonumber
=e^{-\mu f_1}\Phi(f_0,f_1)
\begin{pmatrix} 0 \\ 
(1-e^{\mu e_1})\Phi(e_0,e_1)1_\DR^{\otimes2}
\\ 
\Phi(e_0,e_1)e^{{\mu\over 2}e_0} \Phi(e_\infty,e_0) 
(1-e^{\mu e_\infty})\Phi(e_0,e_\infty) 1_\DR^{\otimes2}
\end{pmatrix}
\\& =e^{-\mu f_1}\Phi(f_0,f_1)
\begin{pmatrix} 0 \\ 
(1-e^{\mu e_1})\Phi(e_0,e_1)1_\DR^{\otimes2}
\\ 
\nonumber
(1-e^{-\mu e_1})\Phi(e_0,e_1)1_\DR^{\otimes2}
\end{pmatrix}
\end{align}
where the first (resp. second, third) equality follows from Corollary \ref{cor:23dec2019} (resp. Corollary \ref{cor:comput:bar:R:muPhi}, 
\eqref{definition:of:Xi}), the fourth equality from the equalities \eqref{interm:1:23dec2019} and \eqref{interm:2:23dec2019}, the fifth 
equality from the commutativity of $e^{{\mu\over 2}f_0}$ with $e_\infty$ and $e_0$, and from $e^{{\mu\over 2}f_0}
1_\DR^{\otimes2}=1_\DR^{\otimes2}$, the sixth equality from the hexagon identity \eqref{hexagon+} for $(a,b,c)
=(f_0,f_1,f_\infty)$, and the last equality follows from \eqref{A:A':A''}. \eqref{equality:col0:17dec2019} then follows from the 
combination of \eqref{value:LHS} and \eqref{value:RHS:24dec2019}. \hfill\qed\medskip 

\subsection{Commutative diagram relating $\mathrm{row}_1\cdot(-)\cdot\mathrm{col}_0$ and 
\underline{$\mathrm{row}$}$_1\cdot(-)\cdot$\underline{$\mathrm{col}$}$_0$ ((M4) in \eqref{diagram:main:mod})}
\label{sect:4:4:crm} 

\begin{lem}
Let $\mu\in\mathbf k^\times$ and $\Phi\in\mathsf M_\mu(\mathbf k)$. One has 
\begin{equation}\label{equality:row:localized:17dec2019}
M_{1\times 3}((\mathrm{comp}_{(\mu,\Phi)}^{\mathcal V_{\mathrm{loc}},(1)})^{\otimes2})\big((X_1-1)^{-1}(1-Y_1^{-1})^{-1}\underline{\mathrm{row}}_1 \big)
=B_\Phi(e_1f_1)^{-1}\mathrm{row}_1(\kappa_{(\mu,\Phi)}\overline P_{(\mu,\Phi)})^{-1}
\end{equation}
(equality in $M_{1\times 3}(\mathcal V^\DR[{1\over e_1}]^{\otimes2\wedge})$). 
\end{lem}

\proof This follows from \eqref{newequality:13122017}, from the fact that 
$(\mathrm{comp}_{(\mu,\Phi)}^{\mathcal V_{\mathrm{loc}},(1)})^{\otimes2}$ is 
an algebra automorphism, and from the equality $(\mathrm{comp}_{(\mu,\Phi)}^{\mathcal V,(1)})^{\otimes2}((X_1-1)(1-Y_1^{-1}))
=e_1f_1 u_{(\mu,\Phi)}B_\Phi^{-1}$, which follows from \eqref{relation:u:B:17dec2019} and \eqref{relation:XY:ef:17dec2019}. 
\hfill\qed\medskip 

\begin{lem} \label{lem:S4:27032019}
Let $\mu\in\mathbf k^\times$ and $\Phi\in\mathsf M_\mu(\mathbf k)$. The diagram 
\begin{equation}\label{diagram:S4}
\xymatrix{
M_3((\mathcal V^\B)^{\otimes 2\wedge})\ar^{(X_1-1)^{-1}(1-Y_1^{-1})^{-1}\underline{\mathrm{row}}_1\cdot(-)\cdot\underline{\mathrm{col}}_0}[rrrr]
\ar_{M_3((\mathrm{comp}^{\mathcal V,(10)}_{(\mu,\Phi)})^{\otimes2})}[d]
& & & &
{\mathcal M}^\B[\frac{1}{X_1-1}]^{\otimes2\wedge}\ar^{(\mathrm{comp}^{{\mathcal M}_{\mathrm{loc}},(10)}_{(\mu,\Phi)})^{\otimes2}}[dd] \\
M_3((\mathcal V^\DR)^{\otimes2\wedge})
\ar_{(\kappa_{(\mu,\Phi)}\overline P_{(\mu,\Phi)})^{-1}\cdot(-)\cdot R_{(\mu,\Phi)}^{-1}}[d] & & & & \\
M_3((\mathcal V^\DR)^{\otimes2\wedge})\ar_{(e_1f_1)^{-1}\mathrm{row}_1\cdot(-)\cdot\mathrm{col}_0}[rr] 
& & {\mathcal M}^\DR[\frac{1}{e_1}]^{\otimes2\wedge}\ar_{B_\Phi\cdot(-)}[rr] 
& & {\mathcal M}^\DR[\frac{1}{e_1}]^{\otimes2\wedge} 
 }
\end{equation}
((M4) in \eqref{diagram:main:mod}) commutes, where $B_\Phi$ is as in \eqref{def:B:Phi}. 
\end{lem}

\proof Let $v\in M_3((\mathcal V^\B)^{\otimes 2\wedge})$. One has 
\begin{align*}
& (\mathrm{comp}_{(\mu,\Phi)}^{\mathcal M_{\mathrm{loc}},(10)})^{\otimes2}\big((X_1-1)^{-1}(1-Y_1^{-1})^{-1}\underline{\mathrm{row}}_1\cdot v\cdot\underline{\mathrm{col}}_0 \big)
\\ & =M_{1\times 3}((\mathrm{comp}_{(\mu,\Phi)}^{\mathcal V_{\mathrm{loc}},(1)})^{\otimes2})\big((X_1-1)^{-1}(1-Y_1^{-1})^{-1}\underline{\mathrm{row}}_1 \big) \cdot 
M_3((\mathrm{comp}_{(\mu,\Phi)}^{\mathcal V,(1)})^{\otimes2}\big(v\big))\cdot\\ & \cdot 
M_{3\times 1}((\mathrm{comp}_{(\mu,\Phi)}^{\mathcal M,(10)})^{\otimes2})\big(\underline{\mathrm{col}}_0 \big) 
\\ &  =M_{1\times 3}((\mathrm{comp}_{(\mu,\Phi)}^{\mathcal V_{\mathrm{loc}},(1)})^{\otimes2})\big((X_1-1)^{-1}(1-Y_1^{-1})^{-1}\underline{\mathrm{row}}_1 \big) \cdot 
M_3((\mathrm{comp}_{(\mu,\Phi)}^{\mathcal V,(10)})^{\otimes2}\big(v\big))\cdot
\\ & \cdot 
\Phi(e_0,e_1)^{-1} \Phi(f_0,f_1)^{-1} 
M_{3\times 1}((\mathrm{comp}_{(\mu,\Phi)}^{\mathcal M,(10)})^{\otimes2})\big(\underline{\mathrm{col}}_0 \big) 
\\ & 
=
B_\Phi(e_1f_1)^{-1}\mathrm{row}_1(\kappa_{(\mu,\Phi)}\overline P_{(\mu,\Phi)})^{-1} 
\cdot M_3((\mathrm{comp}_{(\mu,\Phi)}^{\mathcal V,(10)})^{\otimes2})(v)\cdot 
R_{(\mu,\Phi)}^{-1}\mathrm{col}_0, 
\end{align*}
where the first identity follows from Lemma \ref{module:structure:compMloc:compVloc:17dec2019}, the second 
identity follows from \eqref{relation:comp10:comp1:16dec2019}, and the last identity follows from 
\eqref{equality:row:localized:17dec2019} and \eqref{equality:col0:17dec2019}. This implies
the commutation of the diagram.  \hfill\qed\medskip 

\subsection{Associators and harmonic module coproducts
}\label{sect:aamhc:24dec2019}

\begin{thm}\label{thm:new:crm}
Let $\mu\in\mathbf k^\times$ and $\Phi\in\mathsf M_\mu(\mathbf k)$. Then the
diagram 
$$
\xymatrix{
\hat{\mathcal M}^\B
\ar^{{\hat\Delta}^{\mathcal M,\B}}[rrr]
\ar_{\mathrm{comp}^{\mathcal M,(10)}_{(\mu,\Phi)}}[d]
&&&({\mathcal M}^\B)^{\otimes 2\wedge}
\ar^{(\mathrm{comp}^{\mathcal M,(10)}_{(\mu,\Phi)})^{\otimes 2}}[d]
\\
\hat{\mathcal M}^\DR
\ar_{\!\!{\hat\Delta}^{\mathcal M,\DR}}[r]
&
({\mathcal M}^\DR)^{\otimes 2\wedge}
\ar_{B_\Phi\cdot(-)}[rr]
&&
({\mathcal M}^\DR)^{\otimes 2\wedge}
}
$$
commutes, where $B_\Phi$ is given by \eqref{def:B:Phi}.  
\end{thm}

\proof
Consider the following diagram. It is divided into the subdiagrams (M1), ..., (M5).  
\begin{equation}\label{diagram:main:mod}
{\xymatrix{
{\hat{\mathcal M}^\B} 
\ar^{{\hat\Delta}^{\mathcal M,\B}}[rr]
\ar^*[@]{\vbox to 4pt{\hbox to 4pt {$\!\!\!\!\!\!\scriptstyle{\mathrm{comp}^{\mathcal M,(10)}_{(\mu,\Phi)}}$}}}[dddd]
& & { ({\mathcal M}^\B)^{\otimes2\wedge}}\ar@{^(->}[rr]& & 
F^{-1}({{\mathcal M}^\B[\frac{1}{X_1-1}])^{\otimes 2\wedge}}
\ar^{(\mathrm{comp}^{{\mathcal M}_{\mathrm{loc}},(10)}_{(\mu,\Phi)})^{\otimes2}}[dddd] 
\\
&\hat{\mathcal V}^\B
\ar@{}[ur]|{\mathrm{(M1)}}
\ar@{}[dr]|{\mathrm{(M3)}}
\ar@{}[dddl]|{\mathrm{(M2)}}
\ar^{\!\!\!\!\!\!\!\underline{\hat\BB}}[r]
\ar_{(-)\cdot 1_\B}[ul]
\ar^*[@]{\vbox to 4pt{\hbox to 4pt {$\!\!\!\!\!\!\scriptstyle{\mathrm{comp}^{\mathcal V,(10)}_{(\mu,\Phi)}}$}}}[dd]
&
{ M_3((\mathcal V^\B)^{\otimes 2\wedge})}
\ar^*[@]{\vbox to 4pt{\hbox to 10pt {$\!\!\!\!\!\!\!\!\!\!\!\!\!\!\!\!\!\!\!\!\!\!\!\!\!\!\!\!\!\!\!\!\!\!\!\!\!\scriptstyle{(X_1-1)^{-1}(1-Y_1^{-1})^{-1}\underline{\mathrm{row}}_1\cdot(\text{-})\cdot\underline{\mathrm{col}}_0}$}}}[urr]
\ar^{M_3((\mathrm{comp}^{\mathcal V,(10)}_{(\mu,\Phi)})^{\otimes2})}[d]
\ar@{}[ddrr]|{\mathrm{(M4)}}&&
\\
&&
{M_3((\mathcal V^\DR)^{\otimes2\wedge})}
\ar^{(\kappa_{(\mu,\Phi)}\overline P_{(\mu,\Phi)})^{-1}\cdot (\text{-})\cdot R_{(\mu,\Phi)}^{-1}}[d]
&&
\\
& \hat{\mathcal V}^\DR
\ar^{\!\!\!\!\!\!\hat\BB}[r]
\ar[dl]
\ar@{}[dr]|{\mathrm{(M5)}}
& M_3((\mathcal V^\DR)^{\otimes2\wedge})
\ar^*[@]{\vbox to 4pt{\hbox to 10pt {$\!\!\!\!\!\!\!\!\!\!\!\!\!\!\scriptstyle{(e_1f_1)^{-1}\mathrm{row}_1\cdot(\text{-})\cdot\mathrm{col}_0}$}}}[rd]
&&
\\
\hat{\mathcal M}^\DR
\ar_{\hat\Delta^{\mathcal M,\DR}}[rr]
\ar@{}^*[@]{{\hbox to 0pt{\vbox to 10pt{$\!\!\!\!\!\!\!\!\!\!\!\!\scriptstyle{(-)\cdot 1_\DR}$}}}}[ur] 
&&({\mathcal M}^\DR)^{\otimes2\wedge}\ar@{^(->}[r]&
({\mathcal M}^\DR[\frac{1}{e_1}]_{\geq-1})^{\otimes2\wedge}
\ar_{\scriptstyle{B_\Phi\cdot(\text{-})}}[r]
&
({\mathcal M}^\DR[\frac{1}{e_1}]_{\geq-1})^{\otimes2\wedge}
}}
\end{equation}
The commutativity of subdiagram (M1) (resp. (M2), (M3), (M4), (M5)) follows from Lemma \ref{lemma:completion:10dec2019} 
(resp. Lemma \ref{lem:com:pi:V:M:comp:17dec2019}, Lemma \ref{lem:S3:27032019}, Lemma \ref{lem:S4:27032019}, 
Lemma \ref{lem:completion:diag:LA}). 

Using the commutativities of these diagrams in the following order M1-M4-M3-M5-M2, we obtain that the composition of the external 
square with the morphism $\pi^\B:\hat{\mathcal V}^\B\to\hat{\mathcal M}^\B$ is commutative. Since this morphism is surjective, 
this implies that the external diagram commutes. 

The commutativity of the announced diagram then follows from Lemma \ref{compat:comM:compMloc:17dec2019} and 
from the fact that the injection $(\mathcal M^\DR)^{\otimes2\wedge}\hookrightarrow(\mathcal M^\DR[{1\over e_1}])^{\otimes2\wedge}$ 
intertwines the left multiplication $B_\Phi\cdot(-)$ on both sides. \hfill\qed\medskip 

\newpage 

\printindex


\begin{thebibliography}{EFu2}

\bibitem[An]{An} Y. Andr\'e,
{\it Une introduction aux motifs (motifs purs, motifs mixtes, p\'eriodes). } 
Panoramas et Synth\`eses, 17. Soci\'et\'e Math\'ematique de France, Paris, 2004. 

\bibitem[Ar]{A} E. Artin,
{\it Theory of braids}, Ann. of Math. (2) 48, (1947). 101--126. 

\bibitem[BN]{BN} D. Bar-Natan,
{\it On associators and the Grothendieck-Teichm\"uller group I},
Selecta Mathematica, New Series 
4 (1998), 183--212. 

\bibitem[Bir]{Bir} J. Birman, 
{\it Braids, links, and mapping class groups},
Annals of Mathematics Studies, No. 82. Princeton University Press, 
Princeton, N.J.; University of Tokyo Press, Tokyo, 1974.

\bibitem[Bbk]{Bbk} N. Bourbaki, {\it \'El\'ements de math\'ematique. Fasc. XXXVII. Groupes et alg\`ebres de Lie. 
Chapitre II: Alg\`ebres de Lie libres. Chapitre III: Groupes de Lie,} Actualit\'es Scientifiques et Industrielles, 
No. 1349. Hermann, Paris, 1972. 

\bibitem[De]{Del:P1} P. Deligne, 
{\it Le groupe fondamental de la droite projective moins trois points.} Galois groups over $\mathbb Q$, 72--297,
MSRI publications, 16, Springer-Verlag, 1989.

\bibitem[DeG]{DG} P. Deligne, A. Goncharov, 
Groupes fondamentaux motiviques de Tate mixte. 
Ann. Sci. Ecole Norm. Sup. (4) 38 (2005), no. 1, 1--56. 

\bibitem[DeT]{DT} P. Deligne, T. Terasoma,
{\it Harmonic shuffle relation for associators}, preprint (2005). 

\bibitem[Dr]{Dr} V. Drinfeld,
{\it On quasitriangular quasi-Hopf algebras and on a group that is closely connected with 
$\mathrm{Gal}(\overline{\mathbb Q}/\mathbb Q)$},
Leningrad Math. J. 2 (1991), no. 4, 829--860.  

\bibitem[E]{E} B. Enriquez,
{\it On the Drinfeld generators of $\mathfrak{grt}_1(k)$ and $\Gamma$-functions for associators}, 
Math. Res. Lett. 13 (2006), no. 2-3, 231--243.

\bibitem[EF0]{EFu} B. Enriquez, H. Furusho, 
{\it A stabilizer interpretation of double shuffle Lie algebras, }, 
International Mathematics Research Notices 22 (2018), 6870-6907. 

\bibitem[EF2]{EF2} B. Enriquez, H. Furusho, {\it The Betti side of the double shuffle theory. II. Double shuffle relations for
associators.} Preprint arXiv:1807.07786, v3.

\bibitem[EF3]{EF3} B. Enriquez, H. Furusho, {\it The Betti side of the double shuffle theory. III. Bitorsor structures.} Preprint 
arXiv:1908.00444, v3.  

\bibitem[FaB]{FvB} E. Fadell, J. Van Buskirk,
{\it The braid groups of $E^2$ and $S^2$}, 
Duke Math. J. 29 (1962), 243--257. 

\bibitem[Fox]{Fox} R. Fox, {\it Free Differential Calculus, I: Derivation in the Free Group Ring},
Ann. of Math. (3) 57 (1969). 547--560.

\bibitem[F1]{Fu0} H. Furusho, 
{\it The multiple zeta value algebra and the stable derivation algebra.}
Publ. Res. Inst. Math. Sci. 39 (2003), no. 4, 695--720.

\bibitem[F2]{F:pentagon} H. Furusho,
{\it Pentagon and hexagon equations}, 
Ann. of Math. (2) 171 (2010), no. 1, 545--556. 

\bibitem[F3]{F} H. Furusho,
{\it Double shuffle relation for associators},
Ann. of Math. (1) 174 (2011), no. 1, 341--360. 

\bibitem[G]{G} A. Goncharov, 
{\it Multiple polylogarithms, cyclotomy and modular complexes.}
Math. Res. Lett. 5 (1998), no. 4, 497--516.

\bibitem[IKZ]{IKZ} K. Ihara, M. Kaneko, D. Zagier, 
{\it Derivation and double shuffle relations for multiple zeta values.}
Compos. Math. 142 (2006), no. 2, 307--338. 

\bibitem[Ih1]{Ih:G} Y. Ihara,
{\it Automorphisms of pure sphere braid groups and Galois representations},
The Grothendieck Festschrift, 
Vol. II, 353--373, Progr. Math., 87, Birkh\"auser Boston, Boston, MA, 1990.

\bibitem[Ih2]{Ih} Y. Ihara,
{\it On the stable derivation algebra associated with some braid groups}, 
Israel J. Math. 80 (1992), no. 1-2, 135--153. 

\bibitem[LeM]{LM} T.Q.T. Le,  J. Murakami, {\it Kontsevich's integral for the Kauffman polynomial.} 
Nagoya Math. J. 142 (1996), 39--65.

\bibitem[LoS]{LS} P. Lochak, L. Schneps,
{\it The Grothendieck--Teichm\"uller group and automorphisms of braid groups},
The Grothendieck theory of dessins d'enfants (Luminy, 1993), 323--358, London Math. Soc. Lecture Note Ser., 200, 
Cambridge Univ. Press, Cambridge, 1994. 

\bibitem[R]{Rac} G. Racinet,
{\it Doubles m\'elanges des polylogarithmes multiples aux racines de l'unit\'e}, 
Publ. Math. Inst. Hautes \'Etudes Sci. No. 95 (2002), 185--231. 

\bibitem[Wei]{Wei} C. Weibel, An introduction to homological algebra, Cambridge Studies in Advanced Mathematics 38, 
Cambridge University Press, 1994. 

\end{thebibliography}
\end{document}